\documentclass[11pt]{article} 
\usepackage[latin1]{inputenc}
\usepackage{amsmath}
\usepackage{mathtools}
\usepackage{amsfonts}
\usepackage{amssymb}
\usepackage{amsthm}
\usepackage{amsxtra}
\usepackage{enumerate}
\usepackage{authblk}

\usepackage{titlesec}
\titleformat{\subsection}[hang]{\normalfont\bfseries}{\thesubsection}{1em}{}

\titlespacing\section{0pt}{3.5ex plus 0.5ex minus .2ex}{0.3ex plus .2ex}
\titlespacing\subsection{0pt}{2.5ex plus 0.5ex minus .2ex}{0.3ex plus .2ex}
\titlespacing\subsubsection{0pt}{2.5ex plus 0.5ex minus .2ex}{0.3ex plus .2ex}

\usepackage{multind} 
\usepackage{xr-hyper} 
\usepackage{fullpage}
\usepackage{url} 
\usepackage[all]{xy}
\usepackage{comment}

\usepackage{color}  

\makeindex{notation-ax}

\usepackage{hyperref} 

\numberwithin{equation}{subsection}

\newtheorem{theorem}[equation]{Theorem}
\newtheorem*{theorem*}{Theorem}
\newtheorem{lemma}[equation]{Lemma}
\newtheorem*{lemma*}{Lemma}
\newtheorem{proposition}[equation]{Proposition}
\newtheorem*{proposition*}{Proposition}
\newtheorem{corollary}[equation]{Corollary}
\newtheorem*{corollary*}{Corollary}

\theoremstyle{definition}
\newtheorem{definition}[equation]{Definition}
\newtheorem*{definition*}{Definition}
\newtheorem{notation}[equation]{Notation}
\newtheorem*{notation*}{Notation}
\newtheorem{choice}[equation]{Choice}
\newtheorem*{choice*}{Choice}
\newtheorem{remark}[equation]{Remark}
\newtheorem*{remark*}{Remark}
\newtheorem{axiom}[equation]{Axiom}
\newtheorem*{axiom*}{Axiom}

\allowdisplaybreaks[1]

\newcommand{\ciso}[2]{{c_{#1,#2}}}

\newcommand{\abs}[1]{\left\lvert#1\right\rvert}

\renewcommand{\restriction}{|}

\DeclareMathOperator{\Irr}{Irr}
\DeclareMathOperator{\End}{End}

\DeclareMathOperator{\ind}{ind}

\DeclareMathOperator{\Hom}{Hom}

\DeclareMathOperator{\red}{red}
\DeclareMathOperator{\Mod}{Mod-}

\DeclareMathOperator{\id}{id}
\DeclareMathOperator{\ord}{ord}
\DeclareMathOperator{\supp}{supp}
\DeclareMathOperator{\aff}{aff}

\DeclareMathOperator{\normal}{norm}
\DeclareMathOperator{\terms}{terms}
\DeclareMathOperator{\GL}{GL}
\DeclareMathOperator{\SL}{SL}

\DeclareMathOperator{\res}{res}
\DeclareMathOperator{\trace}{tr}

\newcommand{\cpt}{{\mathrm{cpt}}}
\newcommand{\gen}{{\textup{gen}}}

\newcommand{\rel}{{\textup{rel}}}
\newcommand{\Krel}{{\mathcal{K}\textup{-rel}}}
\newcommand{\Kzrel}{{\mathcal{K}^0\textup{-rel}}}
\newcommand{\flength}{{\ell_{\Krel}}}
\newcommand{\ellsubstitute}{r}
\newcommand{\gpalg}{{b}}
\newcommand{\isoarrow}{\stackrel{\sim}{\longrightarrow}}

\newcommand{\Nheart}{N(\rho_{M})^{\heartsuit}_{[x_{0}]_{M}}}
\newcommand{\Nzeroheart}{N(\rho_{M^{0}})^{\heartsuit}_{[x_{0}]_{M^{0}}}}
\newcommand{\Wheart}{W(\rho_{M})^{\heartsuit}_{[x_{0}]_{M}}}
\newcommand{\Wzeroheart}{W(\rho_{M^{0}})^{\heartsuit}_{[x_{0}]_{M^{0}}}}
\newcommand{\muT}{\mu^{\cT}}
\newcommand{\muTzero}{\mu^{\cT^0}}
\newcommand{\Waff}{W(\rho_M)_{\mathrm{aff}}}
\newcommand{\Waffz}{W(\rho_{M^0})_{\mathrm{aff}}}
\newcommand{\Wzero}{\Omega(\rho_{M})}
\newcommand{\Wzeroz}{\Omega(\rho_{M^0})}
\newcommand{\sfG}{\mathsf{G}}   
\newcommand{\sfM}{\mathsf{M}}   
\newcommand{\Coeff}{\cC}   
\newcommand{\Coeffinvnontriv}{\Coeff^\times \smallsetminus \{1\}}
\newcommand{\Coeffplus}{\Coeff_{>1}}   
   
\newcommand{\Rclsd}{\cR} 
\newcommand{\Isom}{\cI} 
 
\DeclareMathOperator{\Rep}{Rep}
\newcommand{\IEC}{\mathfrak{I}} 

\newcommand{\bC}{\mathbb{C}}
\newcommand{\bQ}{\mathbb{Q}}
\newcommand{\bR}{\mathbb{R}}
\newcommand{\bT}{\mathbb{T}}
\newcommand{\bZ}{\mathbb{Z}}

\newcommand{\cA}{\mathcal{A}}
\newcommand{\cB}{\mathcal{B}}
\newcommand{\cC}{\mathcal{C}}

\newcommand{\cH}{\mathcal{H}}
\newcommand{\cI}{\mathcal{I}}
\newcommand{\cJ}{\mathcal{J}}
\newcommand{\cK}{\mathcal{K}}
\newcommand{\cO}{\mathcal{O}}

\newcommand{\cR}{\mathcal{R}}
\newcommand{\cT}{\mathcal{T}}
\newcommand{\cU}{\mathcal{U}}

\newcommand{\ff}{\mathfrak{f}}

\newcommand{\fs}{{\mathfrak{s}}}

\newcommand{\fH}{\mathfrak{H}}

\newcommand{\fS}{{\mathfrak{S}}}
\newcommand{\fSz}{{\mathfrak{S}_0}}

\newcommand{\spacingatend}[1]{}

\begin{document}
\externaldocument[HAIKY-][nocite]{Adler--Fintzen--Mishra--Ohara_Reduction_to_depth_zero_for_tame_p-adic_groups_via_Hecke_algebra_isomorphisms}

\author{
	Jeffrey D. Adler,
	Jessica Fintzen,
	Manish Mishra,
	and
	Kazuma Ohara
}
\AtEndDocument{\bigskip{\footnotesize%
	\par
       \textsc{%
	Department of Mathematics and Statistics,
	American University,
	4400 Massachusetts Ave NW,
	Washington, DC 20016-8050, USA} \par
       \textit{E-mail address}: \texttt{jadler@american.edu}
}}
\AtEndDocument{\bigskip{\footnotesize%
	\par
       \textsc{
	Universit\"at Bonn,
           Mathematisches Institut,
           Endenicher Allee 60,
           53115 Bonn,
           Germany } \par
       \textit{E-mail address}: \texttt{fintzen@math.uni-bonn.de}
}}
\AtEndDocument{\bigskip{\footnotesize%
	\par
	\textsc{%
	Department of Mathematics,
        IISER Pune,
        Dr.\ Homi Bhabha Road,
        Pune, Maharashtra 411008, India}
\par
       \textit{E-mail address}: \texttt{manish@iiserpune.ac.in}  
}}
\AtEndDocument{\bigskip{\footnotesize%
	\par
	\textsc{%
	Graduate School of Mathematical Sciences, The University of Tokyo, 
	3-8-1 Komaba, Meguroku, Tokyo 153-8914, Japan}
	\par
       \textit{E-mail address}: \texttt{kohara@ms.u-tokyo.ac.jp}
}}

\title{Structure of Hecke algebras arising from types}
\date{}

\maketitle
\begin{abstract}
Let $G$ denote a connected reductive group over a 
nonarchimedean local field $F$ of residue characteristic $p$,
and let $\mathcal{C}$ denote
an algebraically closed field of characteristic $\ell\neq p$.
If $\rho$ is an irreducible smooth $\mathcal{C}$-representation of a compact,
open subgroup $K$ of $G(F)$, then the pair $(K,\rho)$ gives rise
to a Hecke algebra $\mathcal{H}(G(F),(K, \rho))$.
For a large class of pairs $(K,\rho)$, we show that
$\mathcal{H}(G(F),(K, \rho))$ is a semi-direct product
of an affine Hecke algebra with explicit parameters with a twisted
group algebra,
and that it is isomorphic to $\mathcal{H}(G^0(F),(K^0, \rho^0))$ for some reductive subgroup $G^0 \subset G$ with compact, open subgroup $K^0$ and depth-zero representation $\rho^0$ of $K^0$.

The class of pairs that we consider includes all depth-zero types.
In describing their Hecke algebras,
we thus recover a result of Morris as a special case.
In a second paper, we will show that our class also contains all the types constructed by Kim and Yu, and hence we obtain as a corollary that arbitrary Bernstein blocks are equivalent to depth-zero Bernstein blocks under minor tameness assumptions.

The pairs to which our results apply are described in an axiomatic way so that the results can be applied to other constructions of types by only verifying that the relevant axioms are satisfied. The Hecke algebra isomorphisms are given in an explicit manner and are support preserving.
\end{abstract}

{
	\renewcommand{\thefootnote}{}  
	\footnotetext{MSC2020: Primary 22E50, 22E35, 20C08, 20C20. Secondary 22E35} 
	\footnotetext{Keywords: $p$-adic groups, smooth representations, Hecke algebras, types, Bernstein blocks, mod-$\ell$ coefficients}

\footnotetext{The first-named author
was partially supported by
the American University College of Arts and Sciences
Faculty Research Fund.}
	\footnotetext{The second-named author was partially supported by NSF Grants DMS-2055230 and DMS-2044643, a Royal Society University Research Fellowship, a Sloan Research Fellowship and the European Research Council (ERC) under the European Union's Horizon 2020 research and innovation programme (grant agreement no. 950326).}
\footnotetext{
The third-named author was partially supported by
a Fulbright-Nehru Academic and Professional Excellence Fellowship and
a SERB Core Research Grant (CRG/2022/000415).}
\footnotetext{
The fourth-named author is supported by the FMSP program at Graduate School of Mathematical Sciences, the University of Tokyo and JSPS KAKENHI Grant number JP22J22712.
}
}

\setcounter{tocdepth}{2}

\newpage
\tableofcontents

\section{Introduction} 
The category of all smooth, complex representations of a $p$-adic group $G$ decomposes
as a product of indecomposable
full subcategories, called Bernstein blocks, each of which is equivalent to modules over a Hecke algebra under minor tameness assumptions.
 Therefore knowing the explicit structure of these Hecke algebras and their modules yields an understanding of the category of smooth representations. The famous example of the Iwahori--Hecke algebra has already been described in the 1960s and the structure of the above Hecke algebras for $\GL_n$ have been known since the 1990s and all of them played an important role in the representation theory. However, comparatively little has been known about the structure of these Hecke algebras in the general setting above.
In this paper, we provide an explicit description of a large class of such Hecke algebras as a semi-direct product of an affine Hecke algebra with a twisted group algebra. The Hecke algebras that we treat
 are endomorphism-valued functions on the $p$-adic group that transform on an appropriate compact, open subgroup $K$ via a sufficiently nice irreducible representation $\rho$. By \cite{HAIKY}, these Hecke algebras include among others the prior-mentioned Hecke algebras of Bernstein blocks that exist under minor tameness assumptions.
Moreover, we obtain an isomorphisms between Hecke algebras attached to our general pairs $(K ,\rho)$ with Hecke algebras attached to much simpler pairs of reductive subgroups of $G$, e.g. pairs $(K^0, \rho^0)$ with $\rho^0$ a depth-zero representation. In the special case of Bernstein blocks, this provides an isomorphism between a Bernstein block of positive-depth and a depth-zero Bernstein block. The Hecke algebra of the latter was essentially already known by Morris \cite{Morris}, while the positive-depth Hecke algebras
for general $p$-adic groups
remained a mystery until now.
To be precise, Morris' work only provides the Hecke algebras attached to (potentially non-singleton) finite products of depth-zero Bernstein blocks, but we also obtain a description of the Hecke algebras for single Bernstein blocks in the present paper.

The general axiomatic set-up of this present paper allows one to also apply the results to, for example,
Hecke algebras attached to Bernstein blocks arising from other (including future) constructions of pairs $(K, \rho)$ that do not rely on the above minor tameness assumptions. Moreover, we allow arbitrary algebraically closed fields of characteristic different from $p$ as our coefficents.

\subsection{Overview of the main results}
To explain our results in more detail,
let $F$ denote a non-archimedean local field,
$G$ a connected reductive group over $F$,
and $\Coeff$ an algebraically closed field of characteristic different from $p$.

To a pair consisting of a compact, open subgroup $K$ of $G(F)$ and a smooth, irreducible $\Coeff$-representation $(\rho, V_\rho)$ of $K$, we attach the Hecke algebra $\cH(G(F), (K, \rho))$ of compactly supported functions $\varphi: G(F) \rightarrow \End_\Coeff(V_\rho)$ that satisfy $\varphi(k_1gk_2)=\rho(k_1)\circ\varphi(g)\circ\rho(k_2)$ for all $k_1, k_2 \in K$ and $g \in G(F)$. The algebra structure arises from the convolution recalled in Section \ref{Hecke algebras and endomorphism algebras}. These algebras are closely related to smooth representations of $G(F)$ as follows.
The category of smooth $\Coeff$-representations decomposes into a product of indecomposable full subcategories.
In the case that $\Coeff=\bC$,
one calls these subcategories \emph{Bernstein blocks},
and under additional mild tameness assumptions,
this decomposition takes the form
\[
\Rep(G(F))=\prod_{(K, \rho)}\Rep^{(K, \rho)}(G(F)),
\]
where the product is taken over appropriate pairs $(K, \rho)$ 
 such that each block $\Rep^{(K, \rho)}(G(F))$ is equivalent
(via an explicit equivalence recalled in Section \ref{subsec:application})
to the category of unital right modules over $\cH(G(F), (K, \rho))$:
\[
\Rep^{(K, \rho)}(G(F)) \simeq \Mod \cH(G(F), (K, \rho))
\]
(\cite{MR771671, BK-types, Yu, Kim-Yu, MR4357723, Fi-exhaustion}). The pairs $(K, \rho)$ in the above Bernstein decomposition were
constructed by Kim and Yu (\cite{Kim-Yu, MR4357723}) and are a special case of what Bushnell and Kutzko (\cite{BK-types}) called \emph{$\fs$-types}.\footnote{For the experts, in this introduction, we use ``$\fs$-type'' to mean ``a type corresponding to a single Bernstein block'', the $\fs$ does not refer to a specific block.}

The pairs $(K, \rho)$ that we consider in this paper are described in terms of four axioms, Axioms \ref{axiomaboutHNheartandK}, \ref{axiombijectionofdoublecoset}, \ref{axiomexistenceofRgrp}, and \ref{axiomaboutdimensionofend}, that distill the key properties of the $\fs$-types constructed by Kim and Yu that are used to prove our main results below about the structure of the attached Hecke algebra $\cH(G(F), (K, \rho))$. Working with general pairs that satisfy only these four axioms means that our results will be applicable in a much broader setting than just for the $\fs$-types constructed by Kim and Yu. This is expected to include among others future constructions of $\fs$-types in the non-tame setting. 

We now fix a pair $(K, \rho)$ that satisfies our axioms. The first main result, Theorem \ref{theoremstructureofhecke},
describes the structure of $\cH(G(F), (K, \rho))$ via an explicit isomorphism 
\begin{equation}
\label{eq:intro-main1}
\cH(G(F), (K, \rho)) \isoarrow \bC[\Wzero, \muT] \ltimes \cH(\Waff, q),
\end{equation}
where $\Waff$ is an affine Weyl group
that is a normal subgroup of a larger symmetry group $\Wzero \ltimes \Waff$,
$q \colon S \longrightarrow \bC^\times$
is a parameter function on a set $S$ of simple reflections generating $\Waff$,
$\cH(\Waff, q)$ denotes the corresponding affine Hecke algebra,
and
$\bC[\Wzero, \muT]$ denotes the group algebra of $\Wzero$
twisted by a $2$-cocycle $\muT$.
Both of these latter algebras, and the meaning of their semi-direct
product, are recalled in Notation \ref{notn:algebras}.

The second main result (see Theorems \ref{thm:isomorphismtodepthzero} and \ref{thm:explicitisom} and Corollary \ref{cor:starpreservation}) is the existence of a full-rank reductive subgroup $G^0 \subseteq G$ and an irreducible representation $\rho^0$ of $K^0 \coloneqq  K \cap G^0(F)$  with $\rho = \rho^0 \otimes \kappa$ for some smooth representation $\kappa$, where we view $\rho^0$ also as a representation of $K$ via an inflation map described in Axiom \ref{axiomaboutK0vsK}\eqref{axiomaboutK0vsKinflation}, 
such that we obtain an explicit, support-preserving 
isomorphism of Hecke algebras
\begin{equation}
\label{eq:intro-main2}
\cH(G^0(F),(K^0, \rho^0)) \stackrel{\sim}{\longrightarrow} \cH(G(F),(K, \rho)).
\end{equation}
In the case where the pairs $(K,\rho)$ and $(K^0,\rho^0)$
are both $\fs$-types,
this leads to an equivalence of categories
(Theorem \ref{thm:equiv-blocks}):
\begin{equation}
\label{eq:equiv-of-categories}
\Rep^{(K,\rho)}(G(F)) \isoarrow \Rep^{(K^0,\rho^0)}(G^0(F)).
\end{equation}
The possibilities for the triple $(G^0, K^0, \rho^0)$ are described in an axiomatic way, via Axioms \ref{axiomaboutKM0vsKM}, \ref{axiomaboutK0vsK}, and \ref{axiomextensionoftheinductionofkappa}.
In the case of the $\fs$-types constructed by Kim and Yu, one can always choose a triple  $(G^0, K^0, \rho^0)$ with $\rho^0$ of depth zero. In fact, in this case, one option for the triple is already part of the input of the construction by Kim and Yu, twisted by a quadratic character introduced by \cite{FKS}. Since under minor tameness assumptions every Bernstein block has an $\fs$-type of the form constructed by Kim and Yu, this allows one to reduce a plethora of problems about representations of $p$-adic groups to the depth-zero setting.

While we have chosen an axiomatic treatment distilling the key properties that are really necessary for our proofs so that the results can be applied to a variety of different constructions of types, including future ones, and different settings, including coefficients of positive characteristic, we also show in the short Section \ref{sec:depth-zero} that all depth-zero $\fs$-types as well as the depth-zero types corresponding to connected parahoric subgroups satisfy the axioms. 
Thus our first main result in that case recovers and generalizes a result of Morris (\cite{Morris}). The proof that all $\fs$-types constructed by Kim and Yu satisfy our axioms is deferred to \cite{HAIKY} as it involves the extension of the quadratic twist introduced by \cite{FKS} to a group of representatives for the whole support of the Hecke algebra and a careful analysis of the Heisenberg--Weil representations. This uses a different style of arguments and is of independent interest.

\subsection{Some prior and related work} 
\label{subsec:history}
While our works are the first to provide an explicit description of Hecke algebras attached to such a broad class of generalizations of types of arbitrary reductive groups, which include in particular types for all Bernstein blocks if $\Coeff=\bC$, $G$ splits over a tamely ramified extension and $p$ does not divide the order of the absolute Weyl group of $G$, mathematicians have studied and used such Hecke algebras since the 1960s and obtained results about their structure in special cases.

A famous example of the Hecke algebras that we consider is the Iwahori--Hecke algebra attached to the Iwahori subgroup and a trivial representation thereof. The Iwahori--Hecke algebra has already been described in 1965 by Iwahori and Matsumoto for adjoint, split semisimple groups (\cite{Iwahori-Matsumoto}) and their work has been fundamental for future developments in the area.

If $G$ is split and semisimple,
and one replaces
the trivial character of the Iwahori subgroup $I$ by a depth-zero character $\chi$, then
Goldstein \cite{goldstein:thesis}
computed the Hecke algebras attached to $(I, \chi)$,
showing that such an algebra is isomorphic to the Iwahori--Hecke
algebra of a smaller group.
In a vast generalization, Morris (\cite{Morris}) described in 1993 all the Hecke algebras attached to pairs $(K, \rho)$ where $K$ is a parahoric subgroup of $G(F)$ and $\rho$ is an irreducible cuspidal representation of the quotient of $K$ by its pro-$p$ unipotent radical, i.e., in other words, $(K, \rho)$ is a depth-zero type for a finite union of Bernstein blocks by the work of Moy and Prasad (\cite{MR1253198, MR1371680}) and independently of Morris (\cite{MR1713308}). 

Knowing the Hecke algebras of types of Bernstein blocks allows one to study the representations in the Bernstein block via modules of the corresponding Hecke algebra. An example of this approach is Lusztig's famous work (\cite{Lusztig95unipotent, Lusztig02unipotentII}) in which he classified all unipotent representations of adjoint simple algebraic groups that split over an unramified extension by classifying the representations of the corresponding Hecke algebras using their explicit structure. Unipotent representations are a special class of depth-zero representations, and Lusztig's work provides an important case of an explicit local Langlands correspondence. The restriction to adjoint groups was later removed by Solleveld (\cite{Solleveld23unipotentLLC}), again using the explicit structure of the affine Hecke algebras attached to types of Bernstein blocks consisting of unipotent representations.

Beyond depth-zero representations, building up on the work of a lot of mathematicians on special cases over several decades, Bushnell and Kutzko (\cite{MR1204652}) provided a description of all the Hecke algebras attached to types for all Bernstein blocks for the group  $\GL_n$ in parallel to constructing these types. Their description of the Hecke algebras played an important role for the construction of the types themselves as well as for proving the exhaustiveness of their construction of supercuspidal representation.
Using results
from Bushnell and Kutzko
on supercuspidal representations of $\SL(n)$
(\cite{bushnell-kutzko:sln-1,bushnell-kutzko:sln-2}),
Goldberg and Roche 
(\cite{MR1901371,GoldbergRoche-Hecke})
provided a complete collection of types
for all Bernstein blocks of $\SL(n)$
and described the structure of the resulting Hecke algebras.
S\'echerre and Stevens (\cite{secherre-stevens:glnd-4})
achieved the same for all inner forms of $\GL(n)$, and
Miyauchi and Stevens
(\cite{MR3157998})
provided types for all classical groups assuming that $p\neq 2$ and 
described the  Hecke algebras corresponding to types associated to maximal proper Levi subgroups.
For general split reductive groups, Roche (\cite{MR1621409})
described types and the corresponding Hecke algebras for all principal series Bernstein
blocks under mild hypotheses on $p$.

For reductive groups $G$ that split over a tamely ramified field extension, Kim and Yu (\cite{Kim-Yu}) provided a construction of types based on Yu's construction of supercuspidal representations (\cite{Yu, MR4357723}) that provides types for all Bernstein blocks if $p$ does not divide the order of the absolute Weyl group of $G$ (\cite{Fi-exhaustion}). These types are a special case of the setting in which the results of the present paper apply, as we show in \cite{HAIKY}. 

While results about the structure of Hecke algebras, analogous to our result \eqref{eq:intro-main1}, have  been achieved previously in many different families, as described above, results analogous to \eqref{eq:intro-main2} have to our knowledge previously only been obtained in very special cases and for $\GL_n$ (\cite{MR1204652}), principal series Bernstein blocks of split reductive groups (\cite{MR1621409}) and supercuspidal Bernstein blocks (\cite{2021arXiv210101873O}).

\subsection{Sketch of the proofs of the main statements}
 We define the two isomorphisms \eqref{eq:intro-main1} and \eqref{eq:intro-main2} explicitly by defining explicit basis elements of the Hecke algebras $\cH(G(F), (K, \rho))$ and $\cH(G^0(F), (K^0, \rho^0))$ that will get mapped to each other and to the corresponding basis elements of $\Coeff[\Wzero, \muT] \ltimes \cH_{\Coeff}(\Waff, q)$ under the isomorphisms. The difficult task consists of doing this in a way that preserves the algebra structures.
	 
	 To provide a few more details, let us note that the pairs $(K, \rho)$ and $(K^0, \rho^0)$ come with Levi subgroups $M \subseteq G$ and $M^0 \subseteq G^0$ satisfying $M^0 \subseteq M$, which in the setting of $\fs$-types record Levi subgroups appearing in the supercuspidal supports of the corresponding Bernstein blocks,
	 and with a point $x_0$ in the Bruhat--Tits building of $G$ that is contained in the image of the Bruhat--Tits building of $M^0$. We also recall that $K^0=K \cap G^0(F)$ and $\rho=\rho^0 \otimes \kappa$.

	The bases for the Hecke algebras are indexed by $K$- and $K^0$-double cosets, respectively, that satisfy appropriate intertwining properties, and we show that these indexing sets agree and are isomorphic to a quotient $\Wzeroheart$ of a subgroup $\Nzeroheart$
		 of the normalizer $N_{G^0}(M^0)(F)$ of $M^0$ in $G^0$ that preserves the image of $x_0$ in the reduced Bruhat--Tits building of $M^0$, see Proposition \ref{propositiondoublecosetinjection}. This set-theoretic isomorphism endows the indexing set with a group structure, which we prove is isomorphic to a semi-direct product $\Wzero \ltimes \Waff$ whose normal factor is an affine Weyl group, see Propositions \ref{propositiondecompositionofW} and \ref{Rrhosimeqaffineweyl}.

	The basis elements of the Hecke algebras $\cH(G(F), (K, \rho))$ and $\cH(G^0(F), (K^0, \rho^0))$ are then defined by reinterpreting the two Hecke algebras as
	 $\End_{G(F)} \bigl(\ind_{K}^{G(F)} \rho\bigr)$ and
	 $\End_{G^0(F)} \bigl(\ind_{K^0}^{G(F)} \rho^0 \bigr)$,
	 and composing two different intertwining operators.	
	 In order to define the (spaces for the) intertwining operators, we introduce a whole family of pairs $(K_{x}, \rho_{x}=\rho_x^0 \otimes \kappa_x)$ and $(K^0_{x}, \rho^0_{x})$ for $x$ in an open, dense subset $\cA_{\gen}$ of an appropriate affine subspace $\cA_{x_{0}}$ of the Bruhat--Tits building of $G$ on which the group $\Wzero \ltimes \Waff$ acts.
	 The affine space underlying the affine Weyl group $\Waff$ is a quotient of $\cA_{x_{0}}$, see Proposition \ref{Rrhosimeqaffineweyl}. 
	 The families are set up in a way such that $(K, \rho) = (K_{x_{0}}, \rho_{x_{0}})$ and $(K^0, \rho^0) = (K^0_{x_{0}}, \rho^0_{x_{0}})$.  
	 We define the basis element $\Phi_w \in \End_{G(F)} \bigl(\ind_{K}^{G(F)} \rho\bigr) \simeq \cH(G(F), (K, \rho))$ attached to $w \in  \Wzero \ltimes \Waff$ to be the composition of the two intertwining operators
\begin{equation}
\label{eq:two-intertwining-ops}
	 \Theta^{\normal}_{w^{-1} x_0 \mid x_0} \colon \ind_{K_{x_0}}^{G(F)} (\rho_{x_0}) \rightarrow \ind_{K_{w^{-1} x_0}}^{G(F)} (\rho_{w^{-1} x_0})
	 \quad
	 \text{and}
	 \quad 
	 \ciso{w^{-1} x_0}{w} \colon \ind_{K_{w^{-1} x_0}}^{G(F)} (\rho_{w^{-1} x_0}) \rightarrow \ind_{K_{x_0}}^{G(F)} (\rho_{x_0}),
\end{equation}
	where the first intertwining operator, $\Theta^{\normal}_{w^{-1} x_0 \mid x_0} $, is constructed via a normalized integration, see Section \ref{subsec:intertwiningop}, in particular Lemma \ref{lemma:intertwining-operator} and Definition \ref{definitionnormalize}, and the second intertwining operator, $\ciso{w^{-1} x_0}{w}$, results from choosing an element $T_n$ in the one-dimensional $\Coeff$-vector space  $\Hom_{K_{nx_0}}({^n\rho_{x_0}},\rho_{nx_0})$ for $n$ a lift of $w$ in $\Nzeroheart$ and $^n\rho_{x_0}$ the $n$-conjugate of $\rho_{x_0}=\rho$, see Definition \ref{definiotionofciso} and \eqref{eq:defofciso} for details. The analogous definitions are made for $G^0, K^0, \rho^0$, adding a superscript ``0'' as appropriate.
	 
	In order to show that the basis elements defined for an appropriate choice of  $\{ T_n \}$ and $\{ T^0_n \}$ lead to the desired algebra isomorphisms \eqref{eq:intro-main1} and \eqref{eq:intro-main2}, we equip the indexing group $\Wzero \ltimes \Waff$ with a length function $\flength$ that is trivial on $\Wzero$ and is the standard length function on the affine Weyl group $\Waff$ arising from a choice of generators $S$. We then have to check that the isomorphisms \eqref{eq:intro-main1} and \eqref{eq:intro-main2} preserve the multiplication of two basis elements $\Phi_w$ and $\Phi_{w'}$ in the case where $\flength(ww')=\flength(w)+\flength(w')$, which we call the \emph{length-additive case}, and the case of multiplying $\Phi_s$ with itself for $s \in S \subset \Waff$, which we refer to as the \emph{quadratic relations}.
	 To achieve both cases, we analyse the two intertwining operators separately as well as their interaction. 
	
	To prove the required properties for the first intertwining operator we introduce a more general operator
	\(
    \Theta_{y \mid x}
	\colon
	\ind_{K_{x}}^{G(F)} (\rho_{x})
	\longrightarrow
	\ind_{K_{y}}^{G(F)} (\rho_{y})
	\)
	that is defined for all $x, y \in \cA_{\gen}$ and prove a variety of compatibility properties of these operators, 
	 see, for example, Lemma \ref{lemma:intertwining-operator}, Proposition \ref{transitivityofthetadver}, and Lemma \ref{lemmaaboutreplacingpoints} for more detailed statements.
	Moreover, a key step for proving the isomorphism \eqref{eq:intro-main2} consists of relating the operator $\Theta_{y \mid x}$ defined for $G$ directly with the corresponding operator $\Theta^{0}_{y \mid x}$ defined for $G^0$, which we achieve in Lemma \ref{lemmacompativilityofThetadepth0vspositivedepth}. This result relies on a compatibility of the representations $\kappa_x$ (recall: $\rho_{x}=\rho_x^0 \otimes \kappa_x$) for varying nearby points $x$, which is formalized in Axiom \ref{axiomextensionoftheinductionofkappa}\eqref{axiomextensionoftheinductionofkappacompatibilitywiththecompactind}.  In the setting of types constructed by Kim and Yu, this means a compatibility of various Heisenberg--Weil representations, which is proven in \cite[Corollary \ref{HAIKY-corollaryofpropositioninductiontwistedsequencever}]{HAIKY} and requires one to twist the initial construction of Kim and Yu by a quadratic character introduced in \cite{FKS}.
	
	The second family of intertwining operators $\ciso{w^{-1} x_0}{w}$ depends on the choice of elements $\{ T_n \}$, which are only determined up to scalars and which lead to a cocycle $\muT$ on $\Wzero \ltimes \Waff$, see Notation \ref{notationofthetwococycle}. We prove that we can choose the elements $\{ T_n \}$ such that the only part of the resulting cocycle $\muT$ that matters is its restriction to $\Wzero$ (see Proposition~\ref{prop:muTtrivial}). This is the cocycle that appears in the twisted group algebra $\Coeff[\Wzero, \muT]$. Moreover, we show that the choices $\{ T_n \}$ for $G$ and the choices $\{ T^0_n \}$ for $G^0$ can be made in a compatible way that leads to the same cocycles and therefore to the desired Hecke algebra isomorphism \eqref{eq:intro-main2}. This step uses an extension of (the restriction of) $\kappa$ to $\Nheart$, which is provided by Axiom \ref{axiomaboutKM0vsKM}\eqref{axiomaboutKM0vsKMaboutanextension}. 
	 Checking this axiom in the setting of the twisted construction of Kim and Yu is one of the key results of \cite{HAIKY}, see \cite[Proposition~\ref{HAIKY-propproofofaxiomaboutKM0vsKM}]{HAIKY}, which is based on \cite[Theorem \ref{HAIKY-twistextension}]{HAIKY}. Verifying the quadratic relations also requires again the compatibility of the $\kappa_x$ for nearby points $x$ mentioned above and forms one of the key steps of the present paper.
	
	We note that we said that the pairs $(K, \rho)$ and $(K^0, \rho^0)$ need to each satisfy a few axioms,  Axioms \ref{axiomaboutHNheartandK}, \ref{axiombijectionofdoublecoset}, \ref{axiomexistenceofRgrp}, and \ref{axiomaboutdimensionofend}, in order for us to conclude \eqref{eq:intro-main1}, and a few additional axioms describing the compatibility between the two pairs, Axioms \ref{axiomaboutKM0vsKM}, \ref{axiomaboutK0vsK}, and \ref{axiomextensionoftheinductionofkappa}, in order to achieve \eqref{eq:intro-main2}. 
	In fact, in order to obtain \eqref{eq:intro-main2} we require $(K, \rho)$ to only satisfy Axiom \ref{axiomaboutHNheartandK} and we prove in Section \ref{sec:comparison} that the remaining axioms are automatically satisfied by knowing them for $(K^0, \rho^0)$. In practice, checking the axioms for $(K^0, \rho^0)$ might be significantly easier than for $(K, \rho)$.
	As an example, we show in the rather short Section \ref{sec:depth-zero} that depth-zero $\fs$-types $(K^0, \rho^0)$ satisfy the Axioms \ref{axiomaboutHNheartandK}, \ref{axiombijectionofdoublecoset}, \ref{axiomexistenceofRgrp}, and \ref{axiomaboutdimensionofend}. In \cite{HAIKY}, we then apply our results from  Section \ref{sec:comparison} to deduce that also all the positive-depth $\fs$-types $(K, \rho)$ constructed by Kim and Yu (\cite{Kim-Yu, MR4357723}) satisfy the Axioms \ref{axiomaboutHNheartandK}, \ref{axiombijectionofdoublecoset}, \ref{axiomexistenceofRgrp}, and \ref{axiomaboutdimensionofend}

Our approach to proving the isomorphism \eqref{eq:intro-main1} was heavily inspired by the work of Morris (\cite{Morris}) who dealt with the case where $(K^0, \rho^0)$ is a complex depth-zero type with $K^0$ being a parahoric subgroup. In particular, we use similar intertwining operators than Morris.
However, to prove that the resulting basis elements of the Hecke algebras satisfy the desired relations, Morris defines a version of the first intertwining operator $\Theta_{w^{-1}x_0 \mid x_0}$ for $w$ contained in a larger group than the indexing group $\Wzero \ltimes \Waff$. Our approach in contrast consists of defining the first intertwining operator not for a larger group, but, instead, of defining intertwining operators $\Theta_{y \mid x}$ for all $x, y \in \cA_{\gen}$, i.e., also for points $x$ and $y$ that do not lie in the $\Wzero \ltimes \Waff$-orbit of $x_{0}$. Although we only need the intertwining operator $\Theta_{w^{-1} x_{0} \mid x_{0}}$ to define the basis element $\Phi_w$ for $w \in  \Wzero \ltimes \Waff$, these general intertwining operators $\Phi_{y \mid x}$ allow us to prove the desired properties about $\Phi_{w}$ in Sections~\ref{subsec:intertwiningop} and \ref{sec:simplereflections}. We believe that this approach makes our arguments look cleaner, and at the same time allows us to treat a more general setting than the one that Morris dealt with. And it provides us with the set-up in which we can prove our second main result, \eqref{eq:intro-main2}.

\subsection{Structure of the paper}

We have tried to state all assumptions that apply in each subsection at the beginning of the subsection, so that it is easy to see what assumptions are in place where. Moreover, in the main results and in propositions that get used later within the paper, we have repeated the assumptions that are in place to provide clarity for the reader and allow for easier backtracking through the results and to make it easier  to follow our proofs. We have also tried to state all the axioms as early as possible in each subsection so that a reader interested in simply knowing what axioms they need to verify to apply the results can reach the statements of the axioms as quickly as possible. A list of axioms is provided on page \pageref{page:listofaxioms}.
We have also created a list of notation (page \pageref{page:listofnotation}).

We give a brief overview of the following sections of this paper.

In \S\ref{subsec:notation} we fix some notation, 
and in \S\ref{Hecke algebras and endomorphism algebras}, we recall the definition of Hecke algebras as a convolution algebra of functions and its relation with the second interpretation as endomorphism algebra, because we will use both points of view.

In Section~\ref{Structure of a Hecke algebra}, we prove the structure theorem \eqref{eq:intro-main1} for the Hecke algebra associated to a pair $(K, \rho)$ that satisfies Axioms~\ref{axiomaboutHNheartandK}, \ref{axiombijectionofdoublecoset}, \ref{axiomexistenceofRgrp}, and \ref{axiomaboutdimensionofend} defined there.
For this, we first construct an explicit vector-space basis $\{\varphi_{w}\}$ of $\cH(G(F), (K, \rho))$ indexed by the group $\Wheart$ under the assumptions of Axioms~\ref{axiomaboutHNheartandK} and \ref{axiombijectionofdoublecoset} 
as follows.
In \S\ref{subsec:familyofcovers}, we introduce the family $(K_{x}, \rho_{x})$ mentioned above indexed by 
 $\cA_{\gen}$. 
In \S\ref{subsec:intertwiningop}, for $w \in \Wheart$, we define the intertwining operators
$\Theta^{\normal}_{w^{-1} x_0 \mid x_0}$
and
$\ciso{w^{-1} x_0}{w}$
(see \eqref{eq:two-intertwining-ops})
and
define the element $\Phi_{w}$ of
$\End_{G(F)} \bigl(
\ind_{K_{x_0}}^{G(F)} (\rho_{x_0})
\bigr)$
as their composition.
We then define the basis element $\varphi_{w} \in \cH(G(F), (K, \rho))$ to be the element corresponding  to $\Phi_{w}$ via the isomorphism
$\cH(G(F), (K, \rho)) \simeq \End_{G(F)} \bigl(
\ind_{K_{x_0}}^{G(F)} (\rho_{x_0})
\bigr)$
given in \S\ref{Hecke algebras and endomorphism algebras}.
Based on discussions about the intertwining operators in \S\ref{subsec:intertwiningop}, we investigate the relations involving the elements $\varphi_{w}$ in the length-additive case.
In \S\ref{subsection:indexinggroup}, we prove that the indexing group $\Wheart$ is isomorphic to the semi-direct product $\Wzero \ltimes \Waff$ mentioned above under the assumption of Axiom~\ref{axiomexistenceofRgrp}.
In \S\ref{sec:simplereflections}, we introduce Axiom~\ref{axiomaboutdimensionofend}, and under this axiom, we prove quadratic relations for the elements $\varphi_{s}$ corresponding to simple reflections $s \in \Waff$.
When the characteristic $\ell$ of $\Coeff$ is zero or a banal prime,
the coefficients of the quadratic relations are explicitly calculated in \S\ref{subsec:quadratic-banal}.
In \S\ref{subsection:The description of the Hecke algebra}, we prove the structure theorem \eqref{eq:intro-main1} by proving appropriate braid relations and combining them with the discussions from previous subsections.
In \S\ref{Anti-involution of the Hecke algebra}, we assume that the coefficient field $\Coeff$ admits a nontrivial involution.
Under this assumption, we prove that we can take the isomorphism in \eqref{eq:intro-main1} such that it preserves anti-involutions on both sides.
In \S\ref{subsec:isononuniqueness}, we classify the support-preserving isomorphisms in \eqref{eq:intro-main1}.

In Section~\ref{sec:comparison}, we prove the Hecke algebra isomorphism~\eqref{eq:intro-main2} for a pair $(K,\rho)$ consisting of a compact, open subgroup $K$ of $G(F)$ and an irreducible representation $\rho$ of $K$ and a similar pair $(K^0, \rho^0)$ for a reductive subgroup $G^0$ of $G$ that are related according to Axioms
\ref{axiomaboutKM0vsKM},
\ref{axiomaboutK0vsK},
and
\ref{axiomextensionoftheinductionofkappa}.
We assume Axioms~\ref{axiomaboutHNheartandK}, \ref{axiombijectionofdoublecoset}, \ref{axiomexistenceofRgrp}, and \ref{axiomaboutdimensionofend} for the pair $(K^0, \rho^0)$ but only assume Axiom~\ref{axiomaboutHNheartandK} for the pair $(K, \rho)$.
We prove in \S\ref{subsec:Heckeisom} that the other axioms for $(K, \rho)$ automatically follow in this case.
Thus, by applying the results in Section~\ref{Structure of a Hecke algebra} to the pairs $(K^0, \rho^0)$ and $(K, \rho)$, we obtain the descriptions of the Hecke algebras as semi-direct products of affine Hecke algebras with twisted group algebras.
The argument in \S\ref{subsec:relevance} implies that the affine Weyl groups and the subgroups of length-zero elements for $(K^0, \rho^0)$ and $(K, \rho)$ agree.
In Section~\ref{subsec:Heckeisom},
we prove the Hecke algebra isomorphism~\eqref{eq:intro-main2}.
We have already shown the each of $\cH(G(F), (K,\rho))$
and $\cH(G^0(F), (K^0, \rho^0))$
is isomorphic to a semi-direct product of an affine Hecke algebra
(where the underlying affine reflection groups are the same)
with a twisted group algebra (where the underlying groups are the same).
We complete the argument by showing that,
if certain choices are made carefully,
then 
the parameters of the two affine Hecke algebras match up,
as well as the two cocycles that implement the twisting of the two group algebras.
In Section~\ref{subsection:Preserving the anti-involutions}, under the assumption that $\Coeff$ admits a nontrivial involution, we prove that we can take the isomorphism in \eqref{eq:intro-main2} such that it preserves anti-involutions on both sides.
As an application, in Section~\ref{subsec:application}, we prove the equivalence of Bernstein blocks in \eqref{eq:equiv-of-categories}.
Moreover, we also prove that when restricted to irreducible objects,
the equivalence of Bernstein blocks
preserves temperedness, and
preserves the Plancherel measure on the tempered dual up to an explicit constant
factor (see Theorem \ref{thm:plancherel}).

In Section \ref{sec:depth-zero},
we determine the structure of all Hecke algebras arising from depth-zero
pairs $(K,\rho)$.
We review the construction of such pairs in \S\ref{subsec:constructionofdepthzero},
and describe the system of hyperplanes that applies to this case
in \S\ref{subsec:affinehyperplanes-depthzero}.
In \S\ref{subsec:structureofheckefordepthzerotypes},
we obtain the structure of the Hecke algebras
(Theorem \ref{theoremstructureofheckefordepthzero})
by showing 
that depth-zero pairs satisfy those axioms from
Section~\ref{Structure of a Hecke algebra}
necessary to allow us to apply 
Theorem \ref{theoremstructureofhecke}.

\subsection{Guidance for the reader} 
If a reader is interested in our results only in the case of types as constructed by Kim and Yu, we suggest the following approach to the paper:
Start with
\cite[Section \ref{HAIKY-Section-KimYutypes}]{HAIKY}
and read until \cite[Section~\ref{HAIKY-subsec:affinehyperplanesKimYu}]{HAIKY} to learn the definitions of the objects $G$, $M$, $x_0$, $\mathfrak{H}$, $K_M$, $\rho_M$, $K_{x}$, $K_{x, +}$, $\rho_{x}$ and $G^0$, $M^0$, $K_{M^0}$, $\rho_{M^0}$, $K^{0}_{x}$, $K^{0}_{x, +}$, $\rho^{0}_{x}$.
Afterwards, we encourage the reader to read Section \ref{Structure of a Hecke algebra} and \ref{sec:comparison} of the present paper but to replace in their mind the abstract objects $G$, $M$, $x_0$, $\mathfrak{H}$, $K_M$, $\rho_M$, $K_{x}$, $K_{x, +}$, $\rho_{x}$ and $G^0$, $M^0$, $K_{M^0}$, $\rho_{M^0}$, $K^{0}_{x}$, $K^{0}_{x, +}$, $\rho^{0}_{x}$
with the explicit objects introduced in
\cite[Section \ref{HAIKY-Section-KimYutypes}]{HAIKY}.
We have provided a lot of cross-references between
\cite[Section \ref{HAIKY-Section-KimYutypes}]{HAIKY}
and Section \ref{Structure of a Hecke algebra} and \ref{sec:comparison} of this paper to ease this approach.

For a reader who is interested in our general results and who is already familiar with either the construction of Kim and Yu or depth-zero types, we have added after the introduction of each abstract object an explanation or reference regarding what these objects are in these specific settings, and we provide references to where the axioms are proven in these settings. Such a reader might also benefit from skimming or reading \cite[Section \ref{HAIKY-Section-KimYutypes}]{HAIKY} until \cite[Section~\ref{HAIKY-subsec:affinehyperplanesKimYu}]{HAIKY} (for the types as constructed by Kim and Yu) and / or Section \ref{sec:depth-zero} until Section \ref{subsec:affinehyperplanes-depthzero} (for the depth-zero setting) in parallel to reading Section \ref{Structure of a Hecke algebra} and \ref{sec:comparison} of the present paper.

For a reader using the axiomatic set-up to prove similar results in other settings,  we have provided a complete list of axioms on page \pageref{page:listofaxioms} with references to where the axioms can be found. Moreover, we have tried to state the axioms as early as possible in their subsections to allow the reader to receive the desired information, i.e., the axioms to check, as quickly as possible.

\subsection*{Acknowledgments}
Various subsets of the authors benefited from the hospitality
of American University, Duke University, the University of Bonn, 
the  Hausdorff Research Institute for Mathematics in Bonn,
the Max Planck Institute for Mathematics in Bonn, the Indian Institute for Science Education and Research (Pune), and the University of Michigan.
The authors also thank Alan Roche, 
Dan Ciubotaru, and
Tasho Kaletha, for discussions related to this paper. 
This project was started independently by three different subgroups of the authors, and the fourth-named author thanks his supervisor Noriyuki Abe for his enormous support and helpful advice.
The fourth-named author is also grateful to Tasho Kaletha for the discussions during his stay at the University of Michigan in 2023.

\section{Preliminaries}
\subsection{Notation} 
\label{subsec:notation}

Let $F$ be a non-archimedean local field endowed with a discrete valuation $\ord$ on $F^{\times}$ with the value group $\mathbb{Z}$. 
For a finite extension $E$ of $F$,
we also write $\ord$ for the unique extension of this valuation to $E^{\times}$.
We write $\cO_{E}$ for the ring of integers of $E$,
and let $\cO = \cO_F$.
Let $p$ denote the characteristic of the residue field $\ff$ of $F$.

Let $\bR$ and $\bC$ denote the fields of real and complex numbers,
respectively. 
Let $\Coeff$
\index{notation-ax}{C@$\Coeff$ (coefficient field)}
denote an algebraically closed field of characteristic
\index{notation-ax}{l@$\ell$ (characteristic of $\Coeff$)}
$\ell\neq p$.
Except when otherwise indicated,
all representations below are on vector spaces over $\Coeff$.
We fix a square root
\index{notation-ax}{000@$1/2$ in a superscript}%
$p^{1/2}$ of $p$ in $\Coeff$.
For an element $c = p^{n} \in \Coeff$ with $n \in \bZ$,
we write $c^{1/2} \coloneqq \left(p^{1/2}\right)^{n}$.

For a torus $S$ that is defined and splits over $F$, let $X^{*}(S)$ denote the character group of $S$ and $X_{*}(S)$ denote the cocharacter group of $S$.

For a connected reductive group $G$ defined over $F$, we write $Z(G)$ for the center of $G$ and $A_{G}$ for the maximal split torus in $Z(G)$.
For a reductive subgroup $M$ of $G$, let $N_{G}(M)$ (resp.\ $Z_{G}(M)$) denote the normalizer (resp.\ centralizer) of $M$ in $G$.

Let $\cB(G, F)$ denote the enlarged Bruhat--Tits building of $G$, and for a maximal split torus $S$ of $G$, let $\cA(G, S, F)$ denote the apartment of $S$ in $\cB(G, F)$.
We also write $\cB^{\red}(G, F)$ for the reduced building of $G$ and $\cA^{\red}(G, S, F)$ for the apartment of $S$ in $\cB^{\red}(G, F)$.
For $x \in \cB(G, F)$, let $[x]_{G}$ denote the image of $x$ in $\cB^{\red}(G, F)$.
We might also write  $[x]$ instead of  $[x]_{G}$
if the group $G$ is clear from the context.
We write $G(F)_{x}$
and $G(F)_{[x]}$
for the stabilizers of $x$
and $[x]$
in $G(F)$.
Let $G(F)_{x,0}$ denote the group of $\cO$-points of the connected parahoric group scheme
of $G$ associated to the point $x$,
and $G(F)_{x,0+}$ the pro-$p$ radical of $G(F)_{x,0}$.
Let $\sfG_{x}$
\index{notation-ax}{G x@$\sfG_{x}$}%
denote the reductive quotient of the special fiber of the connected parahoric group scheme above.
Thus, $\sfG_{x}(\ff) = G(F)_{x,0} / G(F)_{x, 0+}$.
(We will use this font convention more widely.
Thus, for a Levi subgroup $M$ of $G$, and $x\in \cB(M,F)$,
we have the $\ff$-group $\sfM_{x}$.)

If $P$ is a parabolic subgroup of $G$, then let $U_P$ denote the
unipotent radical of $P$.
We define the set $\cU(M)$
\index{notation-ax}{UM@$\cU(M)$}
as 
\[
\cU(M) = \left\{
U_{P} \mid \text{$P \subseteq G$ is parabolic with Levi factor $M$}
\right\}.
\]

Suppose that $K$ is an open subgroup of a locally profinite group $H$.
For a smooth representation $(\rho, V_{\rho})$ of $K$, let 
\[
\left(
\ind_{K}^{H} (\rho), \ind_{K}^{H} (V_{\rho})
\right)
\]
denote the compactly induced representation of $(\rho, V_{\rho})$.
Here, we realize $\ind_{K}^{H}(\rho)$ as the right regular representation on
\[
\ind_{K}^{H}(V_{\rho}) = \left\{
f \colon H \rightarrow V_{\rho} \colon \text{compactly supported modulo $K$} \mid 
f(kh) = \rho(k) (f(h)) \, (k \in K, h \in H)
\right\}.
\]

Suppose that $K$ is a subgroup of a group $H$ and $h \in H$. We denote $hKh^{-1}$ by $^hK$. If $\rho$ is a representation of $K$, let $^h\!\rho$ denote the representation $x\mapsto \rho(h^{-1}xh)$ of $^hK$. 
If an element $h \in H$ satisfies
\[
\Hom_{K \cap ^h\!K} (^h\!\rho, \rho) \neq \{0\},
\]
we say that $h$ \emph{intertwines} $\rho$.
We write
\[
I_{H}(\rho) = \left \{
h \in H \mid \text{$h$ intertwines $\rho$}
\right \}.
\]
We also write
\[
N_{H}(K) = \left\{
h \in H \mid {^h\!K} = K
\right\}
\]
and
\[
N_{H}(\rho) = \left\{
h \in N_{H}(K) \mid {^h\!\rho} \simeq \rho
\right\}.
\]

For a representation $(\rho, V_{\rho})$ of a group $H$, we identify $\rho$ with its representation space $V_{\rho}$ by abuse of notation.

For any vector space $V$, let $\id_V$
\index{notation-ax}{id@$\id$}
denote the identity map on $V$.

Throughout the paper we let $G$ be a connected reductive group defined over $F$.

\subsection{Hecke algebras and endomorphism algebras}
\label{Hecke algebras and endomorphism algebras}
Let $K$ be a compact, open subgroup of $G(F)$ and let $(\rho, V_{\rho})$ be an irreducible smooth representation of $K$.
We recall the definition of the Hecke algebra associated to the pair $(K, \rho)$.
Let $\cH(G(F), (K, \rho))$ \index{notation-ax}{H(GFKr@$\cH(G(F), (K, \rho))$} denote the space of compactly supported functions
\[
\varphi \colon  G(F)\to \text{End}_{\Coeff}({V_\rho})
\]
satisfying
\[
\varphi(k_1 g k_2)=\rho(k_1)\circ \varphi(g) \circ \rho(k_2)
\]
for all $k_1, k_2 \in K$ and $g \in G(F)$. 
We put an associative $\Coeff$-algebra structure on $\cH(G(F), (K, \rho))$,
from now on called the 
\emph{Hecke algebra associated to the pair $(K, \rho)$},
as follows.
Suppose that $\varphi_1$ and $\varphi_2$ are functions in $\cH(G(F), (K, \rho))$.
We define the function $\varphi_{1} * \varphi_{2} \in \cH(G(F), (K, \rho))$ as
\[
(\varphi_1 * \varphi_2)(g) =
\sum_{h \in G(F) / K}
\varphi_1(h) \circ \varphi_2(h^{-1}g)
\]
for $g \in G(F)$
(see \cite[\S8.6]{vigneras:modular}).
We note that since $\varphi_{1}$ is compactly supported, the sum on the right hand side is a finite sum.
We also note that when the characteristic $\ell$ of $\Coeff$ is zero,
then this multiplication is equivalent to the standard convolution operation
\[
(\varphi_1*\varphi_2)(g)=\int_{G(F)} \varphi_1(h) \circ \varphi_2(h^{-1} g)\, dh,
\]
where $dh$ denotes a Haar measure on $G(F)$ which is
chosen so that $K$ has measure one.
(However, the isomorphism class of $\cH(G(F),(K, \rho))$ does not
depend on the choice of measure.)
If the group $K$ is clear from the context, we drop it from the notation and write $\cH(G(F), \rho)$ for $\cH(G(F), (K, \rho))$.\index{notation-ax}{H(GFr@$\cH(G(F), \rho)$}%

\begin{remark}
\label{rmkrl}
The definition of $\cH(G(F), \rho)$ above is different from the definition of $\cH(G(F), \rho)$ in \cite[Section~2]{BK-types}.
More precisely, our $\cH(G(F), \rho)$ is denoted by $\cH(G(F), \rho^{\vee})$ in \cite[Section~2]{BK-types}, where $\rho^{\vee}$ denotes the contragredient representation of $\rho$.
We also note that our $\cH(G(F), \rho)$ is written as $\check{\cH}(G(F), \rho)$
in \cite[Section~8]{Kim-Yu} and \cite[Section~17]{Yu}.
\end{remark}

For $\varphi \in \cH(G(F), \rho)$, we write $\supp(\varphi)$ for the support of $\varphi$.
For $g \in G(F)$, $\varphi \in \cH(G(F), \rho)$, and $k \in K \cap {^g\!K}$, we have
$$
{\rho}(k) \circ \varphi(g)
=\varphi(kg)
=\varphi(g(g^{-1}kg))
=\varphi(g) \circ {\rho}(g^{-1}kg)
=\varphi(g) \circ {^{g}\!\rho}(k),
$$
and hence
\[
\varphi(g) \in \Hom_{K \cap ^g\!K}(^g\!\rho, \rho).
\]
Therefore we obtain that $\supp(\varphi) \subset I_{G(F)}(\rho)$ for any $\varphi \in \cH(G(F), \rho)$.
For $g \in G(F)$, we define the subspace $\cH(G(F), \rho)_{g}$ \index{notation-ax}{H(GFrg@$\cH(G(F), \rho)_g$} of $\cH(G(F), \rho)$ as
\[
\cH(G(F), \rho)_{g} = \{
\varphi \in \cH(G(F), \rho) \mid \supp(\varphi) \subset K g K
\}.
\]
The subspace $\cH(G(F), \rho)_{g}$ is zero if $g\notin I_{G(F)}(\rho)$,
and it depends only on the double
coset $KgK$, and not on the choice of the coset representative $g$.
As a vector space, we have
\begin{align}
\label{vectorspacedecomposition}
\cH(G(F), \rho) = \bigoplus_{g \in K \backslash I_{G(F)}(\rho) /K} \cH(G(F), \rho)_{g}.
\end{align}

According to \cite[Section~8.5]{vigneras:modular}, there exists an isomorphism of $\Coeff$-algebras
\begin{align}
\label{heckevsend}
\cH(G(F), \rho)
\isoarrow
\End_{G(F)}\bigl(\ind_{K}^{G(F)} (\rho)\bigr).
\end{align}
We write the isomorphism above explicitly.
For $v \in V_\rho$, we define $f_{v} \in \ind_{K}^{G(F)} (V_\rho)$ as
\[
f_v(g) =
\begin{cases}
\rho(g) (v) & (g \in K ), \\
0 & (\text{otherwise}).
\end{cases} 
\]
Then for $\Phi \in \End_{G(F)}\bigr(\ind_{K}^{G(F)} (\rho)\bigr)$, the corresponding element $\varphi \in \cH(G(F), \rho)$ is defined by
\begin{equation} \label{eqvarphifromPhi}
\varphi(g) (v) = \left(\Phi(f_{v})\right)(g)
\end{equation}
for $v \in V_\rho$ and $g \in G(F)$.
Conversely, for $\varphi \in \cH(G(F), \rho)$,
the corresponding element 
\[
\Phi \in \End_{G(F)}\bigr(\ind_{K}^{G(F)} (\rho)\bigr)
\]
is defined by
\[
\left(\Phi(f)\right)(g) 
=
\sum_{h \in G(F)/K} \varphi(h) \left(
 f(h^{-1} g)
\right)
\]
for $f \in \ind_{K}^{G(F)} (V_\rho)$ and $g \in G(F)$.

For a subgroup $K'$ of $G(F)$ containing $K$, we identify the compactly induced representation $\ind_{K}^{K'} (\rho)$ with the $K'$-subrepresentation of $\ind_{K}^{G(F)} (\rho)$ on the space
\[
\Bigl\{
f \in \ind_{K}^{G(F)} (V_\rho) \Bigr. \,\Bigl |\,  \supp(f) \subset K'
\Bigr\},
\]
where $\supp(f)$ denotes the support of $f$.
More generally, for a subset $K'$ of $G(F)$ such that $K \cdot K' = K'$, we define a subspace $\ind_{K}^{K'} (V_\rho)$
\index{notation-ax}{indK@$\ind_{K}^{K'} (V_\rho)$ \: ($K'\subset G(F)$ a subset)}
of $\ind_{K}^{G(F)} (V_\rho)$ as
\[
\ind_{K}^{K'} (V_\rho) = \Bigl\{
f \in \ind_{K}^{G(F)} (V_\rho) \Bigr. \,\Bigl|\, \supp(f) \subset K'
\Bigr\}.
\]
In particular, we regard $(\rho, V_\rho)$ as a $K$-subrepresentation of $(\ind_{K}^{G(F)} (\rho), \ind_{K}^{G(F)} (V_\rho))$ via the isomorphism
\[
\rho \isoarrow \ind_{K}^{K} (\rho)
\]
defined via
\[
V_\rho \ni v \longmapsto f_{v} \in \ind_{K}^{K} (V_\rho).
\]
\begin{lemma}
\label{lemmavarphisupprtedonksKequivalenttoPhirhoinKsK}
Let $g \in K \backslash I_{G(F)}(\rho) /K$.
Let $\Phi \in \End_{G(F)}\bigr(\ind_{K}^{G(F)} (\rho)\bigr)$ and $\varphi \in \cH(G(F), \rho)$  correspond to each other via the isomorphism in \eqref{heckevsend}.
Then the element $\varphi$ is supported on $K g K$ if and only if
\[
\Phi \left(
V_\rho
\right) \subset \ind_{K}^{K g K} (V_\rho).
\]
\end{lemma}
\begin{proof}
According to the construction of the isomorphism in \eqref{heckevsend},
we have 
\[
\varphi(h) (v) = \left(\Phi(f_{v})\right)(h)
\]
for all $h \in G$.
Hence, the function $\varphi$ is supported on $K g K$ if and only if for any $v \in V_\rho$, we have 
\[
\supp \left(
\Phi(f_{v})
\right) \subset K g K,
\]
that is, $\Phi \left(
f_{v}
\right) \in \ind_{K}^{K g K} (V_\rho).$
\end{proof}

For a subgroup $K'$ of $G(F)$ containing $K$, we define a subalgebra $\cH(K', \rho)$ \index{notation-ax}{H(K'r@$\cH(K', \rho)$} of $\cH(G(F), \rho)$ as
\[
\cH(K', \rho) = \left\{
\varphi \in \cH(G(F), \rho) \mid \supp(\varphi) \subset K'
\right\}.
\]
According to Lemma~\ref{lemmavarphisupprtedonksKequivalenttoPhirhoinKsK}, the subalgebra $\cH(K', \rho)$ corresponds to the subalgebra $\End_{G(F)}\bigl(\ind_{K}^{G(F)} (\rho)\bigr)_{K'}$ \index{notation-ax}{EndGF@$\End_{G(F)}\bigl(\ind_{K}^{G(F)} (\rho)\bigr)_{K'}$} of $\End_{G(F)}\bigl(\ind_{K}^{G(F)} (\rho)\bigr)$ defined as
\begin{align*}
\End_{G(F)}\bigl(\ind_{K}^{G(F)} (\rho)\bigr)_{K'} &= \left\{
\Phi \in \End_{G(F)}\bigl(\ind_{K}^{G(F)} (\rho)\bigr) \mid
\Phi \left(
V_\rho
\right) \subset \ind_{K}^{K'} (V_\rho)
\right\} \\
&= \left\{
\Phi \in \End_{G(F)}\bigl(\ind_{K}^{G(F)} (\rho)\bigr) \mid
\Phi \left(
\ind_{K}^{K'} (V_\rho)
\right) \subset \ind_{K}^{K'} (V_\rho)
\right\}
\end{align*}
via the isomorphism in \eqref{heckevsend}.
Similar to \eqref{heckevsend}, we also have an isomorphism
\begin{align}
\label{heckevsendK'}
\cH(K', \rho) \isoarrow \End_{K'}\bigl(\ind_{K}^{K'} (\rho)\bigr),
\end{align}
and one can check easily that
the following diagram commutes:
\begin{equation} \label{DiagramEnd}
\xymatrix{
\cH(K', \rho)
	\ar[d]_-{\id}
	\ar[r]^-{\eqref{heckevsend}}
	\ar@{}[dr]|\circlearrowleft
& \End_{G(F)}\bigl(\ind_{K}^{G(F)} (\rho)\bigr)_{K'}
	\ar[d]^-{\res}\\
\cH(K', \rho)
	\ar[r]^-{\eqref{heckevsendK'}}
& \End_{K'}\bigl(\ind_{K}^{K'} (\rho)\bigr),
}
\end{equation}
where $\res$ denotes the restriction map
defined by
\[
\Phi \longmapsto \Phi \restriction_{\ind_{K}^{K'} (V_\rho)}.
\]
In particular, the restriction map
is an isomorphism.
We regard $\End_{K'}\bigl(\ind_{K}^{K'} (\rho)\bigr)$ as a subalgebra of $\End_{G(F)}\bigl(\ind_{K}^{G(F)} (\rho)\bigr)$ via the inverse of this isomorphism.

We will also need a slightly more general setup. 
For this, let $H$ be a locally profinite group and $K_{1}, K_{2}$ be open subgroups of $H$.
Let $\rho_{1}$, resp., $\rho_{2}$, be a smooth representation of $K_{1}$, resp., $K_{2}$.
For an open subgroup $K'$ of $H$ containing $K_{1}$ and $K_{2}$, we define a subspace $\Hom_{H}\left(
\ind_{K_{1}}^{H} (\rho_{1}), \ind_{K_{2}}^{H}(\rho_{2})
\right)_{K'}$ of $\Hom_{H}\left(
\ind_{K_{1}}^{H} (\rho_{1}), \ind_{K_{2}}^{H}(\rho_{2})
\right)$
by
\begin{gather*}
\Hom_{H}\left(
\ind_{K_{1}}^{H} (\rho_{1}), \ind_{K_{2}}^{H}(\rho_{2})
\right)_{K'} 
=
\Bigl\{
\Phi \in \Hom_{H}\left(
\ind_{K_{1}}^{H} (\rho_{1}), \ind_{K_{2}}^{H}(\rho_{2})
\right) \Bigr. \, \Bigl |  \Phi\left(
V_{\rho_{1}}
\right) \subset \ind_{K_{2}}^{K'} (V_{\rho_{2}})
\Bigr\}
\\
=
\left\{
\Phi \in \Hom_{H}\Bigl(
\ind_{K_{1}}^{H} (\rho_{1}), \ind_{K_{2}}^{H}(\rho_{2})
\right) \Bigr. \, \Bigl| \Phi\left(
\ind_{K_{1}}^{K'} (V_{\rho_{1}})
\right) \subset \ind_{K_{2}}^{K'} (V_{\rho_{2}})
\Bigr\}.
\end{gather*}
\begin{lemma}
\label{lemmarestrictiontoasubspaceofthecompactinductiongeneralver}
The restriction map $\Phi \mapsto \Phi \restriction_{\ind_{K_{1}}^{K'} (V_{\rho_{1}})}$ gives an isomorphism
\[
\Hom_{H}\left(
\ind_{K_{1}}^{H} (\rho_{1}), \ind_{K_{2}}^{H}(\rho_{2})
\right)_{K'} \isoarrow \Hom_{K'}\left(
\ind_{K_{1}}^{K'} (\rho_{1}), \ind_{K_{2}}^{K'}(\rho_{2})
\right).
\]
\end{lemma}
\begin{proof}
Combining Frobenius reciprocity with the transitivity of compact induction, we obtain an isomorphism
\begin{align*}
\Hom_{H}\bigl(
\ind_{K_{1}}^{H} (\rho_{1}), \ind_{K_{2}}^{H}(\rho_{2})
\bigr)
& \simeq \Hom_{H}\bigl(
\ind_{K'}^{H} (
\ind_{K_{1}}^{K'} (\rho_{1})
), \ind_{K_{2}}^{H}(\rho_{2})
\bigr) \\
& \simeq \Hom_{K'}\bigl(
\ind_{K_{1}}^{K'} (\rho_{1}), \ind_{K_{2}}^{H}(\rho_{2}) \restriction_{K'}
\bigr).
\end{align*}
Under this isomorphism the subspace
$\Hom_{H}\bigl(
\ind_{K_{1}}^{H} (\rho_{1}), \ind_{K_{2}}^{H}(\rho_{2})
\bigr)_{K'}$
of
$\Hom_{H}\bigl(
\ind_{K_{1}}^{H} (\rho_{1}), \ind_{K_{2}}^{H}(\rho_{2})
\bigr)$
corresponds to the subspace
$\Hom_{K'}\bigl(
\ind_{K_{1}}^{K'} (\rho_{1}), \ind_{K_{2}}^{K'}(\rho_{2})
\bigr)$
of
$\Hom_{K'}\bigl(
\ind_{K_{1}}^{K'} (\rho_{1}), \ind_{K_{2}}^{H}(\rho_{2}) \restriction_{K'}
\bigr)$,
and the induced isomorphism
\(
\Hom_{H}\bigl(
\ind_{K_{1}}^{H} (\rho_{1}), \ind_{K_{2}}^{H}(\rho_{2})
\bigr)_{K'}
\isoarrow
\Hom_{K'}\bigl(
\ind_{K_{1}}^{K'} (\rho_{1}), \ind_{K_{2}}^{K'}(\rho_{2})
\bigr)
\)
agrees with the restriction map.
\end{proof}

\section{Structure of the Hecke algebra} 

\label{Structure of a Hecke algebra} 
We recall that $G$ is a connected reductive group defined over $F$. In this section, we will provide a description of the Hecke algebra attached to an appropriate compact, open subgroup $K_{x_0} \subset G$ and an appropriate representation $(\rho_{x_0},V_{\rho_{x_0}})$ of $K_{x_0}$ as a semi-direct product of an affine Hecke algebra and a twisted group algebra, see Theorem \ref{theoremstructureofhecke}. We describe the pairs $(K_{x_0}, \rho_{x_0})$ that we consider through several axioms introduced in this section. These are  Axioms \ref{axiomaboutHNheartandK}, \ref{axiombijectionofdoublecoset}, \ref{axiomexistenceofRgrp}, and \ref{axiomaboutdimensionofend} below. The pairs that we consider include among others the case where $(K_{x_0}, \rho_{x_0})$ is a type for a single Bernstein block as constructed by Kim and Yu (\cite{Kim-Yu, MR4357723}) or a depth-zero type attached to either a single Bernstein block or the case where $K_{x_0}$ is a parahoric subgroup, which is the case studied previously by Morris (\cite{Morris}).

Readers interested in depth-zero types
might find it useful to first read
Sections \ref{subsec:constructionofdepthzero}
and \ref{subsec:affinehyperplanes-depthzero},
readers interested in the 
types constructed by
Kim and Yu might find it helpful to first read
Sections \ref{HAIKY-subsec:construction}
and \ref{HAIKY-subsec:affinehyperplanesKimYu}
in \cite{HAIKY},
in order 
to keep these special cases as examples in mind when reading the axiomatic set-up below. 
Those interested only in the set-up in
\cite{HAIKY}
might even completely replace the below axiomatic objects by the explicit objects introduced in
\cite[Section \ref{HAIKY-Section-KimYutypes}]{HAIKY}
in their head.

\subsection{The affine space}
\label{subsec:affine}
In this subsection, we will introduce an affine subspace $\cA_{x_{0}}$ of $\cB(G, F)$, which will be used to index pairs of compact, open subgroups $K_{x}$ and irreducible smooth representations $\rho_{x}$ of $K_{x}$ below. 
More precisely, we will consider a family $\{(K_{x}, \rho_{x})\}_{x \in \cA_{\gen}}$ of compact, open subgroups $K_{x}$ of $G(F)$ and irreducible smooth representations $\rho_{x}$ of $K_{x}$ indexed by an appropriate subset $\cA_{\gen}$ of generic points in $\cA_{x_{0}}$ (see Sections~\ref{subsec:hyperplanes} and \ref{subsec:familyofcovers}).
The family $\{(K_{x}, \rho_{x})\}_{x \in \cA_{\gen}}$ will be used to define basis elements of a Hecke algebra (see Definition~\ref{definitionPhixw}). 
In Proposition~\ref{Rrhosimeqaffineweyl} below, we will also define an affine root system $\Gamma(\rho_{M})$ on a quotient $\cA_{\Krel}$ of $\cA_{x_{0}}$, whose affine Weyl group underlies the affine Hecke algebra appearing in our description of the Hecke algebra attached to $(K_{x_{0}}, \rho_{x_{0}})$ in Theorem~\ref{theoremstructureofhecke}.

To define the affine space $\cA_{x_{0}}$, we fix a Levi subgroup $M$ of $G$.
If we want to describe the Hecke algebra attached to a type for a Bernstein block of a $\bC$-representation, then $M$ is a Levi subgroup of the supercuspidal support of the Bernstein block.
We recall that we write $A_{M}$ for the maximal split torus of the center $Z(M)$ of $M$.
Let $ 
\cB(M, F) \hookrightarrow \cB(G, F)
$ be an admissible embedding of enlarged Bruhat--Tits buildings in the sense of \cite[\S~14.2]{KalethaPrasad}, which exists by  \cite[Section 9.7.5]{KalethaPrasad}.
We regard $\cB(M, F)$ as a subset of $\cB(G, F)$ via this embedding.
Let $x_{0} \in \cB(M, F)$.
We define the subset $\cA_{x_{0}}$
of $\cB(M, F)$ by
\[
\index{notation-ax}{Ax0@$\cA_{x_{0}}$}
\cA_{x_{0}} = x_{0} + \left(
X_{*}(A_{M}) \otimes_{\mathbb{Z}} \bR
\right).
\]	
More precisely,
$x_{0} + \left(
X_{*}(A_{M}) \otimes_{\mathbb{Z}} \bR
\right)$ is an affine subspace of every apartment containing $x_0$
(since $A_M$ is contained in every maximal split torus of $M$),
and the image of this affine space in $\cB(M, F)$ is independent of the choice of apartment.
We fix an $N_{G}(M)(F)$-invariant inner product $(\phantom{x},\phantom{y})_{M}$ on $X_{*}(A_{M}) \otimes_{\bZ} \bR$.
\index{notation-ax}{(,)M@$(\phantom{x},\phantom{y})_{M}$}
(Such an inner product exists, because the action of $N_{G}(M)(F)$ factors
through a finite group.)
Hence, the space $\cA_{x_{0}}$ is a Euclidean space with the vector space of translations $X_{*}(A_{M}) \otimes_{\bZ} \bR$.
We define the subset $N_{G}(M)(F)_{[x_{0}]_{M}}$
of $G(F)$ by \index{notation-ax}{NGMFx0@$N_{G}(M)(F)_{[x_{0}]_{M}}$}
\[
N_{G}(M)(F)_{[x_{0}]_{M}} = \{
n \in N_{G}(M)(F) \mid
n x_{0} \in \cA_{x_{0}}
\}.
\]
\begin{lemma}
\label{lemmaMtoS}
Let $S$ be a maximal split torus of $M$ such that
$
x_{0} \in \cA(G, S, F)
$.
Then we have
\[
N_{G}(M)(F)_{[x_{0}]_{M}} = \left(
N_{G}(M)(F)_{[x_{0}]_{M}} \cap N_{G}(S)(F)
\right) \cdot M(F)_{x_{0}, 0}.
\]
\end{lemma}
\begin{proof}
Let $n \in N_{G}(M)(F)_{[x_{0}]_{M}}$.
Since $x_{0} \in \cA(G, S, F)$, and
$
n x_{0} \in \cA_{x_{0}} \subset \cA(G, S, F)
$,
we have
$
x_{0} \in  \cA(G, S, F) \cap \cA(G, n^{-1} S n, F)
$.
Hence, there exists an element $m \in M(F)_{x_{0}, 0}$ such that
$
m (n^{-1} S n) m^{-1} = S
$.
Thus, we obtain that
\begin{equation*}
n = nm^{-1} \cdot m 
\in \left(
N_{G}(M)(F)_{[x_{0}]_{M}} \cap N_{G}(S)(F)
\right) \cdot M(F)_{x_{0}, 0}.
\qedhere
\end{equation*}
\end{proof}
\begin{lemma}
\label{lemmaaboutstabilizerofx0M}
Let $n \in N_{G}(M)(F)_{[x_{0}]_{M}}$.
Then we have $n x \in \cA_{x_{0}}$ for all $x \in \cA_{x_{0}}$.
\end{lemma}
\begin{proof}
Let $S$ be a maximal split torus of $M$ such that
$
x_{0} \in \cA(G, S, F)
$.
According to Lemma~\ref{lemmaMtoS}, we have
\[
N_{G}(M)(F)_{[x_{0}]_{M}} = \left(
N_{G}(M)(F)_{[x_{0}]_{M}} \cap N_{G}(S)(F)
\right) \cdot M(F)_{x_{0}, 0}.
\]
Since $m x = x$ for all $m \in M(F)_{x_{0}, 0}$ and $x \in \cA_{x_{0}}$, it suffices to show that $n x \in \cA_{x_{0}}$ for all 
\[
n \in N_{G}(M)(F)_{[x_{0}]_{M}} \cap N_{G}(S)(F)
\]
and $x \in \cA_{x_{0}}$.
We write $x = x_{0} + a$ for some 
\[
a \in X_{*}(A_{M}) \otimes_{\mathbb{Z}} \bR \subset X_{*}(S) \otimes_{\mathbb{Z}} \bR.
\]
According to \cite[Proposition~6.2.4]{KalethaPrasad},
we have
\begin{equation*}
n x = n (x_{0} + a) 
= n x_{0} + (Dn) a,
\end{equation*}
where $Dn$ denotes the image of $n$ in the finite Weyl group $N_{G}(S)(F) / Z_{G}(S)(F)$, which acts on the vector space $X_{*}(S) \otimes_{\mathbb{Z}} \bR$.
Since $n$ normalizes $M$ and $a \in X_{*}(A_{M}) \otimes_{\mathbb{Z}} \bR$, we have 
\[
(Dn) a \in X_{*}(A_{M}) \otimes_{\mathbb{Z}} \bR.
\]
Thus, the assumption $n \in N_{G}(M)(F)_{[x_{0}]_{M}}$ implies that
\begin{equation*}
n x  = n x_{0} + (Dn) a 
 \subset \cA_{x_{0}} + \left(
X_{*}(A_{M}) \otimes_{\mathbb{Z}} \bR
\right) 
 = \cA_{x_{0}}.
\qedhere
\end{equation*}
\end{proof}
For later application it will be useful to state the result for an arbitrary point $x'_{0} \in \cA_{x_{0}}$. 
\begin{corollary}
\label{corollaryaboutstabilizerofx0primeM}
Let $x'_{0} \in \cA_{x_{0}}$. 
If an element $n \in N_{G}(M)(F)$ satisfies $n x'_{0} \in \cA_{x_{0}}$,
then we have $n x \in \cA_{x_{0}}$ for all $x \in \cA_{x_{0}}$.
\end{corollary}
\begin{proof}
This follows from Lemma~\ref{lemmaaboutstabilizerofx0M} by replacing $x_{0}$ by $x'_{0}$.
\end{proof}
According to Lemma~\ref{lemmaaboutstabilizerofx0M}, $N_{G}(M)(F)_{[x_{0}]_{M}}$ is a subgroup of $G(F)$, and
the action of $G(F)$ on $\cB(G, F)$ induces an action of $N_{G}(M)(F)_{[x_{0}]_{M}}$ on $\cA_{x_{0}}$.
For $n \in N_{G}(M)(F)_{[x_{0}]_{M}}$ and a subset $X$ of $\cA_{x_{0}}$, we write
$
n(X) = \left\{
n x \mid x \in X
\right\}
$.

\subsection{Affine hyperplanes}
\label{subsec:hyperplanes}
Let $\mathfrak{H}$ \index{notation-ax}{H_@$\mathfrak{H}$}  be a possibly empty locally finite set of affine hyperplanes in $\cA_{x_{0}}$ that do not contain $x_{0}$.
The complement of these hyperplanes in $\cA_{x_{0}}$ decomposes into connected components and the compact open subgroup $K_x$ that we will attach each point $x$ in the complement will be constant on those components. 
The image on the quotient $\cA_{\Krel}$ of $\cA_{x_{0}}$ 
 of an appropriate subset  $\mathfrak{H}_{\Krel} \subseteq \fH$ of $\cK$-relevant hyperplanes will form the set of vanishing hyperplanes of the affine root system $\Gamma(\rho_{M})$ mentioned above in Section \ref{subsec:affine} (see Definition~\ref{definitionKrelevant} and Proposition~\ref{Rrhosimeqaffineweyl}).
In the setting of depth-zero types discussed in Section \ref{sec:depth-zero},
the set of affine hyperplanes $\mathfrak H$ that we want to consider
is described in
\S\ref{subsec:affinehyperplanes-depthzero}.
In the setting of positive-depth types as constructed by Kim and Yu,
the set of affine hyperplanes $\mathfrak H$ that we want to consider is described in
\cite[Section \ref{HAIKY-subsec:affinehyperplanesKimYu}]{HAIKY}.
We define the subset $\cA_{\gen}$ of \emph{generic points}
of $\cA_{x_{0}}$ by
\[
\index{notation-ax}{Agen@$\cA_{\gen}$}
\cA_{\gen} = \cA_{x_{0}} \smallsetminus \left(
\bigcup_{H \in \mathfrak{H}} H
\right).
\]

For $x, y \in \cA_{\gen}$, we define the subset $\mathfrak{H}_{x, y}$ of $\mathfrak{H}$ by
\[
\index{notation-ax}{H_xy@$\mathfrak{H}_{x,y}$}
\mathfrak{H}_{x, y} = \left\{
H \in \mathfrak{H} \mid \text{$x$ and $y$ are on opposite sides of $H$}
\right\}.
\]
Since $\mathfrak{H}$ is locally finite, we have
$
\# \mathfrak{H}_{x, y} < \infty
$.
We write \index{notation-ax}{d(,)@$d(\phantom{x},\phantom{y})$}
$
d(x, y) = \# \mathfrak{H}_{x, y}
$.
\begin{lemma}
\label{lemmaaboutdistancefunctionnoindex}
Let $x, y, z \in \cA_{\gen}$.
Then we have
\[
d(x, y) + d(y, z) \ge d(x, z).
\]
Moreover, the following conditions are equivalent:
\[
\begin{array}{l c  l }
\textup{(a)} \hspace{10em} & d(x, y) + d(y, z) = d(x, z). & \hspace{10em}  \\
\textup{(b)} & \mathfrak{H}_{x, y},\: \mathfrak{H}_{y, z} \subset \mathfrak{H}_{x, z}. \\
\textup{(c)} & \mathfrak{H}_{x, y} \cap \mathfrak{H}_{y, z} = \emptyset. 
\end{array}
\]
\end{lemma}

\begin{proof}
For any $H \in \mathfrak{H}_{x, z}$, exactly one of the following occurs:
\begin{itemize}
\item 
We have $H \in \mathfrak{H}_{x, y}$.
\item
We have $H \in \mathfrak{H}_{y, z}$.
\end{itemize}
Hence, we have
$
d(x, y) + d(y, z) \ge d(x, z)
$,
and equality holds if and only if
$
\mathfrak{H}_{x, y}, \mathfrak{H}_{y, z} \subset \mathfrak{H}_{x, z}
$.
We will prove
that conditions (b) and (c)    
are equivalent.
Suppose that
$
\mathfrak{H}_{x, y}, \mathfrak{H}_{y, z} \subset \mathfrak{H}_{x, z}
$.
Then, for any $H \in \mathfrak{H}_{x, y}$, we also have $H \in \mathfrak{H}_{x, z}$.
Since the points $x$ and $y$ are on opposite sides of $H$, and the points $x$ and $z$ are on opposite sides of $H$, we obtain that the points $y$ and $z$ are on the same side of $H$.
Hence, we have $H \not \in \mathfrak{H}_{y, z}$.
Thus, we conclude that 
$
\mathfrak{H}_{x, y} \cap \mathfrak{H}_{y, z} = \emptyset
$.
On the other hand, suppose that 
$
\mathfrak{H}_{x, y} \cap \mathfrak{H}_{y, z} = \emptyset
$.
Then, for any $H \in \mathfrak{H}_{x, y}$, the points $y$ and $z$ are on the same side of $H$.
Hence, the points $x$ and $z$ are on opposite sides of $H$, that is $H \in \mathfrak{H}_{x, z}$.
Thus, we conclude that $\mathfrak{H}_{x, y} \subset \mathfrak{H}_{x, z}$.
Similarly, we can prove that $\mathfrak{H}_{y, z} \subset \mathfrak{H}_{x, z}$.
\end{proof}

\subsection{Quasi-$G$-covers}
\label{subsec:quasiGcovers}

Let $K_{M}$ be a compact, open subgroup of $M(F)_{x_{0}}$ and let $(\rho_{M}, V_{\rho_{M}})$ be an irreducible smooth representation of $K_{M}$. For example, the pairs $(K_M, \rho_M)$ considered in
\cite[Section \ref{HAIKY-Section-KimYutypes}]{HAIKY}
include the supercuspidal types constructed by Yu (\cite{Yu}) twisted by a quadratic character that arises from the work of \cite{FKS}.
\begin{notation}
\label{notationofNrhoM}
We denote by
$N(\rho_{M})_{[x_{0}]_{M}}$\index{notation-ax}{NrhoMx0@$N(\rho_{M})_{[x_{0}]_{M}}$} the subgroup of $N_{G}(M)(F)_{[x_{0}]_{M}}$ given by
\[
N(\rho_{M})_{[x_{0}]_{M}} = N_{G(F)}(\rho_{M}) \cap N_{G}(M)(F)_{[x_{0}]_{M}}.
\] 
\end{notation}              
\begin{definition}
Let $K$ (resp.\ $K_{M}$) be a compact, open subgroup of $G(F)$ (resp.\ $M(F)$) and let $\rho$ (resp.\ $\rho_{M}$) be an irreducible smooth representation of $K$ (resp.\ $K_{M}$).
We say that $(K, \rho)$ is a \emph{quasi-$G$-cover} of the pair $(K_{M}, \rho_{M})$ if
for every parabolic subgroup $P\subseteq G$ with Levi factor $M$,
we have that the pair $(K,\rho)$ is \emph{decomposed} with respect
to $(M,P)$ (in the sense of \cite[Definition 6.1]{BK-types}).
Equivalently,
for every $U \in \cU(M)$ the following conditions are satisfied:
\begin{enumerate}[(1)]
\item
\label{conditionofquasiGcoverIwahoridecomposition}
We have the decomposition
\[
K = \left(
K \cap U(F)
\right) \cdot (K \cap M(F)) \cdot \left(
K \cap \overline{U}(F)
\right).
\]
\item
We have
$K_M = K\cap M(F)$, 
the restriction of $\rho$ to ${K_{M}}$ is  $\rho_{M}$,
and the restriction of $\rho$ to the groups $K \cap U(F)$ and $K \cap \overline{U}(F)$ is trivial.
\end{enumerate}
\end{definition}

\begin{remark}
\label{remarkaboutthedefinitionfoIwahoridecomposition}
Let $M$ be a Levi subgroup of $G$ and $U \in \cU(M)$.
Then, according to \cite[Proposition~14.21 (iii)]{MR1102012}, the product map
\[
U(F) \times M(F) \times \overline{U}(F) \rightarrow G(F)
\]
is a homeomorphism onto an open subset of $G(F)$.
Hence, Condition \eqref{conditionofquasiGcoverIwahoridecomposition} of the definition of a quasi-$G$-cover is equivalent to the condition that the product map
\[
\left(
K \cap U(F)
\right) \times (K \cap M(F)) \times \left(
K \cap \overline{U}(F)
\right) \rightarrow K
\]
is a homeomorphism of topological spaces.
In particular, any element of $K$ can be written uniquely as a product of three elements in $K \cap U(F), K \cap M(F)$, and $K \cap \overline{U}(F)$, respectively. 
\end{remark}
\begin{remark}
If $(K, \rho)$ is a $G$-cover of $(K_{M}, \rho_{M})$ in the sense of \cite[Definition~8.1]{BK-types}, then $(K, \rho)$ is a quasi-$G$-cover of $(K_{M}, \rho_{M})$ by definition.
\end{remark}

Quasi-$G$-covers will allow us to compare intertwiners of representations of compact, open subgroups of $G(F)$ with intertwiners of representations of compact, open subgroups of $M(F)$. More precisely, we have the following lemma that will be used to study the support of Hecke algebras. 
\begin{lemma}
	\label{lemmacoverintertwining}
	Let $K_{1}$ and $K_{2}$ be compact, open subgroups of $G(F)$, and let $K_{M, 1}$ and $K_{M, 2}$ be compact, open subgroups of $M(F)$.
	Let $\rho_{1}$, $\rho_{2}$, $\rho_{M, 1}$, and $\rho_{M, 2}$ be irreducible smooth representations of $K_{1}$, $K_{2}$, $K_{M, 1}$, and $K_{M, 2}$, respectively.
	Suppose that the pair $(K_{i}, \rho_{i})$ is a quasi-$G$-cover of the pair $(K_{M, i}, \rho_{M, i})$ for $i = 1,2$.
	Then we have
	\[
	\Hom_{K_{1} \cap K_{2}}\left(
	\rho_{1}, \rho_{2}
	\right)
	=
	\Hom_{K_{M, 1} \cap K_{M, 2}}\left(
	\rho_{M, 1}, \rho_{M, 2}
	\right).
	\]
\end{lemma}
\begin{proof}
	We fix $U \in \cU(M)$.
	Since $(K_{i}, \rho_{i})$ is a quasi-$G$-cover of $(K_{M, i}, \rho_{M, i})$, we have
	\begin{align*}
	K_{i} &= \left(
	K_{i} \cap U(F)
	\right) \cdot K_{M, i} \cdot \left(
	K_{i} \cap \overline{U}(F)
	\right),
	\end{align*}
	the representation $\rho_{i}$ is trivial on $K_{i} \cap U(F)$ and $K_{i} \cap \overline{U}(F)$, and we have $\rho_{i}\restriction_{K_{M, i}} = \rho_{M, i}$ for $i = 1,2$.
	Thus we obtain
	\begin{align*}
	\Hom_{K_{1} \cap K_{2}}\left(
	\rho_{1}, \rho_{2}
	\right) &=
	\Hom_{\left(
		K_{1} \cap K_{2} \cap U(F)
		\right) \cdot \left(
		K_{M, 1} \cap K_{M, 2}
		\right) \cdot 
		\left(
		K_{1} \cap K_{2} \cap \overline{U}(F)
		\right)
	}(\rho_{1}, \rho_{2}) \\
	&= \Hom_{K_{M, 1} \cap K_{M, 2}}(\rho_{1}, \rho_{2}) \\
	&= \Hom_{K_{M, 1} \cap K_{M, 2}}(\rho_{M, 1}, \rho_{M, 2}).
	\qedhere
	\end{align*}
\end{proof}

\begin{corollary}
\label{corollarynormalizercontainedinintertwiner}
Let $K$ be a compact, open subgroup of $G(F)$ and $\rho$ be an irreducible smooth representation of $K$.
Suppose that $(K, \rho)$ is a quasi-$G$-cover of $(K_{M}, \rho_{M})$.
Then we have
\[
N(\rho_{M})_{[x_{0}]_{M}} \subset I_{G(F)}(\rho).
\]
\end{corollary}
\begin{proof}
Let $n \in N(\rho_{M})_{[x_{0}]_{M}}$.
Since the pair $(K, \rho)$ is a quasi-$G$-cover of $(K_{M}, \rho_{M})$ and $n$ normalizes the group $M$, the pair $(^n\!K, {^n\!\rho})$ is a quasi-$G$-cover of $(^n\!K_{M}, {^n\!\rho_{M}})$.
Then Lemma~\ref{lemmacoverintertwining} implies that
\begin{align*}
\Hom_{K \cap ^n\!K} (^n\!\rho, \rho) &= \Hom_{K_{M} \cap ^n\!K_{M}}(^n\!\rho_{M}, \rho_{M}).
\end{align*}
Since $n \in N_{G(F)}(\rho_{M})$, the right-hand side is non-trivial.
Hence, we obtain that the left-hand side is also non-trivial, that is, $n \in I_{G(F)}(\rho)$.
\end{proof}

\subsection{Family of quasi-$G$-covers and a group structure on the Hecke algebra support} \label{subsec:familyofcovers}

We will define and study basis elements of the Hecke algebra attached to a quasi-$G$-cover of $(K_{M}, \rho_{M})$ below.
To do this, we do not only study one quasi-$G$-cover, but also consider a family of quasi-$G$-covers of $(K_M, \rho_M)$ indexed by $\cA_{\gen}$ with some additional data and properties.
More precisely, consider a family
\index{notation-ax}{Kay x@$K_{x}$}
\index{notation-ax}{rhox@$\rho_{x}$}
\index{notation-ax}{Kay x +@$K_{x, +}$}
\index{notation-ax}{Kay@$\cK$}
\[
\cK =
\left\{
(K_{x}, K_{x, +}, (\rho_{x}, V_{\rho_{x}}))
\right\}_{x \in \cA_{\gen}},
\]
where each $K_x$ is a compact, open subgroup of $G(F)$,
each $K_{x,+}$
is a normal, open subgroup of $K_x$,
and each $(\rho_x, V_{\rho_x})$ is an irreducible smooth
representation of $K_x$. We will refer to such a family as a \textit{family of quasi-$G$-cover-candidates}.
\index{terminology-ax}{family of quasi-$G$-cover-candidates}
Let $\Nheart$
be a subgroup of 
\index{notation-ax}{NrhoMx0heart@$\Nheart$}
$N(\rho_{M})_{[x_{0}]_{M}}$ containing $A_{M}(F)$.
Eventually (starting with Corollary \ref{corollarybijectiondoublecosetandweylgroup}) we will assume that $\Nheart$ is sufficiently large, more precisely large enough to satisfy Axiom \ref{axiombijectionofdoublecoset} below. For now we first only assume that 
$\cK$ and $\Nheart$ satisfy the following:
\begin{axiom}
\label{axiomaboutHNheartandK}
\mbox{}
\begin{enumerate}[(1)] 
\item
\label{axiomaboutHNheartandKrhoxandthetax}
The restriction of $\rho_x$ to $K_{x, +}$
is $\theta_x$-isotypic for some character
$\theta_x$ of $K_{x, +}$.
\index{notation-ax}{thetax@$\theta_x$}
\item
\label{axiomaboutHNheartandKNinvarianceofH}
The action of $\Nheart$ on $\cA_{x_{0}}$ 
preserves the set $\mathfrak{H}$.
In particular, $\Nheart$
stabilizes the the subset $\cA_{\gen}$ of $\cA_{x_{0}}$.
\item \label{axiomaboutHNheartandKaboutconjugation}
For every $x \in \cA_{\gen}$ and $n \in \Nheart$, we have
\[
K_{n x} = n K_{x} n^{-1}
\qquad
\text{and}
\qquad
K_{n x, +} = n K_{x, +} n^{-1}.
\]
\item
\label{axiomaboutHNheartandKabouttype}
For every $x \in \cA_{\gen}$ and for all $U \in \cU(M)$, we have
\begin{enumerate}
\item 
\label{axiomaboutHNheartandKabouttypeKxisaquasiGcover}
the pair $(K_{x}, \rho_{x})$ is a quasi-$G$-cover of $(K_{M}, \rho_{M})$,
\item
\label{axiomaboutHNheartandKabouttypeKx=KMKx+}
$K_{x} = K_{M} \cdot K_{x, +},$
\item
\label{axiomaboutHNheartandKabouttypeIwahoridecompositionforKx+}
$K_{x, +} = \left(
K_{x, +} \cap U(F)
\right) \cdot \left(K_{x, +} \cap M(F)\right) \cdot \left(
K_{x, +} \cap \overline{U}(F)
\right).
$
\end{enumerate}
Moreover, the group $K_{x, +} \cap M(F)$ is independent of the point $x \in \cA_{\gen}$.
\item
\label{axiomaboutHNheartandKabotgoodunipotentradical}
For $x, y, z \in \cA_{\gen}$ such that 
\[
d(x, y) + d(y, z) = d(x, z),
\]
there exists $U \in \cU(M)$ such that
\[
K_{x} \cap U(F) \subseteq K_{y} \cap U(F) \subseteq K_{z} \cap U(F)
\]
and
\[
K_{z} \cap \overline{U}(F)  \subseteq K_{y} \cap \overline{U}(F) \subseteq K_{x} \cap \overline{U}(F).
\]
In particular, for any $x, y \in \cA_{\gen}$, there exists $U \in \cU(M)$ such that
\[
K_{x} \cap U(F) \subseteq K_{y} \cap U(F)
\quad 
\text{ and }
\quad
K_{y} \cap \overline{U}(F) \subseteq K_{x} \cap \overline{U}(F).
\]
\end{enumerate}
\end{axiom}
Once we assume this axiom for a family of quasi-$G$-cover-candidates $\cK$, we also refer to $\cK$ as a \textit{family of quasi-$G$-covers}.%
\index{terminology-ax}{family of quasi-$G$-covers}
\begin{remark}
The last paragraph of Section~\ref{subsec:constructionofdepthzero}, see p.\ \pageref{depthzerofamilies}, summarizes the families that we consider in Section~\ref{sec:depth-zero}, and
the last paragraph of
\cite[Section \ref{HAIKY-subsec:construction}]{HAIKY}, see p.\ \pageref{HAIKY-KimYufamilies} in \cite{HAIKY},
summarizes the families that we consider in
\cite{HAIKY}.
These cases include the types for single depth-zero Bernstein blocks and those for positive-depth Bernstein blocks as constructed by Kim and Yu (twisted by a quadratic character following Fintzen, Kaletha and Spice), respectively. 
Axiom \ref{axiomaboutHNheartandK} for the families considered in Section~\ref{sec:depth-zero} is vertified in Lemma \ref{proofofaxiomaboutHNheartandK}, and Axiom \ref{axiomaboutHNheartandK} for the families considered in \cite{HAIKY} is verified in
\cite[Lemma \ref{HAIKY-proofofaxiomaboutHNheartandK}]{HAIKY}.
\end{remark}
Recall that the support of the Hecke algebras attached to the pairs $(K_x, \rho_x)$ for $x \in \cA_\gen$  is given by $I_{G(F)}(\rho_{x})$. 
According to Corollary~\ref{corollarynormalizercontainedinintertwiner} and Axiom~\ref{axiomaboutHNheartandK}\eqref{axiomaboutHNheartandKabouttypeKxisaquasiGcover}, we have
\[
\Nheart \subset N(\rho_{M})_{[x_{0}]_{M}} \subset I_{G(F)}(\rho_{x})
\]
for all $x \in \cA_{\gen}$.
Since $I_{G(F)}(\rho_{x}) = K_{x} \cdot I_{G(F)}(\rho_{x}) \cdot K_{x}$ by definition, we also have the inclusion
\[
K_{x} \cdot \Nheart \cdot K_{x} \subset I_{G(F)}(\rho_{x}).
\]
In order to access the support of the relevant Hecke algebras via $\Nheart$, which will allow us to enrich the support with a group structure in Corollary \ref{corollarybijectiondoublecosetandweylgroup} and Definition \ref{definitionofWheart}, we will also suppose starting from Corollary~\ref{corollarybijectiondoublecosetandweylgroup} that this inclusion is in fact an equality, i.e., that the following axiom holds.
\begin{axiom}
\label{axiombijectionofdoublecoset}
We have
\[
K_{x} \cdot \Nheart \cdot K_{x} = I_{G(F)}(\rho_{x})
\]
for all $x \in \cA_{\gen}$.
\end{axiom}
In the depth-zero setting of Section~\ref{sec:depth-zero} this axiom holds by Proposition~\ref{propproofofaxiombijectionofdoublecoset}, and
in the setting of
\cite{HAIKY}
this axiom holds by \cite[Proposition \ref{HAIKY-proofofaxiomaboutK0vsK}]{HAIKY}.
\begin{remark}
To prove Lemma~\ref{lemmaintersectionwithunipotent}--Lemma~\ref{NcapUbarUtrivial} below, we only need Axiom~\ref{axiomaboutHNheartandK}, but we believe that having all axioms in one place will help the reader find them more easily.
\end{remark}
We record some consequences of Axiom~\ref{axiomaboutHNheartandK} that we will use throughout the paper. 
\begin{lemma}
\label{lemmaintersectionwithunipotent}
Let $x \in \cA_{\gen}$.
Then we have
\[
K_{x} \cap U(F) = K_{x, +} \cap U(F)
\]
for all $U \in \cU(M)$.
\end{lemma}
\begin{proof}
According to Axiom~\ref{axiomaboutHNheartandK}\eqref{axiomaboutHNheartandKabouttypeKxisaquasiGcover}, we have
\[
K_{x} = \left(
K_{x} \cap U(F)
\right) \cdot K_{M} \cdot \left(
K_{x} \cap \overline{U}(F)
\right).
\]
On the other hand, according to Axiom~\ref{axiomaboutHNheartandK}\eqref{axiomaboutHNheartandKabouttypeKx=KMKx+}, \eqref{axiomaboutHNheartandKabouttypeIwahoridecompositionforKx+}, we have
\begin{align*}
K_{x} & = K_{M} \cdot K_{x, +} \\
&= K_{M} \cdot \left(
K_{x, +} \cap U(F)
\right) \cdot (K_{x, +} \cap M(F)) \cdot \left(
K_{x, +} \cap \overline{U}(F)
\right) \\
&= \left(
K_{x, +} \cap U(F)
\right) \cdot K_{M} \cdot \left(
K_{x, +} \cap \overline{U}(F)
\right).
\end{align*}
Then the lemma follows from Remark~\ref{remarkaboutthedefinitionfoIwahoridecomposition}.
\end{proof}
\begin{lemma}
\label{lemmaKxcapKy+vsKx+capKy+}
Let $x, y \in \cA_{\gen}$.
Then we have 
\[
K_{x} \cap K_{y, +} = K_{x, +} \cap K_{y, +}.
\]
\end{lemma}
\begin{proof}
We fix $U \in \cU(M)$.
By using Axiom~\ref{axiomaboutHNheartandK}\eqref{axiomaboutHNheartandKabouttype}, we obtain that
\[
K_{x} \cap K_{y, +} = \left(
K_{x} \cap K_{y, +} \cap U(F)
\right) \cdot \left(
K_{x} \cap K_{y, +} \cap M(F)
\right)
\cdot \left(
K_{x} \cap K_{y, +} \cap \overline{U}(F)
\right)
\]
and
\[
K_{x, +} \cap K_{y, +} = \left(
K_{x, +} \cap K_{y, +} \cap U(F)
\right) \cdot \left(
K_{x, +} \cap K_{y, +} \cap M(F)
\right)
\cdot \left(
K_{x, +} \cap K_{y, +} \cap \overline{U}(F)
\right).
\]
Then the claim follows from Lemma~\ref{lemmaintersectionwithunipotent} and the fact that $K_{x, +} \cap M(F) = K_{y, +} \cap M(F)$.
\end{proof}
\begin{lemma}
\label{lemmaquotientvswithquotientofunipotentradical}
Let $x, y \in \cA_{\gen}$ and $U \in \cU(M)$ such that
\[
K_{x} \cap U(F) \subseteq K_{y} \cap U(F)
\quad 
\text{ and }
\quad
 K_{y} \cap \overline{U}(F) \subseteq K_{x} \cap \overline{U}(F).
\]
Then the inclusions $K_{y} \cap U(F) \subset K_{y, +} \subset K_{y}$ induce bijections
\[
\left(
K_{y} \cap U(F) 
\right) / \left(
K_{x} \cap U(F) 
\right) \simeq 
K_{y, +}
/
\left(
K_{x, +} \cap K_{y, +}
\right)
\simeq
K_{y}
/
\left(
K_{x} \cap K_{y}
\right).
\]
\end{lemma}
\begin{proof}
According to Axiom~\ref{axiomaboutHNheartandK}\eqref{axiomaboutHNheartandKabouttypeKxisaquasiGcover}, we have
\[
K_{x} = \left(
K_{x} \cap U(F)
\right) \cdot K_{M} \cdot \left(
K_{x} \cap \overline{U}(F)
\right)
\quad 
\text{ and }
\quad
K_{y} = \left(
K_{y} \cap U(F)
\right) \cdot K_{M} \cdot \left(
K_{y} \cap \overline{U}(F)
\right).
\]
Then the assumptions of the lemma imply that
\[
\left(
K_{y} \cap U(F) 
\right) / \left(
K_{x} \cap U(F) 
\right) \simeq 
K_{y}
/
\left(
K_{x} \cap K_{y}
\right).
\]
Since $K_{y} \cap U(F) = K_{y, +} \cap U(F) \subset K_{y, +}$ by Lemma~\ref{lemmaintersectionwithunipotent}, the claim follows from Lemma~\ref{lemmaKxcapKy+vsKx+capKy+}.
\end{proof}
\begin{lemma}
\label{lemmaquotientnonzeromodell}
Let $x, y \in \cA_{\gen}$.
Then the order of the quotient
$K_{y} / \left( K_{x} \cap K_{y} \right)$
is a power of $p$.
In particular, 
this integer is invertible in $\Coeff$.
\end{lemma}
\begin{proof}
According to Axiom~\ref{axiomaboutHNheartandK}\eqref{axiomaboutHNheartandKabotgoodunipotentradical}, there exists $U \in \cU(M)$ such that
$
K_{x} \cap U(F) \subseteq K_{y} \cap U(F)
$
and
$
K_{y} \cap \overline{U}(F) \subseteq K_{x} \cap \overline{U}(F)
$.
Then the claim follows from Lemma~\ref{lemmaquotientvswithquotientofunipotentradical} and the fact that $K_{y} \cap U(F)$ is a pro-$p$-group.
\end{proof}
\begin{lemma}
\label{lemmathetaonintersection}
Let $x, y \in \cA_{\gen}$.
Then we have
\begin{align*}
\Hom_{K_{x} \cap K_{y}}(\rho_{x}, \rho_{y}) \neq \{0\}.
\end{align*}
In particular, we have
\[
\theta_{x} \restriction_{K_{x, +} \cap K_{y, +}} = \theta_{y} \restriction_{K_{x, +} \cap K_{y, +}}. 
\]
\end{lemma}
\begin{proof}
Since $(K_{x}, \rho_{x})$ and $(K_{y}, \rho_{y})$ are quasi-$G$-covers of $(K_{M}, \rho_{M})$, Lemma~\ref{lemmacoverintertwining} implies that
\begin{align*}
\Hom_{K_{x} \cap K_{y}}(\rho_{x}, \rho_{y}) = \End_{K_{M}}(\rho_{M}) \neq \{0\}.
\end{align*}
The last claim follows from the assumption that the restriction of $\rho_{x}$ to $K_{x, +}$ is $\theta_{x}$-isotypic and the restriction of $\rho_{y}$ to $K_{y, +}$ is $\theta_{y}$-isotypic.
\end{proof}
\begin{lemma}
\label{lemmaeffectofconjugationbyn}
Let $x \in \cA_{\gen}$ and $n \in \Nheart$.
Then the pair $(K_{nx}, {^n\!\rho_{x}})$ is a quasi-$G$-cover of the pair $(K_{M}, {^n\!\rho_{M}})$.
\end{lemma}
\begin{proof}
Since the pair $(K_{x}, \rho_{x})$ is a quasi-$G$-cover of $(K_{M}, \rho_{M})$ and $n$ normalizes the group $M$, the pair $(^n\!K_{x}, {^n\!\rho_{x}})$ is a quasi-$G$-cover of $(^n\!K_{M}, {^n\!\rho_{M}})$.
Then the lemma follows from the facts that $^n\!K_{x} = K_{nx}$ and $^n\!K_{M} = K_{M}$.
\end{proof}
\begin{lemma}
\label{lemmathetaintertwine}
Let $x \in \cA_{\gen}$ and $n \in \Nheart$.
Then we have
\[
\theta_{x}(n^{-1} k n) = \theta_{n x}(k)
\]
for all $k \in K_{n x, +}$.
\end{lemma}
\begin{proof}
According to Axiom~\ref{axiomaboutHNheartandK}\eqref{axiomaboutHNheartandKabouttypeKxisaquasiGcover}, the pair $(K_{n x}, \rho_{n x})$ is a quasi-$G$-cover of the pair $(K_{M}, \rho_{M})$.
On the other hand, according to Lemma~\ref{lemmaeffectofconjugationbyn}, the pair $(K_{n x}, {^n\!\rho_{x}})$ is a quasi-$G$-cover of the pair $(K_{M}, {^n\!\rho_{M}})$.
Since $n \in \Nheart \subset N_{G(F)}(\rho_{M})$, Lemma~\ref{lemmacoverintertwining} implies that
\begin{align*}
\Hom_{K_{n x}}(^n\!\rho_{x}, \rho_{n x}) &= \Hom_{K_{M}}(^n\!\rho_{M}, \rho_{M}) \neq \{0\}.
\end{align*}
Since the restriction of $^n\!\rho_{x}$ to $K_{n x, +}$ is $^n\theta_{x}$-isotypic, and the restriction of $\rho_{n x}$ to $K_{n x, +}$ is $\theta_{n x}$-isotypic, we deduce the claim.
\end{proof}

Axiom~\ref{axiomaboutHNheartandK} also allows us to prove the following proposition that will be used to study the support of the Hecke algebra attached to $(K_{x}, \rho_{x})$.
\begin{proposition}
\label{propositiondoublecosetinjection}
Let $x \in \cA_{\gen}$ and assume Axiom \ref{axiomaboutHNheartandK}.
Then the inclusion
\[
\Nheart \subset I_{G(F)}(\rho_{x}),
\]
induces an injection
\[
\Nheart / \left(
\Nheart  \cap K_{M}
\right) \rightarrow K_{x} \backslash I_{G(F)}(\rho_{x}) / K_{x}.
\]
\end{proposition}
To prove Proposition~\ref{propositiondoublecosetinjection}, we first record the following general lemma:
\begin{lemma}
\label{NcapUbarUtrivial}
Let $U \in \cU(M)$.
Then we have
\[
N_{G}(M)(F) \cap \overline{U}(F) \cdot U(F) = \{1\}.
\]
\end{lemma}
\begin{proof}
Let $n \in N_{G}(M)(F)$, $u \in U(F)$, and $\overline{u} \in \overline{U}(F)$ such that
$
n = \overline{u} u
$.
For any $m \in M(F)$, we have
\[
n m n^{-1} = \overline{u} u m u^{-1} \overline{u}^{-1} 
= \overline{u} u (m u^{-1} m^{-1}) (m \overline{u}^{-1} m^{-1}) m.
\]
We write 
\[
m' = (n m n^{-1}) m^{-1} \in M(F)
\quad 
\text{ and }
\quad
u' = u (m u^{-1} m^{-1}) \in U(F).
\]
Then we have
\[
m' (m \overline{u} m^{-1}) = \overline{u} u'.
\]
Since the natural map
\[
M(F) \times \overline{U}(F) \times U(F) \rightarrow G(F)
\]
is injective, we have
$
m' = 1
$.
Hence, we obtain that $n$ commutes with any element in $M(F)$.
In particular, $n$ is contained in the centralizer of $A_{M}$ in $G(F)$, that is equal to $M(F)$.
Thus, we conclude that
\[
n \in M(F) \cap \overline{U}(F) \cdot U(F) = \{1\}.
\qedhere
\]
\end{proof}

\begin{proof}[Proof of Proposition~\ref{propositiondoublecosetinjection}]
Let $n_{1}, n_{2} \in \Nheart$ and $k_{1}, k_{2} \in K_{x}$ such that
$
n_{2} = k_{1} n_{1} k_{2}
$.
It suffices to show that
$
n_{1} n_{2}^{-1} \in K_{M}
$.
We write
\[
k'_{2} = n_{1} k_{2} n_{1}^{-1} \in n_{1} K_{x} n_{1}^{-1} = K_{n_{1} x}.
\]
By Axiom~\ref{axiomaboutHNheartandK}\eqref{axiomaboutHNheartandKabotgoodunipotentradical}, there exists $U \in \cU(M)$ such that
\[
K_{x} \cap U(F) \subset K_{n_{1} x} \cap U(F)
\quad 
\text{ and }
\quad
K_{n_{1} x} \cap \overline{U}(F) \subset K_{x} \cap \overline{U}(F).
\]
Then according to Axiom~\ref{axiomaboutHNheartandK}\eqref{axiomaboutHNheartandKabouttypeKxisaquasiGcover}, we have
\begin{align*}
n_{2} n_{1}^{-1} &= k_{1} n_{1} k_{2} n_{1}^{-1} \\
&= k_{1} k'_{2} \\
& \in K_{x} \cdot K_{n_{1} x} \\
&=  \left(
K_{x} \cap \overline{U}(F) 
\right) K_{M} 
\left(
K_{x} \cap U(F)
\right)
\cdot K_{n_{1} x} \\
&= \left(
K_{x} \cap \overline{U}(F) 
\right) \cdot K_{n_{1} x} \\
&=  \left(
K_{x} \cap \overline{U}(F) 
\right) \cdot \left(
K_{n_{1} x} \cap \overline{U}(F) 
\right) \left(
K_{n_{1} x} \cap U(F)
\right) K_{M} \\
&= \left(
K_{x} \cap \overline{U}(F) 
\right) \left(
K_{n_{1} x} \cap U(F)
\right) K_{M}.
\end{align*}
Hence, there exists $k_{M} \in K_{M}$ such that
\[
n_{2} n_{1}^{-1} k_{M} \in \Nheart \cdot K_{M} \cap \overline{U}(F) \cdot U(F)
\subset
N_{G}(M)(F) \cap \overline{U}(F) \cdot U(F).
\]
According to Lemma~\ref{NcapUbarUtrivial},
we have
$
n_{2} n_{1}^{-1} k_{M} = 1
$.
Thus, we obtain that
$
n_{1} n_{2}^{-1} = k_{M} \in K_{M}
$.
\end{proof}

From now on, we suppose Axiom~\ref{axiombijectionofdoublecoset} in order to turn the injection in Proposition~\ref{propositiondoublecosetinjection} into a bijection.
\begin{corollary}
\label{corollarybijectiondoublecosetandweylgroup}
Let $x \in \cA_{\gen}$ and assume Axioms \ref{axiomaboutHNheartandK} and \ref{axiombijectionofdoublecoset}. Then the inclusion 
\[
\Nheart \subset I_{G(F)}(\rho_{x})
\]
induces a bijection
\[
\Nheart / \left(
\Nheart  \cap K_{M}
\right) \simeq K_{x} \backslash I_{G(F)}(\rho_{x}) / K_{x}.
\]
\end{corollary}
\begin{proof}
The corollary follows from Axiom~\ref{axiombijectionofdoublecoset} and Proposition~\ref{propositiondoublecosetinjection}.
\end{proof}

\begin{definition}
\label{definitionofWheart}
	We define the group $\Wheart$ as
\[
\index{notation-ax}{WrhoMx@$\Wheart$}
\Wheart = \Nheart / \left(
\Nheart  \cap K_{M}
\right) .
\]
\end{definition}
Since $K_{M} \subset M(F)_{x_{0}}$ fixes every point of $\cA_{x_{0}}$, the action of $N_{G}(M)(F)_{[x_{0}]_{M}}$ on $\cA_{x_{0}}$ induces an action of $\Wheart$ on $\cA_{x_{0}}$.
\begin{remark}
\label{remarkaboutfaithful}
Since the kernel of the action of $N_{G}(M)(F)_{[x_{0}]_{M}}$ on $\cA_{x_{0}}$ is the group $M(F)_{x_{0}}$, the action of $\Wheart$ on $\cA_{x_{0}}$ is faithful if and only if
\[
\Nheart  \cap K_{M} = \Nheart \cap M(F)_{x_{0}}.
\]
In particular, if $K_{M} = M(F)_{x_{0}}$, then the action of $\Wheart$ on $\cA_{x_{0}}$ is faithful.
\end{remark}
\begin{lemma}
\label{lemmaweylactionproperly}
The group $\Wheart$ acts on $\cA_{x_{0}}$ properly, that is, for all compact subsets $C_{1}$ and $C_{2}$ of $\cA_{x_{0}}$, the set
\(
\{
w \in \Wheart \mid w(C_{1}) \cap C_{2} \neq \emptyset
\}
\)
is finite.
\end{lemma}
\begin{proof}
Since $\Nheart \subset N_{G}(M)(F)_{[x_{0}]_{M}}$, and the quotient $M(F)_{x_{0}} / K_{M}$ is finite, it suffices to show that the group 
$
N_{G}(M)(F)_{[x_{0}]_{M}} / M(F)_{x_{0}}
$
 acts properly on $\cA_{x_{0}}$.
Let $S$ be a maximal split torus of $M$ such that
$
x_{0} \in \cA(G, S, F)
$. 
Let $Z_{G}(S)(F)_{\cpt}$ denote the maximal compact subgroup of $Z_{G}(S)(F)$. 
 According to Lemma~\ref{lemmaMtoS}, we have
 \begin{align*}
 N_{G}(M)(F)_{[x_{0}]_{M}} / M(F)_{x_{0}} & \simeq \left(
 N_{G}(M)(F)_{[x_{0}]_{M}} \cap N_{G}(S)(F)
 \right) / \left(
 M(F)_{x_{0}} \cap N_{G}(S)(F)
 \right) \\
 & \twoheadleftarrow \left(
 N_{G}(M)(F)_{[x_{0}]_{M}} \cap N_{G}(S)(F)
 \right) / Z_{G}(S)(F)_{\cpt}.
 \end{align*}
Then the lemma follows from the fact that the Iwahori--Weyl group $N_{G}(S)(F) / Z_{G}(S)(F)_{\cpt}$ acts properly on the apartment $\cA(G, S, F)$ of $S$.
Although the fact is well-known, we record the proof of it for completeness.
Since the finite Weyl group $N_{G}(S)(F)/Z_{G}(S)(F)$ is finite, it suffices to show that the group $Z_{G}(S)(F)/Z_{G}(S)(F)_{\cpt}$ acts properly on $\cA(G, S, F)$.
According to \cite[Proposition~6.2.4]{KalethaPrasad}, for $z \in Z_{G}(S)(F)$ and $x \in \cA(G, S, F)$, we have 
$
z x = x + \nu(z)
$, where $\nu(z)$ is the element of $X_{*}(S) \otimes_{\bZ} \bR$ defined by 
\(
\chi(\nu(z)) = - \ord(\chi(z))
\)
for all $\chi \in X^{*}(S)$.
Thus, the claim follows from the fact that $\nu$ identifies the group $Z_{G}(S)(F)/Z_{G}(S)(F)_{\cpt}$ with a lattice in $X_{*}(S) \otimes_{\bZ} \bR$.
\end{proof}

For 
$w \in \Wheart$,
we define
\index{notation-ax}{H(Grw@$\cH(G(F),\rho_x)_w$}
$\cH(G(F),\rho_x)_w$
to be $\cH(G(F),\rho_x)_n$
for any representative $n$ of $w$ in 
$\Nheart$, i.e., the subspace of $\cH(G(F),\rho_x)$ consisting of functions whose support is contained in $K_xnK_x$.

\begin{proposition}
\label{propositionvectorspacedecomposition}
Let $x \in \cA_{\gen}$ and assume Axioms \ref{axiomaboutHNheartandK} and \ref{axiombijectionofdoublecoset}.
Then for each $w \in \nobreak \Wheart$, the subspace $\cH(G(F), \rho_{x})_{w}$
is one dimensional.
Moreover, we have
\[
\cH(G(F), \rho_{x}) = \bigoplus_{w \in \Wheart} \cH(G(F), \rho_{x})_{w}
\]
as a vector space.
\end{proposition}
\begin{proof}
Let $w \in \Wheart$.
We fix a lift $n$ of $w$ in $\Nheart$.
Then the map $\varphi \mapsto \varphi(n)$ defines an isomorphism of vector spaces
\[
\cH(G(F), \rho_{x})_{w} \isoarrow
\Hom_{K_{x} \cap ^n\!K_{x}}(^n\!\rho_{x}, \rho_{x}).
\]
Since $(K_{x}, \rho_{x})$ is a quasi-$G$-cover of $(K_{M}, \rho_{M})$, and $(^n\!K_{x}, {^n\!\rho_{x}})$ is a quasi-$G$-cover of $(K_{M}, {^n\!\rho_{M}})$, Lemma~\ref{lemmacoverintertwining} implies that
\begin{align*}
\Hom_{K_{x} \cap ^n\!K_{x}} (^n\!\rho_{x}, \rho_{x}) &= \Hom_{K_{M}}(^n\!\rho_{M}, \rho_{M}).
\end{align*}
Since $n \in N_{G(F)}(\rho_{M})$ and the representation $\rho_{M}$ is irreducible, the right-hand side is one dimensional.
The last claim follows from \eqref{vectorspacedecomposition} and Corollary~\ref{corollarybijectiondoublecosetandweylgroup}.
\end{proof}

\subsection{Intertwining operators and a basis of the Hecke algebra} \label{subsec:intertwiningop}
We keep the notation from the previous subsection including the assumption of Axiom \ref{axiomaboutHNheartandK}, but we do not assume Axiom~\ref{axiombijectionofdoublecoset} until the last corollary, Corollary~\ref{corollaryvectorspacedecompositionexplicitver}.
We will define a non-zero element $\varphi_{x, w} \in \cH(G(F), \rho_{x})_{w}$ for every $x \in \cA_{\gen}$ and $w \in \Wheart$ below.
The element $\varphi_{x, w}$ will correspond to the composition of two intertwining operators via the isomorphism in \eqref{heckevsend}. The first intertwining operator
\[
\index{notation-ax}{Thetayx@$\Theta_{y \mid x}$} 
\Theta_{y \mid x} \colon \ind_{K_{x}}^{G(F)} (\rho_{x}) \rightarrow \ind_{K_{y}}^{G(F)} (\rho_{y})
\]
for $x, y \in \cA_{\gen}$ is defined as follows.
We note that by Lemma~\ref{lemmathetaonintersection}
and Axiom \ref{axiomaboutHNheartandK}\eqref{axiomaboutHNheartandKrhoxandthetax},
for every $f \in \ind_{K_{x}}^{G(F)} \left(V_{\rho_{x}}\right)$ and $g \in G(F)$, the element $\theta_{y}(k) \cdot f(k^{-1} g)$ for $k \in K_{y, +}$ only depends on the image $[k]$ of $k$ in the quotient $K_{y, +}/ \left(K_{x, +} \cap K_{y, +} \right)$.
For $f \in \ind_{K_{x}}^{G(F)} \left(V_{\rho_{x}}\right)$, we define the function
\[
\Theta_{y \mid x}(f) \colon G(F) \rightarrow V_{\rho_{y}}
\]
by
\[
\left(\Theta_{y \mid x}(f)\right)(g) = 
\abs{
K_{y, +}
/
\left(
K_{x, +} \cap K_{y, +}
\right)
} ^{-1}
\sum_{[k] \in K_{y, +} / \left(
K_{x, +} \cap K_{y, +}
\right)} \theta_{y}(k) \cdot f(k^{-1} g).
\]
Here, we regard $\abs{
K_{y, +}
/
\left(
K_{x, +} \cap K_{y, +}
\right)
}$ as an element of $\Coeff$, that is non-zero by Lemma~\ref{lemmaquotientvswithquotientofunipotentradical} and Lemma~\ref{lemmaquotientnonzeromodell}.
\begin{remark}
\label{rem:whenisthetaanintegral}
If the characteristic $\ell$ of $\Coeff$
is zero or the group $K_{y, +}$ is a pro-$p$-group,
then we can write $\Theta_{y \mid x}$ as
\[
\left(\Theta_{y \mid x}(f)\right)(g) =
\int_{K_{y, +}} \theta_{y}(k) \cdot f(k^{-1} g) \, dk
\]
for $f \in \ind_{K_{x}}^{G(F)} \left(V_{\rho_{x}}\right)$ and $g \in G(F)$,
where $dk$ denotes the Haar measure on $K_{y, +}$ such that the volume of $K_{y, +}$ is one.
In our application of the theory to the setting of
\cite{HAIKY},
the group $K_{y, +}$ will be a pro-$p$-group,
see \cite[\eqref{HAIKY-definitionofcompactopensubgroupsKimYucase}]{HAIKY}.
\end{remark}
\begin{lemma} \label{lemma:intertwining-operator}
The map $\Theta_{y \mid x}$ defines a $G(F)$-equivariant map 
\[
\Theta_{y \mid x} \colon \ind_{K_{x}}^{G(F)} (\rho_{x}) \rightarrow \ind_{K_{y}}^{G(F)} (\rho_{y}).
\]
\end{lemma}
\begin{proof}
We will prove that 
$
\Theta_{y \mid x}(f) \in \ind_{K_{y}}^{G(F)} \left(V_{\rho_{y}}\right)
$
for all $f \in \ind_{K_{x}}^{G(F)} \left(V_{\rho_{x}}\right)$.
Since we have $K_{y} = K_{M} \cdot K_{y, +}$ and $\rho_{y} \restriction_{K_{M}} = \rho_{M}$, it suffices to show that
\[
\left(
\Theta_{y \mid x}(f) 
\right)(k_{M} g) = 
\rho_{M}(k_{M})
\left( \left(
\Theta_{y \mid x}(f) 
\right)(g)
\right)
\]
and
\[
\left(
\Theta_{y \mid x}(f) 
\right)(k_{+} g) = \rho_{y}(k_{+}) \left(
\left(
\Theta_{y \mid x}(f) 
\right)(g)
\right)
\]
for all $g \in G(F)$, $k_{M} \in K_{M}$, and $k_{+} \in K_{y, +}$.
Let $g \in G(F)$ and $k_{M} \in K_{M}$.
Since we have $K_{M} \subset K_{x} \cap K_{y}$ and the group $K_{x}$ (resp.\ $K_{y}$) normalizes $K_{x, +}$ (resp.\ $K_{y, +}$), the group $K_{M}$ normalizes $K_{x, +}$ and $K_{y, +}$.
Moreover, since $K_{y}$ normalizes $K_{y, +}$ and the restriction of $\rho_{y}$ to $K_{y, +}$ is $\theta_{y}$-isotypic, the group $K_{y}$ normalizes $\theta_{y}$.
Hence, we have
\begin{align*}
\left(
\Theta_{y \mid x}(f) 
\right)(k_{M} g) &= \abs{
K_{y, +}
/
\left(
K_{x, +} \cap K_{y, +}
\right)
} ^{-1}
\sum_{[k] \in K_{y, +} / \left(
K_{x, +} \cap K_{y, +}
\right)} \theta_{y}(k) \cdot f(k^{-1} k_{M} g)\\
&= \abs{
K_{y, +}
/
\left(
K_{x, +} \cap K_{y, +}
\right)
} ^{-1}
\sum_{[k] \in K_{y, +} / \left(
K_{x, +} \cap K_{y, +}
\right)} \theta_{y}(k) \cdot f(k_{M} k_{M}^{-1} k^{-1} k_{M} g) \\
&= \abs{
K_{y, +}
/
\left(
K_{x, +} \cap K_{y, +}
\right)
} ^{-1}
\sum_{[k] \in K_{y, +} / \left(
K_{x, +} \cap K_{y, +}
\right)} \theta_{y}(k) \cdot \rho_{M}(k_{M}) \left(
f(k_{M}^{-1} k^{-1} k_{M}g)
\right) \\
&= \rho_{M}(k_{M}) \left(
\abs{
K_{y, +}
/
\left(
K_{x, +} \cap K_{y, +}
\right)
} ^{-1}
\sum_{[k] \in K_{y, +} / \left(
K_{x, +} \cap K_{y, +}
\right)} \theta_{y}(k) \cdot f(k_{M}^{-1} k^{-1} k_{M}g)
\right) \\
&= \rho_{M}(k_{M}) \left(
\abs{
K_{y, +}
/
\left(
K_{x, +} \cap K_{y, +}
\right)
} ^{-1}
\sum_{[k] \in K_{y, +} / \left(
K_{x, +} \cap K_{y, +}
\right)} \theta_{y}(k) \cdot f((k_{M}^{-1} k k_{M})^{-1} g)
\right) \\
&= \rho_{M}(k_{M}) \left(
\abs{
K_{y, +}
/
\left(
K_{x, +} \cap K_{y, +}
\right)
} ^{-1}
\sum_{[k] \in K_{y, +} / \left(
K_{x, +} \cap K_{y, +}
\right)} \theta_{y}(k_{M} k k_{M}^{-1}) \cdot f(k^{-1} g)
\right) \\
&= \rho_{M}(k_{M}) \left(
\abs{
K_{y, +}
/
\left(
K_{x, +} \cap K_{y, +}
\right)
} ^{-1}
\sum_{[k] \in K_{y, +} / \left(
K_{x, +} \cap K_{y, +}
\right)} \theta_{y}(k) \cdot f(k^{-1} g)
\right) \\
&= \rho_{M}(k_{M}) \left( 
\left(
\Theta_{y \mid x}(f) 
\right)(g)
\right).
\end{align*}
Moreover, the definition of $\Theta_{y \mid x}$ implies that
\[
\left(
\Theta_{y \mid x}(f) 
\right)(k_{+} g) = \theta_{y}(k_{+}) \cdot \left(
\Theta_{y \mid x}(f) 
\right)(g) 
= \rho_{y}(k_{+}) \left( 
\left(
\Theta_{y \mid x}(f) 
\right)(g)
\right)
\]
for all $g \in G(F)$ and $k_{+} \in K_{y, +}$.

The claim that $\Theta_{y \mid x}$ is $G(F)$-equivariant follows from the definition of it.
\end{proof}
We can also write $\Theta_{y \mid x}$ as follows.
\begin{lemma}
\label{lemmarewriteTheta}
Let $x, y \in \cA_{\gen}$ and $U \in \cU(M)$ such that
\[
K_{x} \cap U(F) \subseteq K_{y} \cap U(F)
\quad 
\text{ and }
\quad
 K_{y} \cap \overline{U}(F) \subseteq K_{x} \cap \overline{U}(F).
\]
Then we have
\begin{align*}
\left(\Theta_{y \mid x}(f)\right)(g) =
\int_{K_{y} \cap U(F)} f(u^{-1}g)\, du
\end{align*}
for all $f \in \ind_{K_{x}}^{G(F)} \left(V_{\rho_{x}} \right)$ and $g \in G(F)$, where $du$ denotes the Haar measure on the pro-$p$-group $K_{y} \cap U(F)$ such that the volume of $K_{y} \cap U(F)$ is one.
\end{lemma}
\begin{proof}
Since the representation $\rho_{y}$ is trivial on the group $K_{y} \cap U(F)$ and the restriction of $\rho_{y}$ to $K_{y, +}$ is $\theta_{y}$-isotypic, the character $\theta_{y}$ is trivial on the group $K_{y, +} \cap U(F) = K_{y} \cap U(F)$.
Hence, by using Lemma~\ref{lemmaquotientvswithquotientofunipotentradical}, we obtain that
\begin{align*}
\left(\Theta_{y \mid x}(f)\right)(g) &=
\abs{
K_{y, +}
/
\left(
K_{x, +} \cap K_{y, +}
\right)
} ^{-1}
\sum_{[k] \in K_{y, +} / \left(
K_{x, +} \cap K_{y, +}
\right)} \theta_{y}(k) \cdot f(k^{-1} g) \\
&=
\abs{
\left(
K_{y} \cap U(F)
\right) / \left(
K_{x} \cap U(F)
\right)
} ^{-1}
\sum_{[u] \in \left(
K_{y} \cap U(F)
\right) / \left(
K_{x} \cap U(F)
\right)} \theta_{y}(u) \cdot f(u^{-1} g) \\
&= \int_{K_{y} \cap U(F)} \theta_{y}(u) \cdot f(u^{-1}g)\, du \\
&= \int_{K_{y} \cap U(F)} f(u^{-1}g)\, du.
\qedhere
\end{align*}
\end{proof}
\begin{corollary}
\label{corollarycalculationofconstantterm}
Let $x, y \in \cA_{\gen}$ and assume Axiom \ref{axiomaboutHNheartandK}.
For $v \in V_{\rho_{x}}$, we define $f_{v} \in \ind_{K_{x}}^{G(F)} \left( V_{\rho_{x}} \right)$ as
\[
f_{v}(g) =
\begin{cases}
\rho_{x}(g) (v) & (g \in K_{x}), \\
0 & \text{(otherwise)}.
\end{cases}
\]
Then we have 
\begin{align*}
\left(
\left(
\Theta_{x \mid y} \circ \Theta_{y \mid x}
\right)(f_{v})
\right)(1) = \abs{
K_{y}/\left(
K_{x} \cap K_{y}
\right)
}^{-1} v,
\end{align*} 
where we regard $\abs{
K_{y}/\left(
K_{x} \cap K_{y}
\right)
}$
as an element of $\Coeff$.
This element is invertible by Lemma~\ref{lemmaquotientnonzeromodell}.
In particular, $\Theta_{y \mid x}$ is non-zero.
\end{corollary}
\begin{proof}
By Axiom~\ref{axiomaboutHNheartandK}\eqref{axiomaboutHNheartandKabotgoodunipotentradical}, there exists $U \in \cU(M)$ such that
\[
K_{x} \cap U(F) \subseteq K_{y} \cap U(F)
\quad 
\text{ and }
\quad
 K_{y} \cap \overline{U}(F) \subseteq K_{x} \cap \overline{U}(F).
\]
Then according to Lemma~\ref{lemmaquotientvswithquotientofunipotentradical} and Lemma~\ref{lemmarewriteTheta}, we have
\begin{align*}
\left(
\left(
\Theta_{x \mid y} \circ \Theta_{y \mid x}
\right)(f_{v})
\right)(1) &= \int_{K_{x} \cap \overline{U}(F)} \left(
\Theta_{y \mid x}(f_{v})
\right)(\overline{u}^{-1})\, d\overline{u} \\
&= \int_{K_{x} \cap \overline{U}(F)}
\int_{K_{y} \cap U(F)}
f_{v}(u^{-1} \overline{u}^{-1})
\,du
\,d\overline{u} \\
&= \int_{K_{x} \cap \overline{U}(F)}
\int_{K_{x} \cap K_{y} \cap U(F)}
\rho_{x}(u^{-1} \overline{u}^{-1}) (v)
\,du
\,d\overline{u} \\
&= \int_{K_{x} \cap \overline{U}(F)}
\int_{K_{x} \cap U(F)}
v 
\,du
\,d\overline{u} \\
&= \abs{\left(K_{y} \cap U(F)\right)/\left(K_{x} \cap U(F) \right)}^{-1} v \\
&= \abs{
K_{y}/\left(
K_{x} \cap K_{y}
\right)
}^{-1} v
\qedhere
\end{align*} 
\end{proof}
We will prove a transitivity property of $\Theta_{y \mid x}$.
\begin{proposition}
\label{transitivityofthetadver} 
We assume Axiom \ref{axiomaboutHNheartandK}.
Let $x, y, z \in \cA_{\gen}$ such that
\[
d(x, y) + d(y, z) = d(x, z).
\]
Then we have
\[
\Theta_{z \mid y} \circ \Theta_{y \mid x} = \Theta_{z \mid x}.
\]
\end{proposition}
\begin{proof}
By using Axiom~\ref{axiomaboutHNheartandK}\eqref{axiomaboutHNheartandKabotgoodunipotentradical}, there exists $U \in \cU(M)$ such that
\[
K_{x} \cap U(F) \subset K_{y} \cap U(F) \subset K_{z} \cap U(F)
\]
and
\[
K_{z} \cap \overline{U}(F)  \subset K_{y} \cap \overline{U}(F) \subset K_{x} \cap \overline{U}(F).
\]
Then according to Lemma~\ref{lemmarewriteTheta}, for $f \in \ind_{K_{x}}^{G(F)}\left( V_{\rho_{x}} \right)$ and $g \in G(F)$, we have
\begin{align*}
\left(
\left(
\Theta_{z \mid y} \circ \Theta_{y \mid x}
\right)(f)
\right)(g) &= \int_{K_{z} \cap U(F)} \left(
\Theta_{y \mid x}(f)
\right)(u_{z}^{-1}g)\, du_{z} \\
&= \int_{K_{z} \cap U(F)}
\int_{K_{y} \cap U(F)}
f(u_{y}^{-1} u_{z}^{-1} g)\, du_{y}\, du_{z} \\
&= \int_{K_{y} \cap U(F)}
\int_{K_{z} \cap U(F)}
f(u_{y}^{-1} u_{z}^{-1} g)\, du_{z}\, du_{y} \\
&= \int_{K_{y} \cap U(F)}
\int_{K_{z} \cap U(F)}
f(u_{z}^{-1} g)\, du_{z}\, du_{y} \\
&= \int_{K_{z} \cap U(F)}
f(u_{z}^{-1} g)\, du_{z} \\
&= \left(
\Theta_{z \mid x}(f)
\right)(g).
\qedhere
\end{align*}
\end{proof}
We will prove a similar result under a weaker condition below (see Proposition~\ref{transitivityofThetaGammaver}).
To do this, we define the notion of $\cK$-relevance
for an affine hyperplane $H \in \mathfrak{H}$.
The group generated by the reflections across these $\cK$-relevant hyperplanes will form (under additional assumptions) a basis for a subalgebra of $\cH(G(F),\rho_{x_0})$ that is isomorphic to an affine Hecke algebra.
(See Theorem \ref{theoremstructureofhecke}.)

\begin{definition}\label{definitionKrelevant}
\index{terminology-ax}{K-relevant@$\cK$-relevant}
We say that an affine hyperplane $H \in \mathfrak{H}$ is \emph{$\cK$-relevant} if there exists $x, y \in \cA_{\gen}$ such that 
$
\mathfrak{H}_{x, y} = \left\{
H
\right\}$ and
\begin{equation}
\label{eqn:compositionnotscalar}
\Theta_{x \mid y} \circ \Theta_{y \mid x}
\notin \Coeff \cdot \id_{\ind_{K_{x}}^{G(F)}(\rho_{x})}.
\end{equation}
We denote the set of 
$\cK$-relevant hyperplanes in $\mathfrak{H}$ by $\mathfrak{H}_{\Krel}$
\index{notation-ax}{H_Krel@$\mathfrak{H}_{\Krel}$}.
\end{definition}
We will later see (in Corollary \ref{corollaryweylinvariance})
that the action of 
$\Nheart$
on $\cA_{x_0}$ preserves the set $\mathfrak{H}_{\Krel}$.

\begin{remark}
	By the definition of $\mathfrak{H}_{\Krel}$,
	an element $H \in \mathfrak{H}$ is contained in $\mathfrak{H}_{\Krel}$ if and only if there exists $x, y \in \cA_{\gen}$ such that 
	$
	\mathfrak{H}_{x, y} = \left\{
	H
	\right\}$ and
	\eqref{eqn:compositionnotscalar} holds.
	However, under additional assumptions, we will see in Proposition~\ref{propositionkreloodequivalentforallorthereexists} below that if $H \in \mathfrak{H}_{\Krel}$, then we have
	that \eqref{eqn:compositionnotscalar} holds
	for \emph{all} $x, y \in \cA_{\gen}$ such that 
	$
	\mathfrak{H}_{x, y} = \left\{
	H
	\right\}$.
\end{remark}

The $\cK$-relevant hyperplanes are those for which one has
a more complicated description of
$\Theta_{x \mid y} \circ \Theta_{y \mid x}$ for $x$ and $y$ on opposite sides of the hyperplane (see Lemma~\ref{lemmaphivstheta}, Proposition~\ref{propositionstrongversionofquadraticrelation}).
On the other hand, if $H \in \mathfrak{H}$ is not $\cK$-relevant, then, according to the definition combined with Corollary~\ref{corollarycalculationofconstantterm}, we have
\[
\Theta_{x \mid y} \circ \Theta_{y \mid x} = \abs{
K_{y}/\left(
K_{x} \cap K_{y}
\right)
}^{-1} \cdot \id_{\ind_{K_{x}}^{G(F)}(\rho_{x})}
\]
for all $x, y \in \cA_{\gen}$ such that 
$
\mathfrak{H}_{x, y} = \left\{
H
\right\}$.

We now normalize the operators $\Theta_{x \mid y}$ so that 
\[
\Theta^{\normal}_{x \mid y} \circ \Theta^{\normal}_{y \mid x}
= \id_{\ind_{K_{x}}^{G(F)}(\rho_{x})}
\]
in this case (see Lemma \ref{normalizedlemmalength1case}).
Recall that for an element $c = p^{n} \in \Coeff$ with $n \in \bZ$, we write $c^{1/2} \coloneqq \left(p^{1/2}\right)^{n}$, where $p^{1/2}$ denotes the fixed square root of $p$ in $\Coeff$.

\begin{definition}
	\label{definitionnormalize}
	For $x, y \in \cA_{\gen}$, we define the $G(F)$-equivariant map 
	\[
	\index{notation-ax}{Thetayxnorm@$\Theta^{\normal}_{y \mid x}$} 
	\Theta^{\normal}_{y \mid x}
	\colon
	\ind_{K_{x}}^{G(F)} (\rho_{x})
	\longrightarrow
	\ind_{K_{y}}^{G(F)} (\rho_{y})
	\]
	by
	\[
	\Theta^{\normal}_{y \mid x} = \abs{K_{y}/ \left(
		K_{x} \cap K_{y}
		\right)
	}^{1/2} \cdot \Theta_{y \mid x}.
	\]
\end{definition}

\begin{lemma}
\label{lemmaequalityabouttheabsoftwoquotients}
Suppose that $H \in \mathfrak{H}$ is not $\cK$-relevant.
Then
for all $x, y \in \cA_{\gen}$ such that 
$\mathfrak{H}_{x, y} = \{ H \}$,
we have
\[
\abs{K_{x}/\left( K_{x} \cap K_{y} \right)} = \abs{K_{y}/\left( K_{x} \cap K_{y} \right)}.
\]
\end{lemma}
Note that the image of this number in $\Coeff$ is invertible by
Lemma \ref{lemmaquotientnonzeromodell}.
\begin{proof}
According to Corollary~\ref{corollarycalculationofconstantterm}, we have
\[
\Theta_{x \mid y} \circ \Theta_{y \mid x} = \abs{
K_{y}/\left(
K_{x} \cap K_{y}
\right)
}^{-1} \cdot \id_{\ind_{K_{x}}^{G(F)}(\rho_{x})}
\]
and
\[
\Theta_{y \mid x} \circ \Theta_{x \mid y} = \abs{
K_{x}/\left(
K_{x} \cap K_{y}
\right)
}^{-1} \cdot \id_{\ind_{K_{y}}^{G(F)}(\rho_{y})}.
\]
Hence, we have
\begin{align*}
\Theta_{y \mid x} 
&= \Theta_{y \mid x} \circ \id_{\ind_{K_{x}}^{G(F)}(\rho_{x})} \\
&= \abs{K_{y}/\left( K_{x} \cap K_{y} \right)} \Theta_{y \mid x} \circ \left(
\Theta_{x \mid y} \circ \Theta_{y \mid x}
\right) \\
&= 
\abs{K_{y}/\left( K_{x} \cap K_{y} \right)}
\left(
\Theta_{y \mid x} \circ \Theta_{x \mid y} 
\right) \circ \Theta_{y \mid x} \\
&= \abs{K_{y}/\left( K_{x} \cap K_{y} \right)} \abs{K_{x}/\left( K_{x} \cap K_{y} \right)}^{-1} \cdot \id_{\ind_{K_{y}}^{G(F)}(\rho_{y})} \circ  \Theta_{y \mid x} \\
&=   \abs{K_{y}/\left( K_{x} \cap K_{y} \right)} \abs{K_{x}/\left( K_{x} \cap K_{y} \right)}^{-1} \Theta_{y \mid x}.
\end{align*}
Since $\Theta_{y \mid x}$ is non-zero, we obtain the lemma.
\end{proof}

\begin{lemma}
	\label{normalizedlemmalength1case}
	Suppose that $H \in \mathfrak{H}$ is not $\cK$-relevant.
	Then we have
	\[
	\Theta^{\normal}_{x \mid y} \circ \Theta^{\normal}_{y \mid x}
	= \id_{\ind_{K_{x}}^{G(F)}(\rho_{x})}
	\]
	for all $x, y \in \cA_{\gen}$ such that 
	$
	\mathfrak{H}_{x, y} = \left\{
	H
	\right\}$.
\end{lemma}
\begin{proof}
	According to Lemma~\ref{lemmaequalityabouttheabsoftwoquotients}, we have
	$
	\abs{K_{x}/\left( K_{x} \cap K_{y} \right)} = \abs{K_{y}/\left( K_{x} \cap K_{y} \right)}
	$.
	Then the claim follows from combining this with Corollary~\ref{corollarycalculationofconstantterm} and Definition~\ref{definitionnormalize}.
\end{proof}

As above, we also obtain a transitivity property for the normalized intertwining operators $ \Theta^{\normal}_{x \mid y}$. 
\begin{lemma}
\label{lemmanormalizedtransitivityofthetadver}
Let $x, y, z \in \cA_{\gen}$.
Suppose that 
\[
d(x, y) + d(y, z) = d(x, z).
\]
Then we have
\[
\Theta^{\normal}_{z \mid y}  \circ \Theta^{\normal}_{y \mid x} = \Theta^{\normal}_{z \mid x}.
\]
\end{lemma}
\begin{proof}
By using Axiom~\ref{axiomaboutHNheartandK}\eqref{axiomaboutHNheartandKabotgoodunipotentradical}, we have $U \in \cU(M)$ such that
\[
K_{x} \cap U(F) \subset K_{y} \cap U(F) \subset K_{z} \cap U(F)
\]
and
\[
K_{z} \cap \overline{U}(F)  \subset K_{y} \cap \overline{U}(F) \subset K_{x} \cap \overline{U}(F).
\]
According to Lemma~\ref{lemmaquotientvswithquotientofunipotentradical}, we have
\[
\begin{cases}
\abs{K_{y}/\left(
K_{x} \cap K_{y}
\right)} &= \abs{\left(K_{y} \cap U(F)\right)/\left(K_{x} \cap U(F)\right)}, \\
\abs{K_{z}/ \left(
K_{y} \cap K_{z}
\right)} &= \abs{\left(K_{z} \cap U(F)\right)/\left(K_{y} \cap U(F)\right)}, \\
\abs{K_{z}/ \left(
K_{x} \cap K_{z}
\right)} &= \abs{\left(K_{z} \cap U(F)\right)/\left(K_{x} \cap U(F)\right)}.
\end{cases}
\]
Thus, we obtain
\begin{align*}
\abs{K_{y}/\left(
K_{x} \cap K_{y}
\right)} \abs{K_{z}/ \left(
K_{y} \cap K_{z}
\right)} 
&=
\abs{\left(K_{y} \cap U(F)\right)/\left(K_{x} \cap U(F)\right)} \abs{\left(K_{z} \cap U(F)\right)/\left(K_{y} \cap U(F)\right)} \\
&= \abs{\left(K_{z} \cap U(F)\right)/\left(K_{x} \cap U(F)\right)} \\
&= \abs{K_{z}/ \left(
K_{x} \cap K_{z}
\right)}.
\end{align*}
Combining it with Proposition~\ref{transitivityofthetadver}, we obtain the claim.
\end{proof}
The transitivity property holds under a weaker assumption that only takes the
relevant
$\cK$-hyperplanes into account. In order to state the weaker assumption and stronger result (see Proposition \ref{transitivityofThetaGammaver}), we first need to introduce some notation.

\begin{notation}
	For $x, y \in \cA_{\gen}$, we define the subset
$\mathfrak{H}_{\Krel; x, y}$\index{notation-ax}{H_Krelxy@$\mathfrak{H}_{\Krel; x, y}$}
of $\mathfrak{H}_{\Krel}$ by
	\[
	\mathfrak{H}_{\Krel; x, y} = \left\{
	H \in \mathfrak{H}_{\Krel} \mid \text{$x$ and $y$ are on opposite sides of $H$}
	\right\},
	\]
	and write\index{notation-ax}{d_Krel(@$d_{\Krel}(\phantom{x},\phantom{y})$}
	$
	d_{\Krel}(x, y) = \# \mathfrak{H}_{\Krel; x, y}
	$.
\end{notation}

\begin{lemma}
	\label{lemmaaboutdistancefunctionrel}
	Let $x, y, z \in \cA_{\gen}$.
	Then we have
	\[
	d_{\Krel}(x, y) + d_{\Krel}(y, z) \ge d_{\Krel}(x, z),
	\]
	and the following conditions are equivalent:
\[
	\begin{array}{l c l}
	\textup{(a)} \hspace{1em}
		&d_{\Krel}(x, y) + d_{\Krel}(y, z) = d_{\Krel}(x, z), \\
	\textup{(b)}
		&\mathfrak{H}_{\Krel; x, y}, \mathfrak{H}_{\Krel; y, z} \subset \mathfrak{H}_{\Krel; x, z}, \\
	\textup{(c)}
		&\mathfrak{H}_{\Krel; x, y} \cap \mathfrak{H}_{\Krel; y, z} = \emptyset.
	\end{array}
\]
	Moreover, the condition 
	\[
	d(x, y) + d(y, z) = d(x, z)
	\]
	implies 
	\[
	d_{\Krel}(x, y) + d_{\Krel}(y, z) = d_{\Krel}(x, z).
	\]
\end{lemma}

\begin{proof}
	The first and second claim follow from the same arguments as in the proof of Lemma~\ref{lemmaaboutdistancefunctionnoindex}.
	The last claim follows from the fact 
	$
	\mathfrak{H}_{\Krel; x, y} = \mathfrak{H}_{\Krel} \cap \mathfrak{H}_{x, y}
	$.
\end{proof}

\begin{proposition}
\label{transitivityofThetaGammaver}
Let $x, y, z \in \cA_{\gen}$ and assume Axiom \ref{axiomaboutHNheartandK}.
Suppose that 
\[
d_{\Krel}(x, y) + d_{\Krel}(y, z) = d_{\Krel}(x, z).
\]
Then we have
\[
\Theta^{\normal}_{z \mid y}  \circ \Theta^{\normal}_{y \mid x}
= \Theta^{\normal}_{z \mid x}.
\]
\end{proposition}
\begin{proof}
\addtocounter{equation}{-1}
\begin{subequations}
We use induction on $d(x, y)$.
If $d(x, y) = 0$, we have
\[
d(x, z) = d(y, z) = d(x, y) + d(y, z).
\]
Hence, the proposition follows from Lemma~\ref{lemmanormalizedtransitivityofthetadver}.
Suppose that $d(x, y) > 0$.
We take $x' \in \cA_{\gen}$ such that $d(x, x') = 1$ and
$
d(x, x') + d(x', y) = d(x, y)
$.
According to Lemma~\ref{lemmanormalizedtransitivityofthetadver}, we have
\begin{align}
\label{followsfromlemmatransitivity}
\Theta^{\normal}_{y \mid x'}  \circ \Theta^{\normal}_{x' \mid x} = \Theta^{\normal}_{y \mid x}.
\end{align}
Recall that we are assuming
\begin{align}
\label{condition1assuming}
d_{\Krel}(x, y) + d_{\Krel}(y, z) = d_{\Krel}(x, z).
\end{align}
On the other hand, since 
$
d(x, x') + d(x', y) = d(x, y)
$,
according to Lemma~\ref{lemmaaboutdistancefunctionrel}, we also have
\begin{align}
\label{condition2y'def}
d_{\Krel}(x, x') + d_{\Krel}(x', y) = d_{\Krel}(x, y).
\end{align}
Combining \eqref{condition1assuming} with \eqref{condition2y'def}, we obtain
\begin{align*}
d_{\Krel}(x, z) &= d_{\Krel}(x, y) + d_{\Krel}(y, z) \\
&= d_{\Krel}(x, x') + d_{\Krel}(x', y) + d_{\Krel}(y, z) \\
& \ge d_{\Krel}(x, x') + d_{\Krel}(x', z) \\
& \ge d_{\Krel}(x, z).
\end{align*}
Since all of the inequalities above are thus equalities, we have
\begin{align}
\label{equalityusedtoinduction}
d_{\Krel}(x', y) + d_{\Krel}(y, z) = d_{\Krel}(x', z)
\end{align}
and
\begin{align}
\label{equalityusedtofinalstep}
d_{\Krel}(x, x') + d_{\Krel}(x', z) = d_{\Krel}(x, z).
\end{align}
Then \eqref{equalityusedtoinduction} and the induction hypothesis imply that
\begin{align}
\label{inductionhypothesistransitivityuptoconst}
\Theta^{\normal}_{z \mid y}  \circ \Theta^{\normal}_{y \mid x'} = 
\Theta^{\normal}_{z \mid x'}.
\end{align}
Combining \eqref{followsfromlemmatransitivity} with \eqref{inductionhypothesistransitivityuptoconst}, we have
\[
\Theta^{\normal}_{z \mid y}  \circ \Theta^{\normal}_{y \mid x} = \Theta^{\normal}_{z \mid y} \circ \Theta^{\normal}_{y \mid x'}  \circ \Theta^{\normal}_{x' \mid x} 
= \Theta^{\normal}_{z \mid x'} \circ \Theta^{\normal}_{x' \mid x}.
\]
Thus it suffices to show that 
$
\Theta^{\normal}_{z \mid x'}  \circ \Theta^{\normal}_{x' \mid x} =
\Theta^{\normal}_{z \mid x}
$.
Let $H \in \mathfrak{H}$ denote the unique affine hyperplane in $\cA_{x_{0}}$ such that $x$ and $x'$ are on opposite sides of $H$.
If $x'$ and $z$ are on the same side of $H$, then we have
$
\mathfrak{H}_{x, x'} \cap \mathfrak{H}_{x', z} = \emptyset
$.
Then, according to Lemma~\ref{lemmaaboutdistancefunctionnoindex}, we have
$
d(x, x') + d(x', z) = d(x, z)
$.
In this case, the claim follows from Lemma~\ref{lemmanormalizedtransitivityofthetadver}.
Suppose that $x'$ and $z$ are on opposite sides of $H$.
In this case, $x$ and $z$ are on the same side of $H$.
Hence, we obtain
$
\mathfrak{H}_{x', x} \cap \mathfrak{H}_{x, z} = \emptyset
$,
equivalently,
$
d(x', x) + d(x, z) = d(x', z)
$.
Then, according to Lemma~\ref{lemmanormalizedtransitivityofthetadver}, we have
\begin{align}
\label{lemmatransitivityconverse}
\Theta^{\normal}_{z \mid x}  \circ \Theta^{\normal}_{x \mid x'} = \Theta^{\normal}_{z \mid x'}.
\end{align}
We also note that $H \not \in \mathfrak{H}_{\Krel}$.
Indeed, if $H \in \mathfrak{H}_{\Krel}$, we have
\[
H \in \mathfrak{H}_{\rel; x, x'} \cap \mathfrak{H}_{\rel; x', z} \neq \emptyset,
\]
that contradicts \eqref{equalityusedtofinalstep}.
Hence, we obtain that $H$ is not $\cK$-relevant.
Then, according to Lemma~\ref{normalizedlemmalength1case}, we have
\begin{align}
\label{length1caseafterprovingdgamma=0}
\Theta^{\normal}_{x \mid x'} \circ \Theta^{\normal}_{x' \mid x}
= \id_{\ind_{K_{x}}^{G(F)}(\rho_{x})}.
\end{align}
Combining \eqref{lemmatransitivityconverse} with \eqref{length1caseafterprovingdgamma=0}, we obtain
\[
\Theta^{\normal}_{z \mid x'}  \circ \Theta^{\normal}_{x' \mid x} = \Theta^{\normal}_{z \mid x}  \circ \Theta^{\normal}_{x \mid x'} \circ  \Theta^{\normal}_{x' \mid x} 
= \Theta^{\normal}_{z \mid x}.
\qedhere
\]
\end{subequations}
\end{proof}

We record a corollary of Proposition~\ref{transitivityofThetaGammaver}, that is a generalization of Lemma~\ref{lemmaequalityabouttheabsoftwoquotients} and will be used in Section~\ref{subsection:Preserving the anti-involutions} below.
\begin{corollary}
\label{generalizationoflemmaequalityabouttheabsoftwoquotients}
Let $x, y \in \cA_{\gen}$ such that $d_{\Krel}(x, y) = 0$ and assume Axiom \ref{axiomaboutHNheartandK}.
Then we have
\[
\abs{K_{x}/\left( K_{x} \cap K_{y} \right)} = \abs{K_{y}/\left( K_{x} \cap K_{y} \right)}.
\]
\end{corollary}
\begin{proof}
Since $d_{\Krel}(x, y) + d_{\Krel}(y, x) = d_{\Krel}(x, x) = 0$, Proposition~\ref{transitivityofThetaGammaver} implies that 
\(
\Theta^{\normal}_{x \mid y}  \circ \Theta^{\normal}_{y \mid x}
= \id_{\ind_{K_{x}}^{G(F)}(\rho_{x})}.
\)
Combining this with Corollary~\ref{corollarycalculationofconstantterm} and Definition~\ref{definitionnormalize}, we obtain that
\[
\abs{K_{x}/\left( K_{x} \cap K_{y} \right)}^{1/2}
\abs{K_{y}/\left( K_{x} \cap K_{y} \right)}^{1/2}
\abs{K_{y}/\left( K_{x} \cap K_{y} \right)}^{-1} = 1.
\]
Thus, we obtain the claim.
\end{proof}

Now we construct the non-zero element
$\varphi_{x, w} \in \cH(G(F), \rho_{x})_{w}$
for $x \in \cA_{\gen}$ and $w \in \Wheart$.
In order to do so, we make the following choice. 
\begin{choice} \label{choiceofT}
	We fix a family of non-zero elements
\index{notation-ax}{T@$\cT$}
\index{notation-ax}{T n@$T_n$}
\[
\cT =
\left\{
T_{n} \in \Hom_{K_{M}}\left(
^n\!\rho_{M}, \rho_{M}
\right)
\right\}_{n \in \Nheart}
\]
that satisfies the following conditions:
\begin{enumerate}[(1)]
\item
\label{T1istheidentitymap}
We have
$
T_{1} = \id_{\rho_{M}}
$.
\item
\label{Tnk=TnrhoMk}
For all $n \in \Nheart$ and $k \in K_{M} \cap \Nheart$, we have
$
T_{nk} = T_{n} \circ \rho_{M}(k)
$.
\end{enumerate}
\end{choice}
We will later refine our choice of $\cT$.
See Choices
\ref{choice:tw}
and
\ref{choice:star}.

Let $x \in \cA_{\gen}$ and $n \in \Nheart$. 
According to Axiom~\ref{axiomaboutHNheartandK}\eqref{axiomaboutHNheartandKabouttypeKxisaquasiGcover}, the pair $(K_{nx}, \rho_{nx})$ is a quasi-$G$-cover of $(K_{M}, \rho_{M})$.
On the other hand, according to Lemma~\ref{lemmaeffectofconjugationbyn}, the pair $(K_{nx}, {^{n}\!\rho_{x}})$ is a quasi-$G$-cover of $(K_{M}, {^{n}\!\rho_{M}})$.
Hence, according to Lemma~\ref{lemmacoverintertwining}, we have
\[
T_{n} \in \Hom_{K_{M}}(^n\!\rho_{M}, \rho_{M}) = \Hom_{K_{nx}}\left(
^{n}\!\rho_{x}, \rho_{n x}
\right).
\]
This allows us to define a second type of intertwining operator,
whose composition with an appropriate $\Theta^{\mathrm{norm}}_{y \mid x}$ will allow us to define $\varphi_{x,w} \in \cH(G(F), \rho_{x})_{w}$ in Definition \ref{definitionPhixw} below. 
\begin{definition} \label{definiotionofciso}
For $x \in \cA_{\gen}$ and $n \in \Nheart$ we define the isomorphism

\[
\ciso{x}{n}
\colon
\ind_{K_{x}}^{G(F)}(\rho_{x})
\isoarrow
\ind_{K_{n x}}^{G(F)}(\rho_{n x}) 
\]
by
\index{notation-ax}{cxn@$\ciso{x}{n}$}
\[
\ciso{x}{n} (f) \colon g \longmapsto  T_{n} ( f(n^{-1} g))
\]
for $f \in \ind_{K_{x}}^{G(F)}\left(V_{\rho_{x}} \right)$ and $g \in G(F)$.

\end{definition}

\begin{lemma}
\label{lemmathetaccommute}
Let $x, y \in \cA_{\gen}$ and $n \in \Nheart$.
Then we have
\[
\Theta_{n y \mid n x} \circ \ciso{x}{n} = \ciso{y}{n} \circ \Theta_{y \mid x}
\quad
\text{ and }
\quad
\Theta^{\normal}_{n y \mid n x} \circ \ciso{x}{n} = \ciso{y}{n} \circ \Theta^{\normal}_{y \mid x}.
\]
\end{lemma}
\begin{proof}
By the definitions of $\Theta_{n y \mid n x}$ and $\ciso{x}{n}$ we have for $f \in \ind_{K_{x}}^{G(F)}\left( V_{\rho_{x}} \right)$ and $g \in G(F)$ that
\begin{align*}
\left(
\left(
\Theta_{n y \mid n x} \circ \ciso{x}{n}
\right)(f)
\right)(g) &= 
\abs{
K_{n y, +}
/
\left(
K_{n x, +} \cap K_{n y, +}
\right)
} ^{-1}
\sum_{[k] \in K_{n y, +} / \left(
K_{n x, +} \cap K_{n y, +}
\right)} \theta_{n y}(k) \cdot T_{n} (f(n^{-1} k^{-1} g)).
\end{align*}
On the other hand, by the definitions of $\Theta_{y \mid x}$ and $\ciso{y}{n}$ and Axiom~\ref{axiomaboutHNheartandK}\eqref{axiomaboutHNheartandKaboutconjugation}, we have
\begin{align*}
\left(
\left(
\ciso{y}{n} \circ \Theta_{y \mid x}
\right)(f)
\right)(g) 
&= T_{n} \left( 
\abs{
K_{y, +}
/
\left(
K_{x, +} \cap K_{y, +}
\right)
} ^{-1}
\sum_{[k] \in K_{y, +} / \left(
K_{x, +} \cap K_{y, +}
\right)} \theta_{y}(k) \cdot f(k^{-1} n^{-1} g)
\right) \\
&= \abs{
K_{ny, +}
/
\left(
K_{nx, +} \cap K_{ny, +}
\right)
} ^{-1}
\sum_{[k] \in K_{ny, +} / \left(
K_{nx, +} \cap K_{ny, +}
\right)} \theta_{y}(n^{-1} k n) \cdot T_{n} (f(n^{-1} k^{-1} g)).
\end{align*}
Now the first claim follows from Lemma~\ref{lemmathetaintertwine}, and the second claim follows from the first claim, Definition~\ref{definitionnormalize}, and Axiom~\ref{axiomaboutHNheartandK}\eqref{axiomaboutHNheartandKaboutconjugation}.
\end{proof}

\begin{corollary}
\label{corollaryweylinvariance}
Assume Axiom \ref{axiomaboutHNheartandK} and let $H \in \mathfrak{H}_{\Krel}$ and $n \in \Nheart$.
Then we have $n(H) \in \mathfrak{H}_{\Krel}$.
\end{corollary}

\begin{proof}
According to Axiom~\ref{axiomaboutHNheartandK}\eqref{axiomaboutHNheartandKNinvarianceofH}, we have $n(H) \in \mathfrak{H}$.
We will prove that $n(H)$ is $\cK$-relevant.

Since $H \in \mathfrak{H}_{\Krel}$, there exists $x, y \in \cA_{\gen}$ such that $\mathfrak{H}_{x,y} = \{H\}$ and
\[
\Theta_{x \mid y} \circ \Theta_{y \mid x} \notin \Coeff \cdot \id_{\ind_{K_{x}}^{G(F)}(\rho_{x})}.
\]
According to Lemma~\ref{lemmathetaccommute},
we have
\[
\Theta_{n y \mid n x} = \ciso{y}{n} \circ \Theta_{y \mid x} \circ \ciso{x}{n}^{-1}
\quad 
\text{ and }
\quad
\Theta_{n x \mid n y} = \ciso{x}{n} \circ \Theta_{x \mid y} \circ \ciso{y}{n}^{-1}.
\]
Hence, we obtain that
\[
\Theta_{n x \mid n y} \circ \Theta_{n y \mid n x} =
\left(
\ciso{x}{n} \circ \Theta_{x \mid y} \circ \ciso{y}{n}^{-1}
\right) \circ \left(
\ciso{y}{n} \circ \Theta_{y \mid x} \circ \ciso{x}{n}^{-1}
\right) 
= \ciso{x}{n} \circ \Theta_{x \mid y} \circ \Theta_{y \mid x} \circ \ciso{x}{n}^{-1}.
\]
Since
\[
\Theta_{x \mid y} \circ \Theta_{y \mid x} \notin \Coeff \cdot \id_{\ind_{K_{x}}^{G(F)}(\rho_{x})},
\]
we have
\[
\Theta_{n x \mid n y} \circ \Theta_{n y \mid n x} \notin \Coeff \cdot \id_{\ind_{K_{n x}}^{G(F)}(\rho_{n x})}.
\]
Since 
$
\mathfrak{H}_{n x, n y} = n (\mathfrak{H}_{x, y}) = \{n (H) \}
$,
we conclude that $n(H)$ is $\cK$-relevant.
\end{proof}

\begin{lemma} 
Let $x \in \cA_{\gen}$ and $n \in \Nheart$.
Then the isomorphism $\ciso{x}{n}$ only depends on the image of $n$ in $\Wheart$.
\end{lemma}

\begin{proof}
Let $n \in \Nheart$ and $k \in \Nheart \cap K_{M}$.
We will prove that $\ciso{x}{n} = \ciso{x}{nk}$.
Since $K_{M} \subset~ M(F)_{x_{0}}$ acts trivially on $\cA_{x_{0}}$,
we have $nk x = n x$.
Let $f \in \ind_{K_{x}}^{G(F)}\left( V_{\rho_{x}} \right)$ and $g \in G(F)$.
Then we have
\begin{align*}
(\ciso{x}{nk} (f))(g) & = T_{nk} (f(k^{-1} n^{-1} g)) \\
&= \left(
T_{n} \circ \rho_{M}(k)
\right) 
\left(
\rho_{x}(k^{-1}) f(n^{-1} g)
\right) \\
&= \left(
T_{n} \circ \rho_{M}(k)
\right) 
\left(
\rho_{M}(k^{-1}) f(n^{-1} g)
\right) \\
&= T_{n} ( f(n^{-1} g)) \\
&= (\ciso{x}{n} (f))(g).
\end{align*}
Thus, we obtain the lemma.
\end{proof}
As a consequence of the lemma, 
we may write\spacingatend{}
\begin{equation} \label{eq:defofciso}
\index{notation-ax}{cxw@$\ciso{x}{w}$}
\ciso{x}{w} \coloneqq \ciso{x}{n}
\end{equation}
for $x \in \cA_{\gen}$, $w \in \Wheart$, and $n$ any lift of $w$ in $\Nheart$.

\begin{definition} \label{definitionPhixw}
For $x \in \cA_{\gen}$ and $w \in \Wheart$ we define the element 
\[
\index{notation-ax}{Phixw@$\Phi_{x, w}$}
\Phi_{x, w} \in \End_{G(F)} \left(
\ind_{K_{x}}^{G(F)} (\rho_{x})
\right) 
\]
by
\[
\Phi_{x, w} = \ciso{w^{-1} x}{w} \circ \Theta^{\normal}_{w^{-1} x \mid x}
\]
and let $\varphi_{x, w}$
\index{notation-ax}{phixw@$\varphi_{x, w}$}
denote the element of $\cH(G(F), \rho_{x})$ that corresponds to $\Phi_{x, w}$ via the isomorphism in \eqref{heckevsend}.
We write
$\varphi_{w} \coloneqq \varphi_{x_{0}, w}$\index{notation-ax}{phiw@$\varphi_w$}
and
$\Phi_{w} \coloneqq \Phi_{x_{0}, w}$.\index{notation-ax}{Phiw@$\Phi_w$}
\end{definition}
\begin{remark}
Since $\Theta^{\normal}_{w^{-1} x \mid x}$ is non-zero and $\ciso{w^{-1} x}{w}$ is an isomorphism, we obtain that $\Phi_{x, w}$ and $\varphi_{x, w}$ are non-zero. We will see below in Corollary \ref{corollaryPhixwinvertible} that under additional assumptions the endomorphism $\Phi_{x,w}$ is an isomorphism.
\end{remark}

\begin{remark}\label{remarkaboutphix1} 
The elements $\Phi_{x, w}$ and $\varphi_{x, w}$ depend on the choice of the family
\[
\cT =
\left\{
T_{n} \in \Hom_{K_{M}}\left(
^n\!\rho_{M}, \rho_{M}
\right)
\right\}_{n \in \Nheart}
\]
made in Choice \ref{choiceofT} as follows.
Let $w \in \Wheart$ and $c \in \Coeff^{\times}$.
If we replace $T_{n}$ with $c \cdot T_{n}$ for all lifts $n$ of $w$ in $\Nheart$, then the elements $\Phi_{x, w}$ and $\varphi_{x, w}$ are replaced by $c \cdot \Phi_{x, w}$ and $c \cdot \varphi_{x, w}$, respectively.
We have chosen $T_{1} = \id_{\rho_{M}}$ so that the endomorphism $\Phi_{x, 1}$ is the identity map on $\ind_{K_{x}}^{G(F)} (\rho_{x})$ for every $x \in \cA_{\gen}$.
\end{remark}
\begin{lemma}
\label{supportofphiixw}
Let $x \in \cA_{\gen}$ and $w \in \Wheart$.
Then we have $\supp(\varphi_{x, w}) = K_{x} w K_{x}$.
Moreover, if $n$ is a lift of $w$ in $\Nheart$, then 
$(\varphi_{x, w}(n)) v =\abs{K_{x}/ \left(	K_{wx} \cap K_{x}	\right)}^{-1/2}T_{n}(v)$ for all $v \in V_{\rho_{x}}$.
\end{lemma}
\begin{proof}
We fix a lift $n$ of $w$ in $\Nheart$.
For $g \in G(F)$ and $v \in V_{\rho_{x}}$, we obtain using \eqref{eqvarphifromPhi} and Axiom~\ref{axiomaboutHNheartandK}\eqref{axiomaboutHNheartandKaboutconjugation} that
\begin{align*}
(\varphi_{x, w}(g)) v &= \left(
\Phi_{x, w}(f_{v})
\right)(g) \\
&= 
\left(
(
\ciso{w^{-1} x}{w} \circ \Theta^{\normal}_{w^{-1} x \mid x}
)
(f_{v})
\right)
(g) \\
&=
T_{n} 
\left(
(
\Theta^{\normal}_{w^{-1} x \mid x}(f_{v})
)
(n^{-1}g) 
\right) \\
&=
\abs{K_{w^{-1} x}/ \left(
K_{x} \cap K_{w^{-1} x}
\right)
}^{- 1/2}
T_{n}
\Biggl(
\sum_{[k] \in K_{w^{-1} x, +} / \left(
K_{x, +} \cap K_{w^{-1} x, +}
\right)} \theta_{w^{-1} x}(k) \cdot f_{v}(k^{-1} n^{-1} g)
\Biggr) \\
&=
\abs{K_{w^{-1} x}/ \left(
K_{x} \cap K_{w^{-1} x}
\right)
}^{- 1/2}
T_{n}
\Biggl(
\sum_{[k] \in K_{w^{-1} x, +} / \left(
K_{x, +} \cap K_{w^{-1} x, +}
\right)} \theta_{w^{-1} x}(k) \cdot  f_{v}(n^{-1} (n k n^{-1})^{-1} g) 
\Biggr) \\
&=
\abs{K_{x}/ \left(
K_{w x} \cap K_{x}
\right)
}^{- 1/2}
T_{n}
\Biggl(
\sum_{[k] \in K_{x, +} / \left(
K_{w x, +} \cap K_{x, +}
\right)} \theta_{w^{-1} x}(n^{-1} k n) \cdot f_{v}(n^{-1} k^{-1} g)
\Biggr).
\end{align*}
Since $f_{v}$ is supported on $K_{x}$, the sum vanishes unless 
$
g \in K_{x, +} n K_{x} \subseteq K_{x} w K_{x}
$.
Hence, since $\varphi_{x, w}$ is non-zero, we obtain that
$
\supp(\varphi_{x, w}) = K_{x} w K_{x}
$.
Moreover, we have
$$\sum_{[k] \in K_{x, +} / \left(
	K_{w x, +} \cap K_{x, +}
	\right)} \theta_{w^{-1} x}(n^{-1} k n) \cdot f_{v}(n^{-1} k^{-1} n)
	= v,$$
	 which yields the second claim. 
\end{proof}

\begin{corollary}
\label{corollaryvectorspacedecompositionexplicitver}
Let $x \in \cA_{x_{0}}$ and assume Axioms~\ref{axiomaboutHNheartandK} and \ref{axiombijectionofdoublecoset}.
As a vector space, we have
\[
\cH(G(F), \rho_{x}) = \bigoplus_{w \in \Wheart} \Coeff \cdot \varphi_{x, w}.
\]
\end{corollary}
\begin{proof}
The corollary follows from Proposition~\ref{propositionvectorspacedecomposition} and Lemma~\ref{supportofphiixw}.
\end{proof}

\subsection{Relations in the length-additive case} \label{subsec:lengthadditivecase}
We will now study the structure of the Hecke algebra $\cH(G(F), \rho_{x_{0}})$ as a $\Coeff$-algebra. In this subsection
 we investigate the relations in the length-additive case for a length that we define below in Definition \ref{definitionlength}.
We keep the notation from the previous subsection and assume Axiom~\ref{axiomaboutHNheartandK} (but not Axiom~\ref{axiombijectionofdoublecoset}). 
Let $m, n \in \nobreak \Nheart$.
Recall that we have fixed an isomorphism 
\[
T_{n} \in \Hom_{K_{M}}\left(
^n\!\rho_{M}, \rho_{M}
\right) \subset \End_\Coeff(V_{\rho_M})
\]
in Choice \ref{choiceofT}.
Since the subspaces
$\Hom_{K_{M}}(^{mn}\!\rho_{M}, {^{m}\!{\rho_{M}}})$
and
$\Hom_{K_{M}}(^{n}\!\rho_{M}, {\rho_{M}})$
of 
$\End_\Coeff(V_{\rho_M})$
are equal,
$T_n$ is also an element of
$\Hom_{K_{M}}(^{mn}\!\rho_{M}, {^{m}\!{\rho_{M}}})$.
Then we can form the composition 
$T_{m} \circ T_{n} \in \Hom_{K_{M}}\left(
^{mn}\!\rho_{M}, \rho_{M}
\right)$. 
Since 
$
\dim_{\Coeff} \left(
\Hom_{K_{M}}\left(
^{mn}\!\rho_{M}, \rho_{M}
\right)
\right) = 1
$,
the isomorphisms $T_{m} \circ T_{n}$ and $T_{mn}$ differ by
a non-zero scalar,
and it is straightforward to see that this scalar depends only on the images
$\bar m$ and $\bar n$ of $m$ and $n$ in  $\Wheart$.
\begin{notation}
\label{notationofthetwococycle}
	We denote by \index{notation-ax}{muTmn@$\muT(\phantom{m},\phantom{n})$}
	\[
	\muT \colon \Wheart \times \Wheart \rightarrow \Coeff^{\times}
	\]
	the unique map that satisfies 
	\[
	T_{m} \circ T_{n} = \muT (\bar m, \bar n) \cdot T_{mn}.
	\]
\end{notation}
Standard arguments of projective representations imply that $\muT$ is a $2$-cocycle.
We note that the $2$-cocycle $\muT$ depends on the choice of a family
$
\cT =
\left\{
T_{n} \in \Hom_{K_{M}}\left(
^n\!\rho_{M}, \rho_{M}
\right)
\right\}_{n \in \Nheart},
$ made in Choice \ref{choiceofT},
but its cohomology class does not. We will see in the next lemma that the 2-cocycle also determines the composition of the homomorphisms $\ciso{v^{-1} x}{v}$, which we will then use to study the composition of the operators $\Phi_v$.
\begin{lemma}
\label{lemmatransitivityexceptfottheta}
Let $x \in \cA_{\gen}$ and $v, w \in \Wheart$.
Then the homomorphism
\begin{align*}
\ciso{v^{-1} x}{v} \circ \ciso{w^{-1} v^{-1} x}{w} 
\in 
\Hom_{G(F)} \left(
\ind_{K_{w^{-1} v^{-1} x}}^{G(F)} (\rho_{w^{-1} v^{-1} x}), \ind_{K_{x}}^{G(F)} (\rho_{x})
\right)
\end{align*}
is equal to $\muT(v, w) \cdot \ciso{w^{-1} v^{-1} x}{vw}$.
\end{lemma}
\begin{proof}
We fix a lift $m$ of $v$ and $n$ of $w$ in $\Nheart$.
For $f \in \ind_{K_{w^{-1} v^{-1} x}}^{G(F)} (\rho_{w^{-1} v^{-1} x})$ and $g \in G(F)$, we have
\begin{align*}
\left(
\left(
\ciso{v^{-1} x}{v} \circ \ciso{w^{-1} v^{-1} x}{w} 
\right)(f)
\right)(g) 
&= T_{m}  \left(
\left(
\ciso{w^{-1} v^{-1} x}{w}  (f)
\right)(m^{-1} g) \right) \\
&=
T_{m}  
\left(
T_{n} \left( f(n^{-1} m^{-1} g)
\right) \right) \\
&=  \left(
T_{m} \circ T_{n}
\right)
\left( f(n^{-1} m^{-1} g) \right)\\
&= 
\muT(m, n) \cdot T_{mn} \left( f(n^{-1} m^{-1} g) \right)\\
&= \left(
\muT(v, w) \cdot \ciso{w^{-1} v^{-1} x}{vw} (f)
\right)(g).
\end{align*}
Thus, we obtain the lemma.
\end{proof}

We equip the group $\Wheart$ with the following length function.
\begin{definition} \label{definitionlength}
	For $w \in \Wheart$, we define\index{notation-ax}{lKrel@$\flength$ (length function)}
\[
\flength(w) = d_{\Krel}(x_{0}, w^{-1} x_{0}). 
\]
\end{definition}

\begin{proposition}
\label{propositioniflengthsumthentransitivityuptotwococycle}
Let $v, w \in \Wheart$ and assume Axiom~\ref{axiomaboutHNheartandK}. 
If
\[
\flength(vw) = \flength(v) + \flength(w),
\]
then we have
\[
\Phi_{v} \Phi_{w} = \muT(v, w) \cdot \Phi_{vw}.
\]
\end{proposition}
\begin{proof}
We have
\begin{align*}
d_{\Krel}(x_{0}, w^{-1} x_{0}) + d_{\Krel}(w^{-1} x_{0}, w^{-1}v^{-1} x_{0}) 
&= d_{\Krel}(x_{0}, w^{-1} x_{0}) + d_{\Krel}(x_{0}, v^{-1} x_{0}) \\
&= d_{\Krel}(x_{0}, v^{-1} x_{0}) + d_{\Krel}(x_{0}, w^{-1} x_{0}) \\
&= \flength(v) + \flength(w) \\
&= \flength(vw) \\
&= d_{\Krel}(x_{0}, w^{-1}v^{-1} x_{0}).
\end{align*}
Hence, according to Proposition~\ref{transitivityofThetaGammaver}, we have
\[
\Theta^{\normal}_{w^{-1}v^{-1} x_{0} \mid w^{-1} x_{0}} \circ \Theta^{\normal}_{w^{-1} x_{0} \mid x_{0}} = \Theta^{\normal}_{w^{-1}v^{-1} x_{0} \mid x_{0}}.
\]
On the other hand, according to Lemma~\ref{lemmathetaccommute}, we have
\begin{align*}
\Theta^{\normal}_{v^{-1} x_{0} \mid x_{0}} \circ \ciso{w^{-1} x_{0}}{w}
&= \ciso{w^{-1}v^{-1} x_{0}}{w} \circ \Theta^{\normal}_{w^{-1}v^{-1} x_{0} \mid w^{-1} x_{0}}. \\
\end{align*}
Thus we obtain using Lemma~\ref{lemmatransitivityexceptfottheta}  
\begin{align*}
\Phi_{v} \Phi_{w} &=
\left(
\ciso{v^{-1} x_{0}}{v} \circ \Theta^{\normal}_{v^{-1} x_{0} \mid x_{0}}.
\right) \circ \left(
\ciso{w^{-1} x_{0}}{w} \circ \Theta^{\normal}_{w^{-1} x_{0} \mid x_{0}}
\right) \\
&= 
\ciso{v^{-1} x_{0}}{v} \circ
\left(
\Theta^{\normal}_{v^{-1} x_{0} \mid x_{0}} \circ \ciso{w^{-1} x_{0}}{w}
\right)
\circ  \Theta^{\normal}_{w^{-1} x_{0} \mid x_{0}} \\
&= 
\ciso{v^{-1} x_{0}}{v} \circ \left(
\ciso{w^{-1}v^{-1} x_{0}}{w} \circ \Theta^{\normal}_{w^{-1}v^{-1} x_{0} \mid w^{-1} x_{0}}
\right) \circ  \Theta^{\normal}_{w^{-1} x_{0} \mid x_{0}} \\
&= \left(
\ciso{v^{-1} x_{0}}{v} \circ \ciso{w^{-1}v^{-1} x_{0}}{w}
\right) \circ \left(
\Theta^{\normal}_{w^{-1}v^{-1} x_{0} \mid w^{-1} x_{0}} \circ \Theta^{\normal}_{w^{-1} x_{0} \mid x_{0}}
\right) \\
&= \left(\muT(v, w) \cdot \ciso{w^{-1}v^{-1} x_{0}}{vw}
\right) \circ \Theta^{\normal}_{w^{-1}v^{-1} x_{0} \mid x_{0}} \\
&= \muT(v, w) \cdot \Phi_{vw}.
\qedhere
\end{align*}
\end{proof}

\subsection{The structure of the indexing group} 
\label{subsection:indexinggroup}
We keep the notation from the previous subsection and assume Axioms \ref{axiomaboutHNheartandK} and \ref{axiombijectionofdoublecoset}. 
In this subsection we will introduce an additional axiom about $\Wheart$ containing a nice subgroup, Axiom \ref{axiomexistenceofRgrp}, that allows us to deduce that $\Wheart$ is a semi-direct product of a normal affine Weyl group with the subgroup of $\flength$-length-zero elements. This generalizes the decomposition that Morris (\cite[7.3.~Proposition]{Morris}) obtains for his group $W(\sigma)$ (using his notation) that indexes a basis for the Hecke algebra of a depth-zero type attached to a connected parahoric subgroup.

For $H \in \mathfrak{H}$, let $s_{H}$\index{notation-ax}{sH@$s_H$} denote the orthogonal reflection on $\cA_{x_{0}}$ with respect to the affine hyperplane $H$.
We define  $W_{\Krel}$\index{notation-ax}{WKrel@$W_{\Krel}$} to be the subgroup of the affine transformations of $\cA_{x_{0}}$ generated by $\{ s_H \mid \nobreak H \in \nobreak \mathfrak{H}_{\Krel}\} $, i.e.
\[
W_{\Krel} = 
\langle
s_{H} \mid H \in \mathfrak{H}_{\Krel}
\rangle.
\]
\begin{axiom}
\label{axiomexistenceofRgrp}
There exists a normal subgroup $\Waff$
\index{notation-ax}{Wrhoaff@$\Waff$}
of $\Wheart$ such that the action of $\Wheart$ on $\cA_{x_{0}}$ restricts to an isomorphism
\[
\Waff \isoarrow W_{\Krel}.
\]
\end{axiom}

One way to check that the axiom is satisfied is via the following lemma. 
\begin{lemma}
\label{lemmaaboutaxiomexistenceofRgrp} Assume Axioms \ref{axiomaboutHNheartandK} and \ref{axiombijectionofdoublecoset}, and 
suppose that there exists a normal subgroup $G'$ of $G(F)$ such that 
\[
G' \cap \Nheart  \cap K_{M} = G' \cap \Nheart \cap M(F)_{x_{0}}
\]
and that for all $H \in \mathfrak{H}_{\Krel}$, there exists an element 
\[
s'_{H} \in 
\left(
G' \cap \Nheart
\right) / \left(
G' \cap \Nheart  \cap K_{M}
\right)
\]
such that the action of $s'_{H}$ on $\cA_{x_{0}}$ agrees with the orthogonal reflection $s_{H}$.
Then Axiom~\ref{axiomexistenceofRgrp} is satisfied with
\[
\Waff = \langle
s'_{H} \mid H \in \mathfrak{H}_{\Krel}
\rangle.
\]
\end{lemma}
In the depth-zero setting, i.e., in Section~\ref{sec:depth-zero}, we will apply this lemma in the case that the normal subgroup $G'$ is the kernel of the Kottwitz homomorphism to prove that Axiom \ref{axiomexistenceofRgrp} is satisfied, see Proposition \ref{proofofaxiomexistenceofRgrpandaxiomsHisinKxy}.
In the setting of \cite{HAIKY}, we will observe that $\Wheart$ there is the same as the corresponding group in a depth-zero setting and the relevant hyperplanes are a subset of the hyperplanes in the depth-zero setting, so that Axiom \ref{axiomexistenceofRgrp} in the setting of \cite{HAIKY} follows from the result in the depth-zero setting, see \cite[Proposition~\ref{HAIKY-prop:axiomexistenceofRgrpandaxiomaboutdimensionofend}]{HAIKY}.

\begin{proof}[Proof of Lemma \ref{lemmaaboutaxiomexistenceofRgrp}]
	By Remark~\ref{remarkaboutfaithful} and Corollary \ref{corollaryweylinvariance}, the group 
	$$ \bigl(
	G' \cap \Nheart
	\bigr) / \bigl(
	G' \cap \Nheart  \cap K_{M}
	\bigr) $$ acts faithfully on $\cA_{x_{0}}$ through affine transformations and preserves $\mathfrak{H}_{\Krel}$. Thus we can identify the former group with a subgroup of the group of affine transformations of $\cA_{x_{0}}$ preserving $\mathfrak{H}_{\Krel}$. Since the latter contains the group $W_{\Krel} = 
	\langle 	s_{H} \mid H \in \mathfrak{H}_{\Krel} \rangle$  as a normal subgroup, the claim follows from the assumption about the existence of $s'_{H} \in 
	\bigr(
	G' \cap \Nheart
	\bigr) / \bigl(
	G' \cap \Nheart  \cap K_{M}
	\bigr)
	$ and the observation that $\bigl(
	G' \cap \Nheart
	\bigr) / \bigl(
	G' \cap \Nheart  \cap K_{M}
	\bigr)$ is a normal subgroup of $\Wheart$.
\end{proof}

We will now show that the group $\Waff$ of Axiom \ref{axiomexistenceofRgrp} is an affine Weyl group, which explains our choice of notation.
For $H \in \mathfrak{H}$, let $a_{H}$ denote an affine functional on $\cA_{x_{0}}$ such that
\(
H = \left\{
x \in \cA_{x_{0}} \mid a_{H}(x) = 0
\right\}.
\)
We write $Da_{H}$ for the gradient of $a_{H}$,
which is a linear functional on $X_{*}(A_{M}) \otimes_{\bZ} \bR$.
The subspace
\[
\ker(Da_{H}) = \left\{
v \in X_{*}(A_{M}) \otimes_{\bZ} \bR
\mid
Da_{H}(v) = 0
\right\}
\]
of $X_{*}(A_{M}) \otimes_{\bZ} \bR$ only depends on $H$ and not on the choice of $a_H$.
We define the subspace $V^{\Krel}$\index{notation-ax}{VKrelup@$V^\Krel$} of $X_{*}(A_{M}) \otimes_{\bZ} \bR$ by
\[
V^{\Krel} = \bigcap_{H \in \mathfrak{H}_{\Krel}} \ker(Da_{H}).
\]
We define the affine space $\cA_{\Krel}$ \index{notation-ax}{AKrel@$\cA_{\Krel}$}\label{AKrel-page} to be the quotient affine space $\cA_{x_{0}}/V^{\Krel}$ and write its vector space of translations as
\index{notation-ax}{VKrellow@$V_\Krel$}
\[
V_{\Krel} = \left(
X_{*}(A_{M}) \otimes_{\bZ} \bR 
\right)/V^{\Krel}.
\]
According to Corollary \ref{corollaryweylinvariance}, the action of $\Wheart$ on $\cA_{x_{0}}$ induces a well-defined action of $\Wheart$ on $\cA_{\Krel}$.
Let $(V^{\Krel})^{\perp}$ denote the orthogonal complement of $V^{\Krel}$ in $X_{*}(A_{M}) \otimes_{\bZ} \bR$ with respect to the previously fixed $N_{G}(M)(F)$-invariant inner product $(\phantom{y},\phantom{y})_{M}$ on $X_{*}(A_{M}) \otimes_{\bZ} \bR$.
Then the natural projection 
$
X_{*}(A_{M}) \otimes_{\bZ} \bR \rightarrow V_{\Krel}
$
restricts to an isomorphism
\begin{align}
\label{orthogonalcomplementmorris}
(V^{\Krel})^{\perp} \isoarrow V_{\Krel}.
\end{align}
We define an inner product on $V_{\Krel}$ by restricting the inner product $(\phantom{y},\phantom{y})_{M}$ to $(V^{\Krel})^{\perp}$ and then transporting it to $V_{\Krel}$ via the isomorphism in \eqref{orthogonalcomplementmorris}.
This turns the affine space $\cA_{\Krel}$ into a Euclidean space.
We identify an affine hyperplane $H \in \mathfrak{H}_{\Krel}$ with its image on $\cA_{\Krel}$.
\begin{proposition} 
\label{Rrhosimeqaffineweyl}
Assume Axioms \ref{axiomaboutHNheartandK}, \ref{axiombijectionofdoublecoset} and \ref{axiomexistenceofRgrp}. Then there exists an affine root system $\Gamma(\rho_{M})$ on $\cA_{\Krel}$ whose vanishing hyperplanes are $\fH_\Krel$. In particular, the action of $\Waff$ on $\cA_{\Krel}$ induces an isomorphism
\[
\Waff \simeq W_{\Krel} = W_{\aff}(\Gamma(\rho_{M})),
\]
where $W_{\aff}(\Gamma(\rho_{M}))$ denotes the affine Weyl group of $\Gamma(\rho_{M})$.
(We allow the affine root system to be empty with the associated affine Weyl group being the trivial group.)
\end{proposition}
\begin{proof}
We assume that $\cA_{\Krel}$ has dimension at least one because the statement is otherwise trivial. 
According to \cite[Chapter~V, Section~3.10, Proposition~10]{MR0240238} and the proof of \cite[Chapter~VI, Section~2.5, Proposition~8]{MR0240238}, it suffices to check the following conditions (see also the proof of \cite[2.7 Theorem (b)]{Morris}):
\begin{enumerate}[(1)]
\item 
\label{weylinvariance}
For any $w \in W_{\Krel}$ and $H \in \mathfrak{H}_{\Krel}$, we have $w(H) \in \mathfrak{H}_{\Krel}$.
\item 
\label{actionproperly}
The group $W_{\Krel}$ acts properly on $\cA_{\Krel}$.
\item
\label{rankoftranslations}
For any $H \in \mathfrak{H}_{\Krel}$, there are infinitely many $H' \in \mathfrak{H}_{\Krel}$ that are parallel to $H$. 
\end{enumerate}
Condition \eqref{weylinvariance} follows from Corollary \ref{corollaryweylinvariance} and Axiom~\ref{axiomexistenceofRgrp}, and Condition \eqref{actionproperly} follows from Lemma~\ref{lemmaweylactionproperly} and Axiom~\ref{axiomexistenceofRgrp}.
It remains to prove Condition \eqref{rankoftranslations}.
Let $H \in \mathfrak{H}_{\Krel}$.
We fix an affine functional $a$ on $\cA_{x_{0}}$ such that
$
H = \left\{
x \in \cA_{x_{0}} \mid a(x) = 0
\right\}
$.
We write $\alpha$ for the gradient of $a$.
Let $t \in A_{M}(F)$.
Since $t \in A_{M}(F) \subset \Nheart$, we have $t(H) \in \mathfrak{H}_{\Krel}$ by Corollary \ref{corollaryweylinvariance}.
We define the element
$
\nu(t) \in X_{*}(A_{M}) \otimes_{\bZ} \bR 
$
by
\(
\chi(\nu(t)) = - \ord(\chi(t))
\)
for all $\chi \in X^{*}(A_{M})$.
We note that the set
\(
\left\{
\nu(t) \mid t \in A_{M}(F)
\right\}
\)
is a lattice of full rank in $X_{*}(A_{M}) \otimes_{\bZ} \bR \neq\{ 0 \}$.
Hence, we can take $t \in A_{M}(F)$ such that $\alpha(\nu(t)) \neq 0$.
According to \cite[Proposition~6.2.4]{KalethaPrasad}, we have 
$
t^{-1} x = x - \nu(t)
$.
Hence, we obtain that
\begin{align*}
t(H) &=
\left\{
t x \in \cA_{x_{0}} \mid a(x) = 0
\right\} = \left\{
x \in \cA_{x_{0}} \mid a(t^{-1} x) = 0
\right\} 
= \left\{
x \in \cA_{x_{0}} \mid a(x - \nu(t)) = 0
\right\} \\
&= \left\{
x \in \cA_{x_{0}} \mid a(x) - \alpha(\nu(t)) = 0
\right\}.
\end{align*}
Thus, we obtain that $t(H)$ is parallel to $H$.
Moreover, since $\alpha(\nu(t)) \neq 0$, by applying the same calculations to $t^n$ instead of $t$ for $n \in \mathbb{Z}$, we obtain that the affine hyperplanes 
$t^{n}(H)$ with $n \in \mathbb{Z}$
are all pairwise distinct, contained in $\mathfrak{H}_{\Krel}$, and parallel to $H$.
\end{proof}

We now construct a complement to $\Waff$ in $\Wheart$.
\begin{notation}
\label{notation:Wzero}
	 We call the connected components of the complement of $\fH_\Krel$ in $\cA_{\Krel}$ \emph{chambers} and denote by $C_{\Krel}$ the chamber of $\cA_{\Krel}$ that contains \( x_{0} + V^{\Krel} \).
	We define the subgroup $\Wzero$ 
	\index{notation-ax}{OmegarhoM@$\Wzero$}
	of $\Wheart$ by 
	\[
	\Wzero = \Bigl\{
	w \in \Wheart \Bigr. \,\Bigl | \, w (C_{\Krel}) = C_{\Krel}
	\Bigr\}.
	\]
\end{notation}
Note that $\Wzero$ consists precisely of the length-zero elements of $\Wheart$ with respect to the length $\flength$ introduced in Definition \ref{definitionlength}.

\begin{proposition}
\label{propositiondecompositionofW}Assume Axioms \ref{axiomaboutHNheartandK}, \ref{axiombijectionofdoublecoset} and \ref{axiomexistenceofRgrp}. 
Then we have 
\[
\Wheart = \Wzero \ltimes \Waff.
\]
\end{proposition}
\begin{proof}
According to Proposition~\ref{Rrhosimeqaffineweyl}, the action of $\Waff$ on the set of all chambers of $\cA_{\Krel}$ is simply transitive.
Hence we obtain that
\[
\Wzero \cap \Waff = \{1\}
\]
and
\[
\Wheart = \Wzero \cdot \Waff.
\]
Since $\Waff$ is a normal subgroup of $\Wheart$, the proposition follows.
\end{proof}

\subsection{Simple reflections and quadratic relations} \label{sec:simplereflections} 
In this subsection, we continue our study of the structure of the Hecke algebra $\cH(G(F),\rho_{x_0})$. In Section \ref{subsec:lengthadditivecase}, we investigated the relations of basis elements of the Hecke algebra in the length-additive case, see Proposition \ref{propositioniflengthsumthentransitivityuptotwococycle}. In this subsection we will study the relations in the basic non-length-additive case, the case of (simple) reflections, introduced in Notation \ref{notationsimplereflections} below. This will require one further axiom, Axiom \ref{axiomaboutdimensionofend} below, which ensures that the Hecke algebra element $\Phi_s$ corresponding to a simple reflection $s$ satisfies a quadratic relation. This additional axiom also allows us to prove that the endomorphism $\Phi_w$ for $w \in \Wheart$ is invertible, see Corollary \ref{corollaryPhixwinvertible}. 

We keep the notation from the previous subsections and assume all the above axioms, i.e., Axioms \ref{axiomaboutHNheartandK}, \ref{axiombijectionofdoublecoset}, and \ref{axiomexistenceofRgrp}.

\begin{notation}\label{notationsimplereflections}
	We denote by $S_{\Krel} \subset \nobreak W_{\Krel}$
	\index{notation-ax}{SKrel@$S_{\Krel}$} the subset 
	 of \textit{simple reflections} corresponding to the chamber $C_{\Krel}$, i.e., the reflections across the walls of $C_{\Krel}$. Using the isomorphism of $W_{\Krel}$ with the subgroup $\Waff$ of $\Wheart$  from Axiom \ref{axiomexistenceofRgrp}, we also view  $S_{\Krel}$ as a subset of $\Wheart$. For each $s \in S_{\Krel}$, we denote by $H_{s} \in \mathfrak{H}_{\Krel}$ the corresponding  wall of $C_{\Krel}$. \index{notation-ax}{Hs@$H_s$}
\end{notation}

We note that the restriction of the length function $\flength$ to $\Waff$ agrees with the length function of the Coxeter system $(\Waff, S_{\Krel})$.
We will now impose one more axiom that guarantees that the Hecke algebra elements supported on double cosets of lifts of simple reflections satisfy a quadratic relation.

\begin{axiom}
\label{axiomaboutdimensionofend}
For any $s \in S_{\Krel}$ and $x \in \cA_{\gen}$ such that $\mathfrak{H}_{x, s x} = \{ H_{s} \}$, there exists a compact, open subgroup $K'_{x, s}$ of $G(F)$ containing $K_{x}$ such that
\[ \Wheart \supset
\left(
\Nheart \cap K'_{x, s}
\right) / \left(
\Nheart \cap K_{M}
\right)
 = \{1, s \}.
\]
\end{axiom}
\begin{remark}
In the depth-zero setting discussed in Section \ref{sec:depth-zero},
the axiom is satisfied for the group $K'_{x,s}=K_M \cdot G(F)_{h,0}$ for $h \in H_s$ the unique point for which $h=x+t \cdot (sx-x)$ for some $0 < t < 1$.
For the more general setup in
\cite{HAIKY}
the axiom is satisfied for  $K'_{x,s}=K_{h}$ as defined in
\cite[Equation \eqref{HAIKY-definitionofcompactopensubgroupsKimYucase}]{HAIKY}.
This is proven in
Corollary \ref{HAIKY-corofproofofaxiomexistenceofRgrpandaxiomsHisinKxyaxiomaboutdimensionofend}
of \cite{HAIKY}.
\end{remark}

\begin{notation}
For $H \in \mathfrak{H}_{\Krel}$ and $x \in \cA_{\gen}$ with $\mathfrak{H}_{x, s_{H} x} = \{ H \}$, let
\index{notation-ax}{Kay x s@$K_{x,s}$}
$K_{x,{s_H}} \coloneqq \langle K_x, {s_H} \rangle$.
Here, we regard $s_{H} \in W_{\Krel}$ as an element of $\Waff \subset \Wheart$ via the isomorphism in Axiom \ref{axiomexistenceofRgrp}.
\end{notation}

\begin{remark}
\label{remarkrewriteaxiomaboutdimensionofend}
If a group $K'_{x, s}$ exists as in Axiom \ref{axiomaboutdimensionofend}, 
then we obtain that the axiom also holds for the group $K_{x, s}$ in place of $K'_{x,s}$. This is because $K_{x, s}$ is an open (hence closed) subgroup of $K'_{x, s}$,
and thus is compact and 
\begin{align*}
\{1, s \}
& \subseteq
\left(
\Nheart \cap K_{x, s}
\right) / \left(
\Nheart \cap K_{M}
\right) \\
& \subseteq
\left(
\Nheart \cap K'_{x, s}
\right) / \left(
\Nheart \cap K_{M}
\right) = \{1, s \}.
\end{align*}
Thus, we conclude that
\(
\left(
\Nheart \cap K_{x, s}
\right) / \left(
\Nheart \cap K_{M}
\right)
 = \{1, s \}.
\)
\end{remark}
The reason why we state Axiom \ref{axiomaboutdimensionofend} with an arbitrary subgroup $K'_{x,s}$ rather than with $K_{x,s}$ is that in Section \ref{sec:comparison} below, in order to obtain an isomorphism between different Hecke algebras, we like to state that the Axiom \ref{axiomaboutdimensionofend} is satisfied for a specific choice of $K'_{x,s}$, see the beginning of Section \ref{subsec:Heckeisom} and Theorem \ref{thm:isomorphismtodepthzero}.

\begin{remark}
\label{remarkaboutaxiomaboutdimensionofend} While Axiom \ref{axiomaboutdimensionofend} only concerns simple reflections, it implies the same result for other reflections. More precisely, let $H \in \mathfrak{H}_{\Krel}$ and $x \in \cA_{\gen}$ such that $\mathfrak{H}_{x, s_{H} x} = \{ H \}$.
According to \cite[Chapter~V, Section~3.1, Lemma~2]{MR0240238}, there exists $w \in \Waff$ and $s \in S_{\Krel}$ such that $w(H) = H_{s}$ and hence $s_H=w^{-1}sw$.
Thus, from Axiom~\ref{axiomaboutdimensionofend} and Remark \ref{remarkrewriteaxiomaboutdimensionofend},
we obtain that the group $K_{x, s_{H}} = \langle K_{x}, s_{H} \rangle=w^{-1}K_{wx,s}w$
is a compact, open subgroup of $G(F)$ and
\[
\left(
\Nheart \cap K_{x, s_{H}}
\right) / \left(
\Nheart \cap K_{M}
\right)
 = \{1, s_{H} \}.
\]
\end{remark}
We remind the reader that we have shown in Section \ref{Hecke algebras and endomorphism algebras}, in particular around Diagram \eqref{DiagramEnd}, how to view $\End_{K_{x, s_{H}}}\bigl(
\ind_{K_{x}}^{K_{x, s_{H}}}(\rho_{x})
\bigr)$
as a subalgebra of
$\End_{G(F)}\bigl(
\ind_{K_{x}}^{G(F)}(\rho_{x})
\bigl)$.
\begin{lemma}
\label{lemmadimensionoftheendomorphismalgebraKxsH}
Let $H \in \mathfrak{H}_{\Krel}$ and $x \in \cA_{\gen}$ such that $\mathfrak{H}_{x, s_{H} x} = \{ H \}$.
Then the elements $\Phi_{x, s_{H}}$ and $\Phi_{x, 1}$ 
 form a basis of $\End_{K_{x, s_{H}}}\left(
\ind_{K_{x}}^{K_{x, s_{H}}}(\rho_{x})
\right)$.

In particular, we have
\[
\dim_{\Coeff}\left(
\End_{K_{x, s_{H}}}\left(
\ind_{K_{x}}^{K_{x, s_{H}}}(\rho_{x})
\right)
\right) = 2.
\]
\end{lemma}

\begin{proof}
According to the isomorphism in \eqref{heckevsendK'}, it suffices to show that the elements $\varphi_{x, s_{H}}$ and $\varphi_{x, 1}$ give a basis of $\cH(K_{x, s_{H}}, \rho_{x})$.
According to Proposition~\ref{propositionvectorspacedecomposition} and Lemma~\ref{supportofphiixw}, a basis of the space $\cH(K_{x, s_{H}}, \rho_{x})$ is given by the set
\[
\left\{
\varphi_{x, w} \mid
w \in \left(
\Nheart \cap K_{x, s_{H}}
\right) / \left(
\Nheart \cap K_{M}
\right)
\right\}.
\]
Thus, the lemma follows from Remark~\ref{remarkaboutaxiomaboutdimensionofend}.
\end{proof}
\begin{corollary}
\label{corollaryweakversionofquadraticrelation}
We assume Axioms \ref{axiomaboutHNheartandK}, \ref{axiombijectionofdoublecoset}, \ref{axiomexistenceofRgrp}, and \ref{axiomaboutdimensionofend}.
Let $H \in \mathfrak{H}_{\Krel}$ and $x \in \cA_{\gen}$ such that $\mathfrak{H}_{x, s_{H} x} = \{ H \}$.
Then there exist $p_{x, H}, q_{x, H} \in \Coeff$ such that
\[
\left(
\Phi_{x, s_{H}}
\right)^{2} = p_{x, H} \cdot \Phi_{x, s_{H}} + q_{x, H} \cdot \Phi_{x, 1}.
\]
\end{corollary}
\begin{proof}
Since 
$\Phi_{x, s_{H}}$ belongs to
the subalgebra
$\End_{K_{x, s_{H}}}\bigl(
\ind_{K_{x}}^{K_{x, s_{H}}}(\rho_{x})
\bigr)$
of
$\End_{G(F)}\bigl(
\ind_{K_{x}}^{G(F)}(\rho_{x})
\bigl)$,
the same is true of
$(\Phi_{x, s_{H}})^{2}$.
Thus, the corollary follows from Lemma~\ref{lemmadimensionoftheendomorphismalgebraKxsH}.
\end{proof}
The remainder of this subsection is concerned with strengthening the statement of this Corollary by replacing the condition $\mathfrak{H}_{x, s_{H} x} = \{ H \}$ by the weaker condition $\mathfrak{H}_{\Krel;x, s_{H} x} = \{ H \}$ and by proving that the coefficients $p_{x, H}$ and $q_{x, H}$ in the above quadratic relation are non-zero and independent of the choice $x$.
While $q_{x, H} \neq 0$ follows from a standard observation about the endomorphism $\Theta^{\normal}_{x \mid s_{H} x}
\circ \Theta^{\normal}_{s_{H} x \mid x}$ for a general $H \in \fH$ and $x \in \cA_{\gen}$ such that $\mathfrak{H}_{x, s_{H} x} = \{ H \}$, to prove $p_{x, H} \neq 0$, the condition $H \in \mathfrak{H}_{\Krel}$ is essential.
We remind the reader that the condition $H \in \mathfrak{H}_{\Krel}$ was defined by the existence of $x, y \in \cA_{\gen}$  with $\mathfrak{H}_{x, y} = \left\{H\right\}$ and
$\Theta_{x \mid y} \circ \Theta_{y \mid x}$ a non-scalar operator
on $\ind_{K_{x}}^{G(F)}(\rho_{x})$.
In order to prove $p_{x, H} \neq 0$, we will show that if
$\Theta_{x \mid y} \circ \Theta_{y \mid x}$ is a non-scalar operator
for some $x, y \in \cA_{\gen}$  with $\mathfrak{H}_{x, y} = \left\{H\right\}$,
then the same is true for \emph{all} such $x, y \in \cA_{\gen}$, and even for all $x, y \in \cA_{\gen}$ satisfying only the weaker condition  $\mathfrak{H}_{\Krel; x, y} = \left\{H\right\}$,
see Proposition \ref{propositionkreloodequivalentforallorthereexists}.
Once we know that the coefficients $p_{x, H}$ and $q_{x, H}$ are non-zero and independent of $x$,
 we will refine our choice of \(
\cT =
\left\{
T_{n} \in \Hom_{K_{M}}\left(
^n\!\rho_{M}, \rho_{M}
\right)
\right\}_{n \in \Nheart}
\) that entered the definition of the endomorphism $\Phi_s$
 to obtain a quadratic relation 
 \(
 (\Phi_{s})^{2} = (q_{s} - 1) \cdot \Phi_{s} + q_{s} \cdot \Phi_{1}
 \)
with $q_s \in \Coeffinvnontriv$
for every simple reflection $s$,
see Proposition \ref{propositionaboutquadraticrelationsrevised}. 

We begin by proving that $q_{x, H}$ is non-zero and independent of $x$, for which we first record a consequence of Lemma~\ref{lemmathetaccommute} and Lemma~\ref{lemmatransitivityexceptfottheta}:
\begin{lemma}
\label{lemmaphivstheta}
Let $H \in \mathfrak{H}_{\Krel}$ and $x \in \cA_{\gen}$.
Then we have
\[
\left(
\Phi_{x, s_{H}}
\right)^{2}
=
\muT(s_{H}, s_{H}) \cdot \Theta^{\normal}_{x \mid s_{H} x}
\circ \Theta^{\normal}_{s_{H} x \mid x}.
\]
\end{lemma}
\begin{proof}
We write $s = s_{H}$.
Since $s^{2} = 1$, Lemma~\ref{lemmathetaccommute} and Lemma~\ref{lemmatransitivityexceptfottheta} imply that
\begin{align*}
\left(
\Phi_{x, s}
\right)^{2}
&= 
\ciso{s^{-1} x}{s}
\circ \Theta^{\normal}_{s^{-1} x \mid x}
\circ \ciso{s^{-1} x}{s}
\circ \Theta^{\normal}_{s^{-1} x \mid x} \\
&= 
\ciso{s^{-1} x}{s}
\circ \ciso{x}{s}
\circ \Theta^{\normal}_{x \mid s^{-1} x}
\circ \Theta^{\normal}_{s^{-1} x \mid x} \\
&=
\muT(s, s) \cdot \ciso{x}{1}
\circ \Theta^{\normal}_{x \mid s^{-1} x}
\circ \Theta^{\normal}_{s^{-1} x \mid x} \\
&=
\muT(s, s) \cdot \Theta^{\normal}_{x \mid s^{-1} x}
\circ \Theta^{\normal}_{s^{-1} x \mid x}  \\
&=
\muT(s, s) \cdot \Theta^{\normal}_{x \mid s x}
\circ \Theta^{\normal}_{s x \mid x}.
\qedhere
\end{align*}
\end{proof}
\begin{lemma}
\label{lemmaqHisnonzero}
We assume Axioms \ref{axiomaboutHNheartandK}, \ref{axiombijectionofdoublecoset}, \ref{axiomexistenceofRgrp}, and \ref{axiomaboutdimensionofend}.
Let $H \in \mathfrak{H}_{\Krel}$ and $x \in \cA_{\gen}$ such that $\mathfrak{H}_{x, s_{H} x} = \{ H \}$,
and let $q_{x, H}$ be as in Corollary~\ref{corollaryweakversionofquadraticrelation}.
Then $q_{x, H} = \nobreak \muT(s_H, s_H)$.
In particular, the scalar $q_{x, H}$ is invertible and independent of $x$.
\end{lemma}
\begin{proof}
We write $s = s_{H}$.
For $v \in V_{\rho_{x}}$, we define the element $f_{v} \in \ind_{K_{x}}^{G(F)} \left( V_{\rho_{x}} \right)$ by
\[
f_{v}(g) =
\begin{cases}
\rho_{x}(g) (v) & (g \in K_{x}), \\
0 & \text{(otherwise)}. 
\end{cases}
\]
Then, according to Corollary~\ref{corollarycalculationofconstantterm}, we have 
\begin{align*}
\left(
\left(
\Theta_{x \mid s x} \circ \Theta_{s x \mid x}
\right)(f_{v})
\right)(1) = \abs{
K_{s x}/\left(
K_{x} \cap K_{s x}
\right)
}^{-1} v.
\end{align*}
Hence, we obtain from Definition~\ref{definitionnormalize} and Axiom~\ref{axiomaboutHNheartandK}\eqref{axiomaboutHNheartandKaboutconjugation} that
\[
\left(
\left(
\Theta^{\normal}_{x \mid s x} \circ \Theta^{\normal}_{s x \mid x}
\right)(f_{v})
\right)(1) = \abs{K_{x}/ \left(
		K_{x} \cap K_{s x}
		\right)
	}^{1/2} \abs{
K_{s x}/\left(
K_{x} \cap K_{s x}
\right)
}^{-1/2} v = v.
\]
Combining this with Lemma~\ref{lemmaphivstheta}, we obtain that
\addtocounter{equation}{-1}
\begin{subequations}
\begin{align}
\label{constanttermofthesquareforrelevantH}
\left(
\left(
\Phi_{x, s}^{2}
\right)(f_{v})
\right)(1)
= \muT(s, s) \cdot v.
\end{align}
\end{subequations}
Substituting 
\(
\left(
\Phi_{x, s}
\right)^{2} \) by \( p_{x, H} \cdot \Phi_{x, s} + q_{x, H} \cdot \Phi_{x, 1}
\)
in \eqref{constanttermofthesquareforrelevantH}, we obtain
$
q_{x, H} \cdot v = \muT(s, s) \cdot v
$.
Thus, 
$
q_{x, H} = \muT(s, s)  \in \Coeff^\times
$, as desired.
\end{proof}
\begin{corollary}
\label{corollaryPhixsinvertible}
We assume Axioms \ref{axiomaboutHNheartandK}, \ref{axiombijectionofdoublecoset}, \ref{axiomexistenceofRgrp}, and \ref{axiomaboutdimensionofend}.
Let $H \in \mathfrak{H}_{\Krel}$ and $x \in \cA_{\gen}$ such that $\mathfrak{H}_{x, s_{H} x} = \{ H \}$.
Then the endomorphism $\Phi_{x, s_{H}}$ is invertible.
\end{corollary}
\begin{proof}
Recall from Remark \ref{remarkaboutphix1} that $\Phi_{x, 1}$ is the identity endomorphism. Hence by Corollary~\ref{corollaryweakversionofquadraticrelation} and Lemma~\ref{lemmaqHisnonzero}, the element
\(
q_{x, H}^{-1} \cdot \left(
\Phi_{x, s_{H}} - p_{x, H} \cdot \Phi_{x, 1}
\right)
\)
is a left and right inverse of $\Phi_{x, s_{H}}$.
\end{proof}

From this we can deduce that more generally all endomorphism $\Phi_{x, w}$ with $w \in \Wheart$ are invertible as well as the operators $\Theta_{x \mid y}$ for all $x, y \in \cA_{\gen}$ that we used to define $\Phi_{x, w}$. 

\begin{proposition} \label{propositionThetaisomorphism}
We assume Axioms \ref{axiomaboutHNheartandK}, \ref{axiombijectionofdoublecoset}, \ref{axiomexistenceofRgrp}, and \ref{axiomaboutdimensionofend}.
Let $x, y \in \cA_{\gen}$.
Then the operators $\Theta_{y \mid x}$ and $\Theta^{\normal}_{y \mid x}$ are isomorphisms.
\end{proposition}
\begin{proof}
It suffices to show that $\Theta^{\normal}_{y \mid x}$ is an isomorphism.
If $d(x, y) = 0$, according to Lemma~\ref{lemmanormalizedtransitivityofthetadver}, we have 
\[
\Theta^{\normal}_{x \mid y} \circ \Theta^{\normal}_{y \mid x} = \Theta^{\normal}_{x \mid x} = \id_{\ind_{K_{x}}^{G(F)}(\rho_{x})}
\quad
\text{ and }
\quad
\Theta^{\normal}_{y \mid x} \circ \Theta^{\normal}_{x \mid y} = \Theta^{\normal}_{y \mid y} = \id_{\ind_{K_{y}}^{G(F)}(\rho_{y})}.
\]
Hence $\Theta^{\normal}_{y \mid x}$ is an isomorphism.
Suppose that $d(x, y) \ge1$.
According to Lemma~\ref{lemmanormalizedtransitivityofthetadver}, we can take points
\[
x_{1} = x, x_{2}, x_{3}, \ldots, x_{t} = y \in \cA_{\gen}
\]
such that $d(x_{i}, x_{i+1}) = 1$ for all $1 \le i < t$, and
\[
\Theta^{\normal}_{y \mid x} = \Theta^{\normal}_{x_{t} \mid x_{t-1}} \circ \cdots \circ \Theta^{\normal}_{x_{2} \mid x_{1}}.
\]
Thus, to prove the proposition, it suffices to show that $\Theta^{\normal}_{y \mid x}$ is an isomorphism for all $x, y \in \cA_{\gen}$ such that $d(x, y) = 1$, so we assume $d(x, y) = 1$.
Let $H$ denote the unique affine hyperplane in $\cA_{x_{0}}$ such that $x$ and $y$ are on opposite sides of $H$.
If $H$ is not $\cK$-relevant, Lemma~\ref{normalizedlemmalength1case} implies that $\Theta^{\normal}_{y \mid x}$ is an isomorphism.
Suppose that $H$ is $\cK$-relevant.
We write $s = s_{H}$.
Then there exists $x' \in \cA_{\gen}$ such that $d(x, x') = d(y, s x') = 0$ and $d(x', s x') = 1$.
By Lemma~\ref{lemmanormalizedtransitivityofthetadver}, we have
\[
\Theta^{\normal}_{y \mid x} = \Theta^{\normal}_{y \mid s x'} \circ \Theta^{\normal}_{s x' \mid x'} \circ \Theta^{\normal}_{x' \mid x}.
\]
We already showed that the operators $\Theta^{\normal}_{x' \mid x}$ and $\Theta^{\normal}_{y \mid s x'} $ are isomorphisms.
Moreover, since $\ciso{sx'}{s}$ is an isomorphism, it follows from Corollary~\ref{corollaryPhixsinvertible} that the operator $\Theta^{\normal}_{s x' \mid x'} = (\ciso{sx'}{s})^{-1} \circ \Phi_{x', s}$ is also an isomorphism. 
Thus $\Theta^{\normal}_{y \mid x}$ is also an isomorphism.
\end{proof}

\begin{corollary}
	\label{corollaryPhixwinvertible}
	We assume Axioms \ref{axiomaboutHNheartandK}, \ref{axiombijectionofdoublecoset}, \ref{axiomexistenceofRgrp}, and \ref{axiomaboutdimensionofend}. For $x \in \cA_{\gen}$ and $w \in \nobreak \Wheart$, the endomorphism $\Phi_{x, w}$ is invertible.
\end{corollary}
\begin{proof}
 The claim follows from the definition of $\Phi_{x, w}$ as $\ciso{w^{-1} x}{w} \circ \Theta^{\normal}_{w^{-1} x \mid x}$, where $\ciso{w^{-1} x}{w}$ is an isomorphism and Proposition \ref{propositionThetaisomorphism}.
\end{proof}

In order to show that the coefficient $p_{x,H}$ of the quadratic relation in Corollary \ref{corollaryweakversionofquadraticrelation} is non-zero and independent of $x$, 
 we need a few more lemmas.
\begin{lemma}
\label{lemmafromlemmaaboutdistancefunctionrelandtransitivityofThetaGammaver}
Let $H \in \mathfrak{H}_{\Krel}$ and let $x, x', y \in \cA_{\gen}$.
We assume that the points $x$ and $x'$ are on the same side of $H$ and $\mathfrak{H}_{\Krel; x, y} = \{ H \}$.
Then we have
\[
\Theta^{\normal}_{x' \mid x} \circ \Theta^{\normal}_{x \mid y} = \Theta^{\normal}_{x' \mid y}
\quad 
\text{ and }
\quad
\Theta^{\normal}_{y \mid x} \circ \Theta^{\normal}_{x \mid x'} = \Theta^{\normal}_{y \mid x'}.
\]
\end{lemma}
\begin{proof}
Since the points $x$ and $x'$ are on the same side of $H$ and 
$
\mathfrak{H}_{\Krel; x, y} = \{ H \}
$,
we have
\(
\mathfrak{H}_{\Krel; x, x'} \cap \mathfrak{H}_{\Krel; x, y} = \emptyset.
\)
Then Lemma~\ref{lemmaaboutdistancefunctionrel} implies that
\(
d_{\Krel}(x', x) + d_{\Krel}(x, y) = d_{\Krel}(x', y)
\)
and
\(
d_{\Krel}(y, x) + d_{\Krel}(x, x') = d_{\Krel}(y, x').
\)
Hence, the claim follows from Proposition~\ref{transitivityofThetaGammaver}.
\end{proof}

 \begin{lemma}
 \label{lemmaaboutreplacingpoints}
Let $H \in \mathfrak{H}_{\Krel}$ and let $x, x', y, y' \in \cA_{\gen}$ such that $\mathfrak{H}_{\Krel; x, y} = \mathfrak{H}_{\Krel; x', y'} = \{ H \}$.
Then the following diagram commutes:
\begin{equation} 
\xymatrix@R+1pc@C+1pc{
\ind_{K_{x}}^{G(F)} (\rho_{x})
	\ar[d]_-{\Theta^{\normal}_{y \mid x}}
	\ar[r]^-{\Theta^{\normal}_{x' \mid x}}
	\ar@{}[dr]|\circlearrowleft
& \ind_{K_{x'}}^{G(F)} (\rho_{x'})
	\ar[d]^-{\Theta^{\normal}_{y' \mid x'}}\\
\ind_{K_{y}}^{G(F)} (\rho_{y})
	\ar[r]^-{\Theta^{\normal}_{y' \mid y}}
& \ind_{K_{y'}}^{G(F)} (\rho_{y'}).
}
\end{equation}
\end{lemma}
\begin{proof}
\addtocounter{equation}{-1}
\begin{subequations}
First, we consider the case that $x$ and $x'$ are on the same side of $H$.
Then our assumptions and Lemma~\ref{lemmafromlemmaaboutdistancefunctionrelandtransitivityofThetaGammaver} imply that
\[
\Theta^{\normal}_{y' \mid x'} \circ \Theta^{\normal}_{x' \mid x} = \Theta^{\normal}_{y' \mid x} = \Theta^{\normal}_{y' \mid y} \circ \Theta^{\normal}_{y \mid x}.
\]
Thus, we obtain the claim.

Next, we consider the case that $x$ and $x'$ are on opposite sides of $H$.
In this case, the points $x$ and $y'$ are on the same side of $H$, and the points $y$ and $x'$ are on the same side of $H$.
Then according to Lemma~\ref{lemmafromlemmaaboutdistancefunctionrelandtransitivityofThetaGammaver}, we have
\begin{equation}
\label{equationfromlemmalemmafromlemmaaboutdistancefunctionreland34}
\Theta^{\normal}_{x' \mid y} \circ \Theta^{\normal}_{y \mid x} = \Theta^{\normal}_{x' \mid x}
\quad\text{ and} 
\quad
\Theta^{\normal}_{y' \mid x'} 
\circ \Theta^{\normal}_{x' \mid y} = \Theta^{\normal}_{y' \mid y}.
\end{equation}
Combining the equations in \eqref{equationfromlemmalemmafromlemmaaboutdistancefunctionreland34}, we obtain that
\[
\Theta^{\normal}_{y' \mid x'} \circ \Theta^{\normal}_{x' \mid x} = \Theta^{\normal}_{y' \mid x'} \circ \Theta^{\normal}_{x' \mid y} \circ \Theta^{\normal}_{y \mid x} = \Theta^{\normal}_{y' \mid y}
\circ \Theta^{\normal}_{y \mid x}.
\]
Thus, we also obtain the claim in this case.
\end{subequations}
\end{proof}

Now, we obtain the following corollary of Lemma ~\ref{lemmaaboutreplacingpoints}, that will be used to show that the coefficient $p_{x,H}$ is independent of $x$.
\begin{corollary}
	\label{corollaryPhiconjugatewhenreplacingpoints}
	We assume Axioms \ref{axiomaboutHNheartandK}, \ref{axiombijectionofdoublecoset}, \ref{axiomexistenceofRgrp}, and \ref{axiomaboutdimensionofend}.
	Let $H \in \mathfrak{H}_{\Krel}$ and $x, y \in \nobreak \cA_{\gen}$ such that 
	\[
	\mathfrak{H}_{\Krel; x, s_{H} x} = \mathfrak{H}_{\Krel; y, s_{H} y} = \{ H \}.
	\]
	Then we have
	\[
	\Theta^{\normal}_{y \mid x} \circ \Phi_{x, s_{H}} \circ \left(
	\Theta^{\normal}_{y \mid x}
	\right)^{-1} = \Phi_{y, s_{H}}.
	\]
\end{corollary}
\begin{proof}
	We write $s = s_{H}$ and recall that by Lemma~\ref{lemmathetaccommute} we have
	\[
	\Theta^{\normal}_{y \mid x} \circ \ciso{s^{-1} x}{s}
	=
	\ciso{s^{-1} y}{s} \circ \Theta^{\normal}_{s^{-1} y \mid s^{-1} x}.
	\]
	Using this identity, Proposition~\ref{propositionThetaisomorphism}, and Lemma ~\ref{lemmaaboutreplacingpoints}, we obtain
	\begin{align*}
	\Theta^{\normal}_{y \mid x} \circ \Phi_{x, s} \circ \left(
	\Theta^{\normal}_{y \mid x}
	\right)^{-1} &= 
	\Theta^{\normal}_{y \mid x} \circ 
	\left(
	\ciso{s^{-1} x}{s} \circ \Theta^{\normal}_{s^{-1} x \mid x}
	\right) \circ \left(
	\Theta^{\normal}_{y \mid x}
	\right)^{-1} \\
	&= \left(
	\Theta^{\normal}_{y \mid x} \circ \ciso{s^{-1} x}{s}
	\right) \circ \Theta^{\normal}_{s^{-1} x \mid x} \circ \left(
	\Theta^{\normal}_{y \mid x}
	\right)^{-1} \\
	&= \left(
	\ciso{s^{-1} y}{s} \circ \Theta^{\normal}_{s^{-1} y \mid s^{-1} x}
	\right) \circ \Theta^{\normal}_{s^{-1} x \mid x} \circ \left(
	\Theta^{\normal}_{y \mid x}
	\right)^{-1} \\
	&= \ciso{s^{-1} y}{s} \circ \left(
	\Theta^{\normal}_{s^{-1} y \mid s^{-1} x} \circ \Theta^{\normal}_{s^{-1} x \mid x} \circ \left(
	\Theta^{\normal}_{y \mid x}
	\right)^{-1}
	\right) \\
	& = \ciso{s^{-1} y}{s} \circ \Theta^{\normal}_{s^{-1} y \mid y} = \Phi_{y, s} . 
	\qedhere
	\end{align*}
\end{proof}

By using Lemma~\ref{lemmaaboutreplacingpoints}, we can also prove the following proposition about $\cK$-relevant hyperplanes.
\begin{proposition}
\label{propositionkreloodequivalentforallorthereexists}
We assume Axioms \ref{axiomaboutHNheartandK}, \ref{axiombijectionofdoublecoset}, \ref{axiomexistenceofRgrp}, and \ref{axiomaboutdimensionofend}.
Let $H \in \mathfrak{H}_{\Krel}$ and let $x, y \in \cA_{\gen}$ such that $\mathfrak{H}_{\Krel; x, y} = \{ H \}$.
Then we have
\[
\Theta_{x \mid y} \circ \Theta_{y \mid x} \notin \Coeff \cdot \id_{\ind_{K_{x}}^{G(F)}(\rho_{x})}.
\]
\end{proposition}
\begin{proof}
Since $H \in \mathfrak{H}_{\Krel}$, there exists $x', y' \in \cA_{\gen}$ such that $\mathfrak{H}_{x', y'} = \{ H \}$ and
\[
\Theta_{x' \mid y'} \circ \Theta_{y' \mid x'} \notin \Coeff \cdot \id_{\ind_{K_{x'}}^{G(F)}(\rho_{x'})},
\]
or, equivalently,
\[
\Theta^{\normal}_{x' \mid y'} \circ \Theta^{\normal}_{y' \mid x'} \notin \Coeff \cdot \id_{\ind_{K_{x'}}^{G(F)}(\rho_{x'})}.
\]
According to Lemma~\ref{lemmaaboutreplacingpoints}, we have
\[
\Theta^{\normal}_{y' \mid x'} \circ \Theta^{\normal}_{x' \mid x} = \Theta^{\normal}_{y' \mid y} \circ \Theta^{\normal}_{y \mid x}
\qquad
\text{and}
\qquad
\Theta^{\normal}_{x' \mid y'} \circ \Theta^{\normal}_{y' \mid y} = \Theta^{\normal}_{x' \mid x} \circ \Theta^{\normal}_{x \mid y}.
\]
Combining them, we obtain that
\[
\Theta^{\normal}_{x' \mid y'} \circ \Theta^{\normal}_{y' \mid x'} \circ \Theta^{\normal}_{x' \mid x} = \Theta^{\normal}_{x' \mid y'} \circ \Theta^{\normal}_{y' \mid y} \circ \Theta^{\normal}_{y \mid x} = \Theta^{\normal}_{x' \mid x} \circ \Theta^{\normal}_{x \mid y} \circ \Theta^{\normal}_{y \mid x}.
\]
Using Proposition~\ref{propositionThetaisomorphism}, we have
\[
\Theta^{\normal}_{x' \mid y'} \circ \Theta^{\normal}_{y' \mid x'} = 
\Theta^{\normal}_{x' \mid x}
 \circ \left(
\Theta^{\normal}_{x \mid y} \circ \Theta^{\normal}_{y \mid x} 
\right) \circ 
\left(
\Theta^{\normal}_{x' \mid x}
\right)^{-1}.
\]
Thus, we obtain that
\[
\Theta^{\normal}_{x \mid y} \circ \Theta^{\normal}_{y \mid x} \notin \Coeff \cdot \id_{\ind_{K_{x}}^{G(F)}(\rho_{x})},
\]
and hence
\[
\Theta_{x \mid y} \circ \Theta_{y \mid x} \notin \Coeff \cdot \id_{\ind_{K_{x}}^{G(F)}(\rho_{x})}.
\qedhere
\]
\end{proof}

Now, we can strengthen the statement of Corollary~\ref{corollaryweakversionofquadraticrelation}.
\begin{proposition}
	\label{propositionstrongversionofquadraticrelation}
	We assume Axioms \ref{axiomaboutHNheartandK}, \ref{axiombijectionofdoublecoset}, \ref{axiomexistenceofRgrp}, and \ref{axiomaboutdimensionofend}.
	Let $H \in \mathfrak{H}_{\Krel}$ and $x \in \cA_{\gen}$ such that $\mathfrak{H}_{\Krel; x, s_{H} x} = \{ H \}$.
	Then there exist non-zero $p_{x, H}, q_{x, H} \in \Coeff$
such that
	\[
	\left(
	\Phi_{x, s_{H}}
	\right)^{2} = p_{x, H} \cdot \Phi_{x, s_{H}} + q_{x, H} \cdot \Phi_{x, 1}.
	\]
	Moreover, the coefficients $p_{x, H}$ and $q_{x, H}$ are independent of the point $x \in \cA_{\gen}$ that satisfies $\mathfrak{H}_{\Krel; x, s_{H} x} = \nobreak \{ H \}$.
\end{proposition}

\begin{proof}
Let $y \in \cA_{\gen}$ such that $\mathfrak{H}_{y, s_{H} y} = \{ H \}$.
Then by Corollary~\ref{corollaryweakversionofquadraticrelation} and Lemma~\ref{lemmaqHisnonzero} there exist $p_{y, H}, q_{y, H} \in \Coeff$ such that
$q_{y, H}$ is non-zero and
\[
\left(
\Phi_{y, s_{H}}
\right)^{2} = p_{y, H} \cdot \Phi_{y, s_{H}} + q_{y, H} \cdot \Phi_{y, 1}.
\]
On the other hand, according to Corollary~\ref{corollaryPhiconjugatewhenreplacingpoints}, we have
\[
\Theta^{\normal}_{y \mid x} \circ \Phi_{x, s_{H}} \circ \left(
\Theta^{\normal}_{y \mid x}
\right)^{-1} = \Phi_{y, s_{H}}.
\]
Hence, we also obtain that 
\[
\left(
\Phi_{x, s_{H}}
\right)^{2} = p_{y, H} \cdot \Phi_{x, s_{H}} + q_{y, H} \cdot \Phi_{x, 1}.
\]

Since $H \in \mathfrak{H}_{\Krel}$ and $\mathfrak{H}_{\Krel; x, s_{H} x} = \nobreak \{ H \}$, by combining Proposition~\ref{propositionkreloodequivalentforallorthereexists} 
with Lemma~\ref{lemmaphivstheta}, we obtain that
\(
\left(
\Phi_{x, s_{H}}
\right)^{2} \not \in \Coeff \cdot \id_{\ind_{K_{x}}^{G(F)}(\rho_{x})},
\)
that is, $p_{y, H} \neq 0$. Moreover, $p_{y, H}$ and $q_{y, H}$ are independent of $x$. 
\end{proof}

Scaling the operators $\Phi_{x,s_H}$ we can ensure that the coefficients in the above Proposition have a particularly nice form. 
In order to do so, we choose a partitioning of the coefficient field as follows.
\begin{choice}\label{choiceofCoeffplus}
Let $\Coeffplus$ be a subset of $\Coeffinvnontriv$
such that \index{notation-ax}{Cplus@$\Coeffplus$}
\[
\abs{\{q, q^{-1}\} \cap \Coeffplus}=1
	\quad \text{for all} \quad q \in \Coeffinvnontriv .
\] 
\end{choice}
This choice is necessary as the coefficient $q_{s_H}$ in the quadratic relation in the Proposition below is a priori only determined up to taking its inverse, and we can freely choose any of the (at most) two options.
Strictly speaking it suffices to only make such a choice for the subset of $\Coeff$ that contains all possible values of $q_{s_H}$.
In particular,
if $\ell = 0$, then we will see below in Proposition \ref{value-of-qs} that $q_{s_H}$ is a quotient of two positive integers,
so we only need to choose a subset of the positive rational numbers. One possible choice in that case is to say that $q_{s_H}>1$, which is the choice that, for example, Howlett and Lehrer \cite[(3.19)~Definition and (4.14)~Theorem]{MR570873} and Morris \cite[\S6.9 and 7.12.~Theorem]{Morris} took in their settings, and which is the reason for our choice of notation ``$\Coeffplus$''.

\begin{proposition}
	\label{propositionsuperstrongversionofquadraticrelation}
	We assume Axioms \ref{axiomaboutHNheartandK}, \ref{axiombijectionofdoublecoset}, \ref{axiomexistenceofRgrp}, and \ref{axiomaboutdimensionofend}.
	Then for  $H \in \mathfrak{H}_{\Krel}$ and $x \in \cA_{\gen}$ such that $\mathfrak{H}_{\Krel; x, s_{H} x} = \{ H \}$, there exist unique scalars $d_{s_H} \in \Coeff^\times$ and $q_{s_H} \in \Coeffplus$ such that
	\[
	(d_{{s_H}} \cdot \Phi_{x,{s_H}})^{2} = (q_{{s_H}} - 1) \cdot (d_{{s_H}} \cdot \Phi_{x,{s_H}}) + q_{{s_H}} \cdot \Phi_{x,1}.
	\] 
	Moreover, the scalars $d_{s_{H}}$ and $q_{{s_H}}$ do not depend on the choice of the point $x \in \cA_{\gen}$.
\end{proposition}

\begin{proof}
\addtocounter{equation}{-1}
\begin{subequations}
According to Proposition~\ref{propositionstrongversionofquadraticrelation}, the endomorphism $\Phi_{x, s_{H}}$ satisfies the quadratic relation
        \[
	\left(
	\Phi_{x, s_{H}}
	\right)^{2} = p_{x, H} \cdot \Phi_{x, s_{H}} + q_{x, H} \cdot \Phi_{x, 1}
	\]
with $p_{x, H}$ and $q_{x, H}$ non-zero.
Hence, for $d \in \Coeff^\times$, the endomorphism $d \cdot \Phi_{x,{s_H}}$ satisfies the quadratic relation
\[
(d \cdot \Phi_{x,{s_H}})^{2} = p_{x, H} d \cdot (d \cdot \Phi_{x,{s_H}}) + q_{x, H} d^{2} \cdot \Phi_{x, 1}.
\]
Thus, to prove the first claim of the proposition, it suffices to show that the quadratic equation
\begin{equation}
\label{quadraticequationfordsH}
q_{x, H} d^{2} - p_{x, H} d -1 = 0
\end{equation}
has a unique solution $d = d_{s_{H}}$ with $q_{x, H} d_{s_{H}}^{2} \in \Coeffplus$.
Since $q_{x, H} \neq 0$ and $\Coeff$
is algebraically closed,
Equation \eqref{quadraticequationfordsH} has two solutions $d_{1}, d_{2} \in \Coeff$ 
that are possibly equal.
Note that $d = 0$ is not a solution of Equation \eqref{quadraticequationfordsH}, i.e., $d_{1}, d_{2} \in \Coeff^\times$. 
Let $i = 1$ or $2$.
Since $q_{x, H}, d_{i} \in \Coeff^{\times}$, we have $q_{x, H} d_{i}^{2}  \neq 0$.
Moreover, since $p_{x, H} \neq 0$,
we have $q_{x, H} d_{i}^{2} - 1 = p_{x, H} d_{i} \neq 0$, that is, $q_{x, H} d_{i}^{2} \neq 1$.
Hence, we conclude that $q_{x, H} d_{i}^{2} \in \Coeffinvnontriv$.
Since the solutions $d_{1}$ and $d_{2}$ satisfy $d_{1} d_{2} = -1/ q_{x, H}$, we have $q_{x, H} d_{2}^{2} = \nobreak (q_{x, H} d_{1}^{2})^{-1}$.
Thus, our choice of $\Coeffplus$ in Choice~\ref{choiceofCoeffplus} implies that
\[
\abs{\left\{
q_{x, H} d_{1}^{2}, q_{x, H} d_{2}^{2}
\right\} \cap \Coeffplus}=1.
\] 

If $q_{x, H} d_{1}^{2} \neq q_{x, H} d_{2}^{2}$, this equation implies that exactly one of the solutions $d_{1}$ and $d_{2}$ satisfies $q_{x, H} d_{i}^{2} \in \nobreak \Coeffplus$, as desired.

It remains to consider the case that $q_{x, H} d_{1}^{2} = q_{x, H} d_{2}^{2}$.
Then we have
\[
p_{x, H} d_{1} = (q_{x, H} d_{1}^{2} -1) = (q_{x, H} d_{2}^{2} -1) = p_{x, H} d_{2} ,
\]
so $d_1 = d_2$.
We also note that since $q_{x, H} d_{1}^{2} \neq 1$ and $q_{x, H} d_{1}^{2} = q_{x, H} d_{2}^{2} = (q_{x, H} d_{1}^{2})^{-1}$,
we obtain that the characteristic $\ell$ of $\Coeff$ cannot be two in this case,
and we have $q_{x, H} d_{1}^{2} = q_{x, H} d_{2}^{2} = -1$.
Thus, we conclude that Equation \eqref{quadraticequationfordsH} has the unique solution $d_{s_{H}} \coloneqq d_{1} = d_{2}$, and we have
$
q_{x, H} d_{s_{H}}^{2} = -1 \in \Coeffplus
$.

Since the coefficients $p_{x, H}$ and $q_{x, H}$ are independent of the point $x \in \cA_{\gen}$, Equation \eqref{quadraticequationfordsH} is independent of the point $x \in \cA_{\gen}$.
Hence, the scalars $d_{s_{H}}$ and $q_{s_{H}} = q_{x, H} d_{s_{H}}^{2}$ are also independent of the point $x \in \cA_{\gen}$.
\end{subequations}
\end{proof}

Proposition \ref{propositionsuperstrongversionofquadraticrelation} allows us to refine Choice~\ref{choiceofT} by replacing $T_{\widetilde{s}}$ in Choice~\ref{choiceofT} by $d_{s} \cdot T_{\widetilde{s}}$ for all lifts $\widetilde{s} \in \Nheart$ of $s \in S_{\Krel}$. Then noting Remark~\ref{remarkaboutphix1}, we obtain a choice of $\cT$ that satisfies the following properties.
\begin{choice}
\label{choice:simplereflection}
We choose a family of non-zero elements
\[
\cT =
\left\{
T_{n} \in \Hom_{K_{M}}\left(
^n\!\rho_{M}, \rho_{M}
\right)
\right\}_{n \in \Nheart}
\]
satisfying Conditions \eqref{T1istheidentitymap} and \eqref{Tnk=TnrhoMk} of Choice~\ref{choiceofT} and such that  \index{notation-ax}{qs@$q_s$} for every $s \in S_{\Krel}$ there exists $q_s \in \Coeffplus$ such that
\[
(\Phi_{s})^{2} = (q_{s} - 1) \cdot \Phi_{s} + q_{s} \cdot \Phi_{1} .
\]
\end{choice}

\begin{proposition}
	\label{propositionaboutquadraticrelationsrevised}
	We assume Axioms \ref{axiomaboutHNheartandK}, \ref{axiombijectionofdoublecoset}, \ref{axiomexistenceofRgrp}, and \ref{axiomaboutdimensionofend} and fix a subset  $\Coeffplus \subset \Coeffinvnontriv$
as in Choice \ref{choiceofCoeffplus}. Then we can choose $\cT$ to satisfy all the properties in Choice~\ref{choice:simplereflection}. This means in particular that for each $s \in S_{\Krel}$,
there exists $q_{s} \in \Coeffplus$
such that the element $\Phi_{s}$ satisfies the quadratic relation
	\begin{align*}
	(\Phi_{s})^{2} = (q_{s} - 1) \cdot \Phi_{s} + q_{s} \cdot \Phi_{1}.
	\end{align*}
	Moreover, if $s, s' \in S_{\Krel}$ are $\Wheart$-conjugate, then we have $q_{s} = q_{s'}$.
\end{proposition}
\begin{proof}
Since we have $\mathfrak{H}_{\Krel; x_{0}, s x_{0}} = \{ H_{s} \}$ for all $s \in S_{\Krel}$, the claim that we can choose $\cT$ as in Choice~\ref{choice:simplereflection} follows from Proposition \ref{propositionsuperstrongversionofquadraticrelation} as explained in the paragraph before Choice~\ref{choice:simplereflection} and the quadratic relation is part of the properties stated in Choice~\ref{choice:simplereflection}. It remains to show that we have $q_{s} = q_{s'}$ if $s, s' \in S_{\Krel}$ are $\Wheart$-conjugate.
Let $w \in \Wheart$ such that $w s w^{-1} = s'$.
By replacing $w$ with $ws$ if necessary, we may suppose that 
\[
\flength(ws) = \flength(w) + 1 = \flength(w) + \flength(s).
\]
Then, according to Proposition~\ref{propositioniflengthsumthentransitivityuptotwococycle}, we have
\[
\Phi_{w} \Phi_{s} = \muT(w, s) \cdot \Phi_{ws}.
\]
On the other hand, we have
\[
\flength(s' w) = \flength(ws) = \flength(w) + 1 = \flength(s') + \flength(w).
\]
Hence, Proposition~\ref{propositioniflengthsumthentransitivityuptotwococycle} also implies that
\[
\Phi_{s'} \Phi_{w} = \muT(s', w) \cdot \Phi_{s' w} = \muT(s', w) \cdot \Phi_{ws}.
\]
Writing $d_{s, s'}=\muT(w, s) / \muT(s', w) \in \Coeff^\times$
and using that $\Phi_{w}$ is invertible by Corollary \ref{corollaryPhixwinvertible}, we obtain by  combining the above equations that 
\(
\Phi_{w} \Phi_{s} \Phi_{w}^{-1} = d_{s, s'} \cdot \Phi_{s'}.
\)
Since $\Phi_{w} \Phi_{s} \Phi_{w}^{-1}$ satisfies the same quadratic relation as $\Phi_{s}$, we obtain
\[
\left(
d_{s, s'} \cdot \Phi_{s'}
\right)^{2} = (q_{s} - 1) \cdot \left(
d_{s, s'} \cdot \Phi_{s'}
\right)
+ q_{s} \cdot \Phi_{1},
\]
On the other hand, we have the quadratic relation
\[
(\Phi_{s'})^{2} = (q_{s'} - 1) \cdot \Phi_{s'} + q_{s'} \cdot \Phi_{1}.
\]
Since $q_{s}, q_{s'} \in \Coeffplus$, Proposition \ref{propositionsuperstrongversionofquadraticrelation} implies that $d_{s, s'} = 1$ and $q_{s} = q_{s'}$.
\end{proof}
\begin{remark}
\label{remarkaboutconjugacyofsimplereflections}
The proof of Proposition~\ref{propositionaboutquadraticrelationsrevised} implies the following claim.
Suppose that $s, s' \in \nobreak S_{\Krel}$ and $w \in \Wheart$ satisfy $w s w^{-1} = s'$ and $\flength(ws) = \flength(w) + \flength(s)$.
Then we have $\Phi_{w} \Phi_{s} \Phi_{w}^{-1} = \Phi_{s'}$.
In particular, for all $s \in S_{\Krel}$ and $t \in \Wzero$, we have $\Phi_{t} \Phi_{s} \Phi_{t}^{-1} = \Phi_{t s t^{-1}}$.
\end{remark}

\subsection{The coefficients in the quadratic relations in characteristic zero} 
\label{subsec:quadratic-banal}
In this subsection we keep the notation from the previous subsection.
In addition, we assume that the characteristic $\ell$ of $\Coeff$
is zero or a banal prime for $G(F)$.
Recall that a \emph{banal} prime is one that does not divide the order of any finite quotient
of any compact, open subgroup of $G(F)$.
Doing so allows us to 
give in Proposition \ref{propositionqsincharzero} below a more explicit description of the coefficients $q_s$ in the quadratic relations in Proposition \ref{propositionaboutquadraticrelationsrevised}.
Our restriction on the characteristic implies that
for every finite quotient $H$ of a compact, open subgroup of $G(F)$,
the finite-dimensional $\Coeff$-representations of $H$ are semisimple.
Moreover, 
the order $|H|$ of $H$ is not zero in $\Coeff$
and neither is the dimension of any irreducible $\Coeff$-representation
of $H$.
Thus, we can divide by 
these numbers wherever convenient.

\begin{proposition} \label{propositionqsincharzero}
	\label{value-of-qs}
	We assume Axioms \ref{axiomaboutHNheartandK}, \ref{axiombijectionofdoublecoset}, \ref{axiomexistenceofRgrp}, and \ref{axiomaboutdimensionofend},
fix a subset $\Coeffplus \subset \Coeffinvnontriv$ as in
Choice \ref{choiceofCoeffplus},
and choose $\cT$ as in Choice~\ref{choice:simplereflection}.

	Let $s \in S_{\Krel}$ and let $x \in \nobreak \cA_{\gen}$ such that $\mathfrak{H}_{x, s x} = \{ H_{s} \}$.
	Then the compactly induced representation $\ind_{K_{x}}^{K_{x, s}}(\rho_{x})$ decomposes into a direct sum of two inequivalent irreducible
representations:
	\[
	\ind_{K_{x}}^{K_{x, s}}(\rho_{x}) = \rho_{1} \oplus \rho_{2}. 
	\]
We assume without loss of generality that $\rho_1$ and $\rho_2$  are chosen so that $\dim_{\Coeff}(\rho_{1})/\dim_{\Coeff}(\rho_{2}) \in \Coeffplus \cup \{1\}$. Then 
	\[
	q_{s} = \frac{
		\dim_{\Coeff}(\rho_{1})
	}{
		\dim_{\Coeff}(\rho_{2})
	}.
	\]
\end{proposition}

To prove Proposition~\ref{value-of-qs},
we present a generalization of a result of Howlett and Lehrer \cite{MR570873}.

\begin{lemma}
	\label{lemmacalculationofqparameter}
Let $H$ be a compact topological group and $K$ an open subgroup of $H$. 
If $\ell \neq 0$, then we suppose that the order of no finite quotient of $H$ is divisible by $\ell$. 
Let $(\rho, V_{\rho})$ be an irreducible smooth representation of $K$. We suppose that the induced representation $\ind_{K}^{H} (\rho)$ decomposes into a direct sum of two inequivalent irreducible representations $\rho_{1}$ and $\rho_{2}$.
	
	Then there exists an element $\Phi_{h}$ of 
	\(
	\End_{H}\left(
	\ind_{K}^{H} (\rho)
	\right) \simeq \cH(H, \rho)
	\)
	such that $\supp(\Phi_{h}) = \nobreak K h K$ with $h \in I_{H}(\rho) \smallsetminus K$, and such that we have
	\[
	(\Phi_{h})^{2} = (q - 1) \cdot \Phi_{h} + q \cdot \Phi_{1},
	\]
	where $\Phi_{1}$ denotes the identity map on $\ind_{K}^{H} (V_\rho)$ and \(
	q = \frac{
		\dim_{\Coeff}(\rho_{1})
	}{
		\dim_{\Coeff}(\rho_{2})
	}.
	\)
\end{lemma}
Here we define the Hecke algebra $\cH(H, \rho)$ associated to $(K, \rho)$ analogously to Section~\ref{Hecke algebras and endomorphism algebras}.
\begin{proof}[Proof of Lemma \ref{lemmacalculationofqparameter}]
	\addtocounter{equation}{-1}
	\begin{subequations}
		The proof of the lemma is essentially the same as \cite[Theorem~3.18 (ii)]{MR570873},
but we include it here for the convenience of the reader.
		
		Since $\rho$ is a smooth representation of a compact group $K$, the kernel of $\rho$ is a normal, open subgroup of $K$.
		We define the normal, open subgroup $N$ of $H$ by
		\[
		N = \bigcap_{a \in H/K} a (\ker{\rho}) a^{-1}.
		\]
		By replacing $H$ and $K$ with $H / N$ and $K / N$, respectively, and regarding $\rho$ as a representation of $K / N$, we may suppose that $H$ and $K$ are finite groups.
		
		We fix a non-zero element $\Phi'_{h} \in \End_{H} \left(
		\ind_{K}^{H} (\rho)
		\right)$ such that $\supp(\Phi'_{h}) = K h K$.
		Since $\End_{H}\left(
		\ind_{K}^{H} (\rho)
		\right)$ is two dimensional by Schur's lemma,
$\Phi_{1}$ and $\Phi'_{h}$ give a basis of $\End_{H} \left(
		\ind_{K}^{H} (\rho)
		\right)$.
		Let $p_{1} \in \nobreak \End_{H} \left(
		\ind_{K}^{H} (\rho)
		\right)$ denote the projection onto $\rho_{1}$ with respect to the decomposition
		$
		\ind_{K}^{H} (\rho) = \rho_{1} \oplus \rho_{2}
		$.
		We write
		\begin{align}
		\label{linearcombofprojection}
		p_{1} = \lambda \cdot \Phi_{1} + \mu \cdot \Phi'_{h}
		\end{align}
		for some $\lambda, \mu \in \Coeff$.
		Since the element $p_{1}$ satisfies $p_{1}^{2} = p_{1}$, Equation \eqref{linearcombofprojection} implies that
		\[
		\left(
		\lambda \cdot \Phi_{1} + \mu \cdot \Phi'_{h}
		\right)^{2} = \lambda \cdot \Phi_{1} + \mu \cdot \Phi'_{h}.
		\]
		Hence, we have
		\begin{align}
		\label{quasiquadraticrelation}
		\left(
		\mu \cdot \Phi'_{h}
		\right)^{2} = (1 - 2 \lambda) \cdot \left(
		\mu \cdot \Phi'_{h}
		\right) + \lambda(1 - \lambda) \cdot \Phi_{1}.
		\end{align}
		
		We can calculate the value $\lambda$ as follows.
		We fix a non-zero element
$v \in V_{\rho}$ and define the element $f_{v} \in \ind_{K}^{H} (V_{\rho})$ by
		\[
		f_v(h') =
		\begin{cases}
		\rho(h') (v) & (h' \in K), \\
		0 & (\text{otherwise})
		\end{cases} 
		\]
		for $h' \in H$.
		We will apply both sides of \eqref{linearcombofprojection} to $f_{v}$ and compare the values at the identity element of $H$.
		Recall that the projection $p_{1}$ can be written as
		\[
		p_{1} = \frac{\dim_{\Coeff}\left(
			\rho_{1}
			\right)}{\abs{H}} \sum_{h' \in H} \trace(\rho_{1}(h'^{-1})) \cdot \left(
		\ind_{K}^{H}(\rho)
		\right)(h'),
		\]
		where $\trace(\rho_{1}(h'^{-1}))$ denotes the trace of the linear map $\rho_{1}(h'^{-1})$.
		Hence, we have
		\begin{align*}
		\left(
		p_{1}(f_{v})
		\right)(1) &= \frac{\dim_{\Coeff}\left(
			\rho_{1}
			\right)}{\abs{H}} \sum_{h \in H} \trace(\rho_{1}(h'^{-1})) \cdot 
		\left(
		\left(
		\ind_{K}^{H}(\rho)
		\right)(h') \cdot (f_{v})
		\right)(1) \\
		&= \frac{\dim_{\Coeff}\left(
			\rho_{1}
			\right)}{\abs{H}} \sum_{h' \in H} \trace(\rho_{1}(h'^{-1})) \cdot f_{v}(h') \\
		&= \frac{\dim_{\Coeff}\left(
			\rho_{1}
			\right)}{\abs{H}} \sum_{k \in K} \trace(\rho_{1}(k^{-1})) \cdot \rho(k) (v) \\
		&= \left(
		\frac{\dim_{\Coeff}\left(
			\rho_{1}
			\right)}{\abs{H}} \sum_{k \in K} \trace(\rho_{1}(k^{-1})) \cdot \rho(k)
		\right) (v).
		\end{align*}
		Since 
		\[
		\frac{\dim_{\Coeff}\left(
			\rho_{1}
			\right)}{\abs{H}} \sum_{k \in K} \trace(\rho_{1}(k^{-1})) \cdot \rho(k) \in \End_{K}(\rho),
		\]
		and $\rho$ is an irreducible representation of $K$, Schur's lemma implies that 
		\[
		\frac{\dim_{\Coeff}\left(
			\rho_{1}
			\right)}{\abs{H}} \sum_{k \in K} \trace(\rho_{1}(k^{-1})) \cdot \rho(k)
		\]
		is a scalar multiplication on $V_{\rho}$.
		Moreover,
		letting 
		$(\phantom{x},\phantom{y})_{H}$ and $(\phantom{x},\phantom{y})_{K}$
		denote the inner products
		of finite-dimensional representations of $H$ and $K$,
		the scalar is calculated as
		\begin{align*}
		\frac{1}{\dim_{\Coeff}(\rho)} 
		\frac{\dim_{\Coeff}\left(
			\rho_{1}
			\right)}{\abs{H}} \sum_{k \in K} \trace(\rho_{1}(k^{-1})) \cdot \trace(\rho(k)) 
		&= \frac{1}{\dim_{\Coeff}(\rho)} 
		\frac{\dim_{\Coeff}\left(
			\rho_{1}
			\right)}{\abs{H}} \abs{K} (\rho_{1} \restriction_{K}, \rho)_{K} \\
		&= \frac{1}{\dim_{\Coeff}(\rho)} 
		\frac{\dim_{\Coeff}\left(
			\rho_{1}
			\right)}{\abs{H}} \abs{K} (\rho_{1}, \ind_{K}^{H} (\rho))_{H} \\
		&= \frac{1}{\dim_{\Coeff}(\rho)} 
		\frac{\dim_{\Coeff}\left(
			\rho_{1}
			\right)}{\abs{H}} \abs{K} \\
		&= \dim_{\Coeff}\left(
		\rho_{1}
		\right) \cdot \left(
		\dim_{\Coeff}(\rho) \cdot \abs{H/K}
		\right)^{-1} \\
		&= \frac{\dim_{\Coeff}\left(
			\rho_{1}
			\right)}{\dim_{\Coeff}\left(
			\ind_{K}^{H}(\rho)
			\right)} \\
		&= \frac{
			\dim_{\Coeff}(\rho_{1})
		}{
			\dim_{\Coeff}(\rho_{1}) + \dim_{\Coeff}(\rho_{2})
		}.
		\end{align*}
		On the other hand, the right-hand side of \eqref{linearcombofprojection} is calculated as
		\[
		\left(
		\left(
		\lambda \cdot \Phi_{1} + \mu \cdot \Phi'_{h}
		\right)(f_{v})
		\right)(1) = \lambda \cdot v.
		\]
		Thus, we have
		\begin{align}
		\label{valueoflambda}
		\frac{
			\dim_{\Coeff}(\rho_{1})
		}{
			\dim_{\Coeff}(\rho_{1}) + \dim_{\Coeff}(\rho_{2})
		} = \lambda.
		\end{align}
		Now, we define the element $\Phi_{h}$ of $\End_{H} \left(
		\ind_{K}^{H} (\rho)
		\right)$ by
		\[
		\Phi_{h} = - \frac{
			\dim_{\Coeff}(\rho_{1}) + \dim_{\Coeff}(\rho_{2})
		}{
			\dim_{\Coeff}(\rho_{2})
		} \cdot \mu \cdot \Phi'_{h}.
		\]
		Note that $\Phi_h$ is non-zero,
		because $\lambda(1-\lambda) \neq 0$ implies by \eqref{quasiquadraticrelation} that $\mu$ is non-zero.
		Then according to \eqref{quasiquadraticrelation} and \eqref{valueoflambda}, we have
		\[
		(\Phi_{h})^{2} = (q - 1) \cdot \Phi_{h} + q \cdot \Phi_{1}.
		\qedhere
		\]
	\end{subequations}
\end{proof}
\begin{proof}[Proof of Proposition~\ref{value-of-qs}]
\addtocounter{equation}{-2}
\begin{subequations}
	Since the group $K_{x, s}$ is compact, the representation $\ind_{K_{x}}^{K_{x, s}}(\rho_{x})$ is semi-simple.
	Hence, the first claim follows from Lemma~\ref{lemmadimensionoftheendomorphismalgebraKxsH}.
	To prove the second claim, note that by Proposition \ref{propositionsuperstrongversionofquadraticrelation}, the elements $\Phi_{s} = \Phi_{x_{0}, s}$ and $\Phi_{x, s}$ satisfy the same quadratic relation.
	Hence, we have 
	\begin{equation}
	\label{inpropofqsquadraticrel1}
	(\Phi_{x, s})^{2} = (q_{s} - 1) \cdot \Phi_{x, s} + q_{s} \cdot \Phi_{x, 1}.
	\end{equation}
	On the other hand, applying Lemma \ref{lemmacalculationofqparameter} to $H = K_{x, s}$, $K = K_{x}$, and $\rho = \rho_{x}$, we obtain that there exists an element $\Phi \in \cH(K_{x, s}, \rho_{x})_{s}$ such that
	\begin{equation}
	\label{inpropofqsquadraticrel2}
	(\Phi)^{2} = (q - 1) \cdot \Phi + q \cdot \Phi_{1},
	\end{equation}
	where $q =  \frac{
		\dim_{\Coeff}(\rho_{1})
	}{
		\dim_{\Coeff}(\rho_{2})
	}$.
	Since $\cH(K_{x, s}, \rho_{x})_{s}$ is one dimensional by Proposition~\ref{propositionvectorspacedecomposition}, there exists $d \in \Coeff^{\times}$ such that $\Phi = d \cdot \Phi_{x, s}$.
	Then comparing \eqref{inpropofqsquadraticrel1} with \eqref{inpropofqsquadraticrel2}, we have $q \neq 1$, and replacing $\rho_{1}$ with $\rho_{2}$ if necessary, we may suppose that $q \in \Coeffplus$.
	Then the proposition follows from Proposition \ref{propositionsuperstrongversionofquadraticrelation} and Equations~\eqref{inpropofqsquadraticrel1} and \eqref{inpropofqsquadraticrel2}.
\end{subequations}
\end{proof}

\subsection{The description of the Hecke algebra}
\label{subsection:The description of the Hecke algebra}
In this subsection, we will prove our main theorem of this Section, Theorem \ref{theoremstructureofhecke}, about the structure of the Hecke algebra $\cH(G(F),\rho_{x_0})$. We keep the notation from Section \ref{sec:simplereflections}, i.e., the notation from all previous subsections, but allowing the coefficient field $\Coeff$ again to be of arbitrary characteristic. We also assume that all previous axioms hold, i.e., Axioms \ref{axiomaboutHNheartandK}, \ref{axiombijectionofdoublecoset}, \ref{axiomexistenceofRgrp}, and \ref{axiomaboutdimensionofend}, and fix a subset  $\Coeffplus \subset \Coeffinvnontriv$ as in Choice \ref{choiceofCoeffplus}. 

Recall that we defined the elements $\Phi_w\in  \End_{G(F)} \left( \ind_{K_{x}}^{G(F)} (\rho_{x}) \right)  \simeq \cH(G(F),\rho_x)$
for all $w \in \nobreak \Wheart$ in Definition \ref{definitionPhixw}
and that these endomorphisms depend on the choice of
a family $\cT$ 
Choice \ref{choice:simplereflection}).
In this subsection, we will adjust $\cT$, and thus $\{\Phi_w\}$,
in a way that makes the latter more compatible with the group
structure of 
$w \in \Wheart$, see Choice \ref{choice:tw} below, so that we can use these basis elements to write down an explicit isomorphism between the Hecke algebra $\cH(G(F), \rho_{x_0})$ and a semi-direct product
of an affine Hecke algebra with a twisted group algebra in Theorem \ref{theoremstructureofhecke}.
In order to obtain the subalgebra isomorphic to the affine Hecke algebra, we start by proving appropriate braid relations.

\begin{lemma}
\label{lemmabraidrel}
Let $s_{1}$ and $s_{2}$ be distinct elements of $S_{\Krel}$ such that the order of $s_{1} s_{2}$ in $\Waff$ is $m < \infty$.
Then we have
\[
\underbrace{
\Phi_{s_{1}} \Phi_{s_{2}} \Phi_{s_{1}} \cdots
}_{m \terms}
 = 
 \underbrace{
 \Phi_{s_{2}} \Phi_{s_{1}} \Phi_{s_{2}} \cdots
 }_{m \terms}.
\]
\end{lemma}
\begin{proof}
\addtocounter{equation}{-1}
\begin{subequations}
According to Proposition~\ref{propositioniflengthsumthentransitivityuptotwococycle}, 
there exists $c \in \Coeff^{\times}$ such that
\[
\underbrace{
\Phi_{s_{1}} \Phi_{s_{2}} \Phi_{s_{1}} \cdots
}_{m \terms}
 = c \cdot
\underbrace{
 \Phi_{s_{2}} \Phi_{s_{1}} \Phi_{s_{2}} \cdots
 }_{m \terms}.
\]
We will prove that $c = 1$.
We write
\[
n \coloneqq
\begin{cases}
1 & (2 \mid m) \\
2 & (2 \nmid m)
\end{cases}
\quad
\text{ and }
\quad
\Phi_{2, 1}.
\coloneqq
\underbrace{
 \Phi_{s_{2}} \Phi_{s_{1}} \Phi_{s_{2}} \cdots
 }_{m-1 \terms}
\quad
\text{ so that }
\quad
\Phi_{s_{1}} \Phi_{2, 1} = c \cdot \Phi_{2, 1} \Phi_{s_{n}}.
\]
Then we obtain
\begin{equation*}
	\left( 	\Phi_{s_{1}} \right)^{2} \Phi_{2, 1}=\Phi_{s_1} \circ c \cdot  \Phi_{2,1} \Phi_{s_n}=c^2 \cdot \Phi_{2,1} \left(\Phi_{s_n}\right)^2
\end{equation*}
Since $\Phi_{s_{1}}$ and $\Phi_{s_{n}}$
satisfy the quadratic relations
\[
\left(
\Phi_{s_{1}}
\right)^{2} = (q_{s_{1}} - 1) \cdot \Phi_{s_{1}} + q_{s_{1}} \cdot \Phi_{1}
\quad
\text{ and }
\quad
\left(
\Phi_{s_{n}}
\right)^{2} = (q_{s_{n}} - 1) \cdot \Phi_{s_{n}} + q_{s_{n}} \cdot \Phi_{1},
\]
we obtain that
\begin{align*}
(q_{s_{1}} - 1) \cdot \Phi_{s_{1}} \Phi_{2, 1} + q_{s_{1}} \cdot \Phi_{2, 1}
&= 
c^{2} \cdot \left(
(q_{s_{n}} - 1) \cdot \Phi_{2, 1} \Phi_{s_{n}} + q_{s_{n}} \cdot \Phi_{2, 1}
\right) \\
&= c \cdot (q_{s_{n}} - 1) \cdot \Phi_{s_{1}} \Phi_{2, 1} + c^{2} \cdot q_{s_{n}} \cdot \Phi_{2, 1}.
\end{align*}
Using that $\Phi_{s_1} \Phi_{2, 1}$ and $ \Phi_{2, 1}$ are linearly independent, we conclude that 
\begin{align}
\label{implyc=1}
q_{s_{1}} - 1 &= c \cdot (q_{s_{n}} - 1). 
\end{align}
Since 
\(
s_{n} = 
\Bigl(
\underbrace{
s_{2} s_{1} s_{2} \cdots
}_{m-1 \terms}
\Bigr)^{-1}
s_{1} \Bigl(
\underbrace{
s_{2} s_{1} s_{2} \cdots
}_{m-1 \terms}
\Bigr),
\)
the element $s_{n}$ is $\Wheart$-conjugate to $s_{1}$, and therefore by 
 Proposition~\ref{propositionaboutquadraticrelationsrevised} we have
\(
q_{s_{1}} = q_{s_{n}} \neq 1.
\)
Thus Equation \eqref{implyc=1} implies that $c = 1$.
\end{subequations}
\end{proof}
According to Proposition~\ref{Rrhosimeqaffineweyl}, the group $\Waff \simeq W_{\Krel}$ is a Coxeter group.
Hence, we can define the notion of a reduced expression for $w \in \Waff$ in the usual way.
\begin{corollary}
\label{corollaryofthebraidrelations}
	We assume Axioms \ref{axiomaboutHNheartandK}, \ref{axiombijectionofdoublecoset}, \ref{axiomexistenceofRgrp}, \ref{axiomaboutdimensionofend}, and that $\cT$ is chosen as in Choice~\ref{choice:simplereflection}.
Let $w \in \Waff$.
Then the element 
$
\Phi_{s_{1}} \Phi_{s_{2}} \cdots \Phi_{s_{\ellsubstitute}}
$
does not depend on the choice of a reduced expression
$
w = s_{1} s_{2} \cdots s_{\ellsubstitute}
$
for $w$.
\end{corollary}
\begin{proof}
According to \cite[Theorem~3.3.1 (ii)]{MR2133266}, every two reduced expressions for $w$ can be connected via a sequence of braid-moves.
Then the corollary follows from Lemma~\ref{lemmabraidrel}.
\end{proof}

We now refine Choice \ref{choice:simplereflection}.
Let $t \in \Wzero$ and $w \in \Waff \smallsetminus \{ 1\}$ with reduced expression
$
w = s_{1} s_{2} \cdots s_{\ellsubstitute}
$.
According to Proposition~\ref{propositioniflengthsumthentransitivityuptotwococycle}, there exists $d_{tw} \in \Coeff^{\times}$ such that
$
d_{tw} \cdot \Phi_{tw} =\Phi_t \Phi_{s_{1}} \Phi_{s_{2}} \cdots \Phi_{s_{\ellsubstitute}}
$.
According to Corollary~\ref{corollaryofthebraidrelations}, the scalar $d_{tw}$ does not depend on the choice of a reduced expression for $w$.
By replacing $T_{n}$ with $d_{tw} \cdot T_{n}$ for all lifts $n$ of $tw$ in $\Nheart$ and noting Remark~\ref{remarkaboutphix1}, we obtain the following refinement of Choice \ref{choice:simplereflection}:
\begin{choice}
\label{choice:tw}
We choose a family of non-zero elements
\[
\cT =
\left\{
T_{n} \in \Hom_{K_{M}}\left(
^n\!\rho_{M}, \rho_{M}
\right)
\right\}_{n \in \Nheart}
\]
that satisfies the following conditions:
\begin{enumerate}[(1)]
	\item
	We have
	$
	T_{1} = \id_{\rho_{M}}
	$.
	\label{conditionsofTnkinchoice:triv}
	\item
	\label{conditionsofTnkinchoice:tw}
	For all $n \in \Nheart$ and $k \in K_{M} \cap \Nheart$, we have
	$
	T_{nk} = T_{n} \circ \rho_{M}(k)
	$.
	\item
	\label{conditionofthechoicequadraticrelations}
	 For every $s \in S_{\Krel}$, there exists $q_s \in \Coeffplus$ such that $ (\Phi_{s})^{2} = (q_{s} - 1) \cdot \Phi_{s} + q_{s} \cdot \Phi_{1}$.
	\item
	\label{conditionofthechoicebraidrelations}
	 For all $w \in \Waff$ with reduced expression $	w = s_{1} s_{2} \cdots s_{\ellsubstitute}	$, we have $\Phi_{w} = \Phi_{s_{1}} \Phi_{s_{2}} \cdots \Phi_{s_{\ellsubstitute}}$.
	\item 
	\label{conditionofchoiceproductofWzeroandaff}
	For  all $t \in \Wzero$ and $w \in \Waff$, we have $\Phi_{t} \Phi_{w} = \Phi_{tw}$.
\end{enumerate}
\end{choice}
\begin{proposition}
	\label{propositionchoice:twispossible}
	We assume Axioms \ref{axiomaboutHNheartandK}, \ref{axiombijectionofdoublecoset}, \ref{axiomexistenceofRgrp}, and \ref{axiomaboutdimensionofend} and fix a subset  $\Coeffplus \subset \Coeffinvnontriv$ as in Choice \ref{choiceofCoeffplus}. Then we can choose $\cT$ to satisfy all the properties in Choice~\ref{choice:tw}. 
	\end{proposition}
\begin{proof}
According to Proposition~\ref{propositionaboutquadraticrelationsrevised}, we can choose $\cT$ as in Choice \ref{choice:simplereflection}
 and the paragraph before Choice \ref{choice:tw} explains a rescaling that yields a choice $\cT$ satisfying all the desired properties. 
\end{proof}	

\begin{remark}
\label{remarkaboutchoice:tw}
According to Proposition~\ref{propositionsuperstrongversionofquadraticrelation}, Condition \eqref{conditionofthechoicequadraticrelations} of Choice~\ref{choice:tw} determines the choices of $T_{\widetilde{s}} \in \cT$ for all lifts $\widetilde{s} \in \Nheart$ of $s \in S_{\Krel}$.
Condition \eqref{conditionofthechoicebraidrelations} then determines the choices of $T_{n} \in \cT$ for all lifts $n \in \Nheart$ of $w \in \Waff$.
On the other hand, we can choose $T_{n} \in \cT$ freely for a set of representatives of lifts $n \in \Nheart$ of $t \in \Wzero \smallsetminus \{1\}$. 
Once we fix $T_{n} \in \cT$ for such a set of representatives, Conditions \eqref{conditionsofTnkinchoice:triv} through \eqref{conditionofchoiceproductofWzeroandaff} determine all elements of $\cT$.
We will make a special choice of $T_{n} \in \cT$ for lifts $n \in \Nheart$ of $t \in \Wzero \smallsetminus \{1\}$ in Choice~\ref{choice:star} below.
\end{remark}

We also note the following lemma:
\begin{lemma}
\label{lemmaRvsC}
We choose $\cT$ as in Choice~\ref{choice:tw}.
Then we have
$
\Phi_{t} \Phi_{w} =  \Phi_{t w t^{-1}} \Phi_{t}
$ for all $t \in \Wzero$ and $w \in \Waff$.
\end{lemma}
\begin{proof}
We will prove that
$
\Phi_{t w t^{-1}} = \Phi_{t} \Phi_{w} \Phi_{t}^{-1}
$.
Since we have
$
\Phi_{w} = \Phi_{s_{1}} \Phi_{s_{2}} \cdots \Phi_{s_{\ellsubstitute}}
$
for a reduced expression
$
w = s_{1} s_{2} \cdots s_{\ellsubstitute}
$,
we may suppose that $w = s \in S_{\Krel}$.
Now, the lemma follows from Remark~\ref{remarkaboutconjugacyofsimplereflections}.
\end{proof}

We can rewrite Conditions \eqref{conditionofthechoicebraidrelations} and \eqref{conditionofchoiceproductofWzeroandaff} of Choice~\ref{choice:tw} and Lemma~\ref{lemmaRvsC} in terms of the $2$-cocycle $\muT$.
\begin{proposition}
\label{prop:muTtrivial}
	We assume Axioms \ref{axiomaboutHNheartandK}, \ref{axiombijectionofdoublecoset}, \ref{axiomexistenceofRgrp}, and \ref{axiomaboutdimensionofend}, and we fix $\Coeffplus \subset \Coeffinvnontriv$ as in Choice \ref{choiceofCoeffplus} and $\cT$ as in Choice~\ref{choice:tw}.
Then for $v, w \in \Waff$ such that $\flength(vw) = \flength(v) + \flength(w)$, we have $\muT(v, w) = 1$.
In addition, for $t \in \Wzero$ and $w \in \Waff$, we have $\muT(t, w) = \muT(w, t) = 1$.
\end{proposition}
\begin{proof}
The first claim follows from Proposition~\ref{propositioniflengthsumthentransitivityuptotwococycle} and Condition \eqref{conditionofthechoicebraidrelations} of Choice~\ref{choice:tw}.
The second claim follows by combining Proposition~\ref{propositioniflengthsumthentransitivityuptotwococycle} with Condition \eqref{conditionofchoiceproductofWzeroandaff} of Choice~\ref{choice:tw} and Lemma~\ref{lemmaRvsC} and recalling that $\Wzero$ normalizes $\Waff$.
\end{proof}

\begin{notation}
\label{notn:algebras}
We recall several standard notations.
\begin{enumerate}[(a)]
\item
\label{item:affine-hecke-algebra}
We denote by $\cH_\Coeff(\Waff, q)$
\index{notation-ax}{H(Waffq@$\cH(\Waff, q)$}
the affine Hecke algebra with $\Coeff$-coefficients associated with the affine Weyl group $\Waff$, its set of generators $S_{\Krel}$, and the parameter function $q$.
This means $\cH_\Coeff(\Waff, q)$ is a $\Coeff$-algebra with a vector-space basis
\[
\left\{
\mathbb{T}_{w} \mid w \in \Waff
\right\},
\]
and relations generated by
\begin{equation}
	\label{eqn:abstractrelations}
\mathbb{T}_{s}^{2} = (q_{s} - 1) \cdot \mathbb{T}_{s} + q_s \cdot \mathbb{T}_{1}
\quad
\text{ and }
\quad
\mathbb{T}_{w} = \mathbb{T}_{s_{1}} \mathbb{T}_{s_{2}} \cdots \mathbb{T}_{s_{\ellsubstitute}}
\end{equation}
for all $s \in S_{\Krel}$ and
all $w \in \Waff$ with a reduced expression
$
w = s_{1} s_{2} \cdots s_{\ellsubstitute}
$.
\item
\label{item:twisted-group-algebra}
The twisted group algebra $\Coeff[\Wzero, \muT]$
is defined to be the vector space
\(
\bigoplus_{t \in \Wzero} \Coeff \cdot \gpalg_{t}
\)
equipped with the multiplication 
\(
\gpalg_{t_{1}} \cdot \gpalg_{t_{2}} = \muT(t_{1}, t_{2}) \cdot \gpalg_{t_{1} t_{2}}
\)
for $t_{1}, t_{2} \in \Wzero$.
\item
\label{item:semidirect-product-algebra}
We define the $\Coeff$-algebra $\Coeff[\Wzero, \muT] \ltimes \cH_\Coeff(\Waff, q)$ to be the vector space 
\[
\Coeff[\Wzero, \muT] \otimes \cH_\Coeff(\Waff, q)
\]
 with multiplication rules given by:
\begin{enumerate}[(1)]
\item $\Coeff[\Wzero, \muT]$ and $\cH_\Coeff(\Waff, q)$ are embedded as subalgebras,
\item 
for $t \in \Wzero$ and $w \in \Waff$, we have
\(
\gpalg_{t} \cdot \mathbb{T}_{w} = \mathbb{T}_{t w t^{-1}} \cdot \gpalg_{t}.
\)
\end{enumerate}
\end{enumerate}
\end{notation}
Now we obtain the following structure theorem for our Hecke algebra:

\begin{theorem}
\label{theoremstructureofhecke}
We assume Axioms \ref{axiomaboutHNheartandK}, \ref{axiombijectionofdoublecoset}, \ref{axiomexistenceofRgrp}, and \ref{axiomaboutdimensionofend} and fix a subset  $\Coeffplus \subset \Coeffinvnontriv$ as in Choice \ref{choiceofCoeffplus}. 
Then we can choose $\cT$ as in Choice~\ref{choice:tw}
and the resulting $\Coeff$-linear map 
\[
\cI(\rho_{x_{0}})
\colon
\cH(G(F), \rho_{x_{0}})
\rightarrow
\Coeff[\Wzero, \muT] \ltimes \cH_\Coeff(\Waff, q),
\]
defined by
\[
\cI(\rho_{x_{0}})\left(\varphi_{tw}\right) =  \gpalg_{t} \cdot \mathbb{T}_{w} \quad (t \in \Wzero, w \in \Waff)
\]
is an isomorphism of $\Coeff$-algebras, 
where $\muT$ denotes the restriction to $\Wzero \times \Wzero$ of the $2$-cocycle introduced in Notation~\ref{notationofthetwococycle} and $q$ denotes the parameter function $s \mapsto q_{s}$ appearing in Choice~\ref{choice:tw}\eqref{conditionofthechoicequadraticrelations}.
\end{theorem}
\begin{proof}
By Proposition~\ref{propositionchoice:twispossible} we can choose $\cT$ as in Choice~\ref{choice:tw}. According to Corollary~\ref{corollaryvectorspacedecompositionexplicitver} and Proposition~\ref{propositiondecompositionofW}, $\cI(\rho_{x_{0}})$ is an isomorphism of vector spaces.
Moreover, by Proposition~\ref{propositioniflengthsumthentransitivityuptotwococycle} the restriction of $\cI(\rho_{x_{0}})$ to $\Coeff[\Wzero, \muT]$ is an algebra homomorphism, and
 Conditions \eqref{conditionofthechoicequadraticrelations} and \eqref{conditionofthechoicebraidrelations} of Choice~\ref{choice:tw} imply that the restriction of $\cI(\rho_{x_{0}})$ to $\cH_\Coeff(\Waff, q)$ is also an algebra homomorphism.
Combining these observations with Lemma~\ref{lemmaRvsC}, we obtain that $\cI(\rho_{x_{0}})$ is an isomorphism of $\Coeff$-algebras.
\end{proof}

\subsection{Anti-involution of the Hecke algebra}
\label{Anti-involution of the Hecke algebra}
In this subsection we keep the notation from the previous subsection,
and moreover we assume that $\Coeff$ 
admits a nontrivial involution
$c \mapsto \bar c$.
The fixed field of this involution will then be 
a real closed field that we will denote by $\Rclsd$.
(See Corollary VI.9.3, Theorem IX.2.2, and Proposition IX.2.4 of
\cite{Lang:algebra}.
For instance, $\Rclsd=\bR$ if $\Coeff = \bC$ and $c \mapsto \bar c$ denotes
complex conjugation.
Such an involution exists if and only if $\Coeff$ has characteristic zero.
In this case, one can construct the fixed field $\Rclsd$, and thus
the involution, by choosing a maximal, totally ordered subfield of $\Coeff$.
Note that the choice of $\Rclsd$,
and thus the choice of involution,
is never unique.  
However, in practice, the number of explicit choices that one can make
(i.e., choices that do not require the Axiom of Choice),
is usually zero (say, for $\Coeff = \overline \bQ_\ell$)
or one (say, for $\Coeff = \bC$ or $\overline \bQ$).

Having made such a choice,
write $\abs{c} = c \bar{c}$ for $c \in \Coeff$,
and call $\bar c$ the \emph{$\Coeff$-conjugate}
of $c$ (which depends on our fixed choice of involution even though it is not reflected in the notation).
For a vector space $V$ over $\Coeff$, call a function
$\langle \phantom{x},\phantom{x} \rangle \colon V\times V \rightarrow \Coeff$
\emph{Hermitian}
if it is linear in the first variable, $\Coeff$-conjugate-linear in the second
variable, and we have $\langle v, w \rangle = \overline{
\langle w, v \rangle
}$ for all $v, w \in V$.

Assume that all previous axioms hold, i.e., Axioms \ref{axiomaboutHNheartandK}, \ref{axiombijectionofdoublecoset}, \ref{axiomexistenceofRgrp}, and \ref{axiomaboutdimensionofend}, and fix a subset  $\Coeffplus \subset \Coeffinvnontriv$ as in Choice \ref{choiceofCoeffplus}, and $\cT$ as in Choice~\ref{choice:tw}.
Following \cite[Section~4.4]{ciubotaru:types-unitary}, we define a support-inverting, $\Coeff$-conjugate-linear anti-involution $\varphi \mapsto \varphi^{*}$ 
\index{notation-ax}{*@$(\phantom{\varphi})^*$ ($\Coeff$-conjugate-linear anti-involution)}
on $\cH(G(F),\rho_{x_0})$ as follows.
We fix a $K_{x_0}$-invariant,
positive-definite Hermitian form 
$\langle \phantom{x},\phantom{x} \rangle_{\rho_{x_0}}$ on $V_{\rho_{x_0}}$.
Such a form exists because $K_{x_0}$ is compact, and it is unique
up to $\Rclsd$-scalar multiples because $\rho_{x_0}$ is irreducible.
For $\varphi \in \cH(G(F), \rho_{x_0})$, we define $\varphi^{*} \in \cH(G(F), \rho_{x_0})$ by
\begin{equation}
\label{explicitdescriptionofstar}
\langle
\varphi^{*}(g) (v), w
\rangle_{\rho_{x_0}} = \langle
v, \varphi(g^{-1}) (w)
\rangle_{\rho_{x_0}}
\qquad
\text{for all $g \in G(F)$ and $v, w \in V_{\rho_{x_0}}$.}
\end{equation}
We note that the map
$\varphi \mapsto \varphi^*$ is a support-inverting, $\Coeff$-conjugate-linear anti-involution that does not depend on the choice
of Hermitian form on $V_{\rho_{x_0}}$. Hence, for each $w \in \nobreak \Wheart$,
there is some scalar $c_{w} \in \Coeff^\times$ such that
$\varphi_{w}^{*} = c_{w} \varphi_{w^{-1}}$. 

\begin{lemma}\label{lemmastarscalarforordertwo}
For all $w\in \Wheart$ of order $2$, we have $\abs{c_{w}} = 1$.
Moreover, if $w \in S_{\Krel}$, then $c_{w} = 1$.
\end{lemma}

\begin{proof}
If $w^2 = 1$, then we have
\(
\varphi_w
= \varphi_w^{**}
= (c_{w}\varphi_w)^*
= \bar c_{w}(\varphi_w)^*
= \bar c_{w} c_{w} \varphi_w,
\)
so $\bar c_{w} c_{w} = 1$, as required.
Suppose that $w \in S_{\Krel}$.
Since $*$ is an anti-involution, the element $\varphi_{w}^{*} = c_{w} \varphi_{w}$ satisfies the same quadratic relation as $\varphi_w$.
Thus, Proposition~\ref{propositionsuperstrongversionofquadraticrelation} implies that $c_{w} = 1$.
\end{proof}
\begin{corollary}
\label{corollaryoflemmastarscalarforordertwo}
We have $c_{w} = 1$ for all $w \in \Waff$.
\end{corollary}
\begin{proof}
Let $w \in \Waff$ with reduced expression $w = s_{1} s_{2} \cdots s_{\ellsubstitute}$.
According to Condition \eqref{conditionofthechoicebraidrelations} in Choice~\ref{choice:tw}, we have $\varphi_{w} = \varphi_{s_{1}} \varphi_{s_{2}} \cdots \varphi_{s_{\ellsubstitute}}$ and $\varphi_{w^{-1}} = \varphi_{s_{\ellsubstitute}} \varphi_{s_{\ellsubstitute - 1}} \cdots \varphi_{s_{1}}$ .
Thus, we obtain that
\[
\varphi_{w}^{*} = \left(
\varphi_{s_{1}} \varphi_{s_{2}} \cdots \varphi_{s_{\ellsubstitute}}
\right)^{*} 
= \varphi_{s_{\ellsubstitute}}^{*} \varphi_{s_{\ellsubstitute - 1}}^{*} \cdots \varphi_{s_{1}}^{*} 
= \varphi_{s_{\ellsubstitute}} \varphi_{s_{\ellsubstitute - 1}} \cdots \varphi_{s_{1}} 
= \varphi_{w^{-1}},
\]
as required.
\end{proof}

\begin{proposition}
\label{propositionrefinephitforlengthzerot}
We assume Axioms \ref{axiomaboutHNheartandK}, \ref{axiombijectionofdoublecoset}, \ref{axiomexistenceofRgrp}, and \ref{axiomaboutdimensionofend}, fix a subset  $\Coeffplus \subset \Coeffinvnontriv$ as in Choice \ref{choiceofCoeffplus}, and choose $\cT$ as in Choice~\ref{choice:tw}.
Then we can choose a set of scalars
$\{ d_{t} \in \Coeff^\times \mid t \in \Wzero \}$
such that 
$(d_{t} \varphi_t)^* = d_{t^{-1}} \varphi_{t^{-1}}$
for all $t \in \Wzero$.
\end{proposition}

\begin{proof}
Let $d_1 = 1$.
For all nontrivial $t\in \Wzero$ of order $2$,
from Lemma \ref{lemmastarscalarforordertwo}
and Hilbert's Theorem 90,
we can choose $d_{t} \in \Coeff^\times$
such that $d_{t} / \bar d_{t} = c_t$.
Then
\(
(d_{t} \varphi_t)^*
= \bar d_{t} c_t \varphi_t
= d_{t} \varphi_t,
= d_{t^{-1}} \varphi_{t^{-1}},
\)
as required.

Now consider the set of elements $t \in \Wzero $ such that $t^2$ is nontrivial.
Partition this set into pairs $\{t, t^{-1}\}$.
Given a pair, let us choose an element arbitrarily,
and without loss of generality call it $t$.
Let $d_{t} = 1$ and $d_{t^{-1}} = c_t$.
Then we have
\(
(d_{t} \varphi_t)^*
= \varphi_t^*
= c_t \varphi_{t^{-1}}
= d_{t^{-1}} \varphi_{t^{-1}}
\)
and
\(
(d_{t^{-1}} \varphi_{t^{-1}})^*
= (c_t \varphi_{t^{-1}})^* = \varphi_t^{**}
= \varphi_t
= d_{t} \varphi_t,
\)
as required.
\end{proof}

This allows us to refine our choice of $\{T_n\}$ as follows, see Proposition \ref{propositionchoice:starispossible}.
\begin{choice}
\label{choice:star}
We choose a family of non-zero elements
\[
\cT =
\left\{
T_{n} \in \Hom_{K_{M}}\left(
^n\!\rho_{M}, \rho_{M}
\right)
\right\}_{n \in \Nheart}
\]
that satisfies all the conditions in Choice~\ref{choice:tw}, and
for all $t \in \Wzero$, we have $\varphi_{t}^{*} = \varphi_{t^{-1}}$.
\end{choice}
\begin{proposition}
	\label{propositionchoice:starispossible}
	We assume Axioms \ref{axiomaboutHNheartandK}, \ref{axiombijectionofdoublecoset}, \ref{axiomexistenceofRgrp}, and \ref{axiomaboutdimensionofend} and fix a subset  $\Coeffplus \subset \Coeffinvnontriv$ as in Choice \ref{choiceofCoeffplus}. Then we can choose $\cT$ to satisfy all the properties in Choice~\ref{choice:star}. 
	\end{proposition}
\begin{proof}
The proposition follows from Remark~\ref{remarkaboutchoice:tw} by replacing the previous choice of $T_{n}$ with $d_{t} \cdot T_{n}$ for all lifts $n \in \Nheart$ of $t \in \Wzero$, and Proposition~\ref{propositionrefinephitforlengthzerot} (see also Remark~\ref{remarkaboutphix1}).
\end{proof}	

We define a $\Coeff$-conjugate-linear map $(\phantom{\mathbb T})^*$
\index{notation-ax}{*@$(\phantom{\varphi})^*$ ($\Coeff$-conjugate-linear anti-involution)}
on $\Coeff[\Wzero, \muT] \ltimes \cH_\Coeff(\Waff, q)$ by 
\[
(\gpalg_{t} \cdot \mathbb{T}_{w})^{*} = \gpalg_{t^{-1}} \cdot \mathbb{T}_{t w^{-1} t^{-1}} = \mathbb{T}_{w^{-1}} \cdot \gpalg_{t^{-1}}
\]
for $t \in \Wzero$ and $w \in \Waff$.
\begin{proposition}
\label{propstarpreservationabstractheckevsourhecke}
We assume Axioms \ref{axiomaboutHNheartandK}, \ref{axiombijectionofdoublecoset}, \ref{axiomexistenceofRgrp}, and \ref{axiomaboutdimensionofend} and fix a subset  $\Coeffplus \subset \Coeffinvnontriv$ as in Choice \ref{choiceofCoeffplus}. 
Then we can choose $\cT$ as in Choice~\ref{choice:star}, and the resulting $\Coeff$-algebra isomorphism
\[
\cI(\rho_{x_0})
\colon
\cH(G(F), \rho_{x_0})
\rightarrow
\Coeff[\Wzero, \muT] \ltimes \cH_\Coeff(\Waff, q)
\]
in Theorem~\ref{theoremstructureofhecke} preserves the anti-involutions on both sides.
In particular, the $\Coeff$-conjugate-linear map $(\phantom{\mathbb T})^*$ on $\Coeff[\Wzero, \muT] \ltimes \cH_\Coeff(\Waff, q)$ is an anti-involution.
\end{proposition}
\begin{proof}
By Proposition~\ref{propositionchoice:starispossible} we can choose $\cT$ as in Choice~\ref{choice:star}. 
The claim that $\cI(\rho_{x_0})$ preserves the anti-involutions follows from Corollary~\ref{corollaryoflemmastarscalarforordertwo}, the condition $\varphi_{t}^{*} = \varphi_{t^{-1}}$ for $t \in \Wzero$ in Choice~\ref{choice:star}, and the anti-involution property.
Since $\cI(\rho_{x_0})$ is an isomorphism of $\Coeff$-algebras, we obtain the last claim.
\end{proof}

\subsection{(Non)uniqueness of the Hecke algebra isomorphism}
\label{subsec:isononuniqueness}
In this subsection, we keep the notation from Section \ref{subsection:The description of the Hecke algebra}.
Hence, we allow the coefficient field $\Coeff$
 again to be any algebraically closed field of characteristic $\ell \neq p$.
We also assume that all previous axioms hold, i.e., Axioms \ref{axiomaboutHNheartandK}, \ref{axiombijectionofdoublecoset}, \ref{axiomexistenceofRgrp}, and \ref{axiomaboutdimensionofend}, fix a subset  $\Coeffplus \subset \Coeffinvnontriv$ as in Choice \ref{choiceofCoeffplus}, and choose $\cT$ as in Choice~\ref{choice:tw}.
We can make explicit 
how much choice we had in constructing the particular isomorphism
in Theorem \ref{theoremstructureofhecke}.

We say that a $\Coeff$-linear bijection $\Psi$ on $\Coeff[\Wzero, \muT] \ltimes \cH_\Coeff(\Waff, q)$ is \emph{support-preserving} if for all $t \in \Wzero$ and $w \in \Waff$, we have $\Psi (\gpalg_{t} \cdot \mathbb{T}_{w}) = c \, \gpalg_{t} \cdot \mathbb{T}_{w}$ for some scalar $c \in \Coeff^{\times}$ that might depend on $t$ and $w$.
We also say that a $\Coeff$-linear bijection
$
\cI \colon \cH(G(F), \rho_{x_{0}})
\isoarrow
\Coeff[\Wzero, \muT] \ltimes \cH_\Coeff(\Waff, q)
$
is support-preserving if for all $t \in \Wzero$ and $w \in \Waff$, we have $\cI(\varphi_{tw}) =  c \, \gpalg_{t} \cdot \mathbb{T}_{w}$ for some scalar $c \in \Coeff^{\times}$.
For $\chi \in \Hom_{\bZ}(\Wzero ,\Coeff^\times)$, we define the support-preserving algebra automorphism $\Psi_{\chi}$ of $\Coeff[\Wzero, \muT] \ltimes \cH_\Coeff(\Waff, q)$ by
$\Psi_{\chi}(\gpalg_{t} \cdot \mathbb{T}_{w}) = \chi(t) \gpalg_{t} \cdot \mathbb{T}_{w}$
for $t \in \Wzero$ and $w \in \Waff$.

\begin{proposition}
\label{prop:autos-abstract-hecke-algebra}
A $\Coeff$-linear bijection $\Psi$
on
$\Coeff[\Wzero, \muT] \ltimes \cH_\Coeff(\Waff, q)$
is a support-preserving algebra automorphism if and only if
there exists some $\chi \in \Hom_{\bZ}(\Wzero ,\Coeff^\times)$
such that $\Psi = \Psi_{\chi}$.

If $\Coeff$ admits a nontrivial involution
$c \mapsto \bar c$ and $\cT$ satisfies all the properties in Choice~\ref{choice:star}, then the isomorphism $\Psi_{\chi}$ preserves the anti-involution defined in Section~\ref{Anti-involution of the Hecke algebra} if and only if $|\chi(t)|=1$
for all $t \in \Wzero $.
\end{proposition}

\begin{proof}
Suppose that $\Psi$ is a support-preserving algebra automorphism of
$\Coeff[\Wzero, \muT] \ltimes \cH_\Coeff(\Waff, q)$.
Since $\Psi$ preserves support, 
for all $t \in \Wzero$ and $w \in \Waff$ we must have that
$\Psi (\gpalg_{t} \cdot \mathbb{T}_{w}) = \chi(tw) \gpalg_{t} \cdot \mathbb{T}_{w}$
for some scalar $\chi(tw) \in \Coeff^\times$.
Since $\Psi$ is an algebra isomorphism,
we must have that $\chi(tw) = \chi(t) \chi(w)$.
Using the first relation of \eqref{eqn:abstractrelations}, we have
\[
\chi(s)^{2} ((q_{s} - 1) \cdot \mathbb{T}_{s} + q_s \cdot \mathbb{T}_{1}) =
(\chi(s) \bT_{s})^{2} = 
\Psi(\bT_{s})^{2} = \Psi(\bT_{s}^{2}) =
(q_{s} - 1) \cdot (\chi(s) \bT_{s}) + q_{s} \cdot \bT_{1}
\]
for all $s\in S_\Krel$.
Comparing the coefficients of $\bT_{s}$ and using $q_s \in \Coeffinvnontriv$, we obtain that $\chi(s) = 1$ for all $s\in S_\Krel$.
Then the second relation in \eqref{eqn:abstractrelations} and the fact that $\Psi$ is an algebra isomorphism imply that $\chi(w) = 1$ for all $w\in \Waff$.
Thus, $\Psi$ acts via the identity on $\cH_\Coeff(\Waff, q)$,
and it is determined by the scalars $\chi(t)$ for 
$t \in \Wzero $.
The condition of $\Psi$ preserving the multiplication
of $\Coeff[\Wzero, \muT]$
is equivalent to $\chi$ being a homomorphism.

The last claim follows from the definitions of $\Psi_{\chi}$ and the anti-involution $(\phantom{\mathbb T})^*$ on $\Coeff[\Wzero, \muT] \ltimes \cH_\Coeff(\Waff, q)$.
\end{proof}

Note that the group
$\Hom_{\bZ}(\Wzero ,\Coeff^\times)$
acts on the group of support-preserving algebra automorphisms
via
$\psi:  \Psi_\chi \longmapsto \Psi_{\psi\chi}$.

\begin{proposition}
\label{prop:torsor-of-isos}
We assume Axioms \ref{axiomaboutHNheartandK}, \ref{axiombijectionofdoublecoset}, \ref{axiomexistenceofRgrp}, and \ref{axiomaboutdimensionofend}, we fix a subset  $\Coeffplus \subset \Coeffinvnontriv$ as in Choice \ref{choiceofCoeffplus},
and we choose a family $\cT$ as in Choice \ref{choice:tw}.
Then the set of support-preserving $\Coeff$-algebra isomorphisms
$$
\cH(G(F), \rho_{x_{0}})
\isoarrow
\Coeff[\Wzero, \muT] \ltimes \cH_\Coeff(\Waff, q),
$$
is a torsor under 
$\Hom_{\bZ}(\Wzero ,\Coeff^\times)$.

If $\Coeff$ admits a nontrivial involution
$c \mapsto \bar c$ and $\cT$ satisfies all the properties in Choice~\ref{choice:star}, then the set of such isomorphisms that also preserve the anti-involutions defined in Section~\ref{Anti-involution of the Hecke algebra} is a torsor under
the 
group
$\Hom_{\bZ}\left(\Wzero ,\{ z \in \Coeff^{\times} \, \mid \, |z|=1\}\right)$.
\end{proposition}

\begin{proof}
Fix such an isomorphism
$\cI$, which exists by Theorem~\ref{theoremstructureofhecke}.
For every $\chi \in \Hom_{\bZ}(\Wzero ,\Coeff^\times)$,
we have a support-preserving $\Coeff$-algebra automorphism $\Psi_\chi$ of $\Coeff[\Wzero, \muT] \ltimes \cH_\Coeff(\Waff, q)$,
as in the proof of Proposition \ref{prop:autos-abstract-hecke-algebra}.
Let $\cI_\chi$ denote the composition 
$\Psi_\chi \circ \cI$.
Then it is clear that this is a support-preserving algebra
isomorphism from 
$\cH(G(F), \rho_{x_{0}})$
to
$\Coeff[\Wzero, \muT] \ltimes \cH_\Coeff(\Waff, q)$,
and that all such isomorphisms arise in this way.

Suppose that $\Coeff$ admits a nontrivial involution
$c \mapsto \bar c$ and $\cT$ satisfies all the properties in Choice \ref{choice:star}. 
If we also assume that $\cI$ preserves the anti-involution 
(such an isomorphism exists by Proposition~\ref{propstarpreservationabstractheckevsourhecke}),
then the last part of Proposition \ref{prop:autos-abstract-hecke-algebra} implies 
that $\cI_\chi$ preserves the anti-involution if and only if
$\chi \in \Hom_{\bZ}\left(\Wzero ,\{ z \in \Coeff^{\times} \, \mid \, |z|=1\}\right)$.
\end{proof}

\section{Comparison of Hecke algebras} \label{sec:comparison}

We have seen in Section \ref{Structure of a Hecke algebra} that if a pair $(K,\rho)$ consisting of a compact, open subgroup $K$ of $G(F)$ and an irreducible representation $\rho$ of $K$ satisfies
some axioms, then we can determine the structure of the attached Hecke algebra
$\cH(G(F),\rho)$.
Now we consider a pair $(K^0,\rho^0)$ for a subgroup $G^0$ of $G$.
We will show that if the two pairs $(K,\rho)$ and $(K^0,\rho^0)$ are related
according to some axioms, Axioms \ref{axiomaboutKM0vsKM}, \ref{axiomaboutK0vsK}, and \ref{axiomextensionoftheinductionofkappa},
then we have a support-preserving algebra isomorphism
$\cH(G^0(F),\rho^0) \longrightarrow \cH(G(F),\rho)$, see Theorem \ref{thm:isomorphismtodepthzero} and Theorem \ref{thm:explicitisom}.

As a special case, 
in
\cite[Theorem \ref{HAIKY-heckealgebraisomforKim-Yutype}]{HAIKY},
we will obtain an isomorphism between the Hecke algebra attached to a type of a Bernstein block of arbitrary depth constructed by Kim and Yu and the Hecke algebra of a depth-zero type.
Readers interested in this case might find it helpful to first read
Sections \ref{HAIKY-subsec:construction}
and \ref{HAIKY-subsec:affinehyperplanesKimYu}
of \cite{HAIKY}
to have an example in mind for the objects appearing in the axiomatic set-up below. A reader solely interested in the setting of
\cite{HAIKY}
might also completely replace all the below objects by those introduced in
\cite[Section \ref{HAIKY-Section-KimYutypes}]{HAIKY}
with the same symbols.

\subsection{The set-up} \label{subsec:comparisonsetup}
We let $G^0$ \index{notation-ax}{G0@$G^0$} be a connected reductive subgroup of $G$ of same rank as $G$ and let $j \colon \cB(G^0,F) \hookrightarrow \cB(G,F)$ be an admissible embedding of enlarged Bruhat--Tits buildings in the sense of \cite[\S~14.2]{KalethaPrasad}.
We note that an admissible embedding exists by \cite[Proposition~14.6.1, Corollary~14.7.3, Proposition~14.8.4]{KalethaPrasad}.
For example, in the setting of types constructed by Kim and Yu, $G^0$ is a twisted Levi subgroup of $G$, see \cite[Definition \ref{HAIKY-definitionofGdatum}]{HAIKY}.

Let $M^0$
\index{notation-ax}{M0@$M^0$}
be a Levi subgroup of $G^0$ for which $A_{M^0}=A_M$, where $M$
\index{notation-ax}{M@$M$}
denotes the centralizer $Z_G(A_{M^0})$ of $A_{M^0}$ in $G$. Recall that $A_{M^0}$, respectively, $A_M$, denotes the maximal split torus in the center of $M^0$, respectively $M$. 
Note that $M$ is a Levi subgroup of $G$ and we have $G^{0} \cap M = M^{0}$.
We fix a commutative diagram
\[
\xymatrix{
\cB(G^{0}, F) \ar@{^{(}->}^j[r] & \cB(G, F)
\\
\cB(M^{0}, F) \ar@{^{(}->}[u] \ar@{^{(}->}[r] & \cB(M, F) \ar@{^{(}->}[u]
}
\]
of admissible embeddings of buildings and identify $\cB(M^{0}, F)$, $\cB(M, F)$, and $\cB(G^{0}, F)$ with their images in $\cB(G, F)$.
Let $x_{0} \in \cB(M^{0}, F)$. Then the affine spaces $\cA_{x_{0}}=x_{0} + \left( X_{*}(A_{M^{0}}) \otimes_{\mathbb{Z}} \bR\right)$
and $\cA_{x_{0}}=x_{0} + \left( X_{*}(A_{M}) \otimes_{\mathbb{Z}} \bR\right)$ attached to $(G^0, M^0, x_0)$ and $(G, M, x_0)$ in Section~\ref{subsec:affine} agree, so we may use the same symbol $\cA_{x_{0}}$ without ambiguity.  
By our definition of $M$, the subgroups $N_{G^{0}}(M^{0})(F)_{[x_{0}]_{M^{0}}}$ of $G^{0}(F)$ and $N_{G}(M)(F)_{[x_{0}]_{M}}$ of $G(F)$ introduced in Section~\ref{subsec:affine} satisfy the relation
\begin{align}
\label{NofMvsNofm0}
N_{G^{0}}(M^{0})(F)_{[x_{0}]_{M^{0}}} = G^{0}(F) \cap N_{G}(M)(F)_{[x_{0}]_{M}}.
	\end{align}

We fix a locally finite set of affine hyperplanes $\mathfrak{H}$ in $\cA_{x_{0}}$ that do not contain $x_{0}$. \index{notation-ax}{H_@$\mathfrak{H}$}%
\index{notation-ax}{rhoM@$\rho_{M}$}%
\index{notation-ax}{rhoM0@$\rho_{M^0}$}%
\index{notation-ax}{Kay M@$K_{M}$}%
\index{notation-ax}{Kay M0@$K_{M^0}$}%
We let $K_{M^{0}}$, resp.,\ $K_{M}$, be a compact, open subgroup of $M^{0}(F)_{x_{0}}$, resp.,\ $M(F)_{x_{0}}$, and $(\rho_{M^{0}}, V_{\rho_{M^{0}}})$, resp.,\ $(\rho_{M}, V_{\rho_{M}})$, be an irreducible smooth representation of $K_{M^{0}}$, resp.,\ $K_{M}$. 
Let $\Nzeroheart$ be a subgroup of $N(\rho_{M^{0}})_{[x_{0}]_{M^{0}}}$ containing $A_{M^0}(F)$.
For example, in
\cite{HAIKY},
we take $\Nzeroheart=N(\rho_{M^{0}})_{[x_{0}]_{M^{0}}}$, and $(K_{M^0}, \rho_{M^0})$ and $(K_{M}, \rho_{M})$ are depth-zero and positive-depth supercuspidal types as constructed by Yu (\cite{Yu}) if $\Coeff=\bC$.

We will assume the following axiom that relates the general pairs $(K_{M^0}, \rho_{M^0})$ and $(K_{M}, \rho_{M})$.
\begin{axiom}
\label{axiomaboutKM0vsKM}
\mbox{}

\begin{enumerate}[(1)]
\item
\label{axiomaboutKM0vsKMaboutJ}
There exists a compact, open subgroup $K_{M, 0+}$ of $M(F)$ that is normalized by $\Nzeroheart \cdot K_{M^{0}}$ such that the group $M^{0}(F) \cap K_{M, 0+}$ is contained in the kernel of $\rho_{M^{0}}$, and we have
$
K_{M} = K_{M^{0}} \cdot K_{M, 0+}
$.
\item
\label{axiomaboutKM0vsKMaboutanextension}
There exists an irreducible smooth representation $(\kappa_{M}, V_{\kappa_{M}})$\index{notation-ax}{kappaM@$\kappa_M$}
 of $K_{M}$ that
extends to a smooth representation $\widetilde\kappa_{M}$
\index{notation-ax}{kappaM_tilde@$\widetilde\kappa_M$}
of the group
\[
\index{notation-ax}{Kay Mt@$\widetilde K_M$}
\widetilde K_{M} \coloneqq \Nzeroheart \cdot K_{M}
\]
 such that 
\[
\rho_{M} \simeq \inf\left(
\rho_{M^{0}}
\right) \otimes \kappa_{M},
\]
where $\inf\left(
\rho_{M^{0}}
\right)$ is the inflation of $\rho_{M^{0}}$ to $K_{M}$ via the surjection
\begin{equation*}
K_{M}  = K_{M^{0}} \cdot K_{M, 0+} 
 \twoheadrightarrow (K_{M^{0}} \cdot K_{M, 0+}) / K_{M, 0+} 
 \simeq K_{M^{0}} / \left(
K_{M^{0}}  \cap K_{M, 0+}
\right).
\end{equation*}
\end{enumerate}
\end{axiom}

\begin{remark}
In the setting of
\cite{HAIKY},
the group $K_{M,0+}$ is defined in
\cite[\eqref{HAIKY-definitionofcompactopensubgroupsKimYucase}]{HAIKY},
$\kappa_M$ is a twist of a Weil--Heisenberg representation defined in
\cite[Section \ref{HAIKY-subsec:construction}, p.\ \pageref{HAIKY-kappaM}]{HAIKY},
and the extension $\widetilde \kappa_M$ of $\kappa_M$ is constructed explicitly in
\cite[Definition \ref{HAIKY-definitionofkappatilde}]{HAIKY}
and the preceding text.
That Axiom \ref{axiomaboutKM0vsKM}
holds in this setting is proven in
\cite[Proposition~\ref{HAIKY-propproofofaxiomaboutKM0vsKM}]{HAIKY}.
The existence of an extension  $\widetilde \kappa_M$ of $\kappa_M$ as required in this axiom was the key challenge in proving that all our axioms are satisfied in the setting of \cite{HAIKY}.
\end{remark}

\begin{remark}
\label{remarkaboutintersectionwithdepthzero}
Axiom~\ref{axiomaboutKM0vsKM}\eqref{axiomaboutKM0vsKMaboutJ} in particular includes the statement that $M^{0}(F) \cap K_{M, 0+} \subset K_{M^{0}}$, and hence the requirement $K_{M} = K_{M^{0}} \cdot K_{M, 0+}$ of the axiom
implies that
$
M^{0}(F) \cap K_{M} = K_{M^{0}}
$.
Since $G^{0} \cap M = M^{0}$ and $K_{M} \subset M(F)$, we also have $G^{0}(F) \cap K_{M} = K_{M^{0}}$.
\end{remark}

It will be convenient for us to choose the subgroup $\Nheart$ of $N(\rho_{M})_{[x_{0}]_{M}}$ as $\Nheart =  \Nzeroheart$, but in order to do so, we first need the following lemma.
\begin{lemma}
\label{lemmanormalizerho0vsrho}
We have
\[
\Nzeroheart \subset N(\rho_{M})_{[x_{0}]_{M}}.
\]
\end{lemma}
\begin{proof}
Let $n \in \Nzeroheart$.
According to Equation \eqref{NofMvsNofm0}, we have $n \in N_{G}(M)(F)_{[x_{0}]_{M}}$.
Hence, it suffices to prove that $n \in N_{G(F)}(\rho_{M})$.
According to Axiom~\ref{axiomaboutKM0vsKM}\eqref{axiomaboutKM0vsKMaboutJ}, the element $n$ normalizes the group $K_{M}$.
We will prove that $n$ normalizes the representation $\rho_{M}$.
Let
\(
T^{0}_{n} \colon ^n\!\rho_{M^{0}} \isoarrow \rho_{M^{0}}
\) be an isomorphism.
Then the morphism
\[
 T^{0}_{n} \otimes \widetilde\kappa_{M}(n)
 : \, \,
 ^n\!\rho_{M} = {^n\!\rho_{M^{0}}} \otimes {^n\!\kappa_{M}}
 \isoarrow
 \rho_{M^{0}} \otimes \kappa_{M} = \rho_{M}
\]
is an isomorphism as well, hence $n \in N_{G(F)}(\rho_{M})$.
\end{proof}

\subsection{Compatible families of quasi-$G$-covers}
\label{subsec:quasiGcoverscompatible}
We keep the notation and assumption from the previous subsection, i.e., Axiom \ref{axiomaboutKM0vsKM} holds. In this subsection, we will define the quasi-$G^0$-cover $(K^0_{x_0}, \rho^{0}_{x_0})$ of $(K_{M^0}, \rho_{M^0})$ and the quasi-$G$-cover $(K_{x_0}, \rho_{x_0})$ of $(K_M, \rho_M)$ to which we attach the Hecke algebras that we prove are isomorphic in Theorem \ref{thm:isomorphismtodepthzero}. 
For example, if $\Coeff = \bC$, we can take $(K_{x_0}, \rho_{x_0})$ to be a type for a Bernstein block constructed by Kim and Yu as a $G$-cover of a supercuspidal type $(K_M, \rho_M)$. 
In this case, the pair $(K^0_{x_0}, \rho^{0}_{x_0})$ can be taken as the twist by a quadratic character introduced in \cite{FKS} of the depth-zero type included in the input for the construction of $(K_{x_0}, \rho_{x_0})$.
(See
\cite[Section \ref{HAIKY-subsec:construction}]{HAIKY}
for details, where we include the twist by the quadratic character in the construction of $(K_{x_0}, \rho_{x_0})$ and therefore can take $(K^0_{x_0}, \rho^{0}_{x_0})$ as the depth-zero type input for this twisted construction.)
As in Section \ref{subsec:familyofcovers}, we will define not only one quasi-cover, but a whole family of quasi-covers. Moreover, we will formulate a compatibility condition between the family of quasi-$G^0$-covers and the family of quasi-$G$-covers, see Axiom \ref{axiomaboutK0vsK}.

\index{notation-ax}{rho0x@$\rho^0_{x}$}
\index{notation-ax}{Kay0 x@$K^{0}_{x}$}
\index{notation-ax}{Kay0 x +@$K^{0}_{x, +}$}
Let
\[
\cK^{0} =
\left\{
(K^{0}_{x}, K^{0}_{x, +}, (\rho^{0}_{x}, V_{\rho^{0}_{x}}))
\right\}_{x \in \cA_{\gen}}
\]
be a family
of quasi-$G^0$-cover-candidates
that satisfies Axioms~\ref{axiomaboutHNheartandK} with the subgroup $\Nzeroheart$ of $N(\rho_{M^{0}})_{[x_{0}]_{M^{0}}}$.
According to Lemma~\ref{lemmanormalizerho0vsrho}, $\Nzeroheart$ is also a subgroup of $N(\rho_{M})_{[x_{0}]_{M}}$.
Let
\index{notation-ax}{Kay x@$K_{x}$}
\index{notation-ax}{rhox@$\rho_{x}$}
\index{notation-ax}{Kay x +@$K_{x, +}$}
\[
\cK =
\left\{
(K_{x}, K_{x, +}, (\rho_{x}, V_{\rho_{x}}))
\right\}_{x \in \cA_{\gen}}
\]
be a family of quasi-$G$-cover-candidates that satisfies Axioms~\ref{axiomaboutHNheartandK} with the subgroup $\Nheart\coloneqq\Nzeroheart$ of $N(\rho_{M})_{[x_{0}]_{M}}$. Examples for such families are presented in
the paragraph after
\cite[Lemma \ref{HAIKY-lemmakapparestrictiontoK+KimYuver}]{HAIKY} on p.\ \pageref{HAIKY-KimYufamilies} of \cite{HAIKY}.
We assume that the two families satisfy the following compatibility properties.
\begin{axiom}
\label{axiomaboutK0vsK}
\mbox{}
For each $x \in \cA_{\gen}$, the following properties hold.

\begin{enumerate}[(1)]
\item
There exists a compact, open subgroup $K_{x, 0+}$ of $K_{x}$ that is normalized by $K^{0}_{x}$ such that the group $G^{0}(F) \cap K_{x, 0+}$ is contained in the kernel of $\rho^{0}_{x}$, and we have
$
K_{x} = K^{0}_{x} \cdot K_{x, 0+}
$.
\item \label{axiomaboutK0vsKinflation}
\index{notation-ax}{kappax@$\kappa_{x}$}%
There exists an irreducible smooth representation $(\kappa_{x}, V_{\kappa_{x}})$ of $K_{x}$ such that 
\[
\rho_{x} = \inf \left(
\rho^{0}_{x}
\right) \otimes \kappa_{x},
\]
where $\inf\left(
\rho^{0}_{x}
\right)$ is the inflation of $\rho^{0}_{x}$ to $K_{x}$ via the surjection
\begin{equation*}
K_{x}  = K^{0}_{x} \cdot K_{x, 0+} 
 \twoheadrightarrow K^{0}_{x} \cdot K_{x, 0+} / K_{x, 0+} 
 \simeq K^{0}_{x} / \left(
K^{0}_{x}  \cap K_{x, 0+}
\right).
\end{equation*}
\item
\label{axiomaboutK0vsKaboutintertwiner}
We have 
\[
I_{G(F)}(\rho_{x}) = K_{x} \cdot I_{G^{0}(F)}(\rho^{0}_{x}) \cdot K_{x}.
\]
\end{enumerate}
\end{axiom}
This axiom is verified for the setting of
\cite{HAIKY}
in \cite[Proposition \ref{HAIKY-proofofaxiomaboutK0vsK}]{HAIKY}.

\subsection{$\protect\cK$-relevance vs. $\protect\cK^{0}$-relevance}
\label{subsec:relevance}
We keep the notation and assumptions, i.e., Axioms \ref{axiomaboutHNheartandK}, \ref{axiomaboutKM0vsKM}, and \ref{axiomaboutK0vsK}, from the previous subsection.
Recall that we defined in Section \ref{subsec:intertwiningop} intertwining operators $\Theta_{y\mid x}$ for $x, y \in \cA_\gen$. In order to distinguish the intertwining operators attached to $G$ and $\cK$ from the intertwining operators attached to $G^0$ and $\cK^0$ we will denote the former by 
\[
\Theta_{y \mid x} \colon \ind_{K_{x}}^{G(F)} (\rho_{x}) \rightarrow \ind_{K_{y}}^{G(F)} (\rho_{y})
\] and the latter by
\[
\Theta^{0}_{y \mid x} \colon \ind_{K^{0}_{x}}^{G^{0}(F)} (\rho^{0}_{x}) \rightarrow \ind_{K^{0}_{y}}^{G^{0}(F)} (\rho^{0}_{y}) .
\]
This means  that an affine hyperplane $H \in \mathfrak{H}$ is called $\cK^0$-relevant if there exists $x, y \in \cA_{\gen}$ such that 
$
\mathfrak{H}_{x, y} = \left\{
H
\right\}$
and
\(
\Theta^0_{x \mid y} \circ \Theta^0_{y \mid x} \notin \Coeff \cdot \id_{\ind_{K^0_{x}}^{G^0(F)}(\rho^0_{x})} 
\).
In this subsection, we will exhibit the relation between $\Theta_{y \mid x}$ and $\Theta^0_{y \mid x}$, see Lemma \ref{lemmacompativilityofThetadepth0vspositivedepth}, and show that the notions of $\cK$-relevant and $\cK^0$-relevant coincide, see Corollary \ref{corollaryKrel=K0rel}, under the assumption of the following axiom that ensures a compatibility between the data $\{K^{0}_x, K_{x, 0+}, \kappa_x\}$ for different $x \in \cA_\gen$.
\begin{axiom}
\label{axiomextensionoftheinductionofkappa}
For any $x, y \in \cA_{\gen}$ such that $d(x, y) = 1$,
there exists a compact, open subgroup $K^{0}_{x, y}$ 
\index{notation-ax}{Kay0 x y@$K^0_{x,y}$}
of $G^{0}(F)$, a compact, open subgroup $K_{x, y; 0+}$ of $G(F)$,
\index{notation-ax}{Kay x y 0+@$K_{x,y; 0+}$}
and an irreducible smooth representation $(\kappa_{x, y}, V_{\kappa_{x, y}})$ of $K_{x, y} \coloneqq K^{0}_{x, y} \cdot K_{x, y; 0+}$
\index{notation-ax}{kappaxy@$\kappa_{x,y}$}
\index{notation-ax}{Kay x y@$K_{x, y}$}
such that
\begin{enumerate}[(1)]
\item
\label{axiomextensionoftheinductionofkappaK0xycontainsK0xandy}
$K^{0}_{x, y}$ contains $K^{0}_{x}$ and $K^{0}_{y}$.
\item
\label{axiomextensionoftheinductionofkappaJxyvsJxandJy}
$K_{x, y; 0+}$ is normalized by the group $K^{0}_{x, y}$, and we have
\[
K_{x, 0+} \subset \left(
G^{0}(F) \cap K_{x, 0+}
\right)
\cdot K_{x, y; 0+}
\quad
\text{ and }
\quad
K_{y, 0+} \subset \left(
G^{0}(F) \cap K_{y, 0+}
\right)
\cdot K_{x, y; 0+}.
\]
\item
\label{axiomextensionoftheinductionofkappakrnelofkappaxy}
The group $G^{0}(F) \cap K_{x, y; 0+}$ is contained in the kernels of $\rho^{0}_{x}$ and $\rho^{0}_{y}$.
\item
\label{axiomextensionoftheinductionofkappakappasxirreducible}
The restriction of $\kappa_{x, y}$ to $K_{x, y; 0+}$ is irreducible.
\item
\label{axiomextensionoftheinductionofkappacompatibilitywiththecompactind}
We have isomorphisms
\[
\kappa_{x, y} \restriction_{K^{0}_{x} \cdot K_{x, y; 0+}} \isoarrow \ind_{K_{x}}^{K^{0}_{x} \cdot K_{x, y; 0+}} (\kappa_{x})
\quad
\text{ and }
\quad
\kappa_{x, y} \restriction_{K^{0}_{y} \cdot K_{x, y; 0+}} \isoarrow \ind_{K_{y}}^{K^{0}_{y} \cdot K_{x, y; 0+}} (\kappa_{y}).
\]
\end{enumerate}
\end{axiom}
Since the conditions of Axiom~\ref{axiomextensionoftheinductionofkappa} are symmetric with respect to $x$ and $y$, we may and do assume that $K^{0}_{x, y}=K^{0}_{y, x}$, $K_{x,y; 0+}=K_{y,x; 0+}$, and $\kappa_{x, y} = \kappa_{y, x}$.

In Section \ref{subsec:Heckeisom} below, when we will also assume Axioms \ref{axiomexistenceofRgrp} and \ref{axiomaboutdimensionofend}, then $K^{0}_{x, s x}$ and $K_{x, s x}$  will play the role of $K'_{x,s}$ in Axiom \ref{axiomaboutdimensionofend}.

\begin{remark}
	 In the setting of
\cite{HAIKY},
the objects $K^0_{x,y}$, $K_{x,y;0+}$ and $\kappa_{x,y}$ are defined in
\cite[Notation \ref{HAIKY-notationK0xyKxy0+kappaxy}]{HAIKY},
and \cite[Lemma \ref{HAIKY-proofofaxiomextensionoftheinductionofkappa}]{HAIKY}
shows that these objects satisfy Axiom \ref{axiomextensionoftheinductionofkappa} for the families $\cK^0$ and $\cK$ considered in
\cite{HAIKY}. 
	Axiom \ref{axiomextensionoftheinductionofkappa}\eqref{axiomextensionoftheinductionofkappacompatibilitywiththecompactind} is a reason why we need to twist the construction of Kim and Yu (\cite{Kim-Yu}) in
\cite{HAIKY}
by a quadratic character following \cite{FKS}.
Without this twist, 
 the Hecke algebra isomorphism in \cite[Theorem \ref{HAIKY-heckealgebraisomforKim-Yutype}]{HAIKY}, which is 
 Theorem \ref{thm:isomorphismtodepthzero} applied in that setting,
 would not be true in general, see 
\cite[Section \ref{HAIKY-subsec:quadratictwistisnecessary}]{HAIKY} for an example. 
\end{remark}

\begin{remark}
\label{remarkinclusionsofKxandKxy}
Axiom~\ref{axiomextensionoftheinductionofkappa}\eqref{axiomextensionoftheinductionofkappaK0xycontainsK0xandy} and \eqref{axiomextensionoftheinductionofkappaJxyvsJxandJy} imply that
\[
K_{x} \subset K^{0}_{x} \cdot K_{x, y; 0+} \subset K_{x, y}
\quad
\text{ and }
\quad
K_{y} \subset K^{0}_{y} \cdot \nobreak K_{x, y; 0+} \subset \nobreak K_{x, y}.
\]
\end{remark}

We assume Axiom \ref{axiomextensionoftheinductionofkappa} from now on.
Let $H \in \mathfrak{H}$ and $x, y \in \cA_{\gen}$ such that $\mathfrak{H}_{x, y} = \left\{
H
\right\}$.
Let $(K^{0}_{x, y}, K_{x, y; 0+}, \kappa_{x, y})$  be the triple  from Axiom~\ref{axiomextensionoftheinductionofkappa}.
Since the representation $\rho^{0}_{x}$ is trivial on the group $K^{0}_{x} \cap K_{x, y; 0+} \subset G^{0}(F) \cap K_{x, y; 0+}$, we can inflate $\rho^{0}_{x}$ to the group $K^{0}_{x} \cdot K_{x, y; 0+}$ via the surjection
\[
K^{0}_{x} \cdot K_{x, y; 0+}
\longrightarrow
K^{0}_{x} \cdot K_{x, y; 0+} / K_{x, y; 0+} 
\simeq
K^{0}_{x} / \left( K^{0}_{x} \cap K_{x, y; 0+} \right).
\]
We use the same notation $\inf(\rho^{0}_{x})$
for this inflation as for the inflation of $\rho^{0}_{x}$ to $K_{x}$, which is just the restriction of the former to $K_{x}$.  
Then, by using Axiom~\ref{axiomextensionoftheinductionofkappa}\eqref{axiomextensionoftheinductionofkappacompatibilitywiththecompactind}, we have 
\begin{align*}
\ind_{K_{x}}^{K_{x, y}} (\rho_{x}) &= \ind_{K_{x}}^{K_{x, y}} \left(
\inf\left(
\rho^{0}_{x}
\right) \otimes \kappa_{x}
\right) 
&& \!\!\!\!\!\!
\simeq \ind_{K^{0}_{x} \cdot K_{x, y; 0+}}^{K_{x, y}} \Bigl(
\ind_{K_{x}}^{K^{0}_{x} \cdot K_{x, y; 0+}} \left(
\inf\left(
\rho^{0}_{x}
\right) \otimes \kappa_{x}
\right)
\Bigr) \notag \\
& \simeq \ind_{K^{0}_{x} \cdot K_{x, y; 0+}}^{K_{x, y}} \Bigl(
\inf \left(
\rho^{0}_{x}
\right) \otimes 
\ind_{K_{x}}^{K^{0}_{x} \cdot K_{x, y; 0+}} \kappa_{x} 
\Bigr) 
&& \!\!\!\!\!\! \simeq \ind_{K^{0}_{x} \cdot K_{x, y; 0+}}^{K_{x, y}} \Bigl(
\inf (\rho^{0}_{x}) \otimes 
\kappa_{x, y} \restriction_{K^{0}_{x} \cdot K_{x, y; 0+}} 
\Bigr) \notag \\
& \simeq \ind_{K^{0}_{x} \cdot K_{x, y; 0+}}^{K_{x, y}} \Bigl(
\inf (\rho^{0}_{x}) 
\Bigr) \otimes \kappa_{x, y} 
&& \!\!\!\!\!\! = \ind_{K^{0}_{x} \cdot K_{x, y; 0+}}^{K^{0}_{x, y} \cdot K_{x, y; 0+}} \Bigl(
\inf (\rho^{0}_{x}) 
\Bigr) \otimes \kappa_{x, y} \notag \\
& \simeq \inf \Bigl(
\ind_{K^{0}_{x}}^{K^{0}_{x, y}} (\rho^{0}_{x})
\Bigr) \otimes \kappa_{x, y}, \notag
\end{align*}
where $\inf \left(
\ind_{K^{0}_{x}}^{K^{0}_{x, y}} (\rho^{0}_{x})
\right)$ denotes the inflation of the representation $\ind_{K^{0}_{x}}^{K^{0}_{x, y}} (\rho^{0}_{x})$ to $K_{x, y}$ via the surjection
\[
K_{x, y}
= K^{0}_{x, y} \cdot K_{x, y; 0+} 
\longrightarrow K^{0}_{x, y} \cdot K_{x, y; 0+} / K_{x, y; 0+} 
\simeq K^{0}_{x, y} / \left( K^{0}_{x, y} \cap K_{x, y; 0+} \right).
\]
We write this isomorphism as
\index{notation-ax}{Ixxy@$I^{x, y}_{x}$}
\[
I^{x, y}_{x} \colon
\ind_{K_{x}}^{K_{x, y}} (\rho_{x})
\isoarrow
\inf \left(
\ind_{K^{0}_{x}}^{K^{0}_{x, y}} (\rho^{0}_{x})
\right)
\otimes \kappa_{x, y}. 
\]
Similarly, we obtain an isomorphism
\begin{align*}
I^{x, y}_{y} \colon
\ind_{K_{y}}^{K_{x, y}} (\rho_{y})
\isoarrow
\inf \left(
\ind_{K^{0}_{y}}^{K^{0}_{x, y}} (\rho^{0}_{y})
\right) \otimes \kappa_{x, y}.
\end{align*}

According to Axiom \ref{axiomextensionoftheinductionofkappa} \eqref{axiomextensionoftheinductionofkappaK0xycontainsK0xandy}, Remark~\ref{remarkinclusionsofKxandKxy} and the definitions of $\Theta_{y \mid x}$ and $\Theta^0_{y \mid x}$, we have 
\begin{equation}\label{equationaboutThetainsidesubspace}
\Theta_{y \mid x} \in 
\Hom_{G(F)}\bigl(\ind_{K_{x}}^{G(F)} (\rho_{x}), \ind_{K_{y}}^{G(F)} (\rho_{y})\bigr)_{K_{x, y}} \simeq \Hom_{K_{x, y}}\bigl(\ind_{K_{x}}^{K_{x, y}} (\rho_{x}), \ind_{K_{y}}^{K_{x, y}} (\rho_{y})\bigr)
\end{equation}
\begin{equation}\label{equationaboutThetazeroinsidesubspace}
	\Theta^0_{y \mid x} \in 
	\Hom_{G^0(F)}\bigl(\ind_{K^0_{x}}^{G(F)} (\rho^0_{x}), \ind_{K^0_{y}}^{G(F)} (\rho^0_{y})\bigr)_{K^0_{x, y}} \simeq \Hom_{K^0_{x, y}}\bigl(\ind_{K^0_{x}}^{K^0_{x, y}} (\rho^0_{x}), \ind_{K^0_{y}}^{K^0_{x, y}} (\rho^0_{y})\bigr)
\end{equation}
(see Lemma~\ref{lemmarestrictiontoasubspaceofthecompactinductiongeneralver}) and therefore may view $\Theta_{y \mid x}$ and $\Theta^0_{y \mid x}$ as elements of the latter spaces.
Hence, we obtain an element
\[
I^{x, y}_{y} \circ
\Theta_{y \mid x}
\circ \left(
I^{x, y}_{x}
\right)^{-1}
\in
\Hom_{K_{x, y}}\left(
\inf \bigl(
\ind_{K^{0}_{x}}^{K^{0}_{x, y}} (\rho^{0}_{x})
\bigr) \otimes \kappa_{x, y}, \inf \bigl(
\ind_{K^{0}_{y}}^{K^{0}_{x, y}} (\rho^{0}_{y})
\bigr) \otimes \kappa_{x, y}
\right).
\]
\begin{lemma}
\label{lemmacompativilityofThetadepth0vspositivedepth}
There exists $c \in \Coeff^{\times}$ such that
\[
I^{x, y}_{y} \circ
\Theta_{y \mid x}
\circ \left(
I^{x, y}_{x}
\right)^{-1} = 
c \cdot 
\Theta^{0}_{y \mid x}
 \otimes \id_{V_{\kappa_{x, y}}}.
\]
\end{lemma}
\begin{proof}
\addtocounter{equation}{-1}
\begin{subequations}
Since $K_{x, y; 0+}$ is contained in the kernels of the two 
representations
$\inf \bigl(
\ind_{K^{0}_{x}}^{K^{0}_{x, y}} (\rho^{0}_{x})
\bigr)$
and
$\inf \bigl(
\ind_{K^{0}_{y}}^{K^{0}_{x, y}} (\rho^{0}_{y})
\bigr)$,
and the restriction of $\kappa_{x, y}$ to $K_{x, y; 0+}$ is irreducible by Axiom~\ref{axiomextensionoftheinductionofkappa}\eqref{axiomextensionoftheinductionofkappakappasxirreducible}, we can write
\[
I^{x, y}_{y} \circ
\Theta_{y \mid x}
\circ \left(
I^{x, y}_{x}
\right)^{-1} = 
\Theta'
 \otimes \id_{V_{\kappa_{x, y}}}
\]
for some
\[
\Theta' \in \Hom_{K^{0}_{x, y}}\bigl(
\ind_{K^{0}_{x}}^{K^{0}_{x, y}} (\rho^{0}_{x}), \ind_{K^{0}_{y}}^{K^{0}_{x, y}} (\rho^{0}_{y})
\bigr).
\]
It remains to show that there exists $c \in \Coeff^{\times}$ such that $\Theta' = c \cdot \Theta^{0}_{y \mid x}$. The definition of $\Theta^{0}_{y \mid x}$ implies that 
\begin{align}
\label{Theta0inKyKx}
\Theta^{0}_{y \mid x} \left(
V_{\rho^{0}_{x}}
\right) \subset \ind_{K^{0}_{y}}^{K^{0}_{y} \cdot K^{0}_{x}} \left(
V_{\rho^{0}_{y}}
\right).
\end{align}
On the other hand, we can prove
\begin{align}
\label{Theta'inKyKx}
\Theta'\left(
V_{\rho^{0}_{x}}
\right) \subset \ind_{K^{0}_{y}}^{K^{0}_{y} \cdot K^{0}_{x}} \left(
V_{\rho^{0}_{y}}
\right)
\end{align}
as follows.
Let $v^{0} \in V_{\rho^{0}_{x}}$ and $w \in V_{\kappa_{x}}$, and consider 
\(
v^{0} \otimes w \in V_{\rho^{0}_{x}} \otimes V_{\kappa_{x}} = V_{\rho_{x}} \subset \ind_{K_{x}}^{K_{x, y}} (V_{\rho_{x}}) 
\).
By tracing through the definition of $I^{x, y}_{x}$, we obtain that 
\[
I^{x, y}_{x}(v^{0} \otimes w)=v^{0} \otimes w \in V_{\rho^{0}_{x}} \otimes V_{\kappa_{x}} \subset \ind_{K^{0}_{x}}^{K^{0}_{x, y}} (V_{\rho^{0}_{x}}) \otimes V_{\kappa_{x, y}}, 
\]
where we regard $\kappa_{x}$ as a $K_{x}$-subrepresentation of $\kappa_{x, y}$ via 
$\kappa_{x}  \hookrightarrow \ind_{K_{x}}^{K^{0}_{x} \cdot K_{x, y; 0+}} (\kappa_{x}) \restriction_{K_{x}} \simeq \kappa_{x, y} \restriction_{K_{x}}$.
Hence, we obtain that
$
I^{x, y}_{x}\left(
V_{\rho_x}
\right)
= V_{\rho^{0}_{x}} \otimes V_{\kappa_{x}}
$,
equivalently, 
$ \left(
I^{x, y}_{x}
\right)^{-1}
\left(
V_{\rho^{0}_{x}} \otimes V_{\kappa_{x}}
\right) =
V_{\rho_x}
$.
 The definitions of $I^{x, y}_{y}$ and $\Theta_{y \mid x}$, and the observation $K^{0}_{y} \cdot K^{0}_{x} \cdot K_{x, y; 0+} = K^{0}_{y} \cdot K_{x, y; 0+}  \cdot K^{0}_{x} \cdot K_{x, y; 0+} \supseteq K_y \cdot K_x$ then imply that
\begin{align*}
\left(
I^{x, y}_{y} \circ
\Theta_{y \mid x}
\circ \left(
I^{x, y}_{x}
\right)^{-1}
\right)\left(
V_{\rho^{0}_{x}} \otimes V_{\kappa_{x}}
\right) \subset \ind_{K^{0}_{y}}^{K^{0}_{y} \cdot K^{0}_{x}} \left(
V_{\rho^{0}_{y}}
\right) \otimes V_{\kappa_{x, y}}.
\end{align*}
Thus, we conclude \eqref{Theta'inKyKx}.

According to \eqref{Theta0inKyKx} and \eqref{Theta'inKyKx}, we have
\begin{align*}
\Theta^{0}_{y \mid x} \restriction_{V_{\rho^{0}_{x}}},
\:\: \Theta' \restriction_{V_{\rho^{0}_{x}}}
& \in
\Hom_{K^{0}_{x}}\left(
\rho^{0}_{x}, \ind_{K^{0}_{y}}^{K^{0}_{y} \cdot K^{0}_{x}} (\rho^{0}_{y})
\right) 
\simeq 
\Hom_{K^{0}_{x}}\left(
\rho^{0}_{x}, \ind_{K^{0}_{x} \cap K^{0}_{y}}^{K^{0}_{x}} \left(
\rho^{0}_{y}\restriction_{K^{0}_{x} \cap K^{0}_{y}}
\right)
\right),
\end{align*}
where the isomorphism comes from the isomorphism
\(
\ind_{K^{0}_{y}}^{K^{0}_{y} \cdot K^{0}_{x}} (\rho^{0}_{y}) \rightarrow \ind_{K^{0}_{x} \cap K^{0}_{y}}^{K^{0}_{x}} \left(
\rho^{0}_{y}\restriction_{K^{0}_{x} \cap K^{0}_{y}}
\right)
\) of $K^{0}_{x}$-representations
given by
\(
f \mapsto f\restriction_{K^{0}_{x}}.
\)
Since $K^{0}_{x}$ is a compact group,
the compact induction functor $\ind_{K^{0}_{x} \cap K^{0}_{y}}^{K^{0}_{x}}$
is not only the left-adjoint but also the right-adjoint of the restriction functor.
Hence, according to Lemma~\ref{lemmacoverintertwining}, we have
\begin{align*}
\Hom_{K^{0}_{x}}\left(
\rho^{0}_{x}, \ind_{K^{0}_{x} \cap K^{0}_{y}}^{K^{0}_{x}} \left(
\rho^{0}_{y}\restriction_{K^{0}_{x} \cap K^{0}_{y}}
\right)
\right)
& \simeq \Hom_{K^{0}_{x} \cap K^{0}_{y}} \left(
\rho^{0}_{x} \restriction_{K^{0}_{x} \cap K^{0}_{y}}, \rho^{0}_{y} \restriction_{K^{0}_{x} \cap K^{0}_{y}}
\right) 
= \End_{K_{M^{0}}}(\rho_{M^{0}}).
\end{align*}
Since $\rho_{M^{0}}$ is an irreducible representation of $K_{M^{0}}$, we obtain that
\[
\dim_{\Coeff}\left(
\Hom_{K^{0}_{x}}\left(\rho^{0}_{x}, \ind_{K^{0}_{y}}^{K^{0}_{y} \cdot K^{0}_{x}} (\rho^{0}_{y})\right)
\right) = \dim_{\Coeff}\left(
\End_{K_{M^{0}}}(\rho_{M^{0}})
\right) = 1.
\]
Thus, there exists $c \in \Coeff$ such that
\(
\Theta' \restriction_{V_{\rho^{0}_{x}}} = c \cdot \Theta^{0}_{y \mid x} \restriction_{V_{\rho^{0}_{x}}}.
\)
Since the $K^{0}_{x, y}$-representation $\left(
\ind_{K^{0}_{x}}^{K^{0}_{x, y}} (\rho^{0}_{x}), \right.$ $\left. \ind_{K^{0}_{x}}^{K^{0}_{x, y}} \left(
V_{\rho^{0}_{x}}
\right)
\right)$ is generated by the subspace $V_{\rho^{0}_{x}}$
, and the homomorphisms $\Theta'$ and $\Theta^{0}_{y \mid x}$ are $K^{0}_{x, y}$-equivariant, we also obtain that 
\(
\Theta' = c \cdot 
\Theta^{0}_{y \mid x}.
\)
Since $\Theta^{0}_{y \mid x}$ and $\Theta_{y \mid x}$ are non-zero (see Corollary~\ref{corollarycalculationofconstantterm}), we obtain the lemma.
\end{subequations}
\end{proof}
\begin{corollary}
\label{corollaryKrel=K0rel}
We assume Axioms \ref{axiomaboutHNheartandK}, \ref{axiomaboutKM0vsKM}, \ref{axiomaboutK0vsK}, and \ref{axiomextensionoftheinductionofkappa}.
Then an affine hyperplane $H \in \mathfrak{H}$ is $\cK^{0}$-relevant if and only if $H$ is $\cK$-relevant.
\end{corollary}
\begin{proof}
Let $x, y \in \cA_{\gen}$ such that $\mathfrak{H}_{x, y} = \left\{
H
\right\}$.
To prove the corollary, in light of Equations \eqref{equationaboutThetainsidesubspace} and \eqref{equationaboutThetazeroinsidesubspace}, it suffices to show that
\[
\Theta_{x \mid y} \circ \Theta_{y \mid x}
 \in \Coeff \cdot \id_{\ind_{K_{x}}^{K_{x, y}}(\rho_{x})}
\quad \text{
if and only if
} \quad
\Theta^{0}_{x \mid y} \circ \Theta^{0}_{y \mid x} \in \Coeff \cdot \id_{\ind_{K^{0}_{x}}^{K^{0}_{x, y}}(\rho^{0}_{x})}.
\]
According to Lemma~\ref{lemmacompativilityofThetadepth0vspositivedepth}, there exists $c \in \Coeff^{\times}$ such that
\[
I^{x, y}_{y} \circ \Theta_{y \mid x} \circ \left(
I^{x, y}_{x}
\right)^{-1} = c \cdot \Theta^{0}_{y \mid x}
 \otimes \id_{V_{\kappa_{x, y}}}.
\]
Replacing $x$ with $y$, we also obtain that there exists $c' \in \Coeff^{\times}$ such that
\[
I^{x, y}_{x} \circ \Theta_{x \mid y} \circ \left(I^{x, y}_{y}\right)^{-1} = c' \cdot  \Theta^{0}_{x \mid y}
 \otimes \id_{V_{\kappa_{x, y}}}.
\]
Hence, we have
\(
I^{x, y}_{x} \circ \left(
\Theta_{x \mid y} \circ \Theta_{y \mid x}
\right) \circ \left(
I^{x, y}_{x}
\right)^{-1} = c c' \cdot 
\Theta^{0}_{x \mid y} \circ \Theta^{0}_{y \mid x} \otimes \id_{V_{\kappa_{x, y}}},
\)
from which we deduce the desired equivalence.
\end{proof}

\subsection{Hecke algebra isomorphism} \label{subsec:Heckeisom} 
We keep the notation and assumptions, i.e., Axioms \ref{axiomaboutHNheartandK}, \ref{axiomaboutKM0vsKM}, \ref{axiomaboutK0vsK}, and \ref{axiomextensionoftheinductionofkappa}, from the previous subsection.
In this subsection we are going to show that the Hecke algebras attached to $(K^0_{x_0}, \rho_{x_0}^0)$ and to $(K_{x_0}, \rho_{x_0})$, respectively, are isomorphic, see Theorem \ref{thm:isomorphismtodepthzero}.

In order to use the structure of the Hecke algebras as a semi-direct product of an affine Hecke algebra with a twisted group algebra that we exhibited in Theorem \ref{theoremstructureofhecke}, we will need to also assume the axioms that were used in that theorem. More precisely, from now on, we suppose that the family $\cK^0$ also satisfies Axiom \ref{axiombijectionofdoublecoset}. Starting from after Lemma \ref{lemmaabouttaxiombijectionofdoublecoset} we also assume that the group $\Wzeroheart$ satisfies Axiom~\ref{axiomexistenceofRgrp} for a normal subgroup $\Waffz$
\index{notation-ax}{WrhoMaff0@$\Waffz$}
of $\Wzeroheart$, and that the family $\cK^{0}$ satisfies Axiom~\ref{axiomaboutdimensionofend} with
the groups $K'_{x,s} = \nobreak K^{0}_{x, s x}$ for $s \in S_{\Kzrel}$ and $x \in \cA_{\gen}$ with $\mathfrak{H}_{x, s x} = \{ H_{s} \}$, and where $K^{0}_{x, s x}$ denotes the group in Axiom~\ref{axiomextensionoftheinductionofkappa}.
We will first show that the analogous axioms also hold for $\cK$.

\begin{lemma} 
	\label{lemmaabouttaxiombijectionofdoublecoset}
	The family $\cK$ with the subgroup $\Nheart=\Nzeroheart$ of $N(\rho_{M})_{[x_{0}]_{M}}$ satisfies Axiom~\ref{axiombijectionofdoublecoset}.
\end{lemma}
\begin{proof}
	Let $x \in \cA_{\gen}$. Since the group $\Nzeroheart$ and the family $\cK^{0}$ satisfy Axiom~\ref{axiombijectionofdoublecoset}, we have
	\[
	I_{G^{0}(F)}(\rho^{0}_{x}) = K^{0}_{x} \cdot \Nzeroheart \cdot K^{0}_{x} .
	\]
	Combining this with Axiom~\ref{axiomaboutK0vsK}\eqref{axiomaboutK0vsKaboutintertwiner}, we have
	\begin{equation*}
		I_{G(F)}(\rho_{x}) = K_{x} \cdot I_{G^{0}(F)}(\rho^{0}_{x}) \cdot K_{x} 
		= K_{x} \cdot \Nzeroheart \cdot K_{x}.
		\qedhere
	\end{equation*}
\end{proof}

To show that $\cK$ also satisfies the remaining axioms, recall that we chose the group $\Nheart$ to be $\Nzeroheart \subset G^0(F)$. Hence
we have
\begin{equation}
	\label{NcapKM=NzerocapKmzero}
	\Nheart \cap K_{M} = \Nzeroheart \cap K_{M} = \Nzeroheart \cap K_{M^{0}}
\end{equation}
by  Remark~\ref{remarkaboutintersectionwithdepthzero}.
Thus, we obtain that $\Wheart = \Wzeroheart$.
Moreover, according to Corollary~\ref{corollaryKrel=K0rel}, we have $
W_{\Krel} = W_{\Kzrel}
$.

\begin{lemma}
\label{lemmaaxiomstrongerthanaboutdimensionofendimpliesaxiomaboutdimensionofend}
The family $\cK$ satisfies Axiom~\ref{axiomaboutdimensionofend} with the group $K'_{x,s}=K_{x, s x}$ for $s \in S_{\Krel} = S_{\Kzrel}$ and $x \in \cA_{\gen}$ with $\mathfrak{H}_{x, s x} = \{ H_{s} \}$, where $K_{x, s x}$ denotes the group in Axiom~\ref{axiomextensionoftheinductionofkappa}.
\end{lemma}
\begin{proof}
According to Remark~\ref{remarkinclusionsofKxandKxy}, the group $K_{x, sx}$ is a compact, open subgroup of $G(F)$ containing $K_{x}$.
According to Axiom~\ref{axiomextensionoftheinductionofkappa}\eqref{axiomextensionoftheinductionofkappakrnelofkappaxy}, we have 
\[
\begin{split}
\Nzeroheart \cap K_{x, s x} =
\Nzeroheart \cap K^{0}_{x, s x} \cdot K_{x, sx; 0+}  
= \Nzeroheart \cap G^{0}(F) \cap K^{0}_{x, s x} \cdot K_{x, sx; 0+}\\
= \Nzeroheart \cap K^{0}_{x, s x} \cdot \left(
G^{0}(F) \cap K_{x, sx; 0+}
\right) 
= \Nzeroheart \cap K^{0}_{x, s x}.
\end{split}
\]
Combining this with \eqref{NcapKM=NzerocapKmzero}, we obtain that
\[
\bigl(
\Nzeroheart \cap K_{x, s x} 
\bigr)
/
\bigl(
\Nzeroheart \cap K_{M} 
\bigr) \!
= \!
\bigl(
\Nzeroheart \cap K^{0}_{x, s x} 
\bigr)
/ \bigl(
\Nzeroheart \cap K_{M^{0}} 
\bigr) \!
= \! \{1 , s\},\]
where the last equality follows from $\cK^{0}$ satisfying Axiom~\ref{axiomaboutdimensionofend} for the group $K'_{x,s} = K^{0}_{x, s x}$.
\end{proof}

We fix a subset  $\Coeffplus \subset \Coeffinvnontriv$ as in Choice \ref{choiceofCoeffplus}.
According to our assumptions and Proposition~\ref{propositionchoice:twispossible}, we can choose a family 
\[
\cT^{0} = \left\{
T^{0}_{n} \in \Hom_{K_{M^{0}}}\left(
^n\!\rho_{M^{0}}, \rho_{M^{0}}
\right)
\right\}_{n \in \Nzeroheart}
\]
as in Choice~\ref{choice:tw}.
\begin{proposition}
\label{propositionchoiceofcTrelativetocT0}
Assume Axioms \ref{axiomaboutHNheartandK}, \ref{axiombijectionofdoublecoset}, \ref{axiomexistenceofRgrp}, \ref{axiomaboutdimensionofend},  \ref{axiomaboutKM0vsKM}, \ref{axiomaboutK0vsK}, and \ref{axiomextensionoftheinductionofkappa} for the relevant objects as described above.
Then there exists a unique family 
\[
\cT =
\left\{
T_{n} \in \Hom_{K_{M}}\left(
^n\!\rho_{M}, \rho_{M}
\right)
\right\}_{n \in \Nzeroheart}
\]
that satisfies all the properties in Choice~\ref{choice:tw} and the condition that 
\[
T_{n} = T^{0}_{n} \otimes \widetilde\kappa_{M}(n)
\]
for all $n \in \Nzeroheart$ whose projections to $\Wzeroheart$ are contained in
$\Wzeroz$. 
\end{proposition}
\begin{proof}
According to our assumptions, Lemma~\ref{lemmaabouttaxiombijectionofdoublecoset}, Lemma~\ref{lemmaaxiomstrongerthanaboutdimensionofendimpliesaxiomaboutdimensionofend}, and Proposition~\ref{propositionchoice:twispossible}, we can choose a family $\cT$ to satisfy all the properties in Choice~\ref{choice:tw}. 
Moreover, according to Remark~\ref{remarkaboutchoice:tw}, there exists a unique family $\cT$ satisfying all the properties in Choice~\ref{choice:tw} together with the condition that $T_{n} = T^{0}_{n} \otimes \widetilde\kappa_{M}(n)$ for a set of representatives of lifts $n \in \Nzeroheart$ of $t \in \Wzeroz \smallsetminus \{1\}$.
Since $\widetilde\kappa_{M}$ is an extension of $\kappa_{M}$ and we have $\rho_{M} \simeq \inf\left(
\rho_{M^{0}}
\right) \otimes \kappa_{M}$, we also obtain that $T_{n} = T^{0}_{n} \otimes \widetilde\kappa_{M}(n)$ for all lifts $n \in \Nzeroheart$ of elements of $\Wzeroz$.
\end{proof}

We fix a family $\cT^{0}$ as in Choice~\ref{choice:tw} and let $\cT$ denote the family satisfying the conditions in Proposition~\ref{propositionchoiceofcTrelativetocT0}.
We define non-zero elements 
$
\Phi^{0}_{x, w} \in \End_{G^{0}(F)}\left(
\ind_{K^{0}_{x}}^{G^{0}(F)} (\rho^{0}_{x})
\right)
$
and
$\Phi_{x, w} \in \End_{G(F)} \left(
\ind_{K_{x}}^{G(F)} (\rho_{x})
\right)
$
for $x \in \cA_{\gen}$ and $w \in \Wzeroheart$ as in Definition~\ref{definitionPhixw}, and let $\varphi^{0}_{x, w}$, resp., $\varphi_{x, w}$, denote the element of $\cH(G^{0}(F), \rho^{0}_{x})$, resp., $\cH(G(F), \rho_{x})$, that corresponds to $\Phi^{0}_{x, w}$, resp., $\Phi_{x, w}$, via the isomorphism in \eqref{heckevsend}.
We write $\varphi^{0}_{w} \coloneqq \varphi^{0}_{x_{0}, w}$,\index{notation-ax}{phi0w@$\varphi^0_w$} $\varphi_{w} \coloneqq \varphi_{x_{0}, w}$, $\Phi^{0}_{w} \coloneqq \Phi^{0}_{x_{0}, w}$\index{notation-ax}{Phi0w@$\Phi^0_w$}, and $\Phi_{w} \coloneqq \Phi_{x_{0}, w}$.
\label{phi0w-page}

According to Theorem~\ref{theoremstructureofhecke}, the map
\[
\varphi^{0}_{tw} \mapsto \gpalg_{t} \cdot \mathbb{T}_{w}
\qquad (t \in \Wzeroz , w \in \Waffz)
\]
defines an isomorphism of $\Coeff$-algebras
\[
\cI(\rho^{0}_{x_{0}}) \colon  \cH(G^{0}(F), \rho^{0}_{x_{0}})
\isoarrow
\Coeff[\Wzeroz , \muTzero] \ltimes \cH_\Coeff(\Waffz , q^{0}),
\]
and the map
\[
\varphi_{tw} \mapsto \gpalg_{t} \cdot \mathbb{T}_{w} \, (t \in \Wzeroz , w \in \Waffz )
\]
defines an isomorphism of $\Coeff$-algebras
\[
\cI(\rho_{x_{0}}) \colon \cH(G(F), \rho_{x_{0}})
\isoarrow
\Coeff[\Wzeroz , \muT] \ltimes \cH_\Coeff(\Waffz , q),
\]
where $\muTzero$ and $\muT$ denote the restrictions to $\Wzeroz \times \Wzeroz$ of the $2$-cocycles introduced in Notation~\ref{notationofthetwococycle},
and 
$q^{0}$ and $q$ denote the parameter functions
$s \mapsto q^0_s$ and $s \mapsto q_s$
from $S_{\Kzrel}$ to $\Coeffplus$ such that
the elements $\Phi^{0}_{s}$ and $\Phi_{s}$ satisfy the quadratic relations
\begin{equation} \label{equationtwoquadraticrelations}
\left(
\Phi^{0}_{s}
\right)^{2} = (q^{0}_{s} - 1) \cdot \Phi^{0}_{s} + q^{0}_{s} \cdot \Phi^{0}_{1}
\quad 
\text{ and }
\quad
\left(
\Phi_{s}
\right)^{2} = (q_{s} - 1) \cdot \Phi_{s} + q_{s} \cdot \Phi_{1} .
\end{equation}

Thus, in order to prove that the two Hecke algebras $\cH(G^{0}(F), \rho^{0}_{x_{0}})$ and $\cH(G(F), \rho_{x_{0}})$ are isomorphic, we will prove that $\muTzero = \muT$ and $q^{0} = q$.
\begin{lemma}
\label{lemmamuzero=mu}
We have $\muTzero = \muT$ on $\Wzeroz \times \Wzeroz$.
\end{lemma}
\begin{proof}
Let $v, w \in \Wzeroz$.
We fix lifts $m$ of $v$ and $n$ of $w$ in $\Nzeroheart$.
Since $\widetilde\kappa_{M}$ is a representation, we have 
\(
\widetilde\kappa_{M}(mn) = \widetilde\kappa_{M}(m) \circ \widetilde\kappa_{M}(n).
\)
Hence, our choice of $\cT$ implies that
\[
\begin{split}
\muT (v, w) \cdot T_{mn} &= T_{m} \circ T_{n} 
= \left(
T^{0}_{m} \circ T^{0}_{n}
\right) \otimes \left(
\widetilde\kappa_{M}(m) \circ \widetilde\kappa_{M}(n)
\right) \\
&= \left(
T^{0}_{m} \circ T^{0}_{n}
\right) \otimes \widetilde\kappa_{M}(mn) 
= \muTzero (v, w) \cdot T^{0}_{mn} \otimes \widetilde\kappa_{M}(mn) 
= \muTzero (v, w) \cdot T_{mn}.
\qedhere
\end{split}
\]
\end{proof}
\begin{lemma}
\label{lemmaquadraticrelationforPhisuptoconstant}
We have $q^{0}_{s} = q_{s}$ for all $s \in S_{\Kzrel}$.
\end{lemma} 
\begin{proof}
\addtocounter{equation}{-1}
\begin{subequations}
Let $x \in \cA_{x_{0}}$ such that $\mathfrak{H}_{x, s x} = \{ H_{s} \}$.
Since $\mathfrak{H}_{\Krel; x_0, s x_0} = \{ H_s \}$, by Proposition~\ref{propositionsuperstrongversionofquadraticrelation} and Equation \eqref{equationtwoquadraticrelations} we have 
\begin{equation}
\label{quadraticrelationforPhi0xsandPhixs}
\left(
\Phi^{0}_{x, s}
\right)^{2} = (q^{0}_{s} - 1) \cdot \Phi^{0}_{x, s} + q^{0}_{s} \cdot \Phi^{0}_{x, 1}
\quad 
\text{ and }
\quad
\left(
\Phi_{x, s}
\right)^{2} = (q_{s} - 1) \cdot \Phi_{x, s} + q_{s} \cdot \Phi_{x, 1}.
\end{equation}
Hence, using Proposition~\ref{propositionsuperstrongversionofquadraticrelation}, the lemma follows once we show that  there exists $c \in \Coeff^{\times}$ such that
\begin{align}
\label{anotherquadraticrelationforPhixs}
\left(
c \cdot \Phi_{x, s}
\right)^{2} = (q^{0}_{s} - 1) \cdot \left(
c \cdot  \Phi_{x, s}
\right) + q^{0}_{s} \cdot \Phi_{x, 1}.
\end{align}
Recall that we have an isomorphism
\[
I^{x, s x}_{x}
\colon
\ind_{K_{x}}^{K_{x, s x}} (\rho_{x})
\isoarrow
\inf \bigl(
\ind_{K^{0}_{x}}^{K^{0}_{x, s x}} (\rho^{0}_{x})
\bigr) \otimes \kappa_{x, s x}. 
\]
Since the representation
$\inf \bigl(
\ind_{K^{0}_{x}}^{K^{0}_{x, s x}} (\rho^{0}_{x})
\bigr)$
is trivial on $K_{x, sx; 0+}$, and the restriction of the representation $\kappa_{x, s x}$ to $K_{x, sx; 0+}$ is irreducible,
for any 
$
\Phi \in \End_{K_{x, s x}} \Bigl(
\inf \bigl(
\ind_{K^{0}_{x}}^{K^{0}_{x, s x}} (\rho^{0}_{x})
\bigr) \otimes \kappa_{x, s x}
\Bigr)
$,
there exists a unique
$
\Phi' \in \End_{K^{0}_{x, s x}}\bigl(
\ind_{K^{0}_{x}}^{K^{0}_{x, s x}} (\rho^{0}_{x})
\bigr)
$
such that 
$
\Phi = 
\Phi'
 \otimes \id_{V_{\kappa_{x, s x}}}
$.
Thus, we have an isomorphism of $\Coeff$-algebras 
\[
\eta^{x, s x}_{x} \colon 
\End_{K^{0}_{x, s x}}\bigl(
\ind_{K^{0}_{x}}^{K^{0}_{x, s x}} (\rho^{0}_{x})
\bigr)
\isoarrow
\End_{K_{x, s x}}\bigl(
\ind_{K_{x}}^{K_{x, s x}} (\rho_{x})
\bigr)
\]
defined by
\[
\eta^{x, s x}_{x} \left(
\Phi'
\right)
= (I^{x, s x}_{x})^{-1} \circ \bigl(
\Phi'
 \otimes \id_{V_{\kappa_{x, s x}}}
\bigr)
\circ I^{x, s x}_{x}.
\]

Since we assume that the family $\cK^{0}$ satisfies Axiom~\ref{axiomaboutdimensionofend} with the group $K^{0}_{x, s x}$, we have a lift of $s$ in $\nobreak K^{0}_{x, sx}$, which we also denote by $s$. 
Combining this with Axiom \ref{axiomextensionoftheinductionofkappa}\eqref{axiomextensionoftheinductionofkappaK0xycontainsK0xandy} and Remark~\ref{remarkinclusionsofKxandKxy}, we have $K^{0}_{x} s K^{0}_{x} \subset K^{0}_{x, s x}$ and $K_{x} s K_{x} \subset \nobreak K_{x, sx}$.
Hence, according to Lemma~\ref{lemmavarphisupprtedonksKequivalenttoPhirhoinKsK}
and
Lemma~\ref{supportofphiixw},
we have
\[
\Phi^{0}_{x, s} \in \End_{G^{0}(F)}\bigl(
\ind_{K^{0}_{x}}^{G^{0}(F)} (\rho^{0}_{x})
\bigr)_{K^{0}_{x, s x}}
\simeq 
\End_{K^{0}_{x, s x}}\bigl(
\ind_{K^{0}_{x}}^{K^{0}_{x, s x}} (\rho^{0}_{x})
\bigr)
\]
and
\[
\Phi_{x, s} \in \End_{G(F)}\bigl(
\ind_{K_{x}}^{G(F)} (\rho_{x})
\bigr)_{K_{x, s x}}
\simeq
\End_{K_{x, s x}}\bigl(
\ind_{K_{x}}^{K_{x, s x}} (\rho_{x})
\bigr).
\]
We regard $\Phi^{0}_{x, s}$ and $\Phi_{x, s}$ as elements of the latter spaces and claim that there exists $c \in \Coeff^{\times}$ such that
\[
\eta^{x, s x}_{x} \left(
\Phi^{0}_{x, s}
\right) = c \cdot \Phi_{x, s}.
\]
Since $\eta^{x, s x}_{x}$ is an algebra isomorphism, the first equation of \eqref{quadraticrelationforPhi0xsandPhixs} would then imply Equation \eqref{anotherquadraticrelationforPhixs}  
as desired. 

It remains to prove the claim.
The proof
is essentially the same as the proof of Lemma~\ref{lemmacompativilityofThetadepth0vspositivedepth}.
According to Proposition~\ref{propositionvectorspacedecomposition} and Lemma~\ref{supportofphiixw}, it suffices to show that $\supp\left( (\eta^{x, s x}_{x})^{-1} \left(\Phi_{x, s}\right)\right) \subset K^{0}_{x} s K^{0}_{x}$.
As explained in the proof of Lemma~\ref{lemmacompativilityofThetadepth0vspositivedepth}, we have
$
(I^{x, s x}_{x})^{-1}\left(
V_{\rho^{0}_{x}} \otimes V_{\kappa_{x}}
\right) = V_{\rho_{x}}
$,
where we regard $\kappa_{x}$ as a $K_{x}$-subrepresentation of $\kappa_{x, sx}$ via
$
\kappa_{x} \rightarrow \ind_{K_{x}}^{K^{0}_{x} \cdot K_{x, sx; 0+}} (\kappa_{x}) \restriction_{K_{x}} \simeq \kappa_{x, sx} \restriction_{K_{x}}
$.
Moreover, since $\supp \left(\Phi_{x, s} \right) = K_{x} s K_{x}$, Lemma~\ref{lemmavarphisupprtedonksKequivalenttoPhirhoinKsK} implies that
\begin{equation}
\label{equation:PhicircIinverseofVrho0otimesVkappa}
\left(
\Phi_{x, s} \circ (I^{x, s x}_{x})^{-1} 
\right)
\left(
V_{\rho^{0}_{x}} \otimes V_{\kappa_{x}}
\right) =
\Phi_{x, s} \left(
V_{\rho_{x}}
\right) \subset \ind_{K_{x}}^{K_{x} s K_{x}} \left(
V_{\rho_{x}}
\right).
\end{equation}
Since we have $s \in K^{0}_{x, sx}$ and $K^{0}_{x} \subset K^{0}_{x, sx}$, we obtain from Axiom~\ref{axiomextensionoftheinductionofkappa}\eqref{axiomextensionoftheinductionofkappaJxyvsJxandJy} that the element $s$ and the group $K^{0}_{x}$ normalize the group $K_{x, sx; 0+}$, and hence
\[
  K_x s K_x \subseteq K^{0}_{x} K_{x, sx; 0+} s K^{0}_{x} K_{x, sx; 0+} = K^{0}_{x} s K^{0}_{x} K_{x, sx; 0+}
\]
by Remark \ref{remarkinclusionsofKxandKxy}.
Using the definition of $I^{x, s x}_{x}$, this allows us to deduce
\( 
I^{x, s x}_{x} \left(
\ind_{K_{x}}^{K_{x} s K_{x}} \left(
V_{\rho_{x}}
\right) 
\right)
\subset \ind_{K^{0}_{x}}^{K^{0}_{x} s K^{0}_{x}} \left(
V_{\rho^{0}_{x}}
\right) \otimes V_{\kappa_{x, y}} \). 
Precomposing with \eqref{equation:PhicircIinverseofVrho0otimesVkappa}, we have
$
\left(
I^{x, s x}_{x} \circ \Phi_{x, s} \circ (I^{x, s x}_{x})^{-1} 
\right)
\left(
V_{\rho^{0}_{x}} \otimes V_{\kappa_{x}}
\right) \subset \ind_{K^{0}_{x}}^{K^{0}_{x} s K^{0}_{x}} \left(
V_{\rho^{0}_{x}}
\right) \otimes V_{\kappa_{x, y}}
$.
Thus, we obtain that
$
\left((\eta^{x, s x}_{x})^{-1} \left(\Phi_{x, s}\right)\right)\left(
V_{\rho^{0}_{x}}
\right) \subset \ind_{K^{0}_{x}}^{K^{0}_{x} s K^{0}_{x}} \left(
V_{\rho^{0}_{x}}
\right)
$.
Using Lemma~\ref{lemmavarphisupprtedonksKequivalenttoPhirhoinKsK} we conclude that $\supp \left(
(\eta^{x, s x}_{x})^{-1} \left(\Phi_{x, s}\right)
\right) \subset K^{0}_{x} s K^{0}_{x}$.
\end{subequations}
\end{proof}

\begin{theorem}
\label{thm:isomorphismtodepthzero}
We assume the following:
\begin{enumerate}[(1)]
\item
 The family $\cK^{0}$ satisfies Axioms~\ref{axiomaboutHNheartandK} and \ref{axiombijectionofdoublecoset} with the subgroup $\Nzeroheart$ of $N(\rho_{M^{0}})_{[x_{0}]_{M^{0}}}$.
 \item 
 The family $\cK$ satisfies Axiom~\ref{axiomaboutHNheartandK} with the subgroup $\Nzeroheart$ of $N(\rho_{M})_{[x_{0}]_{M}}$.
 \item
 The group $\Wzeroheart$ satisfies Axiom~\ref{axiomexistenceofRgrp} with a normal subgroup $\Waffz$ of $\Wzeroheart$.
\item
The family $\cK^{0}$ satisfies Axiom~\ref{axiomaboutdimensionofend} with the group $K'_{x,s}=K^{0}_{x, s x}$ for each $s \in S_{\Kzrel}$ and $x \in \cA_{\gen}$ such that $\mathfrak{H}_{x, s x} = \{ H_{s} \}$.
 \end{enumerate}
 We also assume Axioms~\ref{axiomaboutKM0vsKM}, \ref{axiomaboutK0vsK}, and \ref{axiomextensionoftheinductionofkappa}.
 We fix a subset  $\Coeffplus \subset \Coeffinvnontriv$ as in Choice \ref{choiceofCoeffplus},
choose a family $\cT^{0}$ as in Choice~\ref{choice:tw} and let $\cT$ denote the family satisfying the conditions in Proposition~\ref{propositionchoiceofcTrelativetocT0}.
Then 
the map 
\[\Isom \coloneqq
\left(
\cI(\rho_{x_{0}})
\right)^{-1} \circ 
\cI(\rho^{0}_{x_{0}})
\colon
\cH(G^{0}(F), \rho^{0}_{x_{0}})
\longrightarrow
\cH(G(F), \rho_{x_{0}})
\]
is a support-preserving algebra isomorphism.
\end{theorem}
\begin{proof}[Proof of Theorem \ref{thm:isomorphismtodepthzero}]
	Combine Theorem~\ref{theoremstructureofhecke} with Lemma~\ref{lemmamuzero=mu} and Lemma~\ref{lemmaquadraticrelationforPhisuptoconstant}.
\end{proof}
We will provide a more explicit description of $\Isom$ in Theorem \ref{thm:explicitisom} below.
In order to do so, we keep the set-up from Theorem \ref{thm:isomorphismtodepthzero} and describe $\cT$ more explicitly.

\begin{lemma} \label{lemma:epsilon} There exists a unique quadratic character $\epsilon:\Nzeroheart \rightarrow \{ \pm1 \}$ that factors through $\Nzeroheart \twoheadrightarrow \Wzeroheart$ and is trivial on $\Wzeroz$ such that 
	\[
	T_{n} = \epsilon(n) \cdot  T^{0}_{n} \otimes \widetilde\kappa_{M}(n) 
	\quad
	\text{ for every }
	\quad n \in \Nzeroheart .
	\]
\end{lemma}
\begin{proof}
Let $w \in \Wzeroheart$ and choose a lift $n \in \Nzeroheart$ of $w$.
Since $T_{n}$ and $T^{0}_{n} \otimes \widetilde\kappa_{M}(n)$ are non-zero elements
in the one-dimensional space $\Hom_{K_{M}}\left(
^n\!\rho_{M}, \rho_{M}
\right)$, there exists $\epsilon(n) \in \Coeff^{\times}$ such that $T_{n} = \epsilon(n) \cdot T^{0}_{n} \otimes \widetilde\kappa_{M}(n)$.
We note that since $\widetilde\kappa_{M}$ is an extension of $\kappa_{M}$ and we have $\rho_{M} \simeq \inf\left(
\rho_{M^{0}}
\right) \otimes \kappa_{M}$, our choices of $\cT^{0}$ and $\cT$ imply that the number $\epsilon(n)$ does not depend on the choice of lift $n$, but only on $w$.
We write $\epsilon(w) = \epsilon(n)$ for any lift $n \in \Nzeroheart$ of $w \in \Wzeroheart$.
Since we chose the family $\cT$ to satisfy the condition in Proposition~\ref{propositionchoiceofcTrelativetocT0}, we have $\epsilon(t) = 1$ for all $t \in \Wzeroz$.
We will prove that $\epsilon(w) \in \{\pm1\}$ for all $w \in \Waffz$.
First, let $s \in S_{\Kzrel}$.
We fix a lift $n_s$ in $\Nzeroheart$ of $s$ and note that $n_s^{-1}$ is also a lift of $s=s^{-1}$.
Using the definitions of $\muTzero$ and $\muT$, we obtain
\begin{align*}
\muT(s, s) \cdot T_{1} = T_{n_s} \circ T_{n_s^{-1}} 
&= \epsilon(s)^{2} \cdot \left(
T^{0}_{n_s} \circ T^{0}_{n_s^{-1}}
\right) \otimes
\left(
\widetilde\kappa_{M}(n_s) \circ \widetilde\kappa_{M}(n_s^{-1})
\right) \\
&= \epsilon(s)^{2} \muTzero(s, s) \cdot T^{0}_{1} \otimes \widetilde\kappa_{M}(1) 
= \epsilon(s)^{2} \muTzero(s, s) \cdot T_{1}.
\end{align*}
According to Lemma~\ref{lemmaqHisnonzero} and Lemma~\ref{lemmaquadraticrelationforPhisuptoconstant}, we have $\muTzero(s, s) = q^{0}_{s} = q_{s} = \muT(s, s)$.
Thus, we conclude that
$
\epsilon(s)^{2} = 1
$, that is, $\epsilon(s) \in \{\pm 1\}$, as desired.
Next, we consider a general $w \in \Waffz$.
We fix a reduced expression $w = s_{1} s_{2} \cdots s_{\ellsubstitute}$ for $w$.
We fix lifts of $s_{i}$ for $1 \le i \le \ellsubstitute$ in $\Nzeroheart$ and denote them by $n_{i}$.
We write $n = n_{1} n_{2} \cdots n_{\ellsubstitute} \in \Nzeroheart$.
Then according to Proposition~\ref{prop:muTtrivial}, we have $T^{0}_{n} = T^{0}_{n_{1}} \circ T^{0}_{n_{2}} \circ \cdots \circ T^{0}_{n_{\ellsubstitute}}$ and $T_{n} = T_{n_{1}} \circ T_{n_{2}} \circ \cdots \circ T_{n_{\ellsubstitute}}$.
Hence, we obtain that $\epsilon(w) = \prod_{i= 1}^{\ellsubstitute} \epsilon(s_{i}) \in \{\pm 1\}$.
Moreover, since $\Waffz$ is a Coxeter group, the arguments above imply that the map $w \mapsto \epsilon(w)$ defines a group homomorphism $\Waffz \rightarrow \nobreak \{\pm 1\}$.
To prove the lemma, it suffices to show that the character $\epsilon \colon \Waffz \rightarrow \nobreak \{\pm 1\}$ is invariant under the conjugation action by $\Wzeroz$.
Let $t \in \Wzeroz$ and $w \in \Waffz$ with lifts $\widetilde t$ and $n$ in $\Nzeroheart$, respectively.
Then according to Proposition~\ref{prop:muTtrivial}, we have $T^{0}_{\widetilde t n \widetilde t^{-1}} T^{0}_{\widetilde t}= T^0_{\widetilde t n}= T^{0}_{\widetilde t} \circ T^{0}_{n}$, hence $T^{0}_{\widetilde t n \widetilde t^{-1}} = T^{0}_{\widetilde t} \circ T^{0}_{n} \circ (T^{0}_{\widetilde t})^{-1}$, and $T_{\widetilde t n \widetilde t^{-1}} = T_{\widetilde t} \circ T_{n} \circ T_{\widetilde t}^{-1}$.
Thus, the definition of $\epsilon$ implies that $\epsilon(t w t^{-1}) = \epsilon(w)$, as required.
\end{proof}

\begin{remark} \label{remark:epsilon}
	The quadratic character $\epsilon$ from Lemma \ref{lemma:epsilon} extends uniquely to a quadratic character of $\widetilde K_M$ that is trivial on $K_M$. Hence $\epsilon \cdot \widetilde \kappa_M$ is a smooth representation of $\widetilde K_M$ that extends $\kappa_M$. Thus if we had chosen $\epsilon \cdot \widetilde \kappa_M$ as the lift $\widetilde \kappa_M$ of $\kappa_M$ in Axiom \ref{axiomaboutKM0vsKM}, then we would obtain $T_{n} =  T^{0}_{n} \otimes \widetilde\kappa_{M}(n)$ for every $n \in \Nzeroheart$.
\end{remark}

The above allows us to describe the isomorphism of Theorem \ref{thm:isomorphismtodepthzero} more explicitly. Note that, as explained in Remark \ref{remark:epsilon}, if one chooses $\widetilde \kappa_M$ appropriately, then $\epsilon=1$.

\begin{theorem} \label{thm:explicitisom}
	We assume the same set-up as in Theorem \ref{thm:isomorphismtodepthzero}. Then there exists a unique quadratic character $\epsilon:\Nzeroheart  \rightarrow \{ \pm1 \}$ that factors through $\Nzeroheart \twoheadrightarrow \Wzeroheart$ and is trivial on $\Wzeroz$ such that the isomorphism $\Isom$ is given as follows.
	 If $\varphi \in \cH(G^{0}(F), \rho^{0}_{x_{0}})$ is supported on $K_{x_{0}}^0nK_{x_{0}}^0$ for some $n \in \Nzeroheart$, then $\Isom(\varphi)$ is supported on $K_{x_{0}}nK_{x_{0}}$ and satisfies
\[
\Isom(\varphi)(n)=
d_n  \cdot \varphi(n) \otimes \bigl(\epsilon(n)\cdot \widetilde\kappa_M(n)\bigr)
\quad
\text{ with } 
\quad
d_{n} = 
\left(
\frac{
	\abs{K^{0}_{x_{0}}/ \left(
		K^{0}_{n x_{0}} \cap K^{0}_{x_{0}}
		\right)
	}
}{
	\abs{K_{x_{0}}/ \left(
		K_{n x_{0}} \cap K_{x_{0}}
		\right)
	}
}
\right)^{1/2} \in \Coeff^\times .
\] 
\end{theorem}
\begin{proof}
	 If $\varphi \in \cH(G^{0}(F), \rho^{0}_{x_{0}})$ is supported on $K_{x_{0}}^0nK_{x_{0}}^0$ for some $n \in \Nzeroheart$, then $\varphi$ is a scalar multiple of $\varphi^0_{x_0,w}=\varphi^0_w$ where $w$ denotes the image of $n$ in $\Wzeroheart$. By definition of the $\Coeff$-linear isomorphism $\Isom$, we have $\Isom(\varphi^0_w)=\varphi_w$. Hence the claim follows using Lemma \ref{supportofphiixw} and Lemma~\ref{lemma:epsilon}.
\end{proof}

\subsection{Preserving the anti-involutions}
\label{subsection:Preserving the anti-involutions}

In this subsection we keep the notation from the previous subsection
except that within this subsection
we assume that $\Coeff$ admits a nontrivial involution
$c \mapsto \bar c$.
We also suppose that all the assumptions of Theorem~\ref{thm:isomorphismtodepthzero} hold.
We fix a family $\cT^{0}$ as in Choice~\ref{choice:star} and let $\cT$ denote the family satisfying the conditions in Proposition~\ref{propositionchoiceofcTrelativetocT0}.
We say that a representation $(\pi, V)$ of a group $H$ is unitary
if there exists an $H$-invariant,
positive-definite Hermitian form 
$\langle \phantom{x},\phantom{x} \rangle_{\pi}$ on $V$.
\begin{lemma}
\label{lemmaunitaryextensionimplieschoice:star}
Suppose that the representation $\widetilde\kappa_{M}$ of $\widetilde K_{M}$ is unitary.
Then the family $\cT$ satisfies all the properties in Choice~\ref{choice:star}. 
\end{lemma}

\begin{proof}
\addtocounter{equation}{-1}
\begin{subequations}
We will prove that $\varphi_{t}^{*} = \varphi_{t^{-1}}$ for all $t \in \Wzero$.
Since $\varphi_{t}^{*}, \varphi_{t^{-1}} \in \cH(G(F), \rho_{x_0})_{t^{-1}}$, it suffices to show that $\varphi_{t}^{*}(n^{-1}) = \varphi_{t^{-1}}(n^{-1})$ for a lift $n \in \Nzeroheart$ of $t$.
According to Equation \eqref{explicitdescriptionofstar}, it suffices to show that
$
\langle
\varphi_{t^{-1}}(n^{-1}) \cdot v, w
\rangle_{\rho_{x_0}} = \langle
v, \varphi_{t}(n) \cdot w
\rangle_{\rho_{x_0}}
$
for all $v, w \in V_{\rho_{x_0}}$.
We may suppose that $v$ and $w$ are of the form $v = v_{0} \otimes v_{\kappa}$ and $w = w_{0} \otimes w_{\kappa}$ for $v_{0}, w_{0} \in V_{\rho^{0}_{x_0}}$ and $v_{\kappa}, w_{\kappa} \in V_{\kappa_{x_0}}$.
We fix a $K^{0}_{x_0}$-invariant, positive-definite Hermitian form 
$\langle \phantom{x},\phantom{x} \rangle_{\rho^{0}_{x_0}}$ on $V_{\rho^{0}_{x_0}} = V_{\rho_{M^{0}}}$ and a $\widetilde K_{M}$-invariant,
positive-definite Hermitian form 
$\langle \phantom{x},\phantom{x} \rangle_{\widetilde\kappa_{M}}$ on $V_{\widetilde\kappa_{M}} = V_{\kappa_{M}}$.
Then the Hermitian form $\langle \phantom{x},\phantom{x} \rangle_{\rho^{0}_{x_0}} \otimes \langle \phantom{x},\phantom{x} \rangle_{\widetilde\kappa_{M}}$ is a $K_{M}$-invariant form on $V_{\rho_{M}} = V_{\rho_{M^{0}}} \otimes \nobreak V_{\kappa_{M}}$.
Moreover, since $(K_{x_0}, \rho_{x_0})$ is a quasi-$G$-cover of $(K_{M}, \rho_{M})$, the Hermitian form $\langle \phantom{x},\phantom{x} \rangle_{\rho^{0}_{x_0}} \otimes \langle \phantom{x},\phantom{x} \rangle_{\widetilde\kappa_{M}}$ is also a $K_{x_0}$-invariant form on $V_{\rho_{x_0}} = V_{\rho_{M}}$.
Hence, we may suppose that
$
\langle \phantom{x},\phantom{x} \rangle_{\rho_{x_0}} = \langle \phantom{x},\phantom{x} \rangle_{\rho^{0}_{x_0}} \otimes \langle \phantom{x},\phantom{x} \rangle_{\widetilde\kappa_{M}}
$.
According to Theorem~\ref{thm:explicitisom}, we have
\[
\varphi_{t}(n) \cdot w = 
\left(
\frac{
\abs{K^{0}_{x_{0}}/ \left(
K^{0}_{t x_{0}} \cap K^{0}_{x_{0}}
\right)
}
}{
\abs{K_{x_{0}}/ \left(
K_{t x_{0}} \cap K_{x_{0}}
\right)
}
}
\right)^{1/2} \varphi^{0}_{t}(n) \cdot w_{0} \otimes \widetilde\kappa_{M}(n) \cdot w_{\kappa}
\]
and
\[
\varphi_{t^{-1}}(n^{-1}) \cdot v = \left(
\frac{
\abs{K^{0}_{x_{0}}/ \left(
K^{0}_{t^{-1} x_{0}} \cap K^{0}_{x_{0}}
\right)
}
}{
\abs{K_{x_{0}}/ \left(
K_{t^{-1} x_{0}} \cap K_{x_{0}}
\right)
}
}
\right)^{1/2} \varphi^{0}_{t^{-1}}(n^{-1}) \cdot v_{0} \otimes \widetilde\kappa_{M}(n^{-1}) \cdot v_{\kappa}.
\]
Hence, we have
\begin{equation}
\label{decompositionofwphit-1n-1v}
\langle
\varphi_{t^{-1}}(n^{-1}) \cdot v, w
\rangle_{\rho_{x_0}} = 
\left(
\frac{
\abs{K^{0}_{x_{0}}/ \left(
K^{0}_{t^{-1} x_{0}} \cap K^{0}_{x_{0}}
\right)
}
}{
\abs{K_{x_{0}}/ \left(
K_{t^{-1} x_{0}} \cap K_{x_{0}}
\right)
}
}
\right)^{1/2}
\langle
\varphi_{t^{-1}}^{0}(n^{-1}) \cdot v_{0}, w_{0}
\rangle_{\rho^{0}_{x_0}} \cdot 
\langle
\widetilde\kappa_{M}(n^{-1}) \cdot v_{\kappa}, w_{\kappa}
\rangle_{\widetilde\kappa_{M}}
\end{equation}
and
\begin{equation}
\label{decompositionofphitnwv}
\langle
v, \varphi_{t}(n) \cdot w
\rangle_{\rho_{x_0}} = \left(
\frac{
\abs{K^{0}_{x_{0}}/ \left(
K^{0}_{t x_{0}} \cap K^{0}_{x_{0}}
\right)
}
}{
\abs{K_{x_{0}}/ \left(
K_{t x_{0}} \cap K_{x_{0}}
\right)
}
}
\right)^{1/2} 
\langle
v_{0}, \varphi^{0}_{t}(n) \cdot w_{0}
\rangle_{\rho^{0}_{x_0}} \cdot 
\langle
v_{\kappa}, \widetilde\kappa_{M}(n) \cdot w_{\kappa}
\rangle_{\widetilde\kappa_{M}}.
\end{equation}

Since $t \in \Wzero$, we have $d_{\Krel}(x_{0}, t^{-1} x_{0}) = \flength(t) = 0$.
Hence, according to Corollary~\ref{generalizationoflemmaequalityabouttheabsoftwoquotients}, we have
$
\abs{K_{x_{0}}/ \left(
K_{t^{-1} x_{0}} \cap K_{x_{0}}
\right)
}
= \abs{K_{t^{-1} x_{0}}/ \left(
K_{t^{-1} x_{0}} \cap K_{x_{0}}
\right)
}
$.
Combining this with Axiom~\ref{axiomaboutHNheartandK}\eqref{axiomaboutHNheartandKaboutconjugation}, we obtain that
\begin{equation}
\label{equationKtxcapKxvsKt-1xcapKx}
\abs{K_{x_{0}}/ \left(
K_{t^{-1} x_{0}} \cap K_{x_{0}}
\right)
} = 
\abs{K_{t^{-1} x_{0}}/ \left(
K_{t^{-1} x_{0}} \cap K_{x_{0}}
\right)
}
=
\abs{K_{x_{0}}/ \left(
K_{t x_{0}} \cap K_{x_{0}}
\right)
}.
\end{equation}
Similarly, we can prove that
\begin{equation}
\label{equationKtxcapKxvsKt-1xcapKxzerover}
\abs{K^{0}_{x_{0}}/ \left(
K^{0}_{t^{-1} x_{0}} \cap K^{0}_{x_{0}}
\right)
} = \abs{K^{0}_{x_{0}}/ \left(
K^{0}_{t x_{0}} \cap K^{0}_{x_{0}}
\right)
}.
\end{equation}

Since the family $\cT^{0}$ satisfies all the properties in Choice~\ref{choice:star}, we have
\begin{equation}
\label{starpreservingforTzero}
\langle
\varphi_{t^{-1}}^{0}(n^{-1}) \cdot v_{0}, w_{0}
\rangle_{\rho^{0}_{x_0}} 
= 
\langle
(\varphi^{0}_{t})^{*}(n^{-1}) \cdot v_{0}, w_{0}
\rangle_{\rho^{0}_{x_0}} 
=
\langle
v_{0}, \varphi^{0}_{t}(n) \cdot w_{0}
\rangle_{\rho^{0}_{x_0}}.
\end{equation}
Moreover, since the Hermitian form 
$\langle \phantom{x},\phantom{x} \rangle_{\widetilde\kappa_{M}}$ on $V_{\widetilde\kappa_{M}} = V_{\kappa_{M}}$ is $\widetilde K_{M}$-invariant, we have
\begin{equation}
\label{followsfromtildekappaMisunitary}
\langle
\widetilde\kappa_{M}(n^{-1}) \cdot v_{\kappa}, w_{\kappa}
\rangle_{\widetilde\kappa_{M}} = \langle
v_{\kappa}, \widetilde\kappa_{M}(n) \cdot w_{\kappa}
\rangle_{\widetilde\kappa_{M}}.
\end{equation}
Combining \eqref{equationKtxcapKxvsKt-1xcapKx}, \eqref{equationKtxcapKxvsKt-1xcapKxzerover}, \eqref{starpreservingforTzero}, and \eqref{followsfromtildekappaMisunitary} with \eqref{decompositionofwphit-1n-1v} and \eqref{decompositionofphitnwv}, we obtain that $\langle
\varphi_{t^{-1}}(n^{-1}) \cdot v, w
\rangle_{\rho_{x_0}} = \langle
v, \varphi_{t}(n) \cdot w
\rangle_{\rho_{x_0}}$, as desired.
\end{subequations}
\end{proof}

\begin{corollary} \label{cor:starpreservation}
We assume the same set-up as in Theorem~\ref{thm:isomorphismtodepthzero}.
We also suppose that there exists a $\widetilde K_{M}$-invariant,
positive-definite Hermitian form 
$\langle \phantom{x},\phantom{x} \rangle_{\widetilde\kappa_{M}}$ on $V_{\widetilde\kappa_{M}} = V_{\kappa_{M}}$.
Then the isomorphism
\[
\Isom
\colon
\cH(G^{0}(F), \rho_{x_0}^{0})
\longrightarrow
\cH(G(F), \rho_{x_0})
\] 
in Theorem~\ref{thm:isomorphismtodepthzero} preserves the anti-involutions on both sides defined in Section~\ref{Anti-involution of the Hecke algebra}.
\end{corollary}
\begin{proof}
The corollary follows from Proposition~\ref{propstarpreservationabstractheckevsourhecke} and Lemma~\ref{lemmaunitaryextensionimplieschoice:star}.
\end{proof}

\subsection{Application: Equivalence of Bernstein blocks}
\label{subsec:application}

Throughout this subsection, we suppose that $\Coeff = \bC$.
We only do so because the literature on types currently makes this assumption.

Let $\Rep(G(F))$ denote the category of smooth $\bC$-representations
of $G(F)$.
The Bernstein decomposition
(see \cite{MR771671})
expresses this category as a direct product
of full subcategories:
\[
\Rep(G(F))=\prod_{\fs\in \IEC(G)}\Rep^\fs(G(F)).
\]
Here $\IEC(G)$ is the set of \emph{inertial equivalence
classes},
i.e., equivalence classes $[L,\sigma]_G$ of cuspidal
pairs $(L,\sigma)$ in $G$, where $L$ is a Levi subgroup of $G$,
$\sigma$ is an irreducible supercuspidal representation of $L(F)$,
and where the equivalence is given by conjugation by $G(F)$
and twisting by unramified characters of $L(F)$.
If $\fs = [L,\sigma]_G$,
then the block $\Rep^{\fs}(G(F))$
consists of those representations $(\pi,V)$ for which each
irreducible subquotient of $\pi$ appears as a subquotient in a parabolic induction
of $\sigma\otimes\chi$ for some unramified character $\chi$ of $L(F)$.

Let $K$ denote a compact, open subgroup of $G(F)$, and $(\rho, V_{\rho})$
an irreducible smooth $\bC$-representation of $K$.
Let $\Rep_{\rho}(G(F))$ denote the full subcategory of $\Rep(G(F))$
whose objects are the $\bC$-representations $(\pi,V)$ of $G(F)$
generated by the $\rho$-isotypic subspace 
of
$V$.
Following \cite[(2.8) and Theorem 4.3(ii)]{BK-types},
we have a functor 
\[
\mathbf{M}_{\rho} \colon
\Rep_{\rho}(G(F))
\longrightarrow
\mathrm{Mod-}\cH(G(F),\rho)
\quad
\text{ given by }
\quad
\pi \longmapsto \Hom_{G(F)}(\ind_{K}^{G(F)} (\rho), \pi),
\]
where $\mathrm{Mod-}\cH(G(F),\rho)$
denotes the category of
unital right modules over $\cH(G(F),\rho)$.
Here, the right action of $\cH(G(F),\rho)$ on $\Hom_{G(F)}(\ind_{K}^{G(F)} (\rho), \pi)$ is given by precomposition using the isomorphism
$\cH(G(F), \rho)
\isoarrow
\End_{G(F)}\bigl(\ind_{K}^{G(F)} (\rho)\bigr)$
in \eqref{heckevsend}.

For $\fs\in \IEC(G)$,
the pair $(K,\rho)$
is called an \emph{$\fs$-type}
if $\Rep_{\rho}(G(F))=\Rep^\fs(G(F))$.
More generally,
if $\fS$ is a finite subset of $\IEC(G)$,
then 
the pair $(K,\rho)$
is called an \emph{$\fS$-type}
if
 \[\Rep_{\rho}(G(F))=\Rep^\fS(G(F)) \coloneqq \prod_{\fs \in \fS} \Rep^{\fs}(G(F)). \]
In either case, by \cite[(4.3) Theorem (ii)]{BK-types},
the functor $\mathbf{M}_{\rho}$ gives an equivalence
of categories
\begin{equation} \label{eq:types1}
\Rep^\fS(G(F)) = \Rep_{\rho}(G(F))
\isoarrow
\mathrm{Mod-}\cH(G(F),\rho) .
\end{equation}

We write 
\begin{equation} \label{eq:types2}
\Isom^* \colon
\mathrm{Mod-}\cH(G(F),\rho_{x_0})
\isoarrow 
\mathrm{Mod-}\cH(G^0(F),\rho^0_{x_0})
\end{equation}
for the equivalence of categories associated to the isomorphism $\Isom
\colon
\cH(G^0(F),\rho^0_{x_0})
\isoarrow 
\cH(G(F),\rho_{x_0})$ of Theorem \ref{thm:isomorphismtodepthzero}.

\begin{theorem}
\label{thm:equiv-blocks}
Let $\fS \subset \IEC(G)$ and $\fSz \subset \IEC(G^0)$
be finite subsets.
We suppose that the pairs $(K_{x_0},\rho_{x_0})$
and
$(K^{0}_{x_0}, \rho^{0}_{x_0})$
are an $\fS$-type and
an $\fSz$-type, respectively.
Then the functor 
$(\mathbf{M}_{\rho^{0}_{x_{0}}})^{-1} \circ \cI^* \circ \mathbf{M}_{\rho_{x_{0}}}$
gives an equivalence of categories \spacingatend{}
\[
\Rep^\fS(G(F)) \isoarrow \Rep^\fSz(G^0(F)).
\]
\end{theorem}

\begin{proof}
This follows from combining the equivalences in \ref{eq:types1} and \ref{eq:types1}.
\end{proof}

\begin{theorem} 
	\label{thm:plancherel}
	Choose Haar measures $\nu$ and $\nu^0$ on $G(F)$ and $G^0(F)$,
	and let 
	$\hat\nu$ and $\hat\nu^0$ denote the corresponding 
	Plancherel measures on
	the Borel spaces
	$\Irr_t^\fS(G(F))$ and $\Irr_t^\fSz(G^0(F))$ 
	of irreducible tempered representations in 
	$\Rep^\fS(G)$ and $\Rep^\fSz(G^0)$.
	Then our equivalence of categories
	$\Rep^\fS(G) \longrightarrow \Rep^\fSz(G^0)$
	from Theorem \ref{thm:equiv-blocks}
	induces a homeomorphism
	$\cJ\colon 
	\Irr_t^\fSz(G^0(F)) \longrightarrow \Irr_t^\fS(G(F))$,
	such that
	$$
	\hat\nu\circ \cJ
	=
	\dim\kappa_{x_0}
	\frac{\nu^0(K_{x_0}^0)}{\nu(K_{x_0})} \hat\nu^0.
	$$
\end{theorem}

\begin{proof}
	The support-preserving, anti-involution-preserving isomorphism
	of Hecke algebras
	from Theorem \ref{thm:isomorphismtodepthzero}
	gives rise to an isomorphism of Hilbert algebras.
	From \cite[\S5.1]{bushnell-henniart-kutzko:plancherel},
	one has a homeomorphism 
	$\cJ\colon 
	\Irr_t^\fSz(G^0(F)) \longrightarrow \Irr_t^\fS(G(F))$,
	such that
	$
	\frac{\nu(K_{x_0})}{\dim \rho_{x_0}} \hat\nu\circ \cJ
	=
	\frac{\nu^0(K_{x_0}^0)}{\dim \rho_{x_0}^0} \hat\nu^0.
	$
	Our result then follows from the fact that 
	$\dim \rho_{x_0} / \dim \rho_{x_0}^0 = \dim \kappa_{x_0}$.
\end{proof}

\section{Hecke algebras of depth-zero pairs} 
\label{sec:depth-zero}

In this section we will 
show that all of the axioms of
Section~\ref{Structure of a Hecke algebra}
apply to the special case of a pair $(K_{x_0},\rho_{x_0})$
where $K_{x_0}$ is a normal, compact, open subgroup of $G(F)_{[x_0]}$
that contains the parahoric subgroup $G(F)_{x_0,0}$
and the restriction of $\rho_{x_0}$ to $G(F)_{x_0,0}$ contains the inflation
of a cuspidal representation of $G(F)_{x_0,0}/G(F)_{x_0,0+}$.
Thus, Theorem~\ref{theoremstructureofhecke} applies to such pairs.
This means that we obtain an explicit description of the corresponding
Hecke algebras as a semi-direct product of an affine Hecke algebra with a twisted group algebra,
see Theorem \ref{theoremstructureofheckefordepthzero}.
The special case where $K_{x_0}=G(F)_{x_0,0}$ and $\Coeff =\bC$ 
has already been treated by Morris \cite[Theorem 7.12]{Morris}.
The resulting types describe certain finite products of Bernstein blocks.
Our construction also includes the case $K_{x_0} = G(F)_{x_0}$, 
when the resulting types describe single Bernstein blocks.

\subsection{Construction of depth-zero pairs}
\label{subsec:constructionofdepthzero}

We recall the notion of a depth-zero $G$-datum following Kim and Yu (\cite{Kim-Yu}), but adjusted to our more general coefficient field $\Coeff$, from which the pair $(K_{x_0},\rho_{x_0})$ will be constructed.
\begin{definition}[cf.\ {\cite[7.1]{Kim-Yu}}]
\label{definition:depthzeroGdatum}
	A \emph{depth-zero $G$-datum} is a triple
	$\big((G, M), (x_{0}, \iota), (K_{M}, \rho_{M})\big)$ such that
	\begin{enumerate}[(1)]
		\item $M$ is a Levi subgroup of $G$.
		\item $x_{0} \in \cB(M, F)$ is a point whose image under the projection to
		$\cB^{\red}(M, F)$ is a vertex, and
		\(
		\iota \colon \cB(M, F) \hookrightarrow \cB(G, F)
		\)
		is a $0$-generic admissible embedding relative to $x_{0}$ in the sense of \cite[Definition~3.2]{Kim-Yu}, i.e., $M(F)_{x_0,0}/M(F)_{x_0,0+} \simeq G(F)_{\iota(x_0),0}/G(F)_{\iota(x_0),0+}$. We use this embedding to identify $\cB(M, F)$ with its image in $\cB(G, F)$.
		\item $K_{M}$ is a compact, open subgroup of $M(F)_{x_{0}}$ containing $M(F)_{x_{0}, 0}$, and $\rho_{M}$ is an irreducible smooth $\Coeff$-representation of $K_{M}$ such that $\rho_{M}\restriction_{M(F)_{x_{0}, 0}}$
		is the inflation of a cuspidal representation of $\sfM_{x_{0}}(\ff) = M(F)_{x_{0}, 0}/M(F)_{x_{0}, 0+}$.
	\end{enumerate}
\end{definition}
If $\Coeff=\bC$, a depth-zero $G$-datum is used to construct types for all depth-zero Bernstein blocks based on works of Moy and Prasad (\cite{MR1371680}) and Bushnell and Kutzko (\cite{BK-types}),
which is a special case of the construction below.

From now on, we let $\Sigma = \big((G, M), (x_{0}, \iota), (K_{M}, \rho_{M})\big)$ be a depth-zero $G$-datum.
\index{notation-ax}{SZigma@$\Sigma$}
We note that for each $x \in \cA_{x_{0}} \coloneqq x_{0} + \left(
X_{*}(A_{M}) \otimes_{\mathbb{Z}} \bR
\right)$ such that $\iota: \cB(M,F) \rightarrow \cB(G,F)$ is 0-generic relative to $x$, the triple
$\Sigma_{x} = \big((G, M), (x, \iota), (K_{M}, \rho_{M})\big)$
is also a $G$-datum.
\index{notation-ax}{SZigmax@$\Sigma_x$}
For $x \in \cA_{x_{0}}$, we set 
\index{notation-ax}{Kay x@$K_{x}$}
\index{notation-ax}{Kay x +@$K_{x, +}$}
\begin{equation} \label{definiotionofcompactopensubgroupsdepthzerocase}
	K_{x} = K_{M} \cdot G(F)_{x, 0}
	\quad
	\text{ and }
	\quad K_{x, +} = G(F)_{x, 0+} ,
\end{equation}
which are compact, open subgroups of $G(F)$.
We note that if $\iota: \cB(M,F) \hookrightarrow \cB(G,F)$ is 0-generic relative to $x$, then we have
$K_{x} = K_{M} \cdot K_{x, +}$ since $K_{M}$ contains $M_{x_{0},0} = M_{x, 0}$. In this case, by \cite[4.3~Proposition~(b)]{Kim-Yu}, 
the inclusion $K_{M} \subseteq K_{x}$ induces an isomorphism
\[
K_{M} / M(F)_{x, 0+} \isoarrow K_{x}/K_{x, +},
\] 
\index{notation-ax}{rhox@$\rho_{x}$}%
and we define the irreducible smooth representation $\rho_{x}$ of $K_{x} / K_{x, +}$ as the composition of $\rho_{M}$ with the inverse of the isomorphism above. 
We also regard $\rho_{x}$ as an irreducible smooth representation of $K_{x}$ that is trivial on $K_{x, +}$. If $\Coeff=\bC$, the pair $(K_{x}, \rho_{x})$ is a type.

If $N_{G}(M)(F)_{[x_{0}]_{M}}$ normalizes the group $K_{M}$, e.g., if $K_{M} = M(F)_{x_{0}}$ or $M(F)_{x_{0}, 0}$,
then
we will show in Section \ref{subsec:structureofheckefordepthzerotypes} that
the objects \label{depthzerofamilies} $G$, $M$, $x_0$, $K_{M}$, $\rho_M$ and the families $\left\{(K_{x}, K_{x, +}, \rho_{x})\right\}$ (for appropriate $x \in \cA_{x_0}$)
satisfy all the desired axioms of Section~\ref{Structure of a Hecke algebra}
for the choice of $\Nheart = N(\rho_{M})_{[x_{0}]_{M}}$.

\subsection{Affine hyperplanes}\label{subsec:affinehyperplanes-depthzero}
In order to apply
Section~\ref{Structure of a Hecke algebra} to the objects introduced in Section \ref{subsec:constructionofdepthzero}, we also need an appropriate set of affine hyperplanes as in Section~\ref{subsec:hyperplanes}.
We define these hyperplanes as follows.
For a maximal split torus $S$ of $G$, let $\Phi(G, S)$\index{notation-ax}{Phi0GS@$\Phi(G, S)$} 
denote the relative root system of $G$ with respect to $S$ and let $\Phi_{\aff}(G, S)$\index{notation-ax}{PhiaffGS@$\Phi_{\aff}(G, S)$}
denote the (relative) affine root system associated to $(G, S)$ by the work of Bruhat and Tits  \cite{MR327923}.
We fix a maximal split torus $S$ of $M$ such that $x_{0} \in \cA(G, S, F)$.
For $a \in \nobreak \Phi_{\aff}(G, S) \smallsetminus \Phi_{\aff}(M, S)$,
we define the affine hyperplane $H_{a}$
\index{notation-ax}{Ha@$H_{a}$}
in $\cA(G, S, F)$ by \spacingatend{}
\[
H_{a} = 
\left\{
x \in \cA(G, S, F) \mid a(x) = 0
\right\}.
\]
Since $a \not \in \Phi_{\aff}(M, S)$, the intersection
$
\cA_{x_{0}} \cap H_{a} 
$
is an affine hyperplane in $\cA_{x_{0}}$.
We define the locally finite set $\mathfrak{H}_{S}$
of affine hyperplanes in $\cA_{x_{0}}$ by
\[
\mathfrak{H}_{S} = \left\{
\cA_{x_{0}} \cap H_{a} \mid a \in \Phi_{\aff}(G, S) \smallsetminus \Phi_{\aff}(M, S)
\right\}.
\]

\begin{lemma} \label{lemma:independenceofS0} The set of affine functionals
	\(
	\left\{ a\restriction_{\cA_{x_0}} \mid a \in \Phi_{\aff}(G, S) \smallsetminus \Phi_{\aff}(M, S)
	\right\}
	\) on $\cA_{x_0}$ and the set $\mathfrak{H}_{S}$
	do not depend on the choice of a maximal split torus $S$ of $M$.
\end{lemma}
\begin{proof}
	Let $S'$ be another maximal split torus of $M$ such that $x_{0} \in \cA(G, S', F)$.
	Then there exists an element $m \in M(F)_{x_{0}, 0}$ such that 
	$
	m S m^{-1} = S'
	$, and we obtain a bijection between $\Phi_{\aff}(G, S) \smallsetminus \Phi_{\aff}(M, S)$ and $\Phi_{\aff}(G, S') \smallsetminus \Phi_{\aff}(M, S')$ by sending $a \in \Phi_{\aff}(G, S) \smallsetminus \Phi_{\aff}(M, S)$ to the affine root $ma \in \Phi_{\aff}(G, S') \smallsetminus \Phi_{\aff}(M, S')$ defined by
	$
	(ma)(x) = a(m^{-1} x)
	$
	for
	$x \in \cA(G, S', F) = m\cA(G, S, F).
	$
	Since the group $M(F)_{x_{0}, 0}$ acts trivially on $\cA_{x_{0}}$, for every $a \in \Phi_{\aff}(G, S) \smallsetminus \Phi_{\aff}(M, S)$, we have $(ma)(x) = a(x)$ for all $x \in \cA_{x_0}$, i.e., $(ma)\restriction_{\cA_{x_0}}=a\restriction_{\cA_{x_0}}$  and 
	\begin{align*}
		\cA_{x_{0}} \cap H_{a} =
		\left\{
		x \in \cA_{x_{0}} \mid a(x) = 0
		\right\} 
		& = \left\{
		x \in \cA_{x_{0}} \mid (ma)(x) = 0
		\right\} 
		= \cA_{x_{0}} \cap H_{ma}.
	\end{align*}
	Thus, we obtain that $\mathfrak{H}_{S} = \mathfrak{H}_{S'}$.
\end{proof}

\index{notation-ax}{H_@$\mathfrak{H}$}%
Based on the lemma, we can set $\mathfrak{H} = \mathfrak{H}_{S}$, where $S$ is any maximal split torus of $M$ such that $x_{0} \in \cA(G, S, F)$.
Then 
 $\iota: \cB(M^0,F) \rightarrow \cB(G^0,F)$ is 0-generic relative to $x \in \cA_{x_{0}}$ if and only if $x$ is not contained in any affine hyperplane $H \in \mathfrak{H}$, that is, $x \in \cA_{\gen}= \cA_{x_{0}} \smallsetminus \left(
\bigcup_{H \in \mathfrak{H}} H\right)$.
In particular, we have $x_{0} \in \cA_{\gen}$.
\begin{lemma}
	\label{lemmaNzeroheartpreservesH}
	The action of $N_{G}(M)(F)_{[x_{0}]_{M}}$ on $\cA_{x_{0}}$ preserves the set $\mathfrak{H}$.
\end{lemma}
\begin{proof}
	Let $n \in N_{G}(M)(F)_{[x_{0}]_{M}}$.
	We fix a maximal split torus $S$ of $M$ such that $x_{0} \in \cA(G, S, F)$.
	Then the torus $n S n^{-1}$ is also a maximal split torus of $M$, and we have
	\begin{align*}
		x_{0}  \in n x_{0} + \left( X_{*}(A_{M}) \otimes_{\mathbb{Z}} \bR\right)
		\subseteq n \cA(G, S, F)
		= \cA(G, n S n^{-1}, F).
	\end{align*}
	Since the set $\mathfrak{H}_{S}$ does not depend on the choice of such a maximal split torus of $M$, we obtain that
	$
	n \left(
	\mathfrak{H}_{S}
	\right) = \mathfrak{H}_{n S n^{-1}} = \mathfrak{H}_{S}
	$.
	Thus, we conclude that $n\left(
	\mathfrak{H}
	\right) = \mathfrak{H}$.
\end{proof}

\subsection{The structure of Hecke algebras attached to depth-zero pairs}
\label{subsec:structureofheckefordepthzerotypes}
From now on, we suppose that the group $K_{M}$ is normalized by $N_{G}(M)(F)_{[x_{0}]_{M}}$.
For instance, if we choose $K_{M} = M(F)_{x_{0}}$ or $M(F)_{x_{0}, 0}$, then this assumption is satisfied.
We impose this assumption to show that the support of the Hecke algebra attached to $(K_x, \rho_x)$ is given by $K_{x} \cdot N(\rho_{M})_{[x_{0}]_{M}} \cdot K_{x}$, see Proposition~\ref{propproofofaxiombijectionofdoublecoset} below.
If $\Coeff=\bC$, then the case of $K_{M} = M(F)_{x_{0}}$ corresponds to types for single Bernstein blocks, while the case $K_{M} = M(F)_{x_{0}, 0}$ is the one that Morris (\cite{Morris}) studied.

In this subsection, we will prove Theorem~\ref{theoremstructureofheckefordepthzero},
i.e., that the Hecke algebra $\cH(G(F), \rho_{x_{0}})$ is isomorphic to a semi-direct product of an affine Hecke algebra with a twisted group algebra,
by verifying all the required axioms from Section~\ref{Structure of a Hecke algebra}
that allow us to apply Theorem~\ref{theoremstructureofhecke}.
We recall that we have constructed in Section \ref{subsec:constructionofdepthzero} the family
\[
\cK =
\left\{
(K_{x}, K_{x, +}, (\rho_{x}, V_{\rho_{x}}))
\right\}_{x \in \cA_{\gen}} 
\] 
of quasi-$G$-cover-candidates as defined at the beginning of Section~\ref{subsec:familyofcovers}.
We will now prove that the family also satisfies Axioms~\ref{axiomaboutHNheartandK} and \ref{axiombijectionofdoublecoset} for the group $\Nheart \coloneqq N(\rho_{M})_{[x_{0}]_{M}}$.
\begin{lemma}
	\label{proofofaxiomaboutHNheartandK}
	\mbox{}
	\begin{enumerate}[(1)] 
		\item
		For every $x \in \cA_{\gen}$, we have
		\begin{enumerate}
			\item \(
			K_{n x} = n K_{x} n^{-1}
			\)
			and
			\(
			K_{n x, +} = n K_{x, +} n^{-1}
			\) for $n \in N(\rho_{M^{0}})_{[x_{0}]_{M^{0}}}$,
			\item 
			the pair $(K_{x}, \rho_{x})$ is a quasi-$G$-cover of $(K_{M}, \rho_{M})$,
			\item
			$K_{x} = K_{M} \cdot K_{x, +},$
			\item
			$K_{x, +} = \left(
			K_{x, +} \cap U(F)
			\right) \cdot \left(K_{x, +} \cap M(F)\right) \cdot \left(
			K_{x, +} \cap \overline{U}(F)
			\right)$ for all $U \in \cU(M)$.
		\end{enumerate}
		Moreover, the group $K_{x, +} \cap M(F)$ is independent of the point $x \in \cA_{\gen}$.
		\item
		For $x, y, z \in \cA_{\gen}$ such that 
		\(
		d(x, y) + d(y, z) = d(x, z),
		\)
		there exists $U \in \cU(M)$ such that
		\[
		K_{x} \cap U(F) \subseteq K_{y} \cap U(F) \subseteq K_{z} \cap U(F)
		\quad
		\text{ and }
		\quad
		K_{z} \cap \overline{U}(F)  \subseteq K_{y} \cap \overline{U}(F) \subseteq K_{x} \cap \overline{U}(F).
		\]
	\end{enumerate}
	Thus, the family $\cK$ satisfies Axiom~\ref{axiomaboutHNheartandK}
	for the group $\Nheart = N(\rho_{M})_{[x_{0}]_{M}}$.
\end{lemma}
\begin{proof}
	\addtocounter{equation}{-1}
	\begin{subequations}
		
		The first claim follows from the definitions and \cite[4.3~Proposition, Theorem~7.5]{Kim-Yu}.
		We will prove the second claim.
		Let $x, y, z \in \cA_{\gen}$ such that 
		$
		d(x, y) + d(y, z) = d(x, z)
		$.
		Recall that $z - x$ is an element of $X_{*}(A_{M}) \otimes_{\bZ} \bR$.
		We take $U \in \cU(M)$ such that
		$
		\alpha(z - x) \ge 0
		$
		for any non-zero weight $\alpha \in X^{*}(A_{M})$ occurring in the adjoint representation of $A_{M}$ on the Lie algebra of $U$.
		Then the definition of the groups $K_{x}$ and $K_{z}$ implies that
		\begin{align}
			\label{defofmoyprasadchoiceofPU}
			\begin{cases}
				K_{x} \cap U(F) & \subseteq K_{z} \cap U(F), \\
				K_{z} \cap \overline{U}(F) & \subseteq K_{x} \cap \overline{U}(F).
			\end{cases}
		\end{align}
		We will prove that
		\[
		K_{x} \cap U(F) \subseteq K_{y} \cap U(F) \subseteq K_{z} \cap U(F).
		\]
		Suppose that
		$
		K_{x} \cap U(F) \not \subseteq K_{y} \cap U(F)
		$.
		Then the definitions of $K_{x}$ and $K_{y}$ imply that there exists $a \in \Phi_{\aff}(G, S)$ such that the gradient $Da \in \Phi(G, S)$
occurs in the adjoint representation of $S$ on the Lie algebra of $U$, and
		\[
		a(x) > 0
		\qquad \text{and} 
		\qquad
		a(y) < 0.
		\]
		In particular we have 
		$
		\cA_{x_{0}} \cap H_{a} \in \mathfrak{H}_{x, y}
		$.
		Since 
		$
		d(x, y) + d(y, z) = d(x, z)
		$, Lemma~\ref{lemmaaboutdistancefunctionnoindex}
		implies that
		$
		\mathfrak{H}_{x, y}, \mathfrak{H}_{y, z} \subset \mathfrak{H}_{x, z}
		$.
		Hence, we also obtain that
		$
		\cA_{x_{0}} \cap H_{a} \in \mathfrak{H}_{x, z}
		$.
		Since $a(x) > 0$, we have $a(z) < 0$.
		Then the definitions of $K_{x}$ and $K_{z}$ imply that
		$
		K_{x} \cap U_{Da}(F) \not \subseteq K_{z} \cap U_{Da}(F)
		$,
		where $U_{Da}$ denotes the root subgroup corresponding to $Da$.
		However, this contradicts \eqref{defofmoyprasadchoiceofPU}.
		Hence, we obtain that
		$
		K_{x} \cap U(F) \subseteq K_{y} \cap U(F)
		$.
		Similarly, we can prove that
		$
		K_{y} \cap U(F) \subseteq K_{z} \cap U(F)
		$.
		Replacing $x$ with $z$ and $U$ with $\overline{U}$, we also obtain that
		\[
		K_{z} \cap \overline{U}(F) \subseteq K_{y} \cap \overline{U}(F) \subseteq K_{x} \cap \overline{U}(F).
		\]
	\end{subequations}
	This proves the second claim. 
	The first two properties of Axiom~\ref{axiomaboutHNheartandK}
	follow from Lemma~\ref{lemmaNzeroheartpreservesH}, the remaining properties from the first two claims proven above.
\end{proof}

Next, we will prove that the family $\cK$ satisfies
Axiom~\ref{axiombijectionofdoublecoset}.
The following proposition is essentially a translation of \cite[4.15~Theorem]{Morris} into our slightly more general setting.
\begin{proposition}
	\label{propproofofaxiombijectionofdoublecoset}
	We have
	\[
	K_{x} \cdot N(\rho_{M})_{[x_{0}]_{M}} \cdot K_{x} = I_{G(F)}(\rho_{x})
	\]
	for all $x \in \cA_{\gen}$, that is, 
	the family $\cK$ satisfies
	Axiom~\ref{axiombijectionofdoublecoset}
	for $\Nheart=N(\rho_{M})_{[x_{0}]_{M}}$.
\end{proposition}
To prove Proposition~\ref{propproofofaxiombijectionofdoublecoset}, we prepare the following lemma.
For a maximal split torus $S$ of $G$ and $x \in \cA(G, S, F)$, we write\index{notation-ax}{PhiaffxGS@$\Phi_{\aff, x}(G, S)$}
\[
\Phi_{\aff, x}(G, S) = \left\{
a \in \Phi_{\aff}(G, S) \mid a(x) = 0
\right\}.
\]

\begin{lemma}
	\label{Phiaffxminimalforgenericpoints}
	Let $x, y \in \cA_{x_{0}}$. As above, we denote by $S$ a maximal split torus of $M$ such that $x_{0} \in \cA(G, S, F)$. 
	Suppose that $x \in \cA_{\gen}$.
	Then we have
	\[
	\Phi_{\aff, x}(G, S) \subseteq \Phi_{\aff, y}(G, S).
	\]
\end{lemma}
\begin{proof}
	\addtocounter{equation}{-1}
	\begin{subequations}
		Let $b \in \Phi_{\aff, x}(G, S)$. Since $x \not \in H_{a}$ for all $a \in \Phi_{\aff}(G, S) \smallsetminus \Phi_{\aff}(M, S)$, we have $b \in \Phi_{\aff}(M, S)$. Thus, using that $y - x \in X_{*}(A_{M}) \otimes_{\bZ} \bR$, we obtain
		\(
		b(y) = b(x + (y - x)) = b(x) + Db(y - x) = b(x) = 0.
		\)
		Hence
		$
		\Phi_{\aff, x}(G, S) \subseteq \Phi_{\aff, y}(G, S)
		$.
	\end{subequations}
\end{proof}
\begin{proof}[Proof of Proposition~\ref{propproofofaxiombijectionofdoublecoset}]
	Let $x \in \cA_{\gen}$.
	It suffices to prove that
	\(
	K_{x} \cdot N(\rho_{M})_{[x_{0}]_{M}} \cdot K_{x} \supseteq I_{G(F)}(\rho_{x})
	\)
	because the reverse inclusion follows from Corollary~\ref{corollarynormalizercontainedinintertwiner}.
	Hence let $g \in I_{G(F)}(\rho_{x})$.
	We fix a maximal split torus $S$ of $M$ such that $x_{0} \in \cA(G, S, F)$.
	According to \cite[Theorem~7.8.1]{KalethaPrasad}, we have
	$
	G(F) = K_{x} \cdot N_{G}(S)(F) \cdot K_{x}
	$.
	Hence, we may suppose that 
	$
	g \in I_{G(F)}(\rho_{x}) \cap N_{G}(S)(F)
	$.
	Since
	$
	\Hom_{K_{x} \cap ^g\!K_{x}}\left(
	^g\!\rho_{x}, \rho_{x}
	\right) \neq \{0\}
	$,
	and the representation $^g\!\rho_{x}$ is trivial on the group $^gG(F)_{x, 0+}$, we obtain that $\rho_{x}$ has a non-zero $(K_{x} \cap {^gG(F)_{x, 0+}})$-fixed vector.
	In particular, we obtain that the representation $\rho_{x}$ has a non-zero $(G(F)_{x, 0} \cap {^gG(F)_{x, 0+}})$-fixed vector.
We have that
	\begin{align*}
		\left(
		G(F)_{x, 0} \cap {{^gG}(F)_{x, 0+}}
		\right) G(F)_{x, 0+} / G(F)_{x, 0+} = \left(
		G(F)_{x, 0} \cap G(F)_{g x, 0+}
		\right) G(F)_{x, 0+} / G(F)_{x, 0+}
	\end{align*}
	is the group of $\ff$-points of the unipotent radical of a parabolic subgroup of $\sfG_{x}$,
	and the restriction of $\rho_{x}$ to $\sfG_{x}(\ff) = \sfM_{x}(\ff)$ is a cuspidal representation,
so we obtain that
	\[
	\left(
	G(F)_{x, 0} \cap G(F)_{g x, 0+}
	\right) G(F)_{x, 0+} / G(F)_{x, 0+} = \{1\}.
	\]
	In particular, we have
	\[
	\left(
	M(F)_{x, 0} \cap M(F)_{g x, 0+}
	\right) M(F)_{x, 0+} / M(F)_{x, 0+}= \{1\}.
	\]
	Since the point $[x]_{M}$ is a vertex in $\cA^{\red}(M, S, F)$, this equation implies that $[g x]_{M} = [x]_{M}$, that is,
	\[
	g x \in x + \left(
	X_{*}(A_{M}) \otimes_{\bZ} \bR
	\right) = \cA_{x_{0}}.
	\]
	Since $x \in \cA_{\gen}$, Lemma~\ref{Phiaffxminimalforgenericpoints} implies that
	$
	\Phi_{\aff, x}(G, S) \subset \Phi_{\aff, g x}(G, S)
	$.
	Since
	\(
	\abs{\Phi_{\aff, x}(G, S)} = \abs{g \Phi_{\aff, x}(G, S)} = \abs{\Phi_{\aff, g x}(G, S)},
	\)
	we obtain that $\Phi_{\aff, x}(G, S) = \Phi_{\aff, g x}(G, S)$.
	In particular, we have $gx \not \in H_{a}$ for all
	$a \in \Phi_{\aff}(G, S) \smallsetminus \Phi_{\aff}(M, S)$.
	Since $x, g x \not \in H_{a}$ for all
	$a \in \Phi_{\aff}(G, S) \smallsetminus \Phi_{\aff}(M, S)$, and the projection $[g x]_{M} = [x]_{M}$ is a vertex, we have 
	\[
	A_{M} = \Bigl(
	\bigcap_{a \in \Phi_{\aff, x}(G, S)} \ker(Da)
	\Bigr)^{\circ} 
	= \Bigl(
	\bigcap_{a \in \Phi_{\aff, g x}(G, S)} \ker(Da)
	\Bigr)^{\circ} 
	= {^g\!A_{M}}.
	\]
	Thus, we obtain that $g \in N_{G}(M)(F)$.
	Moreover, since $g x \in \cA_{x_{0}}$, Corollary~\ref{corollaryaboutstabilizerofx0primeM}
	implies that $g \in N_{G}(M)(F)_{[x_{0}]_{M}}$.
	As $N_{G}(M)(F)_{[x_{0}]_{M}}$ normalizes $K_{M}$, this also implies that $g$ normalizes the group $K_{M}$.
	Combining this with the assumption $g \in I_{G(F)}(\rho_{x})$ and Lemma~\ref{lemmacoverintertwining}, we obtain that $g \in N_{G(F)}(\rho_{M})$.
	Hence
	\(
	g \in N_{G(F)}(\rho_{M}) \cap N_{G}(M)(F)_{[x_{0}]_{M}}
	=
	N(\rho_{M})_{[x_{0}]_{M}}
	\).
\end{proof}

We recall from Definition~\ref{definitionofWheart} that
\[
\Wheart
=
\Nheart / \left( \Nheart  \cap K_{M} \right) 
=
N(\rho_{M})_{[x_{0}]_{M}} / K_{M}
\]
and from Section~\ref{subsection:indexinggroup}
that
\(
W_{\Krel} = \langle s_{H} \mid H \in \mathfrak{H}_{\Krel} \rangle
\)
with set of simple reflections $S_{\Krel}$, see Notation~\ref{notationsimplereflections}.
To prove Axioms~\ref{axiomexistenceofRgrp} and \ref{axiomaboutdimensionofend} we first introduce some notation. 
\begin{notation}
\label{notationK0xyKxy0+kappaxy}
Let $x, y \in \cA_{\gen}$ with $d(x, y) = 1$.
We denote by $H_{x,y} \in \mathfrak{H}$
the unique hyperplane that satisfies  $\mathfrak{H}_{x, y} = \{ H_{x,y} \}$  and define the compact, open subgroup
$K_{x, y}$\index{notation-ax}{Kay x y@$K_{x, y}$}
of $G(F)$ by $K_{x, y} = K_{h} = K_{M} \cdot G(F)_{h, 0}$,
where $h \in H_{x,y}$ is the unique point for which $h = x + t \cdot (y - x)$ for some $0 < t < 1$.
\end{notation} 
Since $d(x, y) = 1$, the definition of $\mathfrak{H}$ implies that we have $G(F)_{x, 0}, G(F)_{y, 0} \subseteq G(F)_{h, 0}$.
Thus, we have $K_{x}, K_{y} \subseteq K_{x, y}$.

\begin{proposition}
	\label{proofofaxiomexistenceofRgrpandaxiomsHisinKxy}
	\mbox{}
	\begin{enumerate}[(1)]
		\item
		\label{proofofaxiomexistenceofRgrpandaxiomsHisinKxyaxiomexistenceofRgrp}
		There exists a normal subgroup $\Waff$
		of $\Wheart$ such that the action of $\Wheart$ on $\cA_{x_{0}}$ restricts to an isomorphism\spacingatend{}
		\[
		\Waff \isoarrow W_{\Krel},
		\]
		that is, 
		Axiom~\ref{axiomexistenceofRgrp} is satisfied. 
		Moreover, the elements of $\Waff$ can be represented by elements in $G_{\cpt, 0} \cap N(\rho_{M})_{[x_{0}]_{M}}$, where $G_{\cpt, 0}$ denotes the kernel of the Kottwitz homomorphism on $G(F)$.
		\item
		\label{proofofaxiomexistenceofRgrpandaxiomsHisinKxyaxiomaboutdimensionofend}
		For every $s \in S_{\Krel}$ and $x \in \cA_{\gen}$ such that $\mathfrak{H}_{x, s x} = \{ H_{s} \}$, we have
		\[ 
		\left(
		N(\rho_{M})_{[x_{0}]_{M}} \cap K_{x, s x}
		\right) / K_{M}
		= \{1, s \},
		\]
		where $K_{x, s x}$ denotes the group in Notation~\ref{notationK0xyKxy0+kappaxy}.
		Thus, the family $\cK$ satisfies
		Axiom~\ref{axiomaboutdimensionofend}
		with $\Nheart = N(\rho_{M})_{[x_{0}]_{M}}$ and the group $K'_{x,s}=K_{x, s x}$ for each $s \in S_{\Krel}$ and $x \in \cA_{\gen}$ such that $\mathfrak{H}_{x, s x} = \{ H_{s} \}$.
	\end{enumerate}
\end{proposition}
\begin{proof}
	\addtocounter{equation}{-1}
	\begin{subequations}
		First, we will prove \eqref{proofofaxiomexistenceofRgrpandaxiomsHisinKxyaxiomexistenceofRgrp}.
		Note that we have $G_{\cpt, 0} \cap M(F)_{x_{0}} = G_{\cpt, 0} \cap M(F)_{x_{0}, 0}$.
		Hence, according to Lemma~\ref{lemmaaboutaxiomexistenceofRgrp},
		to prove
		\eqref{proofofaxiomexistenceofRgrpandaxiomsHisinKxyaxiomexistenceofRgrp},
		it suffices to show that for all $H \in \mathfrak{H}_{\Krel}$,
		there exists an element 
		\[
		s'_{H} \in 
		\left(
		G_{\cpt, 0} \cap N(\rho_{M})_{[x_{0}]_{M}}
		\right) / \left(
		G_{\cpt, 0} \cap K_{M}
		\right)
		\]
		such that the action of $s'_{H}$ on $\cA_{x_{0}}$ agrees with the orthogonal reflection $s_{H}$.
		Let $H \in \mathfrak{H}_{\Krel}$, and $x, y \in \cA_{\gen}$ such that $\mathfrak{H}_{x,y} = \{H\}$ and 
		\begin{align}
			\label{equationtoprovethedimensiongreaterthanone}
			\Theta_{x \mid y} \circ \Theta_{y \mid x} \not \in \Coeff \cdot \id_{\ind_{K_{x}}^{G(F)}(\rho_{x})}.
		\end{align}
		Since $K_{x}, K_{y} \subset K_{x, y}$, the definitions of $\Theta_{x \mid y}$ and $\Theta_{y \mid x}$ imply that
		\[
		\Theta_{x \mid y} \circ \Theta_{y \mid x} \in 
		\End_{G(F)} \left(
		\ind_{K_{x}}^{G(F)}(\rho_{x})
		\right)_{K_{x, y}}
		\simeq 
		\End_{K_{x, y}} \left(
		\ind_{K_{x}}^{K_{x, y}}(\rho_{x})
		\right),
		\]
		where the isomorphism follows from Lemma \ref{lemmarestrictiontoasubspaceofthecompactinductiongeneralver}.
		Hence, \eqref{equationtoprovethedimensiongreaterthanone}
		implies that the dimension of the $\Coeff$-vector space
		\[
		\End_{K_{x, y}} \left(
		\ind_{K_{x}}^{K_{x, y}}(\rho_{x})
		\right) \simeq \cH(K_{x, y}, \rho_{x})
		\]
		is greater than one.
		According to Proposition~\ref{propositionvectorspacedecomposition},
		a basis of the space $\cH(K_{x, y}, \rho_{x})$ is indexed by the group  
		\[
		\left(
		N(\rho_{M})_{[x_{0}]_{M}} \cap K_{x, y}
		\right) / K_{M}
		\subset \Wheart.
		\]
		In particular, we can take a non-trivial element 
		$
		s'_{H} \in 
		\left(
		N(\rho_{M})_{[x_{0}]_{M}} \cap K_{x, y}
		\right) / K_{M}
		$.
		Recall that $K_{x, y} = K_{M} \cdot G(F)_{h, 0}$, where 
		$h \in H$ is the unique point such that
		$
		h = x + t \cdot (y - x)
		$
		for some $0 < t < 1$. Hence,
		we have
		\begin{align*}
			s'_{H}  \in 
			\left(
			N(\rho_{M})_{[x_{0}]_{M}} \cap K_{x, y}
			\right) / K_{M} & \simeq \left(
			G_{h, 0} \cap N(\rho_{M})_{[x_{0}]_{M}}
			\right) / \left(
			G_{h, 0} \cap K_{M}
			\right) \\
			& \subset \left(
			G_{\cpt, 0} \cap N(\rho_{M})_{[x_{0}]_{M}}
			\right) / \left(
			G_{\cpt, 0} \cap K_{M}
			\right),
		\end{align*}
		where the isomorphism is given by the inclusion of $G_{h, 0}$ into $K_{x,y}$.
		We will prove that the action of the element $s'_{H}$ on the space $\cA_{x_{0}}$ agrees with the orthogonal reflection with respect to the affine hyperplane $H$.
		It suffices to show the following three properties.
		\begin{enumerate}[(i)]
			\item
			The gradient of the action of $s'_{H}$ on $\cA_{x_{0}}$ preserves the inner product $(\phantom{x},\phantom{y})_{M}$ on $X_{*}(A_{M}) \otimes_{\bZ} \bR$.
			\item 
			The action of $s'_{H}$ on $\cA_{x_{0}}$ is nontrivial.
			\item
			For any $z \in H$, we have $s'_{H}(z) = z$.
		\end{enumerate}
		The first property follows from the fact that $(\phantom{x},\phantom{x})_{M}$ is preserved by $N_{G}(M)(F)$.
		The second property follows from the facts that $s'_{H} \neq 1$
		and that the group
		\(
		\bigl(
		G_{\cpt, 0} \cap N(\rho_{M})_{[x_{0}]_{M}}
		\bigr) / \bigl(
		G_{\cpt, 0} \cap K_{M}
		\bigr)
		\)
		acts faithfully on $\cA_{x_{0}}$ by Remark \ref{remarkaboutfaithful}
		combined with $G_{\cpt, 0} \cap M(F)_{x_0}=G_{\cpt, 0} \cap K_{M}$.
		
		We will prove the third property.
		It suffices to show that there exists a non-empty open subset $U$ of $H$ such that $s'_{H}(z) = z$ for all $z \in U$.
		Since $\mathfrak{H}_{x,y} = \{H\}$, the definitions of $\mathfrak{H}$ and $h$ and the parahoric subgroups imply that there exists an open ball $U$ in $H$ with center $h$ such that $G_{h, 0} = G_{z, 0}$ for all $z \in U$.
		Hence, we have 
		$
		s'_{H} \in 
		\left(
		G_{z, 0} \cap N(\rho_{M})_{[x_{0}]_{M}}
		\right) / \left(
		G_{z, 0} \cap K_{M}
		\right) 
		$
		for all $z \in U$.
		In particular, $s'_{H}(z) = z$ for all $z \in U$.
		
		Next, we will prove \eqref{proofofaxiomexistenceofRgrpandaxiomsHisinKxyaxiomaboutdimensionofend}.
		Let $s \in S_{\Krel}$ and $x \in \cA_{\gen}$ such that $\mathfrak{H}_{x, s x} = \{ H_{s} \}$.
		By Part \eqref{proofofaxiomexistenceofRgrpandaxiomsHisinKxyaxiomexistenceofRgrp} the element $s$ can be represented by an element in $G_{\cpt, 0} \cap N(\rho_{M})_{[x_{0}]_{M}}$, which we also denote by $s$. Since $s$ fixes $H_{s}$, we have $s \in G_{\cpt, 0} \cap N(\rho_{M})_{[x_{0}]_{M}} \cap G_{h'}$ for any $h' \in H_{s}$. Hence $s \in K_{x,sx}$. Thus 
		\( 
		\left(
		N(\rho_{M})_{[x_{0}]_{M}} \cap K_{x, s x}
		\right) / K_{M}
		\supset \{1, s \}
		\).
		If $s'_H \in \left(
		N(\rho_{M})_{[x_{0}]_{M}} \cap K_{x, s x}
		\right) / K_{M}$ is nontrivial, then the same arguments as in the proof of Part \eqref{proofofaxiomexistenceofRgrpandaxiomsHisinKxyaxiomexistenceofRgrp} show that that $s'_H=s$.
	\end{subequations}
\end{proof}

Now we have shown that all the axioms of
Section \ref{Structure of a Hecke algebra} 
are satisfied in the setting of the present section.
Thus, we obtain the following result.

\begin{theorem}
	\label{theoremstructureofheckefordepthzero}
	We have an isomorphism of $\Coeff$-algebras
	\[
	\cH(G(F), \rho_{x_{0}}) \simeq \Coeff[\Wzero, \muT] \ltimes \cH_\Coeff(\Waff, q),
	\] 
	where $\muT$ denotes the restriction to $\Wzero \times \Wzero$ of the $2$-cocycle introduced in Notation~\ref{notationofthetwococycle}
	for a choice of a family $\cT$ satisfying the properties of Choice~\ref{choice:tw}
	and $q$ denotes the parameter function $s \mapsto q_{s}$ appearing in Choice~\ref{choice:tw}\eqref{conditionofthechoicequadraticrelations}.
	If $\Coeff$ admits a nontrivial involution,
then we can choose $\cT$ as in Choice~\ref{choice:star},
and the above isomorphisms can be chosen to preserve the anti-involutions
on each algebra defined in Section~\ref{Anti-involution of the Hecke algebra}.
\end{theorem}
\begin{proof}
	The statement follows from Theorem \ref{theoremstructureofhecke} and Proposition \ref{propstarpreservationabstractheckevsourhecke}, whose assumptions are satisfied by 
	Lemma~\ref{proofofaxiomaboutHNheartandK}, 
	Proposition~\ref{propproofofaxiombijectionofdoublecoset}, 
	and Proposition~\ref{proofofaxiomexistenceofRgrpandaxiomsHisinKxy}. 
\end{proof}
In the case where $K_{x_0}$ is a parahoric subgroup of $G(F)$, this result was proven by Morris (\cite[7.12~Theorem]{Morris}).


\section*{List of axioms}
\addcontentsline{toc}{section}{List of axioms} \label{page:listofaxioms}
	Axiom \ref{axiomaboutHNheartandK}, p. \pageref{axiomaboutHNheartandK}
	
	Axiom \ref{axiombijectionofdoublecoset}, p. \pageref{axiombijectionofdoublecoset}
	
	Axiom \ref{axiomexistenceofRgrp}, p. \pageref{axiomexistenceofRgrp}
	
	Axiom \ref{axiomaboutdimensionofend}, p. \pageref{axiomaboutdimensionofend}
	
	Axiom \ref{axiomaboutKM0vsKM}, p. \pageref{axiomaboutKM0vsKM}
	
	Axiom \ref{axiomaboutK0vsK}, p. \pageref{axiomaboutK0vsK}
	
	Axiom \ref{axiomextensionoftheinductionofkappa}, p. \pageref{axiomextensionoftheinductionofkappa}

\printindex{notation-ax}{Selected notation} \label{page:listofnotation}


\bibliographystyle{halpha-abbrv}
\addcontentsline{toc}{section}{References}
\bibliography{ourbib}

\end{document}